\documentclass[11pt]{article}

\usepackage[T1]{fontenc}
\usepackage{lmodern}
\usepackage{amsfonts}
\usepackage{amssymb}
\usepackage{amsthm}
\usepackage{amsmath}
\usepackage{bm}
\usepackage{graphicx}
\usepackage{etoolbox,xstring,mfirstuc,textcase}
\usepackage{empheq}
\usepackage{indentfirst}
\usepackage{cite} 
\usepackage{mathrsfs}
\usepackage{graphics}
\usepackage{xcolor}  
\usepackage{cases}
\usepackage{changes} 
\usepackage{lineno} 
\usepackage{soul} 
\sethlcolor{yellow} 
\usepackage[colorlinks=true,linkcolor=black,anchorcolor=blue,citecolor=blue]{hyperref} 
\newtheorem{theorem}{\textbf{Theorem}}[section]
\newtheorem{lemma}{\textbf{Lemma}}[section]
\newtheorem{proposition}{\textbf{Proposition}}[section]
\newtheorem{corollary}{\textbf{Corollary}}[section]
\newtheorem{remark}{\textbf{Remark}}[section]
\newtheorem{definition}{\textbf{Definition}}[section]
\newtheorem{assumption}{\textbf{Assumption}}[section]

\allowdisplaybreaks[4]

\def\be{\begin{equation}}
\def\ee{\end{equation}}
\def\bea{\begin{eqnarray}}
\def\eea{\end{eqnarray}}
\def\bt{\begin{theorem}}
\def\et{\end{theorem}}
\def\bl{\begin{lemma}}
\def\el{\end{lemma}}
\def\bnum{\begin{numcases}{}}
\def\enum{\end{numcases}}
\def\br{\begin{remark}}
\def\er{\end{remark}}
\def\bp{\begin{proposition}}
\def\ep{\end{proposition}}
\def\bc{\begin{corollary}}
\def\ec{\end{corollary}}
\def\bd{\begin{definition}}
\def\ed{\end{definition}}
\def\ba{\begin{assumption}}
\def\ea{\end{assumption}}

\newcommand{\NormNormal}[1]{\| #1 \|}

\newcommand{\AbsNormal}[1]{| #1 |}

\def\r{\mathbb{R}}

\def\sp{\bm{u}}

\newcommand{\refform}[1]{\hyperref[#1]{\ref{#1}}}

\def\ul{_{\mathrm{uloc}}}

\def\define {\stackrel{\mathrm{def}}{=}}

\begin{document}
	

\title{Well-posedness of Strong Solutions for a Thermodynamically Consistent Diffuse-interface Model for Incompressible \\ Two-phase Flows with a Soluble Surfactant}

\author{
Maurizio Grasselli,\!\!
\footnote{Dipartimento di Matematica,
Politecnico di Milano, 20133 Milano, Italy.
Email: \texttt{maurizio.grasselli@polimi.it}}
\quad
Bohan Ouyang,\!\!
\footnote{School of Mathematical Sciences,
Fudan University, Handan Road 220, 200433 Shanghai, P. R. China.
Email: \texttt{22110180033@m.fudan.edu.cn}}
\quad
Hao Wu\!\!
\footnote{Corresponding author. School of Mathematical Sciences, Fudan University, 200433 Shanghai, P. R. China.
Email: \texttt{haowufd@fudan.edu.cn}}
}

\date{\today}

\maketitle

\begin{abstract}
\noindent We consider a thermodynamically consistent diffuse-interface model for binary mixtures of incompressible viscous Newtonian fluids with a soluble surfactant.
This hydrodynamic system generalizes the Abels--Garcke--Gr\"un model for two-phase flows with unmatched densities by incorporating an additional Cahn--Hilliard equation for the surfactant concentration. We study the strong well-posedness of the initial-boundary value problem with physically relevant singular potentials, subject to a no-slip boundary condition for the fluid velocity and homogeneous Neumann boundary conditions for the phase-field variables and the corresponding chemical potentials.
First, we establish the existence of local-in-time strong solutions in three dimensions and global-in-time strong solutions in two dimensions.
In the two-dimensional case, we further show that, under natural growth assumptions on the singular potentials, strong solutions remain strictly separated from the pure phases, thereby guaranteeing uniqueness.
In the three-dimensional case, we establish local-in-time uniqueness of strong solutions, provided that the initial data are strictly separated from the pure phases. Our results extend previous work on strong solutions to the Abels--Garcke--Gr\"un model from the special case of constant mobility to the more general scenario with concentration-dependent mobilities and interactions with a surfactant.
\end{abstract}

\section{Introduction} \label{intro}
\setcounter{equation}{0}

Surfactants have wide applications in both nature and industry. As surface-active substances, surfactants typically adsorb from bulk fluids to free interfaces, where they reduce interfacial tension and influence the dynamics of fluid mixtures \cite{MS2020}.
Incorporating surfactant effects into multiphase flows leads to complex interfacial dynamics, which has attracted considerable attention in modeling, analysis, and numerical simulations. Phase-field models are of particular interest because of their advantage in handling topological changes of free interfaces; see, e.g.,  \cite{DG2019,EDAT2013,GLS2014,KK1997,La1992,LK2012,TSLV2011,vander2006,ZKYWYS2019} and the references therein.
Within this framework, an order parameter (phase-field variable) is introduced to model immiscible fluid mixtures, in which the two components are separated by a thin and smooth transition layer of width $O(\varepsilon)$, i.e., the diffuse interface. The free interface can then be effectively captured by the zero level set of the phase-field variable, which satisfies a suitable evolution equation in Eulerian coordinates.
In this work, our aim is to study the initial-boundary value problem for a thermodynamically consistent diffuse-interface model that describes a two-phase flow of two viscous incompressible Newtonian fluids with a soluble surfactant in the general scenario of unmatched densities and variable mobilities, whose derivation can be found in \cite{EDAT2013,ZKYWYS2019}.

\subsection{Model description}
Let $\Omega \subset \mathbb{R}^d$ ($d\in\{2,3\}$) be a bounded domain with a smooth boundary $\partial \Omega$, and let $T>0$ be the final time. We consider the following coupled system (cf. \cite{OGW2026}):
\begin{numcases}{}
    \partial_{t} \big(\rho (\phi) \bm{u}\big) + \mathrm{div}\,\big( \bm{u} \otimes (\rho (\phi) \bm{u} + \mathbf{J})  \big)
    - \mathrm{div}\, \big(\nu (\phi) D\sp\big)
    + \nabla\pi = \mu_\phi \nabla \phi + \mu_\psi \nabla \psi,
    \label{eq:nsch1} \\
    \mathrm{div}\, \bm{u} =0,
    \label{eq:nsch2}  \\
    \partial_{t} \phi + \bm{u} \cdot \nabla \phi + \sigma_1 \big( \overline{\phi} -c \big) = \mathrm{div}\, \big(m_\phi (\phi) \nabla \mu_\phi\big),
    \label{eq:nsch3} \\
    \mu_\phi = -\Delta \phi + \sigma_2 \mathcal{N} \big( \phi - \overline{\phi} \big) + F_\phi'(\phi) + \partial_\phi G(\phi,\psi),
    \label{eq:nsch4} \\
    \partial_{t} \psi + \bm{u} \cdot \nabla \psi = \mathrm{div}\, \big(m_\psi (\psi) \nabla \mu_\psi\big),  \label{eq:nsch5} \\
    \mu_\psi = -\beta \Delta \psi + F_\psi'(\psi) + \partial_\psi G(\phi,\psi), \label{eq:nsch6}
\end{numcases}
in $\Omega \times (0,T)$, subject to the boundary and initial conditions:
\begin{numcases}{}
    \bm{u} = \mathbf{0}, \ \partial_{\mathbf{n}} \phi = \partial_{\mathbf{n}} \mu_\phi = \partial_{\mathbf{n}} \psi = \partial_{\mathbf{n}} \mu_\psi = 0, \label{eq:nschb} & on $\partial \Omega \times (0,T)$,\\
    \bm{u}|_{t=0} = \bm{u}_0(x), \ \phi|_{t=0} = \phi_0(x), \ \psi|_{t=0} = \psi_0(x), \label{eq:nschi} & in $\Omega$.
\end{numcases}
Here, $\mathbf{n}$ is the unit outer normal vector of $\partial\Omega$ and $\partial_\mathbf{n}$ is the outward normal derivative on the boundary.

The hydrodynamic system \eqref{eq:nsch1}--\eqref{eq:nsch6} is a natural generalization of the thermodynamically consistent diffuse-interface model for incompressible two-phase flows with unmatched densities proposed in Abels, Garcke and Gr\"{u}n \cite{AGG2012} (AGG model for short). In particular, it incorporates interactions with a soluble surfactant and possible mass transfer/reaction. The state variables of the system include the volume-averaged velocity $\bm{u}:\Omega \times [0,T) \to \mathbb{R}^{d}$, the pressure of the fluid mixture $\pi:\Omega \times [0,T) \to \mathbb{R}$, the phase-field variable $\phi:\Omega\times[0,T)\to [-1,1]$ that denotes the difference in the volume fraction between two fluid components and the phase-field variable $\psi:\Omega\times[0,T)\to [0,1]$ that denotes the volume fraction of the surfactant.
In the momentum equation \eqref{eq:nsch1}, $\rho(\phi)$ denotes the density of the fluid mixture given by
\[
\rho( \phi ) = \frac{\rho_1 - \rho_2}{2} \phi + \frac{\rho_1 + \rho_2}{2},
\]
where the positive constants $\rho_1$ and $\rho_2$ are the specific mass densities of the unmixed fluid components. Furthermore, $\nu ( \phi )$ is the viscosity coefficient, typically taking the form
\[
\nu( \phi ) = \frac{\nu_1 - \nu_2}{2} \phi + \frac{\nu_1 + \nu_2}{2},
\]
where $\nu_1$ and $\nu_2$ are viscosity coefficients of the unmixed fluid components. The matrix-valued function
$D \bm{u} = \frac{1}{2} \big( \nabla \bm{u} + (\nabla \bm{u})^T \big)$ represents the symmetrized gradient of $\bm{u}$, and the relative (mass) flux
$$
\mathbf{J} = - \frac{\rho_1 - \rho_2}{2} m_\phi(\phi) \nabla \mu_\phi
$$
accounts for diffusion driven by density contrast.
In \eqref{eq:nsch3}--\eqref{eq:nsch6}, the strictly positive continuous functions $m_\phi:[-1,1] \to \mathbb{R}_+$ and $m_\psi:[0,1] \to \mathbb{R}_+$ are mobility coefficients for the binary fluids and the surfactant, respectively.
Besides, $\sigma_1 \in \mathbb{R}_+$, $\sigma_2 \in \mathbb{R}$, $c \in (-1,1)$, and $\beta \in \mathbb{R}_+$ are given constants. The operator $\mathcal{N}$ denotes the inverse of the minus Laplacian with a homogeneous Neumann boundary condition (see Section \ref{mr}) and $\overline{f}$ denotes the spatial average of a function $f$ over $\Omega$, i.e., $\overline{f} = | \Omega |^{-1} \int_{\Omega} f \, \mathrm{d}x$.
The chemical potentials associated with the phase-field variable and the surfactant concentration are given by $\mu_\phi:\Omega \times [0,T) \to \mathbb{R}$ and $\mu_\psi:\Omega \times [0,T) \to \mathbb{R}$, respectively.
They are defined as the (partial) variational derivatives of the free energy functional
\[
E_{\mathrm{free}} \left(\phi,\psi\right) = E_\phi(\phi) + E_\psi(\psi) + \int_{\Omega}G\left( \phi,\psi \right) \, \mathrm{d}x.
\]
The first two terms in $E_{\mathrm{free}}$ are Ginzburg--Landau type free energies associated with $\phi$ and $\psi$, respectively.
Setting some constants equal to unity for simplicity (e.g., the interfacial thickness), we take
 \begin{align}
 \label{enphi}
    & E_\phi( \phi ) =  \int_{\Omega} \left(\frac{1}{2}\vert\nabla \phi\vert^2 + F_\phi(\phi) + \frac{\sigma_2}{2} \vert \nabla \mathcal{N}\big( \phi - \overline{\phi} \big)\vert^2\right) \, \mathrm{d}x,
    \\
 \label{enpsi}
    & E_\psi( \psi ) =  \int_{\Omega} \left(\frac{\beta}{2}\vert\nabla \psi\vert^2 + F_\psi(\psi)\right) \, \mathrm{d}x.
\end{align}
The functions $F_\phi$ and $F_\psi$ are the bulk energy densities corresponding to the binary fluid mixture and the surfactant, respectively. Throughout this paper, we assume that they are singular, with the Flory--Huggins potential as a prototypical example. These singular potentials force the phase-field variables to take values in the corresponding physical domains, that is, $\phi\in [-1,1]$ and $\psi\in [0,1]$. When $\sigma_2\neq 0$, $E_\phi$ yields an Ohta--Kawasaki-type functional that accounts for possible nonlocal interactions between the two fluid components (see \cite{NO1995}). The bivariate function $G$ denotes the density of the interaction energy. A typical form is given by
\begin{equation*}
    G(\phi,\psi) = \frac{\gamma_1}{2} \psi \phi^2 - \frac{\gamma_2}{4} \psi ( \phi^2 -1 )^2,
\end{equation*}
where $\gamma_1, \gamma_2 \in \mathbb{R}_+$ are given constants (see \cite{EDAT2013}). The first term in $G$ penalizes the free surfactant in the bulk phases, while the second term contributes to the high concentration of surfactant at the interface. For the total energy of the system \eqref{eq:nsch1}--\eqref{eq:nschi}, which is defined as
\begin{equation*}
    E_{\mathrm{tot}} (\bm{u},\phi,\psi) =  \int_{\Omega} \frac{\rho(\phi)}{2} | \bm{u} |^2 \, \mathrm{d}x
    + E_{\mathrm{free}}(\phi,\psi),
\end{equation*}
we can formally derive the following energy identity
\begin{align}
    & \frac{\mathrm{d}}{\mathrm{d}t} E_{\mathrm{tot}} (\bm{u},\phi,\psi) + \int_{\Omega}  \big(  \nu(\phi) | D \bm{u} |^2 + m_\phi(\phi) | \nabla \mu_\phi |^2 + m_\psi(\psi) | \nabla \mu_\psi |^2  \big)  \, \mathrm{d}x \nonumber \\
    & \quad + \sigma_1 \big( \overline{\phi} -c \big) \Big( \int_\Omega \mu_\phi\, \mathrm{d}x - \frac{\rho_1 - \rho_2}{4} \| \bm{u} \|_{L^2(\Omega)}^2  \Big) = 0, \quad \forall\, t\in (0,T). \label{basic_energy_identity}
\end{align}
In addition, a straightforward calculation shows that any (sufficiently smooth) solution to problem \eqref{eq:nsch1}--\eqref{eq:nschi} satisfies the following mass balance relations
\begin{align}
& \frac{\mathrm{d}}{\mathrm{d}t} \overline{\phi} +  \sigma_1 \big(\overline{\phi} -c\big)=0,
\quad \frac{\mathrm{d}}{\mathrm{d}t}  \overline{\psi}=0,
\quad \forall\, t\in (0,T).
\label{mass-relation}
\end{align}
For more detailed background information on the model \eqref{eq:nsch1}--\eqref{eq:nschi}, one may refer to \cite{OGW2026,EDAT2013,ZKYWYS2019} and the references therein.

\subsection{Related literature}
The mathematical analysis of the coupled system \eqref{eq:nsch1}--\eqref{eq:nsch6} remains largely open. In this work, we consider the standard boundary conditions in \eqref{eq:nschb}.
This setting maintains the mass and energy balance relations \eqref{basic_energy_identity}--\eqref{mass-relation}, which serve as a starting point for the theoretical analysis of the initial boundary value problem \eqref{eq:nsch1}--\eqref{eq:nschi}. Recently, in \cite{OGW2026}, we established the existence of global weak solutions to problem \eqref{eq:nsch1}--\eqref{eq:nschi} in both cases of non-degenerate mobilities and degenerate mobilities, under rather general structural assumptions.

There are other physically relevant choices of boundary conditions for the system \eqref{eq:nsch1}--\eqref{eq:nschi}, such as the dynamic boundary condition for $\phi$ and the generalized Navier boundary condition for $\bm{u}$, which were introduced in \cite{WLZ2024,ZKYWYS2019} to describe possible contact line dynamics at the solid boundary. These generalized boundary conditions are of great importance in practical applications. However, they lead to more complicated bulk-surface coupling structures, and thus further theoretical obstacles in the corresponding analysis. For related Navier--Stokes--Cahn--Hilliard systems without surfactant interaction, see \cite{GGM2016,GGW2019,GLW2024,GK2023} and the references therein.

When the interaction with the surfactant is neglected, numerous papers have been devoted to the mathematical analysis of the initial boundary value problem subject to the standard no-slip and homogeneous Neumann boundary conditions for the AGG model (see, for instance, \cite{ADG2013ndm,ADG2013dm,AGG2024, AGP2024,GLW2026,G2021,G2022,GP2025b} and the references therein).
Regarding strong well-posedness, which is the main focus of this study, the existence of global-in-time strong solutions in the two-dimensional case was proved in \cite{G2021} under periodic boundary conditions and later in \cite{AGG2024} for the standard no-slip/Neumann boundary conditions. In the three-dimensional case, the existence of local-in-time strong solutions was established in \cite{G2022}. It is worth noting that although singular potentials were treated in these works, the mobilities were assumed to be constant. We also mention \cite{AW2021}, where the authors applied the monotone operator theory and the semigroup approach to establish the local well-posedness of strong solutions for polynomial-like potentials and non-constant, non-degenerate mobility.

For the Cahn--Hilliard equation without fluid interaction, concentration-dependent mobility combined with singular potentials yields several difficulties in the analysis, due to the resulting highly nonlinear structure.
When the mobility is non-degenerate, the existence of global weak solutions was established in \cite{BB1999}. In addition, the existence and uniqueness of local strong solutions were obtained in three dimensions.
In the two-dimensional case, the authors of \cite{CGGG2025} proved the uniqueness of weak solutions, the regularity propagation of weak solutions at positive times, and that every global weak solution converges towards a single steady state as $t\to \infty$.
We refer to \cite{GP2025b} for a different approach that yields convergence for any global weak solution only satisfying an energy inequality, under the assumption that the mobility is a non-degenerate continuous function.
If regular potentials are considered instead, the two-dimensional initial-boundary value problem admits a unique global-in-time classical solution with proper higher-order Schauder-type estimates (see \cite{LQY2006}).
Finally, we note that the argument in \cite{BB1999} for the existence of local-in-time strong solutions can be applied to more complicated coupled systems, for example, the diffuse-interface model describing a ternary system constituted by a conductive diblock copolymer and a homopolymer acting as solvent, under the effect of an electric field \cite{ADG2024}.

\subsection{Summary of results and strategy of the proofs}

In this study, we investigate the strong well-posedness of problem \eqref{eq:nsch1}--\eqref{eq:nschi}. First, we establish the existence of local-in-time strong solutions in three dimensions and global-in-time strong solutions in two dimensions. When the spatial dimension is two, we further show that, under natural growth assumptions on the singular potentials, strong solutions remain strictly separated from the pure phases and a weak-strong uniqueness result holds.
In the three-dimensional case, we demonstrate the local-in-time uniqueness of strong solutions, provided that the initial data are strictly separated from the pure phases. These results are summarized in Theorems \ref{nschsse3d}, \ref{nschsse2d}, see Section \ref{sec:main results} for details.
To the best of our knowledge, no analogous results are yet available for Navier--Stokes--Cahn--Hilliard (NSCH) type systems for incompressible two-phase flows with unmatched densities and concentration-dependent mobilities, even in the absence of surfactant interaction.

We now highlight some novelties of this work, along with the associated difficulties.
First, we extend the previous result in the three-dimensional case on the local-in-time strong well-posedness of the NSCH system with unmatched densities and constant mobilities to the concentration-dependent mobility case.
In \cite{G2022}, an appropriate semi-Galerkin scheme was used to construct approximate strong solutions, where a Galerkin approximation was applied to the Navier--Stokes equations and a viscous regularization was introduced in the convective Cahn--Hilliard equation.
We adopt a similar approximation strategy here in order to preserve the positivity of the averaged density. Indeed, in the unmatched-density setting, the density profile $\rho(\phi)$ depends linearly on the concentration difference of the two fluids.
Therefore, the positivity of the density, which is needed to avoid the possible degeneracy of the momentum equation \eqref{eq:nsch1} and its Galerkin approximation (especially in the term $\partial_t (\rho \bm{u})$), may be lost if one cannot ensure that $\phi$ takes its values in the corresponding physical range $[-1,1]$. One possible approach to overcome this difficulty is to extend the density function nonlinearly to $\mathbb{R}$, while an additional higher-order correction term must be introduced in \eqref{eq:nsch1} to maintain an analog of the energy identity \eqref{basic_energy_identity}, thereby making the analysis more complicated (see, e.g.,  \cite{Fri16,GGW2019}). Alternatively, one can retain the singular potentials and directly restrict the phase-field variables to their physically relevant intervals. However, this immediately raises another difficulty, as pointed out in \cite{G2022}: in three dimensions, singular potentials prevent one from obtaining the time regularity of the chemical potentials required for higher-order estimates. The aforementioned difficulty can be overcome by introducing a viscous regularization in the convective Cahn--Hilliard subsystem. The resulting strict separation property keeps the phase-field variables away from the singular endpoints, so that the desired time regularity of the chemical potentials can be recovered.
However, it is exactly the analysis of the viscous counterpart of the convective Cahn--Hilliard subsystem \eqref{eq:nsch3}--\eqref{eq:nsch6} that constitutes the main new obstacle in the present case of concentration-dependent mobility. In the constant mobility case, the analysis of such a viscous convective Cahn--Hilliard system is standard, but in the concentration-dependent mobility case, even the uniqueness of strong solutions to this system is not straightforward (cf. \cite[Appendix]{G2022}). This issue is crucial in defining the mapping for solving the approximate system using a suitable fixed-point argument. We circumvent it by introducing two weighted Laplace operators (see \eqref{ahi_definition} and \eqref{asi_definition}) as in \cite{BB1999,CGGG2025}.
This allows us to construct test functions that are tailored to concentration-dependent mobilities and singular potentials. The convexity of the potential functions then enables us to obtain the continuous dependence of solutions on the initial data, which yields the uniqueness. Furthermore, we establish the desired regularity of strong solutions to the convective viscous Cahn--Hilliard subsystem, which in turn allows us to construct suitable approximate solutions to \eqref{eq:nsch1}--\eqref{eq:nschi}.

The global well-posedness of strong solutions to problem \eqref{eq:nsch1}--\eqref{eq:nschi} in dimension two is more delicate, since the nonlinear structure impedes direct higher-order estimates in general.
Here we can extend the results on the existence of global-in-time strong solutions and weak-strong uniqueness to the case of concentration-dependent mobility (cf. \cite{AGG2024} for the AGG model with constant mobility and without surfactant). Instead of constructing strong solutions via standard compactness arguments, we start with the global weak solutions whose existence was obtained in \cite{OGW2026} and prove that they are automatically strong for sufficiently regular initial data.
Applying the strategy of \cite[Theorem 2.4]{AGG2024}, we first show that with the given solenoidal velocity field of Leray--Hopf type, if the initial data are smooth enough, the weak solution to the convective Cahn--Hilliard subsystem \eqref{eq:nsch3}--\eqref{eq:nsch6} is indeed a strong one.
To achieve this goal, one has to control the additional nonlinear terms generated by the concentration-dependent mobilities together with the convective term, which requires a nontrivial generalization of the arguments in \cite{CGGG2025,AGG2024}.
The key ingredients are the existence of global-in-time strong solutions to this subsystem and uniqueness of weak solutions (or at least a weak-strong uniqueness).
Although the uniqueness of weak solutions is standard for constant mobility, it remained open even for the classical Cahn--Hilliard equation with concentration-dependent mobility and singular potential until the recent work \cite{CGGG2025}. Building on the approach therein, we exploit the properties of weighted Laplace operators and utilize the $L^4_tH^2_x$-boundedness of the phase-field variables. This allows us to derive a continuous dependence estimate, which immediately yields the uniqueness. However, proving the existence of global-in-time strong solutions requires a more delicate analysis. Indeed, the introduction of viscous regularization to enhance the time regularity of the chemical potentials gives rise to an additional difficulty: the $L^4_tH^2_x$ estimate of phase-field variables depends on the $L^2_tH^1_x$-norm of their time derivatives. To obtain the required uniform estimates, we exploit the dissipative structure of the system, which enables us to choose the viscous regularization coefficient sufficiently small so that these troublesome terms can be absorbed into the higher-order dissipative terms. Once the desired regularity of the phase-field variables and the chemical potentials have been established, the improved regularity of the velocity field follows from a linearization argument (cf. \cite{AGG2024}) applied to the momentum equation \eqref{eq:nsch1}.

\textit{Plan of the paper.} In Section \ref{mr}, we first introduce the functional settings, the notation, and the basic assumptions. After that, we state our main results.
Section \ref{3d_strong_proof} is devoted to the local-in-time well-posedness of problem \eqref{eq:nsch1}--\eqref{eq:nschi} in three dimensions.
In Section \ref{2d_strong_proof}, we establish the global-in-time well-posedness of problem \eqref{eq:nsch1}--\eqref{eq:nschi} in two dimensions.
Finally, appendices \ref{appendix} and \ref{appendix-2} contain some auxiliary results on the weighted Laplace operator and a convective viscous Cahn--Hilliard system.

\section{Main Results}\label{mr}
\setcounter{equation}{0}

\subsection{Preliminaries and notation}
\label{pre}
We first introduce some notation and conventions used in this paper. Let $\mathcal{X}$ be a (real) Banach space with the norm $\left\|\cdot\right\|_{\mathcal{X}}$. We denote by $\mathcal{X}^*$ its dual space and by $\langle \cdot,\cdot \rangle_{\mathcal{X}^*,\mathcal{X}}$ the associated duality pairing. Given a (real) Hilbert space $\mathcal{H}$, its inner product will be denoted by $(\cdot,\cdot)_{\mathcal{H}}$.
Throughout this paper, we assume that $\Omega \subset \mathbb{R}^d$ ($d\in\{2,3\}$) is a bounded domain with a sufficiently smooth boundary $\partial \Omega$.
For the standard Lebesgue and Sobolev spaces in $\Omega$, we use the notation $L^{p}(\Omega)$, $W^{k,p}(\Omega)$ for any $p \in [1,\infty]$ and $k\in \mathbb{N}$,
equipped with the corresponding norms
$\|\cdot\|_{L^{p}(\Omega)}$, $\|\cdot\|_{W^{k,p}(\Omega)}$, respectively. The space $W^{k,p}_0(\Omega)$ denotes the closure of $C^{\infty}_0 ( \Omega )$ in $W^{k,p}(\Omega)$ and its dual space is denoted by $ W^{-k,p'}(\Omega) \define  \big(  W^{k,p}_0(\Omega)\big)^*$,
where $(1/p)+(1/p')=1$, $p\in (1,\infty)$.
For $p = 2$, we use the convention $H^{k}(\Omega) \define  W^{k,2}(\Omega)$.
The $L^2$-Bessel potential spaces are denoted by $H^{s}(\Omega)$, $s\in\r$, which are defined by the restriction of distributions in $H^s( \mathbb{R}^{d} )$ to $\Omega$. For $s >0$, $H^{s}_0 (\Omega)$ denotes the closure of $C^{\infty}_0( \Omega )$ in $H^{s}(\Omega)$.
For simplicity, the norm and the inner product of $L^2(\Omega)$ will be indicated by $\|\cdot\|$ and $(\cdot,\cdot)$, respectively, while the duality pairing between $H^1(\Omega)$ and $(H^1(\Omega))^*$ will be indicated by $\langle \cdot,\cdot \rangle$.
Bold letters will be used for vector-valued spaces, for example, $\bm{L}^p(\Omega)= L^p(\Omega;\mathbb{R}^d)$, $p \in [1,\infty]$. When there is no ambiguity, for a vector-valued function $\bm{f}:\Omega\to \mathbb{R}^{d}$, the expression $\bm{f} \in L^{p}(\Omega)$ should be understood as $\bm{f} \in \bm{L}^p(\Omega) $.

Given a measurable set $I$ of $\mathbb{R}$, we introduce the function space $L^p(I;\mathcal{X})$ with $p\in [1,\infty]$, which consists of Bochner measurable $p$-integrable functions (if $p \in [1,\infty)$) or essentially bounded functions (if $p =\infty$) with values in a given Banach space $\mathcal{X}$. If $I=(a,b)$, we simply write $L^p(a,b;\mathcal{X})$. In addition, $f \in L^{p}_{\mathrm{loc}} ( [0,\infty);\mathcal{X} )$ if and only if $f \in L^{p} \big( 0,T;\mathcal{X} \big)$ for every $T\in (0,\infty)$. The space $L^p\ul([0,\infty);\mathcal{X})$ denotes the uniformly local variant of $L^p(0,\infty;\mathcal{X})$ consisting of all strongly measurable $f:[0,\infty)\to \mathcal{X}$ such that
\begin{equation*}
    \|f\|_{L^p\ul([0,\infty);\mathcal{X})}
    \define  \mathop\mathrm{sup}\limits_{t\ge0} \|f\|_{L^p(t,t+1;\mathcal{X})} < \infty.
\end{equation*}
If $T\in(0,\infty)$, we set $L^p\ul([0,T);\mathcal{X})= L^p(0,T;\mathcal{X})$.
For every $k\in \mathbb{N}_+$, $f\in W^{k,p} ( 0,T;\mathcal{X} )$ if and only if $ \{ (\frac{\mathrm{d}}{\mathrm{d}t})^jf \}_{j=0}^k \subset L^{p} ( 0,T;\mathcal{X} )$, where $\frac{\mathrm{d}}{\mathrm{d} t} f$ denotes the vector-valued distributional derivative of $f$ and $(\frac{\mathrm{d}}{\mathrm{d}t})^0 f \define f$.
Furthermore, $W^{k,p}_{\mathrm{uloc}} ( [0,\infty);\mathcal{X})$ is defined by replacing $L^{p} ( 0,T;\mathcal{X} )$ with $L^{p}_{\mathrm{uloc}} \big( [0,\infty);\mathcal{X} \big)$, and we set
$$
H^{k} ( 0,T;\mathcal{X} ) =W^{k,2} ( 0,T;\mathcal{X} ),
\quad
H^{k}_{\mathrm{uloc}} ( [0,\infty);\mathcal{X} ) = W^{k,2}_{\mathrm{uloc}} ( [0,\infty);\mathcal{X} ).
$$
Let $I=[0,T]$ if $T \in (0,\infty)$ or $I=[0,\infty)$ if $T=\infty$. Then $BC( I;\mathcal{X} )$ denotes the Banach space of all bounded and continuous functions $f:I \to \mathcal{X}$ equipped with the supremum norm, and $BUC( I;\mathcal{X} )$ is the subspace of all bounded and uniformly continuous functions.
Furthermore, for every $k \in \mathbb{N}_+$, $BC^k( I;\mathcal{X} )$ (resp. $BUC^k( I;\mathcal{X} )$) denotes all functions $f \in BC( I;\mathcal{X} )$ (resp. $f \in BUC( I;\mathcal{X} )$), whose Fr\'echet derivatives of order no larger than $k$ exist and belong to $BC( I;\mathcal{X} )$ (resp. $BUC( I;\mathcal{X} )$).
Finally, we denote by $BC_{w}( I;\mathcal{X} )$ the topological vector space of all bounded and weakly continuous functions $f:I \to \mathcal{X}$.

For convenience, we introduce the following shorthand notation:
\begin{align*}
    & H \define  L^2\left(\Omega\right),
    \quad V \define  H^1\left(\Omega\right),
    \quad  W \define \big\{u\in H^2(\Omega)\ \big|\ \partial_\mathbf{n}u=0\ \text{a.e. on}\ \partial\Omega\big\}.
\end{align*}
$H$ is identified with its dual, and we have the following continuous, dense, and compact embeddings:
$W \hookrightarrow V  \hookrightarrow H \hookrightarrow V^*$. In addition, we find $V=D \big((I-\Delta_N)^\frac{1}{2}\big)$, $W=D (I-\Delta_N)$, where $I$ is the identity operator on $H$ and $-\Delta_N: W \subset H \to H$ is the minus Laplace operator on $\Omega$ subject to the homogeneous Neumann boundary condition. In general, for all $k \in \mathbb{N}_+$, we define $H_N^k(\Omega)$ as $ D \big((I-\Delta_N)^\frac{k}{2}\big)$.
For every $f\in V^*$, $\overline{f}$ denotes its generalized mean value over $\Omega$, given by
$\overline{f}=|\Omega|^{-1}\langle f,1\rangle$. If $f\in L^1(\Omega)$, its mean value is given by $\overline{f}=|\Omega|^{-1}\int_\Omega f \, \mathrm{d}x$.
Then we recall the well-known Poincar\'{e}--Wirtinger inequality:
\begin{equation*}
\big\|f-\overline{f}\big\|\leq C_P \|\nabla f\|,\quad \forall\,
f\in V,
\end{equation*}
where the positive constant $C_P$ depends only on $\Omega$.
Let $A_N: V \to V^*$ denote the extension of $-\Delta_N$ on $V$, that is,
\begin{equation*}
    \langle A_N f,g \rangle
    = \int_{\Omega} \nabla f \cdot \nabla g \,\ \mathrm{d}x, \quad \forall\,f,g \in V.
\end{equation*}
Setting
\begin{align*}
    & L^p_{(0)}(\Omega) \define \big\{f\in L^{p}(\Omega) \ \big|\ \overline{f} =0\big\},
    \quad H_{(0)} \define  L^2_{(0)}(\Omega),
    \\
    & V_{(0)} \define  V\cap H_{(0)},
    \quad V^*_{(0)} \define  \big\{ L \in V^* \ \big|\   \langle L,1 \rangle = 0 \big\},
\end{align*}
it is easy to verify that the restriction of $A_N$ to $V_{(0)}$ is a linear isomorphism between $V_{(0)}$ and $V^*_{(0)}$.
Thus, we can define the inverse operator $\mathcal{N} \define  \big( A_N|_{V_{(0)}} \big)^{-1}:V_{(0)}^* \to V_{(0)}$, and obtain the following identities
\begin{align*}
&\left\langle A_N u,\mathcal{N}L \right\rangle
= \left\langle L,u \right\rangle,\quad \forall\, u \in V_{(0)},\ L \in V_{(0)}^*,
\\
&\left\langle L_1,\mathcal{N}L_2 \right\rangle = ( \nabla \left(\mathcal{N}L_1\right), \nabla \left(\mathcal{N}L_2\right) ),
\quad \forall\, L_1,\,L_2\ \in V_{(0)}^*.
\end{align*}
Define
\begin{align*}
    &\|L\|_* \define  \| \nabla \left(\mathcal{N}L\right)\|
    = \sqrt{\langle L,\mathcal{N}L \rangle},\quad \forall\, L\in V^*_{(0)}, \quad \|L\|_{-1}^2 \define  \|L-\overline{L}\|_*^2 + |\overline{L}|^2,
    \quad \forall\, L\in V^*.
\end{align*}
We observe that $\left\|\cdot\right\|_*$ and $\left\|\cdot\right\|_{-1}$ are equivalent to the standard dual norms on $V^*_{(0)}$ and $V^*$, respectively (cf. \cite{MZ2004}). A key feature of this study is that we need to handle nonconstant mobilities.
For this purpose, we define the following weighted Laplace operators and their inverses:
\begin{align}
    & \mathcal{A}_{\phi, h} : V \to V^*, \ \langle \mathcal{A}_{\phi, h} (f),g \rangle = \int_{\Omega} m_\phi( h ) \nabla f \cdot \nabla g  \, \mathrm{d}x, \ \mathcal{N}_{\phi, h} \define  \big(  \mathcal{A}_{\phi, h} |_{V_{(0)}}  \big)^{-1} : V_{(0)}^* \to V_{(0)},
    \label{ahi_definition} \\
    & \mathcal{A}_{\psi, h} : V \to V^*, \ \langle \mathcal{A}_{\psi, h} (f),g \rangle = \int_{\Omega} m_\psi( h ) \nabla f \cdot \nabla g  \, \mathrm{d}x, \ \mathcal{N}_{\psi, h} \define  \big(  \mathcal{A}_{\psi, h} |_{V_{(0)}}  \big)^{-1} : V_{(0)}^* \to V_{(0)},
    \label{asi_definition}
\end{align}
where $h$ is a Lebesgue measurable function defined in $\Omega$. The non-degeneracy of $m_\phi$ and $m_\psi$ listed in (\textbf{H1}) justifies the definition of $\mathcal{N}_{\phi, h}$ and $\mathcal{N}_{\psi, h}$ under suitable regularity assumptions in $h$. For basic properties of these operators, we refer to Appendix \ref{appendix} (see also \cite{BB1999,CGGG2025}).

The following interpolation inequality will be useful in the subsequent analysis:
\begin{equation*}
    \|f\|^2 \le \xi \|\nabla f\|^2 + C_\xi \|f\|^2_{V^*}, \quad \forall f \in V,
\end{equation*}
where $\xi \in (0,1)$ is an arbitrary constant and $C_\xi>0$ depends only on $\xi$ and $\Omega$.
In addition, we recall the classical Gagliardo--Nirenberg (GN for short) inequality:
\begin{equation}\label{gns_inequality}
    \| D^j f \|_{L^{p}(\Omega)} \le C_1 \| D^m f \|^a_{L^{r}(\Omega)} \| f \|^{1-a}_{L^{q}(\Omega)}
    + C_2 \| f \|_{L^{q}(\Omega)},
    \ \ \forall\, f \in W^{m,r}(\Omega) \cap L^{q}(\Omega),
\end{equation}
where $a\in [0,1]$, $p,q,r \in [1,\infty]$, and $j,m \in \mathbb{N}$ satisfy
\begin{align*}
    & 0 \le j < m, \quad  \frac{j}{m} \le a \le 1, \quad  \frac{1}{p} - \frac{j}{d} = a \Big( \frac{1}{r}-\frac{m}{d} \Big) + ( 1-a ) \frac{1}{q},
\end{align*}
and when $r>1$ as well as $m-j-(d/r)\in \mathbb{N}$, the constant $a$ is additionally required to be in $[j/m,1)$. In \eqref{gns_inequality}, $C_1$ and $C_2$ are positive constants only depending on $a, \ q, \ r, \ j, \ m, \ d$ and $\Omega$.
Agmon's inequality will also be frequently used, that is, for all $f \in H^{2}(\Omega)$,
\begin{equation}\label{Agmon_inequality}
    \| f \|_{L^{\infty}(\Omega)}^2 \le \left\{
    \begin{aligned}
     & C \| f \| \| f \|_{H^{2}(\Omega)},
    && \mathrm{if} \ d=2,
    \\
    & C \| f \|_{H^{1}(\Omega)} \| f \|_{H^{2}(\Omega)}, && \mathrm{if} \ d=3.
    \end{aligned}
    \right.
\end{equation}

For the velocity field, we define $\bm{C}^{\infty}_{0,\sigma}(\Omega) \define  \{ \bm{f} \in C^{\infty}_0(\Omega)^d \ | \ \mathrm{div}\, \bm{f} =0 \}$
and
\begin{align*}
    &\bm{L}^{p}_\sigma(\Omega) \define  \overline{\bm{C}^{\infty}_{0,\sigma}(\Omega)}^{\bm{L}^p(\Omega)},
    \quad
    \bm{H}^{s}_{0,\sigma}(\Omega) \define  \overline{\bm{C}^{\infty}_{0,\sigma}(\Omega)}^{\bm{H}^s(\Omega)},
    \quad
    \bm{W}^{k,p}_{0,\sigma}(\Omega) \define  \overline{\bm{C}^{\infty}_{0,\sigma}(\Omega)}^{\bm{W}^{k,p}(\Omega)},
\end{align*}
for all $p \in (1,\infty)$, $k \in \mathbb{N}_+$ and $s > 0$.
The Helmholtz projection corresponding to $\bm{L}^{2}_\sigma(\Omega)$ is denoted by $\mathbb{P}_{\sigma}$ (see, for example, \cite{S2001}). Then $\mathbb{P}_{\sigma} \bm{f} = \bm{f} - \nabla p$, where $p \in V_{(0)}$ is the solution to the (weak) Neumann problem
$$
( \nabla p, \nabla \phi )
= ( \bm{f}, \nabla \phi ),
\quad  \forall\, \phi \in C^{\infty} \big( \overline{\Omega} \big).
$$
In addition, if $\partial \Omega$ is sufficiently smooth (for example, of class $C^2$), we can also define the Helmholtz--Weyl decomposition of $\bm{L}^{p}(\Omega)$ for all $p \in (1,\infty)$
(see \cite[Section III.1]{G2011}).
When there is no confusion, $\mathbb{P}_{\sigma}$ also denotes the projection from $\bm{L}^{p}(\Omega)$ to $\bm{L}^{p}_\sigma(\Omega)$, whose null space is
\[
\big\{
\bm{w} \in L^{p}(\Omega):\bm{w} = \nabla p, \ \mathrm{for} \ \mathrm{some} \ p \in W^{1,p}(\Omega)   \big\}.
\]

In the subsequent analysis, we use the notation
\[
\rho_* \define  \min \{ \rho_1,\rho_2 \}, \quad  \rho^* \define  \max \{ \rho_1,\rho_2 \},
\quad \gamma=\frac{\rho_1 - \rho_2}{2},
\]
where $\rho_1$, $\rho_2$ denote positive homogeneous densities of the two components of the fluid mixture.
The capital letter $C$ denotes a generic positive constant that depends on the problem's structural data. Its meaning may change from line to line and even within the same chain of computations. Specific dependence will be pointed out if necessary.
Finally, for simplicity, we set $Q_{(s,t)}=\Omega \times (s,t), \ Q_{t}=Q_{(0,t)}$, $Q=Q_{(0,\infty)}$,
and analogously $S_{(s,t)} = \partial \Omega \times (s,t)$, $S_t = S_{(0,t)}$, and $S=S_{(0,\infty)}$.

\subsection{Statement of results}
\label{sec:main results}

Let us introduce the basic assumptions (cf. \cite{G2022,AGG2012,GP2025}).
\rm
\begin{itemize}
    \item[$(\mathbf{H0})$] The averaged density is defined as  $$
    \rho(\phi) = \frac{\rho_1 - \rho_2}{2}\phi + \frac{\rho_1 + \rho_2}{2}.
    $$
    \item[$(\mathbf{H1})$] We assume that $m_\phi(\cdot),m_\psi(\cdot) \in BUC^2(\mathbb{R})$.
    They are bounded from below by
\[
m_{\phi,*} \define  \inf\limits_{s \in \mathbb{R}} \{ m_\phi(s) \} >0 \quad  \mathrm{and} \quad m_{\psi,*} \define  \inf\limits_{s \in \mathbb{R}} \{ m_\psi(s) \} >0,
\]
    while their upper bounds are denoted by $m_\phi^* \define  \sup\limits_{s \in \mathbb{R}} \{ m_\phi(s) \}$ and $m_\psi^* \define  \sup\limits_{s \in \mathbb{R}} \{ m_\psi(s) \}$.

    Moreover, the viscosity coefficient satisfies $\nu \in C^1( [-1,1] )$ with
\[
\nu_* \define  \inf\limits_{s \in [-1,1]} \{ \nu(s) \} >0.
\]
    Its upper bound is denoted by $\nu^* \define  \sup\limits_{s \in [-1,1]} \{ \nu(s) \}$.

    \item[$(\mathbf{H2})$] The potential function $F_\phi\in C\big([-1,1]\big) \cap C^2\big((-1,1)\big)$ satisfies
    \begin{equation*}
        \mathop{\mathrm{lim}}\limits_{s \to -1^+} F_\phi'(s) = -\infty,\quad \mathop{\mathrm{lim}}\limits_{s \to 1^-} F_\phi'(s) = \infty, \quad F_\phi''(s) \ge \theta_\phi > 0, \quad \forall s \in (-1,1),
    \end{equation*}
    where $\theta_\phi$ is a given constant.
    Similarly, $F_\psi\in C\big([0,1]\big) \cap C^2\big((0,1)\big)$ satisfies
    \begin{equation*}
        \mathop{\mathrm{lim}}\limits_{s \to 0^+} F_\psi'(s) = -\infty,\quad \mathop{\mathrm{lim}}\limits_{s \to 1^-} F_\psi'(s) = \infty, \quad F_\psi''(s) \ge \theta_\psi > 0, \quad \forall s \in (0,1).
    \end{equation*}
    We set $F_\phi(s) = +\infty $ whenever $|s| > 1$ and $F_\psi(s) = +\infty $ whenever $s \in (-\infty,0) \cup (1,\infty)$. Without loss of generality, we take
    \begin{equation*}
        F_\phi(0)= F_\phi'(0) = F_\psi\left(\frac{1}{2}\right) = F_\psi'\left(\frac{1}{2}\right) = 0.
    \end{equation*}
    The bi-variate function $G$ satisfies $G(\cdot,\cdot)\in BUC^2(\mathbb{R}^{2})$.
\end{itemize}

Next, we introduce the notion of strong solutions to problem \eqref{eq:nsch1}--\eqref{eq:nschi}.
\begin{definition} \label{nschssd} \rm
Let $T \in (0,\infty]$ and set $I=[0,\infty)$ if $T= \infty$ or $I= [0,T] $ if $T<\infty$.
Suppose that $\Omega\subset \mathbb{R}^{3}$ is a bounded domain of class $C^3$ and Assumptions $(\mathbf{H0})$--$(\mathbf{H2})$ are satisfied. For any initial data
$\bm{u}_0 \in \bm{H}^{1}_{0,\sigma}(\Omega)$, $\phi_0, \psi_0 \in W$ such that  $\overline{\phi_0} \in (-1,1)$, $\overline{\psi_0} \in (0,1)$ and
$\mu_{\phi,0} \define - \Delta \phi_0 + F_\phi' (\phi_0)  \in V$, $\mu_{\psi,0} \define  - \beta \Delta \psi_0 + F_\psi' (\psi_0) \in V$,
we call $( \bm{u},\pi,\phi,\psi,\mu_\phi,\mu_\psi )$ that satisfies the regularity properties
\begin{align*}
    &\bm{u} \in BC (  I;\bm{H}^{1}_{0,\sigma}(\Omega)  )
    \cap L^{2}_{\mathrm{uloc}} ( I;\bm{H}^2(\Omega)  )
    \cap H^{1}_{\mathrm{uloc}} ( I;\bm{L}^{2}_\sigma(\Omega)),
    \\
    &\pi \in L^{2}_{\mathrm{uloc}} ( I;V ),
    \\
    &\phi,\psi \in L^{\infty} ( I;W^{2,p}(\Omega) ), \quad  \partial_{t} \phi, \partial_{t} \psi \in L^{\infty} ( I;V^* ) \cap L^{2}_{\mathrm{uloc}} ( I;V ),
    \\
    &\phi,\psi \in L^{\infty} ( Q_T ) \ \text{with} \ ( \phi,\psi ) \in (-1,1) \times (0,1) \ \mathrm{a.e.} \ \mathrm{in} \ Q_T,
    \\
    &\mu_\phi,\mu_\psi \in L^{\infty} ( I;V ) \cap L^{2}_{\mathrm{uloc}} ( I;H^3( \Omega ) ),
    \\
    &F_\phi'( \phi ),F_\psi'( \psi ) \in L^{\infty} ( I;L^{p}(\Omega) ),
\end{align*}
where $p = 6 $ if $d=3$ and $p \in [2,\infty)$ if $d=2$,
a strong solution to problem \eqref{eq:nsch1}--\eqref{eq:nschi} on $I$, if the system \eqref{eq:nsch1}--\eqref{eq:nsch6} holds almost everywhere in $Q_T$, and the boundary conditions \eqref{eq:nschb} (resp. initial conditions \eqref{eq:nschi}) hold almost everywhere on $S_T$ (resp. in $\Omega$).
\end{definition}

We now state our first result on the existence and uniqueness of local-in-time strong solutions in three dimensions:

\begin{theorem} \label{nschsse3d}
Let $\Omega\subset \mathbb{R}^{3}$ be a bounded domain of class $C^3$. Suppose that Assumptions $(\mathbf{H0})$--$(\mathbf{H2})$ are satisfied.

(1) For any initial data $\bm{u}_0 \in \bm{H}^{1}_{0,\sigma}(\Omega)$, $\phi_0,\psi_0 \in W$ satisfying $\overline{\phi_0} \in (-1,1)$, $\overline{\psi_0} \in (0,1)$ and $-\Delta \phi_0 + F_\phi'(\phi_0), -\beta \Delta \psi_0 + F_\psi'(\psi_0) \in V$, there exists a positive time $T_0$ depending on $\| \bm{u}_0 \|_{\bm{H}_{0,\sigma}^1(\Omega)}$, $\| \phi_0 \|_V$, $\| \psi_0 \|_V$, $\| -\Delta \phi_0 + F_\phi'(\phi_0) \|_V$, $\| -\beta \Delta \psi_0 + F_\psi'(\psi_0) \|_V$, $\overline{\phi_0}$ and $\overline{\psi_0}$ such that problem \eqref{eq:nsch1}--\eqref{eq:nschi} admits a strong solution on $[0,T_0]$ in the sense of Definition \ref{nschssd}.

(2) Assume in addition that $\Omega\subset \mathbb{R}^{3}$ is a bounded domain of class $C^4$, $F_\phi \in C^4\big( (-1,1) \big)$, $F_\psi \in C^4\big( (0,1) \big)$, and $\| \phi_0 \|_{L^{\infty}(\Omega)} < 1$, $\| \psi_0 - 1/2 \|_{L^{\infty}(\Omega)} < 1/2$. Then the strong solution is \emph{locally unique}: there exists a time $T_1 \in (0, T_0]$, such that the solution is unique on $[0, T_1]$.
\end{theorem}

Our second result concerns the existence of global-in-time strong solutions and a weak-strong uniqueness result for problem \eqref{eq:nsch1}--\eqref{eq:nschi} in two dimensions. In view of \cite{GP2025}, we impose the following structural assumptions on the singular potential functions $F_\phi$ and $F_\psi$.
\begin{itemize}
    \item[$(\mathbf{H2*})$] There exists some constant $\kappa> 1/2$ such that, as $\delta \to 0^+$, it holds
    \begin{equation}\notag
        \frac{1}{F_\phi'(1-2\delta)} = O\left(\frac{1}{| \ln \delta|^{\kappa}}\right),
        \quad \frac{1}{F_\phi'(-1+2\delta)} = O\left(\frac{1}{| \ln \delta|^{\kappa}}\right),
    \end{equation}
    and
    \begin{equation}\notag
        \frac{1}{F_\psi'(2\delta)} = O\left(\frac{1}{| \ln \delta|^{\kappa}}\right),
        \quad \frac{1}{F_\psi'(-1+2\delta)} = O\left(\frac{1}{| \ln \delta|^{\kappa}}\right).
    \end{equation}
\end{itemize}
On the other hand, Proposition \ref{cvcho_2d_global_sse} implies that any weak solution to problem \eqref{eq:nsch1}--\eqref{eq:nschi} in the sense of \cite[Definition 2.1]{OGW2026} with sufficiently regular initial data is indeed strong in the phase-field variables and chemical potentials.
Therefore, inspired by \cite[Theorem 4.1]{AGG2024}, we introduce the following notion of weak-strong uniqueness.
\begin{definition} \label{nsch_2d_weak_strong_unique} \rm
Let $\Omega \subset \mathbb{R}^2$ be a bounded domain of class $C^3$.
Suppose that Assumptions $(\mathbf{H0})$--$(\mathbf{H2})$ and $(\mathbf{H2}*)$ are satisfied.
In addition, $\bm{u}_0 \in \bm{H}^1_{0,\sigma}(\Omega)$, $\phi_0,\psi_0 \in W$ such that $\overline{\phi_0} \in (-1,1)$, $\overline{\psi_0} \in (0,1)$ and $-\Delta \phi_0 + F_\phi'(\phi_0), -\beta \Delta \psi_0 + F_\psi'(\psi_0) \in V$.
We say that the \emph{weak-strong uniqueness} holds for problem \eqref{eq:nsch1}--\eqref{eq:nschi} on $[0,T]$ if any weak solution $(\bm{u},\phi,\psi,\mu_\phi,\mu_\psi)$ and any strong solution $(\bm{U},\Pi,\Phi,\Psi,M_\phi,M_\psi)$ emanating from the same initial data $(\bm{u}_0,\phi_0,\psi_0)$ and satisfying the following properties coincide, that is,
\[
\bm{u} = \bm{U}, \quad \phi = \Phi, \quad \psi = \Psi, \quad \mathrm{a.e.} \ \text{in } Q_T.
\]
Specifically, the two solutions are, respectively, required to satisfy:
\begin{itemize}
    \item $(\bm{u},\phi,\psi,\mu_\phi,\mu_\psi)$ is a weak solution on $[0,T]$ exhibiting strong regularity in the phase-field variables and chemical potentials, such that
    \begin{align*}
        & \bm{u} \in BC_w ([0,T];\bm{L}^2_\sigma(\Omega)) \cap L^2 (0,T;\bm{H}^1_{0,\sigma}(\Omega)), \\
        & \phi,\psi \in L^\infty (0,T;  H_N^3(\Omega)) \cap H^1 (0,T;V) \cap C(\overline{\Omega} \times [0,T]),
        \\
        & \mu_\phi, \mu_\psi \in L^2 (0,T; H_N^3(\Omega) ) \cap L^\infty (0,T;V),
    \end{align*}
    with the uniform strict separation properties:
    \[
        \| \phi \|_{C(\overline{Q_T})} < 1 \quad \text{and} \quad \Big\|\psi - \frac12\Big\|_{C(\overline{Q_T})} < \frac12.
    \]
    It satisfies the weak formulation:
    \begin{align*}
        & \int_\Omega \rho(\phi(t)) \bm{u}(t) \cdot \bm{w}(t) \ \mathrm{d} x
        - \int_\Omega \rho(\phi_0) \bm{u}_0 \cdot \bm{w}(0) \, \mathrm{d} x
        - \int_0^t \big(  \rho(\phi) \bm{u} , \partial_t \bm{w} \big) \, \mathrm{d} \tau
        \\
        & \qquad - \int_0^t \big(  \rho(\phi) \bm{u} \otimes \bm{u} , \nabla \bm{w} \big) \, \mathrm{d} \tau
        + \int_0^t \big( \nu(\phi) D \bm{u}, D \bm{w} \big) \, \mathrm{d} \tau
        - \int_0^t \big( \bm{u} \otimes \mathbf{J}, \nabla \bm{w} \big) \, \mathrm{d} \tau
        \\
        & \quad = \int_0^t \big(  \mu_\phi \nabla \phi + \mu_\psi \nabla \psi , \bm{w} \big) \, \mathrm{d} \tau,
        \quad \forall\, t\in (0,T),
    \end{align*}
    for all $\bm{w} \in H^1(0,T;\bm{L}_\sigma^2(\Omega)) \cap L^4 (0,T;\bm{H}^1_{0,\sigma}(\Omega))$,
    where $\mathbf{J} = - \gamma m_\phi(\phi) \nabla \mu_\phi$ almost everywhere in $Q_T$, and equations \eqref{eq:nsch3}--\eqref{eq:nsch6} almost everywhere in $Q_T$, with
    \[
    (\bm{u},\phi,\psi)|_{t=0} = (\bm{u}_0,\phi_0,\psi_0),
    \quad \text{a.e. in}\ \Omega.
    \]
    Moreover, the energy inequality holds
    \begin{align*}
        & E_{\mathrm{tot}}(\bm{u}(t),\phi(t),\psi(t)) + \int_{Q_t} \big( \nu(\phi) | D \bm{u} |^2 + m_\phi(\phi) | \nabla \mu_\phi |^2 + m_\psi(\psi) | \nabla \mu_\psi |^2 \big) \, \mathrm{d} (x,\tau)
        \\
        & \quad + \int_{Q_t}  \sigma_1 \big(  \overline{\phi} - c  \big) \Big(  \mu_\phi - \frac12 \gamma | \bm{u} |^2  \Big)  \, \mathrm{d} (x,\tau) \le E_{\mathrm{tot}}(\bm{u}_0,\phi_0,\psi_0),\quad \forall\, t \in [0,T].
    \end{align*}
    \item $(\bm{U},\Pi,\Phi,\Psi,M_\phi,M_\psi)$ is a strong solution on $[0,T]$ such that
    \begin{align*}
        & \bm{U} \in BC ([0,T];\bm{H}^1_{0,\sigma}(\Omega)) \cap L^2 (0,T;\bm{H}^2 (\Omega))
        \cap H^1 (0,T;\bm{L}^2_\sigma(\Omega)),
        \\
        & \Phi,\Psi \in L^\infty (0,T; H_N^3(\Omega)) \cap H^1 (0,T;V) \cap C(\overline{\Omega} \times [0,T]),
        \\
        & M_\phi, M_\psi \in L^2 (0,T; H_N^3(\Omega) ) \cap L^\infty (0,T;V),
    \end{align*}
    along with the strict separation properties $\| \Phi \|_{C(\overline{Q_T})} < 1$ and $\| \Psi - 1/2 \|_{C(\overline{Q_T})} < 1/2$. It satisfies the system \eqref{eq:nsch1}--\eqref{eq:nsch6} almost everywhere in $Q_T$, and the initial condition
    \[
    (\bm{U},\Phi,\Psi)|_{t=0} = (\bm{u}_0,\phi_0,\psi_0),\quad \text{a.e. in}\ \Omega.
    \]
\end{itemize}
\end{definition}

Our second result reads as follows:
\begin{theorem} \label{nschsse2d}
Let $\Omega\subset \mathbb{R}^{2}$ be a bounded domain of class $C^3$. Suppose that Assumptions $(\mathbf{H0})$--$(\mathbf{H2})$ hold.

(1) For any initial data $\bm{u}_0 \in \bm{H}^{1}_{0,\sigma}(\Omega)$, $\phi_0,\psi_0 \in W$ with $\overline{\phi_0} \in (-1,1)$, $\overline{\psi_0} \in (0,1)$ and $-\Delta \phi_0 + F_\phi'(\phi_0),\, -\beta \Delta \psi_0 + F_\psi'(\psi_0) \in V$, problem \eqref{eq:nsch1}--\eqref{eq:nschi} admits a global strong solution on $[0,\infty)$ in the sense of Definition \ref{nschssd}.

(2) In addition, if the assumption $(\mathbf{H2}*)$ holds, then $(\phi, \psi, \mu_\phi, \mu_\psi)$ satisfies the following regularity properties:
\begin{align*}
    & \phi,\psi \in L^{\infty} ( 0,\infty;H^{3}_N(\Omega) ) \quad \text{and}\quad \phi \in [-1+\delta_\phi,1-\delta_\phi], \ \psi \in [\delta_\psi,1-\delta_\psi] \ \mathrm{in} \ Q,
    \\
    & \mu_\phi, \mu_\psi \in BC ( [0,\infty);V ), \quad \partial_{t} \mu_\phi, \partial_{t} \mu_\psi \in L^{2} ( 0,\infty;V^* ),
\end{align*}
where the positive constant $\delta_\phi\in (0,1)$ depends on $\Omega$, $F_\phi$, $\| G \|_{C^2( [-1,1]\times[0,1] )}$, and $\| \mu_\phi \|_{L^{\infty} ( 0,\infty;V )}$, while $\delta_\psi\in (0,1/2)$ depends on $\Omega$, $\beta$, $F_\psi$, $\| G \|_{C^2( [-1,1]\times[0,1] )}$, and $\| \mu_\psi \|_{L^{\infty} ( 0,\infty;V )}$. If we further assume that $F_\phi \in C^3\big( (-1,1) \big)$ and $F_\psi \in C^3\big( (0,1) \big)$, then we can conclude the weak-strong uniqueness as described in Definition \ref{nsch_2d_weak_strong_unique}.
\end{theorem}

\section{Local Well-posedness in 3D}
\label{3d_strong_proof}
\setcounter{equation}{0}

This section is devoted to the proof of Theorem \ref{nschsse3d}. We establish the existence of local-in-time strong solutions in three dimensions via an appropriate semi-Galerkin approximation scheme and compactness arguments. The result naturally extends to the two-dimensional case.

\subsection{The semi-Galerkin approximation}
\label{loc-3D-1}
Let $\{ \bm{w}_j \}_{j=1}^{\infty}$ and $\{ \lambda_j \}_{j=1}^{\infty}$ denote the eigenfunctions and corresponding eigenvalues of the Stokes operator $\mathbf{A} =  \mathbb{P}_{\sigma} (-\Delta)$ subject to the non-slip boundary condition.
For any $m \in \mathbb{N}_+$, we define the finite-dimensional subspace
$\mathbf{V}_m \define \mathrm{span} \{ \bm{w}_1,...,\bm{w}_m \} \subset \bm{L}^{2}_\sigma(\Omega)$, endowed with the standard $\bm{L}^{2}(\Omega)$-norm.
The orthogonal projection on $\mathbf{V}_m$ with respect to the inner product in $\bm{L}^{2}_\sigma(\Omega)$ is denoted by $\mathbb{P}_m$. Since $\Omega$ is of class $C^3$, the regularity theory of the Stokes operator yields that $\bm{w}_j\in \bm{H}_{0,\sigma}^3(\Omega)$ for all $j\in \mathbb{N}_+$.

Let $T>0$. For any $m,k\in \mathbb{N}_+$ and $\alpha\in (0,1)$,  we consider the following approximate problem for \eqref{eq:nsch1}--\eqref{eq:nschi}:
\begin{numcases}{}
    \big( \rho( \phi^m ) \partial_{t} \bm{u}^m, \bm{\eta} \big)
    + \big( \rho( \phi^m ) ( \bm{u}^m \cdot \nabla ) \bm{u}^m , \bm{\eta} \big)
    + \big(  \nu( \phi^m ) D \bm{u}^m , \nabla \bm{\eta} \big)
    \nonumber\\
    \qquad - \gamma \big(  m_\phi(\phi^m)( \nabla \mu_\phi^m \cdot \nabla ) \bm{u}^m + \sigma_1 (  \overline{\phi^m} - c ) \bm{u}^m, \bm{\eta} \big)
    \nonumber\\
    \quad = \big( \mu_\phi^m \nabla \phi^m + \mu_\psi^m \nabla \psi^m, \bm{\eta} \big), \label{eq:semi_galerkin_nsch1}
    &$\forall\, \bm{\eta} \in \mathbf{V}_m$, $t \in [0,T]$,
    \\
    \partial_{t} \phi^m + \bm{u}^m \cdot \nabla \phi^m + \sigma_1 \big( \overline{\phi^m} -c \big) = \mathrm{div}\, \big(m_\phi (\phi^m) \nabla \mu_\phi^m\big), \label{eq:semi_galerkin_nsch2} & a.e. in $Q_T$,
    \\
    \mu_\phi^m = \alpha \partial_{t} \phi^m -\Delta \phi^m + \sigma_2 \mathcal{N} \big( \phi^m - \overline{\phi^m} \big) + F_\phi'(\phi^m) + \partial_\phi G(\phi^m,\psi^m), \label{eq:semi_galerkin_nsch3}  & a.e. in $Q_T$,
    \\
    \partial_{t} \psi^m + \bm{u}^m \cdot \nabla \psi^m = \mathrm{div}\, \big(m_\psi (\psi^m) \nabla \mu_\psi^m\big) \label{eq:semi_galerkin_nsch4}, & a.e. in $Q_T$,
    \\
    \mu_\psi = \alpha \partial_{t} \psi^m -\beta \Delta \psi^m + F_\psi'(\psi^m) + \partial_\psi G(\phi^m,\psi^m), \label{eq:semi_galerkin_nsch5} & a.e. in $Q_T$,
\end{numcases}
with $\gamma=(\rho_1 - \rho_2)/2$, subject to the following boundary and initial conditions:
\begin{numcases}{}
    \bm{u}^m =\mathbf{0}, \ \ \partial_{\mathbf{n}} \phi^m = \partial_{\mathbf{n}} \mu_\phi^m = \partial_{\mathbf{n}} \psi^m = \partial_{\mathbf{n}} \mu_\psi^m = 0, \label{eq:semi_galerkin_nschb} & a.e. on $S_T$, \\
    \bm{u}^m|_{t=0} =  \mathbb{P}_m \bm{u}_0, \ \phi^m|_{t=0} = \phi_{0,k},\ \psi^m|_{t=0} = \psi_{0,k}, \label{eq:semi_galerkin_nschi} & a.e. in $\Omega$.
\end{numcases}
The semi-Galerkin scheme \eqref{eq:semi_galerkin_nsch1}--\eqref{eq:semi_galerkin_nschi} is inspired by \cite{G2022} for the AGG model. More precisely, we apply the Galerkin approximation only to the Navier--Stokes part \eqref{eq:nsch1}--\eqref{eq:nsch2} and introduce a viscous regularization to the convective Cahn--Hilliard subsystem \eqref{eq:nsch3}--\eqref{eq:nsch6}, where $\alpha\in (0,1)$. For the initial datum $(\bm{u}_0,\phi_0,\psi_0)$ prescribed in Theorem \ref{nschsse3d}, we perform a finite-dimensional projection of the initial velocity field and approximate the initial phase-field variables with their regularization $( \phi_{0,k},\psi_{0,k} )$ as in \cite{GMT2019,G2021}. For $k\in \mathbb{N}_+$, we define
\begin{align*}
    & \mu_{\phi,0}^k =  h_k \circ \mu_{\phi,0},
    \quad \mu_{\psi,0}^k =  h_k \circ \mu_{\psi,0},
\end{align*}
where $\mu_{\phi,0} = -\Delta \phi_0 + F_\phi' ( \phi_0 )$, $\mu_{\psi,0} = -\beta \Delta \psi_0 + F_\psi' ( \psi_0 )$, and $h^k$ is a global Lipschitz function given by
\begin{equation*}
    h^k( z )=\left\{
    \begin{aligned}
        & -k, && z < -k, \\
        & z,  && z \in [-k,k], \\
        & k, && z >k.
    \end{aligned}
    \right.
\end{equation*}
Let $\phi_{0,k},\psi_{0,k}$ be two functions that solve the following elliptic problems subject to homogeneous Neumann boundary conditions, respectively,
\begin{numcases}{}
    -\Delta \phi_{0,k} + F_\phi' ( \phi_{0,k} ) = \mu_{\phi,0}^k, & in $\Omega$, \nonumber\\
    \partial_{\mathbf{n}} \phi_{0,k} = 0, & on $\partial \Omega$,\nonumber
\end{numcases}
\begin{numcases}{}
    -\beta \Delta \psi_{0,k} + F_\psi' ( \psi_{0,k} ) = \mu_{\psi,0}^k, & in $\Omega$, \nonumber\\
    \partial_{\mathbf{n}} \psi_{0,k} = 0, & on $\partial \Omega$.\nonumber
\end{numcases}
Some basic properties of the approximate initial data $(\phi_{0,k},\psi_{0,k})$ are summarized below (cf. \cite{G2021,G2022}):
\begin{lemma} \label{approximate_initial_phase_field}
Suppose that the assumption $(\mathbf{H2})$ holds and $(\phi_0,\psi_0)$ is given as in Theorem \ref{nschsse3d}.

(1) For any $k\in \mathbb{N}_+$, we have $\mu_{\phi,0}^k,\, \mu_{\psi,0}^k\in V\cap L^\infty(\Omega)$ such that
$$
\|\mu_{\phi,0}^k\|_V\leq \|\mu_{\phi,0}\|_V,\quad \|\mu_{\psi,0}^k\|_V\leq \|\mu_{\psi,0}\|_V,\quad \forall\, k\in \mathbb{N}_+.
$$
The functions $\phi_{0,k}, \psi_{0,k}\in W^{2,6}(\Omega)\cap W$ are uniquely determined with $F_\phi' ( \phi_{0,k} ),\, F_\psi' ( \psi_{0,k} )\in L^6(\Omega)$. In addition, $\| \phi_{0,k} \|_{W^{2,6}(\Omega)}$ (resp. $\| \psi_{0,k} \|_{W^{2,6}(\Omega)}$) and $\| F_\phi'( \phi_{0,k} ) \|_{L^{6}(\Omega)}$ (resp. $\| F_\psi'( \psi_{0,k} ) \|_{L^{6}(\Omega)}$) are uniformly bounded by $\|\mu_{\phi,0}\|_V$ (resp. $\|\mu_{\psi,0}\|_V$).

(2) As $k \to \infty$, we have $\mu_{\phi,0}^k\to \mu_{\phi,0}$ and $\mu_{\psi,0}^k\to \mu_{\psi,0}$ in $H$, which implies that  $\phi_{0,k} \to \phi_0$ and $\psi_{0,k} \to \psi_0$ in $V$.

(3) For any $k\in \mathbb{N}_+$, $\| F_\phi'( \phi_{0,k} ) \|_{L^{\infty}(\Omega)}\leq k$ and $\| F_\psi'( \psi_{0,k} ) \|_{L^{\infty}(\Omega)}\leq k$, which implies that
$$
\phi_{0,k} \in [-1+\delta_\phi(k),1-\delta_\phi(k)], \quad  \psi_{0,k} \in [\delta_\psi(k),1-\delta_\psi(k)],
\quad \text{in} \ \Omega
$$
for some $\delta_\phi(k)\in (0,1)$, $\delta_\psi(k)\in (0,1/2)$. As a consequence, we have $F_\phi' ( \phi_{0,k} ),\, F_\psi' ( \psi_{0,k} )\in V$, which implies  $\phi_{0,k},\,\psi_{0,k} \in H^3(\Omega)$.
\end{lemma}

\subsection{Solvability of the approximate problem}
We solve the approximate problem \eqref{eq:semi_galerkin_nsch1}--\eqref{eq:semi_galerkin_nschi}
by a fixed-point argument applied to the velocity field $\bm{u}^m$. This necessitates a preliminary well-posedness analysis of a viscous Cahn--Hilliard subsystem with convection. Consider the following initial-boundary value problem:
\begin{numcases}{}
    \partial_{t} \phi + \bm{u} \cdot \nabla \phi + \sigma_1 \big( \overline{\phi} -c \big) = \mathrm{div}\, \big(m_\phi (\phi) \nabla \mu_\phi\big),
    \label{eq:cvcho1}\\
    \mu_\phi = \alpha \partial_{t} \phi -\Delta \phi + \sigma_2 \mathcal{N} \big( \phi - \overline{\phi} \big) + F_\phi'(\phi) + \partial_\phi G(\phi,\psi), \label{eq:cvcho2}\\
    \partial_{t} \psi + \bm{u} \cdot \nabla \psi = \mathrm{div}\, \big(m_\psi (\psi) \nabla \mu_\psi\big),
    \label{eq:cvcho3}\\
    \mu_\psi = \alpha \partial_{t} \psi -\beta \Delta \psi + F_\psi'(\psi) + \partial_\psi G(\phi,\psi), \label{eq:cvcho4}
\end{numcases}
in $Q_T$ with $T\in(0,\infty]$, subject to the boundary and initial conditions:
\begin{numcases}{}
    \partial_{\mathbf{n}} \phi = \partial_{\mathbf{n}} \mu_\phi = \partial_{\mathbf{n}} \psi = \partial_{\mathbf{n}} \mu_\psi = 0, \label{eq:cvchob} & on $S_T$,\\
    \phi|_{t=0} = \phi_0(x), \ \psi|_{t=0} = \psi_0(x), \label{eq:cvchoi} & in $\Omega$.
\end{numcases}
Here, $\alpha \in (0,1)$ and $\bm{u}$ is a given solenoidal velocity field defined in $Q_T$.
A key distinction from \cite{G2022} is that we now study a coupled Cahn--Hilliard system with nondegenerate concentration-dependent mobilities and a nonlocal interaction term of Oono's type, which admits only local strong solutions in three dimensions.
This renders the corresponding analysis considerably more subtle.

\begin{definition} \label{cvchossd} \rm
Let $\alpha\in (0,1)$, $T \in (0,\infty]$. Set $I=[0,\infty)$ if $T= \infty$ or $I = [0,T]$ if $T<\infty$.
Suppose that Assumptions $(\mathbf{H1})$ and $(\mathbf{H2})$ are satisfied and $\bm{u}\in L^{2} \big( I;\bm{L}^{2}_\sigma(\Omega) \big)$.
For any given initial data $\phi_0, \psi_0 \in V \cap L^\infty(\Omega)$ with $(\phi_0,\psi_0) \in [-1,1]\times[0,1]$ almost everywhere in $\Omega$ and $\overline{\phi_0} \in (-1,1)$, $\overline{\psi_0} \in (0,1)$, we call $(\phi,\psi,\mu_\phi,\mu_\psi)$ a weak solution to problem \eqref{eq:cvcho1}--\eqref{eq:cvchoi} on $I$, if
\begin{align*}
    & \phi,\,\psi \in BC ( I;V ) \cap L^{2}_{\mathrm{uloc}} ( I;W ), \quad \partial_{t} \phi,\,\partial_{t} \psi \in L^{2} ( I;H ),
    \\
    & \phi,\,\psi \in L^{\infty}( Q_T ) \ \ \text{such that } -1<\phi<1,\ \ 0<\psi<1 \quad \text{a.e. in } Q_T,
    \\
    & \mu_\phi,\,\mu_\psi \in L^{2}_{\mathrm{uloc}} ( I;V ) \cap L^{\frac{4}{3}}_{\mathrm{uloc}} ( I;W^{2,\frac{4}{3}}(\Omega) ), \quad \nabla \mu_\phi,\,\nabla \mu_\psi \in L^{2} ( I;H ),
    \\
    & F_\phi'( \phi ),\,F_\psi'( \psi ) \in L^{2}_{\mathrm{uloc}} ( I;H ),
\end{align*}
and $(\phi,\psi,\mu_\phi,\mu_\psi)$ satisfies equations \eqref{eq:cvcho1}--\eqref{eq:cvcho4} almost everywhere in $Q_T$, the boundary conditions \eqref{eq:cvchob} almost everywhere on $S_T$ and the initial conditions \eqref{eq:cvchoi} almost everywhere in $\Omega$.
\end{definition}

The following proposition establishes the well-posedness of problem \eqref{eq:cvcho1}--\eqref{eq:cvchoi} and the regularity properties of its solutions.

\begin{proposition} \label{cvchossr}
Let $\Omega \subset \mathbb{R}^{d} \ (d \in \{ 2,3 \})$ be a bounded domain of class $C^3$.
Suppose that $\alpha\in (0,1)$, $T\in (0,\infty]$, Assumptions $(\mathbf{H1})$--$(\mathbf{H2})$ are satisfied and $\bm{u}\in L^{\infty} ( I;\bm{L}^{3}_\sigma(\Omega) )$, $\partial_t \bm{u} \in L^2 ( I;\bm{L}^{\frac{6}{5}}(\Omega) )$.
For any initial data $\phi_0,\psi_0 \in H^3_N(\Omega)$ satisfying
\[
-1+\delta_{0,\phi} \le \phi_0(x)\leq 1-\delta_{0,\phi},
\qquad
\delta_{0,\psi} \le \psi_0(x)\leq 1-\delta_{0,\psi},
\quad \forall \, x\in\Omega,
\]
for some constants $\delta_{0,\phi} \in (0,1)$ and
$\delta_{0,\psi} \in (0,1/2)$, there is a positive time $T_r \in I$ such that problem \eqref{eq:cvcho1}--\eqref{eq:cvchoi} admits a unique strong solution on $[0,T_r]$ that satisfies the properties in Definition \ref{cvchossd} and
\begin{align*}
        & \phi,\psi \in L^{\infty} ( 0,T_r; H_N^{3}(\Omega)), \quad  \partial_{t} \phi, \partial_{t} \psi \in L^{\infty} ( 0,T_r;V ) \cap L^{2}_{\mathrm{uloc}} ( 0,T_r;W ) \cap H^{1}_{\mathrm{uloc}} ( 0,T_r;H ),
        \\
        & \mu_\phi, \mu_\psi \in L^\infty ( 0,T_r;W ), \quad  \partial_{t} \mu_\phi,\partial_{t} \mu_\psi \in L^{2}_{\mathrm{uloc}} ( 0,T_r;V ).
\end{align*}
Moreover, it holds
\[
(  \phi(x,t),\psi(x,t)  ) \in [-1+\delta_\phi,1-\delta_\phi] \times [\delta_\psi,1-\delta_\psi],
\quad \forall\, (x,t) \in \Omega \times [0,T_r],
\]
for some constants $\delta_\phi\in (0,1)$ and $\delta_\psi\in (0,1/2)$.
When $d=2$, we simply have $T_r=T$.
\end{proposition}
\noindent \textbf{Proof.} The proof of Proposition \ref{cvchossr} relies on the preliminary results established in Proposition \ref{cvchosse} (see the appendix) for problem \eqref{eq:cvcho1}--\eqref{eq:cvchoi}.
Here, under stronger regularity assumptions on the initial data and the prescribed velocity
$\bm{u}$, we obtain improved regularity for the solutions.

Let $( \phi,\psi,\mu_\phi,\mu_\psi )$ be the strong solution to problem \eqref{eq:cvcho1}–\eqref{eq:cvchoi} on $[0,T_r]$ corresponding to the initial data $(\phi_0,\psi_0)$ and possessing the regularity specified in Proposition \ref{cvchosse}. Under the present assumption on $\bm{u}$, we derive additional regularity properties. To rigorously justify the subsequent \textit{a priori} estimates, we first work with a sufficiently smooth velocity field $\bm{u} \in C^1_0 ( I;\bm{H}^{2}_{0,\sigma}(\Omega) )$.
Then, applying standard compactness arguments, we can recover the result for $\bm{u}$ that satisfies $\bm{u} \in L^{\infty} ( I;\bm{L}^{3}_\sigma(\Omega) )$ and $\partial_t \bm{u} \in L^2 (I;\bm{L}^{\frac{6}{5}}_\sigma(\Omega) )$.

Using the time difference argument as in the proof of Proposition \ref{cvchosse}, we can show that
\begin{align*}
    & \partial_{t} \phi, \partial_{t} \psi \in L^{2} ( 0,T_r;W ) \cap L^{\infty} ( 0,T_r;V ) \cap H^1 (  0,T_r;H  ),\\
    & \mu_\phi,\mu_\psi \in H^1 (  0,T_r;V ).
\end{align*}
Hence, we perform only the formal estimates below.
Differentiating equations \eqref{eq:cvcho1}--\eqref{eq:cvcho4} with respect to time, we obtain
\begin{numcases}{}
    \partial_{t} ^2 \big( \phi - \overline{\phi} \big)  + \partial_{t} \bm{u} \cdot \nabla \phi
    + \bm{u} \cdot \nabla \partial_{t} \phi
    = \mathrm{div}\, \big(  m_\phi' ( \phi ) \partial_{t} \phi \nabla \mu_\phi  \big)
    + \mathrm{div}\, \big(  m_\phi ( \phi ) \nabla \partial_{t} \mu_\phi  \big),
    \label{eq:cvcho1tdif} \\
    \partial_{t} ^2 \psi  + \partial_{t} \bm{u} \cdot \nabla \psi + \bm{u} \cdot \nabla \partial_{t} \psi
    = \mathrm{div}\, \big(  m_\psi' ( \psi ) \partial_{t} \psi \nabla \mu_\psi  \big)
    + \mathrm{div}\, \big(  m_\psi ( \psi ) \nabla \partial_{t} \mu_\psi  \big), \label{eq:cvcho3tdif}
\end{numcases}
in $L^{2} ( 0,T_r;V^* )$, and
\begin{numcases}{}
    \partial_{t} \mu_\phi
    = \alpha \partial_{t}  ^2 \phi  -\Delta \partial_{t} \phi
    + \sigma_2 \mathcal{N} \partial_{t}  \big( \phi - \overline{\phi} \big)
    + F_\phi''(\phi) \partial_{t} \phi
    + \partial_\phi ^2 G(\phi,\psi) \partial_{t} \phi
    + \partial_{\phi,\psi} ^2 G(\phi,\psi) \partial_{t} \psi,
    \label{eq:cvcho2tdif} \\
    \partial_{t} \mu_\psi
    = \alpha \partial_{t}  ^2 \psi
    - \beta \Delta \partial_{t} \psi
    + F_\psi''(\psi) \partial_{t} \psi
    + \partial_\psi ^2 G(\phi,\psi) \partial_{t} \psi
    + \partial_{\phi,\psi} ^2 G(\phi,\psi) \partial_{t} \phi,
    \label{eq:cvcho4tdif}
\end{numcases}
almost everywhere in $\Omega\times (0,T_r)$.
Testing equations \eqref{eq:cvcho1tdif} and \eqref{eq:cvcho3tdif} with $\partial_{t} \mu_\phi$ and $\partial_{t} \mu_\psi$, respectively, adding the resultants together, we get
\begin{align*}
    & \frac{1}{2} \frac{\mathrm{d}}{\mathrm{d t}}
    \big(  \|  \nabla \partial_{t} \phi  \|^2 + \beta \|  \nabla \partial_{t} \psi  \|^2 + \sigma_2 \|  \partial_{t} \phi  - \partial_{t} \overline{\phi}  \|^2_*  \big)
    + \alpha \|  \partial_{t} ^2 (\phi  - \overline{ \phi})  \|^2
    + \alpha \|  \partial_{t} ^2 \psi  \|^2
    \nonumber \\
    & \qquad + \int_{\Omega}  \big( m_\phi ( \phi ) \big| \nabla \partial_{t} \mu_\phi \big|^2 + m_\psi ( \psi ) \big| \nabla \partial_{t} \mu_\psi \big|^2 \big) \, \mathrm{d}x
    \nonumber \\
    & \quad = - \int_{\Omega} \big( m_\phi' ( \phi ) \partial_{t} \phi \nabla \mu_\phi \cdot \nabla \partial_{t} \mu_\phi + m_\psi' ( \psi ) \partial_{t} \psi \nabla \mu_\psi \cdot \nabla \partial_{t} \mu_\psi \big) \, \mathrm{d}x
    \nonumber \\
    & \qquad - \int_{\Omega} \big( \partial_{t} \mu_\phi \partial_{t} \bm{u} \cdot \nabla \phi + \partial_{t} \mu_\phi \bm{u} \cdot \nabla \partial_{t} \phi + \partial_{t} \mu_\psi \partial_{t} \bm{u} \cdot \nabla \psi + \partial_{t} \mu_\psi \bm{u} \cdot \nabla \partial_{t} \psi \big) \, \mathrm{d}x \nonumber \\
    & \qquad - \big(  F_\phi''(\phi) \partial_{t} \phi + \partial_\phi ^2 G(\phi,\psi) \partial_{t} \phi + \partial_{\phi,\psi} ^2 G(\phi,\psi) \partial_{t} \psi ,  \partial_{t} ^2 \big( \phi - \overline{\phi} \big)   \big)
    \nonumber \\
    & \qquad - \big(  F_\psi''(\psi) \partial_{t} \psi + \partial_\psi ^2 G(\phi,\psi) \partial_{t} \psi + \partial_{\phi,\psi} ^2 G(\phi,\psi) \partial_{t} \phi ,  \partial_{t} ^2 \psi \big)
    \nonumber \\
    & \quad \le \frac{1}{4} m_{\phi,*} \|  \nabla \partial_{t} \mu_\phi  \|^2
    + \frac{1}{4} m_{\psi,*} \|  \nabla \partial_{t} \mu_\psi  \|^2
    + C \| \nabla \mu_\phi \|^2_{L^{6}(\Omega)} \| \partial_{t} \phi \|^2_{L^{3}(\Omega)}
    + C \| \nabla \mu_\psi \|^2_{L^{6}(\Omega)} \| \partial_{t} \psi \|^2_{L^{3}(\Omega)}
    \nonumber \\
    & \qquad + \|  \partial_{t}  \mu_\phi  - \overline{\partial_{t}  \mu_\phi }  \|_{L^{6}(\Omega)} \|  \partial_{t} \bm{u}  \|_{L^{\frac{6}{5}}(\Omega)} \| \nabla \phi \|_{L^{\infty}(\Omega)}
    + \|  \partial_{t}  \mu_\psi  - \overline{\partial_{t}  \mu_\psi }  \|_{L^{6}(\Omega)} \|  \partial_{t} \bm{u}  \|_{L^{\frac{6}{5}}(\Omega)} \| \nabla \psi \|_{L^{\infty}(\Omega)} \nonumber \\
    & \qquad + \|  \partial_{t}  \mu_\phi  - \overline{\partial_{t}  \mu_\phi }  \|_{L^{6}(\Omega)} \| \bm{u} \|_{L^{3}(\Omega)} \| \nabla \partial_{t} \phi \|
    + \|  \partial_{t}  \mu_\psi  - \overline{\partial_{t}  \mu_\psi }  \|_{L^{6}(\Omega)} \| \bm{u} \|_{L^{3}(\Omega)} \| \nabla \partial_{t} \psi \|
    \nonumber \\
    & \qquad + C_{\delta} \big(  \| \partial_{t} \phi \| + \| \partial_{t} \psi \|  \big) \big(  \| \partial_{t} ^2 \big( \phi - \overline{\phi} \big)  \| + \| \partial_{t} ^2 \psi  \|  \big)
    \nonumber \\
    & \quad \leq  \frac{1}{2} \int_{\Omega}  m_\phi ( \phi ) \big| \nabla \partial_{t} \mu_\phi \big|^2  \, \mathrm{d}x
    + \frac{1}{2} \int_{\Omega}  m_\psi ( \psi ) \big| \nabla \partial_{t} \mu_\psi \big|^2  \, \mathrm{d}x
    + C \| \mu_\phi \|_{W}^2 \| \nabla \partial_{t} \phi \|^2
    + C \| \mu_\phi \|_{W}^2 \big| \partial_{t} \overline{\phi} \big|^2
    \\
    & \qquad  + C \| \mu_\psi \|_{W}^2 \| \nabla \partial_{t} \psi \|^2
    + C \| \partial_{t} \bm{u} \|^2_{L^{\frac{6}{5}}(\Omega)} \big(  \| \nabla \phi \|_{H^{2}(\Omega)}^2 + \| \nabla \psi \|_{H^{2}(\Omega)}^2  \big)
    \\
    & \qquad + C \| \bm{u} \|_{L^{3}(\Omega)}^2 \big( \| \nabla \partial_{t} \phi \|^2 + \| \nabla \partial_{t} \psi \|^2 \big)
    + C_{\delta} \big(  \| \partial_{t} \phi \|^2 + \| \partial_{t} \psi \|^2  \big)
    + \frac{\alpha}{2} \big(  \| \partial_{t} ^2 \big( \phi - \overline{\phi} \big)  \|^2 + \| \partial_{t} ^2 \psi \|^2  \big),
\end{align*}
where the constant $C_\delta>0$ depends on $\delta_\phi$ and $\delta_\psi$.
In addition, by the standard elliptic regularity theory, from equations \eqref{eq:cvcho2} and \eqref{eq:cvcho4}, we infer that
\begin{align*}
    \| \nabla \phi \|_{H^{2}(\Omega)} & \le C(\| \nabla \Delta \phi \| + \|\phi\|_V)
    \\
    & \le C\| \nabla \mu_\phi \| + C \alpha \| \nabla \partial_{t} \phi \| + C_{\delta} \| \phi \|_V + C \| \nabla \psi \|,
    \\
    \beta \| \nabla \psi \|_{H^{2}(\Omega)}
    & \le C \beta (\| \nabla \Delta \psi \| +\|\psi\|_V) \\
    & \le C \| \nabla \mu_\psi \| + C \alpha \| \nabla \partial_{t} \psi \| + C_{\delta} \| \psi \|_V + C \| \nabla \phi \|.
\end{align*}
Combining the above estimates, we get
\begin{align*}
    & \frac{\mathrm{d}}{\mathrm{d t}}  \big(  \|  \nabla \partial_{t} \phi  \|^2 + \beta \|  \nabla \partial_{t} \psi  \|^2 + \sigma_2 \|  \partial_{t} \phi  - \partial_{t} \overline{\phi}  \|^2_*  \big)  + \alpha \big\|  \partial_{t} ^2 \big(\phi  - \overline{ \phi}\big)  \big\|^2
    + \alpha \|  \partial_{t} ^2 \psi  \|^2
    \nonumber \\
    & \qquad + \int_{\Omega} \big( m_\phi ( \phi ) | \nabla \partial_{t} \mu_\phi |^2 + m_\psi ( \psi ) | \nabla \partial_{t} \mu_\psi |^2 \big) \, \mathrm{d}x \nonumber \\
    & \quad \le C H_1 \big(  \|  \nabla \partial_{t} \phi  \|^2 + \beta \|  \nabla \partial_{t} \psi  \|^2  \big) + C H_2,
\end{align*}
with
\begin{align*}
    & H_1 = \| \mu_\phi \|^2_W + \| \mu_\psi \|^2_W + \| \partial_{t} \bm{u} \|^2_{L^{\frac{6}{5}}(\Omega)} + \| \bm{u} \|_{L^{3}(\Omega)}^2, \nonumber \\
    & H_2 = \big[  \| \nabla \mu_\phi \|^2 + \| \nabla \mu_\psi \|^2 + C_{\delta} ( \| \phi \|_V^2 + \| \psi \|_V^2 )  \big] \| \partial_{t} \bm{u} \|^2_{L^{\frac{6}{5}}(\Omega)}
    \nonumber \\
    & \qquad\  + \| \mu_\phi \|^2_W \big| \partial_{t} \overline{\phi} \big|^2 + C_{\delta} \big(  \| \partial_{t} \phi \|^2 + \| \partial_{t} \psi \|^2  \big).
    \nonumber
\end{align*}
Hence, applying Gronwall's lemma on any finite interval (and the uniform Gronwall inequality on $[1,\infty)$ if $T_r=\infty$), we can conclude that the $L^{\infty} ( 0,T_r;V )$ and $H^{1}_{\mathrm{uloc}} ( 0,T_r;H )$-norms of $\partial_{t} \phi$ and $\partial_{t} \psi$, as well as the $H^{1}_{\mathrm{uloc}} ( 0,T_r;V )$-norms of $\mu_\phi$ and $\mu_\psi$, are bounded by the $L^{\infty} ( I;\bm{L}^{3}_\sigma(\Omega) )$-norm of $\bm{u}$ and $L^2 ( I;\bm{L}^{\frac{6}{5}}(\Omega) )$-norm of $\partial_t\bm{u}$.
Consequently, by standard compactness arguments, we obtain the desired result for the velocity field
$\bm{u}$ under the stated assumptions.
The $L^{2}_{\mathrm{uloc}} ( 0,T_r;H^{2}(\Omega) )$-bound of $\partial_{t} \phi$ follows from the comparison of the terms in \eqref{eq:cvcho2tdif} and similarly for $\partial_{t} \psi$ from \eqref{eq:cvcho4tdif}.
Finally, using the strict separation property of $\phi$ (see Proposition \ref{cvchosse}-$(\mathbf{R2})$) together with the established facts that $\mu_\phi \in L^\infty ( 0,T_r;W )$ (see Proposition \ref{cvchosse}-$(\mathbf{R1})$) and $\partial_t \phi\in L^\infty(0,T_r;V)$, we can apply the elliptic estimate to equation \eqref{eq:cvcho2} and obtain $\phi\in L^\infty(0,T_r;H^3_N(\Omega))$. Similarly, the property $\psi\in L^\infty(0,T_r;H^3_N(\Omega))$ follows by invoking \eqref{eq:cvcho4}. This completes the proof.
\qed

\medskip

Next, we prove the local well-posedness of the approximate system \eqref{eq:semi_galerkin_nsch1}--\eqref{eq:semi_galerkin_nsch5} in three dimensions subject to the boundary condition \eqref{eq:semi_galerkin_nschb} and generic regular initial data.
\begin{proposition} \label{local_approximate_fixpoint}
Let $\Omega \subset \mathbb{R}^{3}$ be a bounded domain of class $C^3$.
Suppose that $\alpha\in (0,1)$, $m\in \mathbb{N}_+$, Assumptions $(\mathbf{H0})$--$(\mathbf{H2})$ are satisfied, $\bm{u}_0 \in \mathbf{V}_m$ and $\phi_0,\, \psi_0 \in H_N^{3}(\Omega)$ such that $(  \phi_0,\psi_0 ) \in [-1+\delta,1-\delta] \times [\delta,1-\delta]$ in $\Omega$ for some constant $\delta \in (0,1/2)$. Then there exists some time $T_\alpha>0$ depending on $\alpha$, $m$ such that the approximate problem \eqref{eq:semi_galerkin_nsch1}--\eqref{eq:semi_galerkin_nschb} subject to the initial condition
\begin{equation}
\label{icappr}
( \bm{u}^m,\phi^m,\psi^m )|_{t=0} = ( \bm{u}_0,\phi_{0},\psi_{0} ),
\end{equation}
admits a unique solution $( \bm{u}^m,\phi^m,\psi^m,\mu_\phi^m,\mu_\psi^m )$ on $[0,T_\alpha]$ that satisfies the following regularity properties
\begin{align*}
    & \bm{u}^m \in H^1 ( 0,T_\alpha;\mathbf{V}_m ),\quad  \phi^m,\psi^m \in L^{\infty} ( 0,T_{\alpha};H^{3}_N(\Omega) ),
    \\
    & ( \phi^m,\psi^m ) \in [-1+\widetilde{\delta},1-\widetilde{\delta}] \times [\widetilde{\delta},1-\widetilde{\delta}] \ \mathrm{a.e.} \ \mathrm{in} \ \Omega \times (0,T_{\alpha}),
    \\
    &\partial_{t} \phi^m, \partial_{t} \psi^m \in L^{\infty} ( 0,T_{\alpha};V ) \cap L^{2} ( 0,T_{\alpha};W ) \cap H^1 ( 0,T_{\alpha};H ),
    \\
    & \mu_\phi^m,\mu_\psi^m \in L^{\infty} ( 0,T_{\alpha};W ) \cap H^1 ( 0,T_{\alpha};V ),
\end{align*}
for some $\widetilde{\delta}\in (0,\delta]$ depending on $\alpha$ and $\delta$.
\end{proposition}

\noindent \textbf{Proof.} For any given finite-dimensional ansatz $\bm{v}^m \in H^1(0,T_0;\mathbf{V}_m)$, where $T_0 \in (0,1]$ is a fixed time to be determined later, we consider the following convective viscous Cahn--Hilliard system:
\begin{numcases}{}
    \partial_{t} \phi^m + \bm{v}^m \cdot \nabla \phi^m + \sigma_1 \big( \overline{\phi^m} -c \big) = \mathrm{div}\, \big(m_\phi (\phi^m) \nabla \mu_\phi^m\big),
    \label{eq:fixedpointch1}\\
    \mu_\phi^m = \alpha \partial_{t} \phi^m -\Delta \phi^m + \sigma_2 \mathcal{N} \big( \phi^m - \overline{\phi^m} \big) + F_\phi'(\phi^m) + \partial_\phi G(\phi^m,\psi^m), \label{eq:fixedpointch2} \\
    \partial_{t} \psi^m + \bm{v}^m \cdot \nabla \psi^m = \mathrm{div}\, \big(m_\psi (\psi^m) \nabla \mu_\psi^m\big),
    \label{eq:fixedpointch3} \\
    \mu_\psi^m = \alpha \partial_{t} \psi^m -\beta \Delta \psi^m + F_\psi'(\psi^m) + \partial_\psi G(\phi^m,\psi^m),
    \label{eq:fixedpointch4}
\end{numcases}
almost everywhere in $\Omega \times (0,T_0)$, subject to the boundary and initial conditions
\begin{numcases}{}
    \partial_{\mathbf{n}} \phi^m = \partial_{\mathbf{n}} \mu_\phi^m = \partial_{\mathbf{n}} \psi^m = \partial_{\mathbf{n}} \mu_\psi^m = 0, \label{eq:fixedpointchb} & a.e. on $\partial \Omega \times (0,T_0)$,
    \\
    ( \phi^m,\psi^m) |_{t=0}= ( \phi_{0},\psi_{0} ), \label{eq:fixedpointchi} & a.e. in $\Omega$.
\end{numcases}
By the definition of $\bm{V}_m$, we have $\bm{v}^m \in H^1(0,T_0;\bm{H}_{0,\sigma}^3(\Omega))$. This  enables us to apply Proposition \ref{cvchossr} to find a time $\widehat{T}_0 \in (0,T_0]$ depending on $\bm{v}^m$, norms of the initial data, $\alpha$ and other parameters of the system \eqref{eq:fixedpointch1}--\eqref{eq:fixedpointch4} such that there exists a unique strong solution $( \phi^m,\psi^m )$ to problem \eqref{eq:fixedpointch1}--\eqref{eq:fixedpointchi} on $[0,\widehat{T}_0]$ satisfying
\begin{align*}
    & \phi^m,\psi^m \in L^{\infty} ( 0,\widehat{T}_0;H^{3}_N(\Omega) ),
    \\
    & ( \phi^m,\psi^m ) \in [-1+\widetilde{\delta},1-\widetilde{\delta}] \times [\widetilde{\delta},1-\widetilde{\delta}] \quad \mathrm{a.e.} \ \mathrm{in} \ \Omega \times (0,\widehat{T}_0),
    \\
    &\partial_{t} \phi^m, \partial_{t} \psi^m \in L^{\infty} ( 0,\widehat{T}_0;V \big) \cap L^{2} \big( 0,\widehat{T}_0;W ) \cap H^1 ( 0,\widehat{T}_0;H ),
    \\
    & \mu_\phi^m,\mu_\psi^m \in L^{\infty} ( 0,\widehat{T}_0;W ) \cap H^1 ( 0,\widehat{T}_0;V ),
\end{align*}
where $\widetilde{\delta} \in (0,\delta]$ depends on $\alpha$ and $\delta$. Moreover, we can derive the following energy estimate by direct calculations (cf. \cite{AGG2024} for the treatment of the convection term and \cite{HW2024} for the term involving mass variation):
\begin{align*}
    & \frac{\mathrm{d}}{\mathrm{d t}}  \widetilde{E}_{\mathrm{free}}( \phi^m,\psi^m )  +  \frac{\alpha}{2} \big(  \| \partial_{t} \phi^m \|^2 + \| \partial_{t} \psi^m \|^2  \big)
    + \frac{1}{2} \int_{\Omega} \big( m_\phi ( \phi^m ) |  \nabla \mu_\phi^m  |^2 + m_\psi ( \psi^m ) |  \nabla \mu_\psi^m  |^2 \big) \, \mathrm{d}x
    \\
    & \quad \le C_1 \sigma_1 e^{-\sigma_1 t} + \frac{2}{m_0} \| \bm{v}^m \|^2.
\end{align*}
Here, $$
m_0=\min\{m_{\phi,*},m_{\psi,*}\},\quad \widetilde{E}_{\mathrm{free}}(\phi, \psi) \define E_{\mathrm{free}}(\phi, \psi) + \| G \|_{C(\mathbb{R}^2)},
$$
and the constant $C_1>0$ depends only on $\sigma_1$, $c$, $\overline{\phi_0}$ and $m_{\phi,*}$.

Given the pair $(\phi^m, \psi^m)$, we denote by $\bm{u}^m \in H^1 ( 0,\widehat{T}_0;\mathbf{V}_m )$ the unique solution to the following linear (ODE) system:
\begin{align}
    & \big(  \rho ( \phi^m ) \partial_{t}  \bm{u}^m  , \bm{\eta} \big)
    + \big(  \rho ( \phi^m ) (  \bm{v}^m \cdot \nabla ) \bm{u}^m , \bm{\eta} \big)
    + \big(  \nu (  \phi^m  ) D \bm{u}^m , \nabla \bm{\eta} \big)
    - \gamma \big(  m_\phi ( \phi^m ) (  \nabla \mu_\phi^m \cdot \nabla  ) \bm{u}^m , \bm{\eta} \big) \nonumber \\
    & \quad = \gamma \sigma_1 \big(  \overline{\phi^m} -c  \big) \big(  \bm{u}^m , \bm{\eta} \big)
    + \big(  \mu_\phi^m \nabla \phi^m + \mu_\psi^m \nabla \psi^m , \bm{\eta} \big),
    \quad  \forall\, \bm{\eta} \in \mathbf{V}_m, \ t \in [0,\widehat{T}_0],
    \label{eq:v_to_u}
\end{align}
subject to the initial condition
\[
\bm{u}^m |_{t=0} = \bm{u}_0(x), \quad \text{in} \ \Omega.
\]
The existence of the finite-dimensional ansatz $\bm{u}^m$ follows from the same argument as in \cite{G2022}. Taking $\bm{\eta} = \bm{u}^m$ in \eqref{eq:v_to_u} and applying integration by parts, we obtain
\begin{align*}
    & \frac{1}{2} \frac{\mathrm{d}}{\mathrm{d t}}   \int_{\Omega}  \rho ( \phi^m)  | \bm{u}^m |^2  \, \mathrm{d}x
    + \int_{\Omega}  \nu ( \phi^m ) |  D \bm{u}^m  |^2  \, \mathrm{d}x
    - \frac{\sigma_1}{2} \gamma \big(  \overline{\phi^m} - c  \big) \| \bm{u}^m \|^2
    \\
    & \qquad - \int_{\Omega}  \big[  \mathrm{div}\, (  \rho( \phi^m ) \bm{v}^m  ) + \partial_{t} \rho( \phi^m ) + \mathrm{div}\, \mathbf{J}^m + \sigma_1 \gamma \big(  \overline{ \phi^m } -c  \big)  \big] \frac{1}{2} | \bm{u}^m |^2  \, \mathrm{d}x
    \\
    & \quad = \big(  \mu_\phi^m \nabla \phi^m + \mu_\psi^m \nabla \psi^m , \bm{u}^m \big),
\end{align*}
where $\mathbf{J}^m = - \gamma m_\phi (  \phi^m  ) \nabla \mu_\phi^m$. It follows from \eqref{eq:fixedpointch1} that
\[
 \partial_{t} \rho( \phi^m )+ \mathrm{div}\, (  \rho( \phi^m ) \bm{v}^m  )  + \mathrm{div}\, \mathbf{J}^m + \sigma_1 \gamma \big(  \overline{ \phi^m } -c  \big) = 0.
\]
With the above identity, we can deduce that
\begin{align*}
    & \frac{1}{2} \frac{\mathrm{d}}{\mathrm{d t}}   \int_{\Omega}  \rho ( \phi^m ) | \bm{u}^m |^2  \, \mathrm{d}x
    + \int_{\Omega}  \nu ( \phi^m ) |  D \bm{u}^m  |^2  \, \mathrm{d}x
    \\
    & \quad = \frac{\sigma_1}{2} \gamma \big(  \overline{\phi^m} - c  \big) \| \bm{u}^m \|^2
    - \big(  \phi^m \nabla \mu_\phi^m + \psi^m \nabla \mu_\psi^m , \bm{u}^m \big)
    \\
    & \quad \le | \gamma | \sigma_1 e^{-\sigma_1 t} \| \bm{u}^m \|^2
    + \big(  \| \nabla \mu_\phi^m \| + \|  \nabla \mu_\psi^m  \| \big ) \| \bm{u}^m \|
    \\
    & \quad \le \frac{| \gamma |}{\rho_*} \sigma_1 e^{-\sigma_1 t} \int_{\Omega}  \rho ( \phi^m ) | \bm{u}^m |^2  \, \mathrm{d}x
    + \frac{1}{2} \int_{\Omega}  \nu ( \phi^m ) |  D \bm{u}^m  |^2  \, \mathrm{d}x
    + \frac{2}{\lambda_1 \nu_*} \big(  \| \nabla \mu_\phi^m \|^2 + \|  \nabla \mu_\psi^m  \|^2  \big),
\end{align*}
where in the last line we have used Poincar\'{e}'s and Korn's inequalities for $\bm{u}^m$. Applying Gronwall's lemma, we obtain
\begin{align*}
    \rho_* \| \bm{u}^m(t) \|^2
    & \le \int_{\Omega}  \rho ( \phi^m(t) ) | \bm{u}^m(t) |^2  \, \mathrm{d}x
    \\
    & \le e^{\frac{2| \gamma |}{\rho_*}} \left(  \int_{\Omega}  \rho( \phi_0 ) | \bm{u}_0 |^2  \, \mathrm{d}x
    + \frac{2}{\lambda_1 \nu_* m_0} \int_{0}^{t}  m_\phi (  \phi^m  ) | \nabla \mu_\phi^m |^2 + m_\psi (  \psi^m  ) | \nabla \mu_\psi^m |^2 \, \mathrm{d}\tau  \right)
    \\
    & \le e^{\frac{2| \gamma |}{\rho_*}} \left(  \rho^* \| \bm{u}_0 \|^2 + \frac{4 C_1 + 4\widetilde{E}_{\mathrm{free}} (  \phi_0,\psi_0  ) }{\lambda_1 \nu_* m_0} + \frac{8}{\lambda_1 \nu_* m_0^2} \int_{0}^{t} \| \bm{v}^m(\tau) \|^2 \,\mathrm{d}\tau  \right).
\end{align*}
This gives
\[
\|  \bm{u}^m(t)  \|^2 \le C_2 + C_3 \int_{0}^{t} \| \bm{v}^m (\tau)\|^2 \, \mathrm{d}\tau,
\quad \forall\, t \in [0,\widehat{T}_0],
\]
where
\[
C_2 = e^{\frac{2| \gamma |}{\rho_*}} \left(  \frac{\rho^*}{\rho_*} \| \bm{u}_0 \|^2 + \frac{4 C_1 + 4\widetilde{E}_{\mathrm{free}} (  \phi_0,\psi_0  )}{\lambda_1 \nu_* m_0 \rho_*}  \right),
\quad C_3 = \frac{8 e^{\frac{2| \gamma |}{\rho_*}} }{\lambda_1 \nu_* m_0^2 \rho_*}.
\]
We assume that $\bm{v}^m$ satisfies
\[
\int_{0}^{t} \| \bm{v}^m (\tau)\|^2 \,\mathrm{d}\tau \le C_2 e^{C_3 t}, \quad \forall\, t \in [0,T_0],
\]
then it holds
\begin{equation*}
    \int_{0}^{t}  \| \bm{u}^m (\tau)\|^2 \, \mathrm{d}\tau \le C_2 t + C_3\int_{0}^{t} C_2   e^{C_3 \tau}\,  \mathrm{d}\tau = C_2(t-1) + C_2 e^{C_3 t}
    \le C_2 e^{C_3 t},\quad \forall\, t \in [0,\widehat{T}_0],
\end{equation*}
where we have used the fact that $\widehat{T}_0\leq 1$. Furthermore, we have
\[
\sup\limits_{t \in [0,\widehat{T}_0]} \| \bm{u}^m(t) \| \le \big(  C_2 + C_2 C_3 e^{C_3}  \big)^\frac{1}{2}=: K_0.
\]
Next, we derive estimates on the $L^{2} ( 0,\widehat{T}_0;\mathbf{V}_m )$-norm of $\partial_{t} \bm{u}^m$. As in \cite{G2022}, taking $\bm{\eta} = \partial_{t} \bm{u}^m$ in \eqref{eq:v_to_u} and applying integration by parts, we obtain
\begin{align*}
    \rho_* \| \partial_{t} \bm{u}^m \|^2 & = \gamma \big(  m_\phi ( \phi^m ) (  \nabla \mu_\phi^m \cdot \nabla  ) \bm{u}^m ,  \partial_{t} \bm{u}^m  \big)
    + \gamma \sigma_1 \big( \overline{\phi^m}-c \big) \big(  \bm{u}^m ,  \partial_{t} \bm{u}^m  \big)
    \\
    & \quad - \big(  \rho ( \phi^m ) (  \bm{v}^m \cdot \nabla  ) \bm{u}^m ,  \partial_{t} \bm{u}^m  \big)
    - \big(  \nu ( \phi^m ) D \bm{u}^m ,  \nabla \partial_{t} \bm{u}^m  \big)
    - \big(  \phi^m \nabla \mu_\phi^m + \psi^m \nabla \mu_\psi^m ,  \partial_{t} \bm{u}^m  \big)
    \\
    & \le | \gamma | m_\phi^* \| \nabla \mu_\phi^m \| \| \nabla \bm{u}^m \|_{L^{\infty}(\Omega)} \|  \partial_{t} \bm{u}^m  \| +  2 | \gamma | \sigma_1 e^{-\sigma_1 t} \| \bm{u}^m \| \| \partial_{t} \bm{u}^m \|  \\
    & \quad + \rho^* \| \bm{v}^m \| \|  \nabla \bm{u}^m  \|_{L^{\infty}(\Omega)} \|  \partial_{t} \bm{u}^m  \| + \nu^* \| D \bm{u}^m \| \|  \nabla \partial_{t} \bm{u}^m  \| + (  \| \nabla \mu_\phi^m \| + \| \nabla \mu_\psi^m \|  ) \| \partial_{t} \bm{u}^m \| \\
    & \le \big(  C_m | \gamma | m_\phi^* \| \nabla \mu_\phi^m \| + 2 | \gamma | \sigma_1 e^{-\sigma_1 t} + C_m \rho^* \| \bm{v}^m \| + C_m^2 \nu^*  \big) \| \bm{u}^m \| \| \partial_{t} \bm{u}^m \| \\
    & \quad + \big(  \| \nabla \mu_\phi^m \| + \| \nabla \mu_\psi^m \|  \big) \| \partial_{t} \bm{u}^m \|,
\end{align*}
where $C_m>0$ is the constant in the inverse Sobolev embedding inequalities for $\bm{u}^m$:
$$
\|\bm{u}^m\|_{V}\leq C_m \|\bm{u}^m\|,\quad \|\bm{u}^m\|_{H^2(\Omega)}\leq C_m \|\bm{u}^m\|,
\quad \|\bm{u}^m\|_{H^3(\Omega)}\leq C_m\|\bm{u}^m\|.
$$
$C_m$ depends only on the dimension $m$ and the domain $\Omega$.
Integrating with respect to time and using Young's inequality, we find that
\begin{align*}
    \int_{0}^{\widehat{T}_0}  \| \partial_{t} \bm{u}^m(\tau) \|^2 \, \mathrm{d}\tau
    & \le \frac{16 K_0^2}{\rho_*^2m_0}   C_m^2 \gamma^2 (m_\phi^*)^2 \Big(  C_1 +  \widetilde{E}_{\mathrm{free}}( \phi_0,\psi_0 )
    + \frac{2}{m_0} C_2 e^{C_3}   \Big)
    \\
    &\quad + \frac{16 K_0^2}{\rho_*^2} \big[ 4\gamma^2 (\sigma_1)^2 + C_m^2 ( \rho^* )^2 C_2 e^{C_3} +   C_m^4 ( \nu^* )^2 \big]
    \\
    & \quad + \frac{8}{\rho_*^2m_0} \Big(  C_1 +  \widetilde{E}_{\mathrm{free}}( \phi_0,\psi_0 ) + \frac{2}{m_0} C_2 e^{C_3}   \Big)=: K_1^2.
\end{align*}

In view of the above estimates, we take
\[
T_0 \define \min \big\{  1, \ T_1(C_m K_0, C_m K_1  ), \ T_2(C_m K_0, C_m K_1  )  \big\},
\]
where $T_1,\,T_2>0$ can be determined explicitly as in Proposition \ref{cvchosse}-(\textbf{R1}). Then we define
\begin{align*}
    S \define  \Big\{  & \bm{u} \in H^1 ( 0,T_0;\mathbf{V}_m ): \| \bm{u} \|_{L^{\infty} ( 0,T_0;L^{2}(\Omega) )} \le K_0, \ \|  \partial_{t} \bm{u}  \|_{L^{2} ( 0,T_0;L^{2}(\Omega) )} \le K_1 , \\
    &\quad \int_{0}^{t} \| \bm{u}(\tau) \|^2\, \mathrm{d}\tau \le C_2 e^{C_3 t}, \ \forall\, t \in [0,T_0]  \Big\},
\end{align*}
which is a closed and convex subset of $L^{2} ( 0,T_0;\mathbf{V}_m )$.
For any $\bm{v}^m\in S$, Proposition \ref{cvchossr} and Remark \ref{length_cvcho} ensure the existence of a unique strong solution $(\phi^m,\psi^m)$ to problem \eqref{eq:fixedpointch1}--\eqref{eq:fixedpointchi} on $[0,T_0]$, that is, we can take $\widehat{T}_0=T_0$.
Consequently, we determine a unique solution $\bm{u}^m \in S$ to the linear ODE system \eqref{eq:v_to_u}.
This allows us to define the mapping
$$
\Lambda:S \to S\quad \text{such that}\quad \Lambda(\bm{v}^m) = \bm{u}^m.
$$
Observing that $S$ is a compact set in $L^{2} ( 0,T_0;\mathbf{V}_m )$, to show the existence of solutions to the system \eqref{eq:semi_galerkin_nsch1}--\eqref{eq:semi_galerkin_nsch5} subject to \eqref{eq:semi_galerkin_nschb} and \eqref{icappr} by Schauder's fixed-point theorem, it remains to verify the continuity of $\Lambda$.

The continuity of $\Lambda$ and the uniqueness of solutions with the regularity listed in Proposition \ref{local_approximate_fixpoint} can be proved simultaneously. Let $i \in \{ 1,2 \}$.
Given functions $\bm{v}_i\in  H^1 ( 0,T;\mathbf{V}_m )$, $\bm{u}_{0,i} \in \mathbf{V}_m$, $\phi_{0,i},\psi_{0,i} \in H^{3}_N(\Omega)$ with $(  \phi_{0,i},\psi_{0,i}  ) \in [-1+\delta,1-\delta] \times [\delta,1-\delta]$  in $\Omega$ for some $\delta >0$, we suppose that the corresponding solutions $( \bm{u}_i,\phi_i,\psi_i,\mu_{\phi_i},\mu_{\psi_i} )$ exist on the common time interval $[0,T]$ with the following regularity
\begin{align*}
    & \bm{u}_i \in H^1 ( 0,T;\mathbf{V}_m ),\quad  \phi_i,\psi_i \in L^{\infty} ( 0,T;H^{3}(\Omega) ),
    \\
    & ( \phi_i,\psi_i ) \in [-1+\delta',1-\delta'] \times [\delta',1-\delta'] \quad \mathrm{a.e.} \ \mathrm{in} \ Q_T,
    \\
    &\partial_{t} \phi_i, \partial_{t} \psi_i \in L^{\infty} ( 0,T;V ) \cap L^{2} ( 0,T;W ) \cap H^1 ( 0,T;H ),
    \\
    & \mu_{\phi_i},\mu_{\psi_i} \in L^{\infty} ( 0,T;W ) \cap H^1 ( 0,T;V ),
\end{align*}
for some $\delta' \in (0,\delta]$ depending on $\alpha$ and $\delta$, moreover,
\begin{numcases}{}
    \big( \rho( \phi_i ) \partial_{t} \bm{u}_i, \bm{\eta} \big)
    + \big( \rho( \phi_i ) ( \bm{v}_i \cdot \nabla ) \bm{u}_i , \bm{\eta} \big)
    + \big(  \nu( \phi_i ) D \bm{u}_i , \nabla \bm{\eta} \big)
    \nonumber\\
    \quad - \gamma \big(  m_\phi(\phi_i)( \nabla \mu_{\phi_i} \cdot \nabla ) \bm{u}_i + \sigma_1 \big(  \overline{\phi_i} - c  \big) \bm{u}_i , \bm{\eta} \big)
    = \big( \mu_{\phi_i} \nabla \phi_i + \mu_{\psi_i} \nabla \psi_i, \bm{\eta} \big),
    \nonumber & $ \forall\, \bm{\eta} \in \mathbf{V}_m$, a.e. in $(0,T)$,
    \\
    \partial_{t} \phi_i + \bm{v}_i \cdot \nabla \phi_i + \sigma_1 \big( \overline{\phi_i} -c \big)
    = \mathrm{div}\, \big(m_\phi (\phi_i) \nabla \mu_{\phi_i}\big),
    \nonumber & a.e. in $Q_T$,
    \\
    \mu_{\phi_i} = \alpha \partial_{t} \phi_i -\Delta \phi_i + \sigma_2 \mathcal{N} \big( \phi_i - \overline{\phi_i} \big) + F_\phi'(\phi_i) + \partial_\phi G(\phi_i,\psi_i),
    \nonumber & a.e. in $Q_T$,
    \\
    \partial_{t} \psi_i + \bm{v}_i \cdot \nabla \psi_i
    = \mathrm{div}\, \big(m_\psi (\psi_i) \nabla \mu_{\psi_i}\big),
    \nonumber & a.e. in $Q_T$,
    \\
    \mu_{\psi_i} = \alpha \partial_{t} \psi_i -\beta \Delta \psi_i + F_\psi'(\psi_i) + \partial_\psi G(\phi_i,\psi_i),
    \nonumber & a.e. in $Q_T$,
\end{numcases}
and
\begin{numcases}{}
    \bm{u}_i =\mathbf{0}, \ \partial_{\mathbf{n}} \phi_i = \partial_{\mathbf{n}} \mu_{\phi_i} = \partial_{\mathbf{n}} \psi_i = \partial_{\mathbf{n}} \mu_{\psi_i} = 0,
    \nonumber & a.e. on $S_T$,
    \\
    \bm{u}_i (\cdot , 0) = \bm{u}_{0,i}, \ ( \phi_i(\cdot , 0),\psi_i(\cdot , 0) ) = ( \phi_{0,i},\psi_{0,i} ), \nonumber & a.e. in $\Omega$.
\end{numcases}
Set
$$
\big( \widetilde{\bm{u}},\widetilde{\phi},\widetilde{\psi},\widetilde{\mu_\phi},\widetilde{\mu_\psi} \big)
= ( \bm{u}_1,\phi_1,\psi_1,\mu_{\phi_1},\mu_{\psi_1} )
- ( \bm{u}_2,\phi_2,\psi_2,\mu_{\phi_2},\mu_{\psi_2} ),\quad  \widetilde{\bm{v}} = \bm{v}_1 - \bm{v}_2.
$$
We observe that
\begin{align}
    & \big( \rho( \phi_1 ) \partial_{t} \widetilde{\bm{u}}, \bm{\eta} \big)
    + \big(  (  \rho ( \phi_1 ) - \rho ( \phi_2 )  ) \partial_{t} \bm{u}_2 , \bm{\eta} \big) + (  \nu ( \phi_1 ) D \widetilde{\bm{u}} , \nabla \bm{\eta} )
    + \big(  (  \nu ( \phi_1 ) - \nu ( \phi_2 )  ) D \bm{u}_2 , \nabla \bm{\eta} \big)
    \nonumber \\
    & \qquad - \gamma \sigma_1 \big(  \overline{\phi_1} - c  \big) \big(  \widetilde{\bm{u}} , \bm{\eta} \big)
    - \gamma \sigma_1 \overline{\widetilde{\phi}} ( \bm{u}_2, \bm{\eta} )
    + \big(  \rho ( \phi_1 ) (  \bm{v}_1 \cdot \nabla  ) \widetilde{\bm{u}} , \bm{\eta} \big)
    + \big(  \rho ( \phi_1 ) (  \widetilde{\bm{v}} \cdot \nabla  ) \bm{u}_2 , \bm{\eta} \big)
    \nonumber\\
    & \qquad + \big(  (  \rho ( \phi_1 ) - \rho ( \phi_2 )  ) ( \bm{v}_2 \cdot \nabla ) \bm{u}_2 , \bm{\eta} \big)
    - \gamma \big(  m_\phi ( \phi_1 ) (  \nabla \mu_{\phi_1} \cdot \nabla  ) \widetilde{\bm{u}} , \bm{\eta} \big)
    \nonumber \\
    & \qquad - \gamma \big(  ( m_\phi( \phi_1 )-m_\phi( \phi_2 ) ) ( \nabla \mu_{\phi_1} \cdot \nabla ) \bm{u}_2 , \bm{\eta} \big)
    - \gamma \big(  m_\phi ( \phi_2 ) (  \nabla \widetilde{\mu_\phi} \cdot \nabla  ) \bm{u}_2 , \bm{\eta} \big)
    \nonumber \\
    & \quad= \big(  \widetilde{\mu_\phi} \nabla \phi_1 + \widetilde{\mu_\psi} \nabla \psi_1 , \bm{\eta} \big)
    + \big(  \mu_{\phi_2} \nabla \widetilde{\phi} + \mu_{\psi_2} \nabla \widetilde{\psi} , \bm{\eta} \big), \quad
    \forall\, \bm{\eta} \in \mathbf{V}_m,\ \text{a.e. in}\  (0,T).
    \label{diff-um}
\end{align}
Taking $\bm{\eta} = \widetilde{\bm{u}}$ in \eqref{diff-um} and using integration by parts, with the following identity
\[
\big(  \partial_{t} \rho( \phi_1 )  + \sigma_1 \gamma \big( \overline{\phi_1}-c \big) + \bm{v}_1 \cdot \nabla \rho ( \phi_1 ) - \mathrm{div}\, (  \gamma m_\phi ( \phi_1 ) \nabla \mu_{\phi_1}  ) , | \widetilde{\bm{u}} |^2 \big) = 0,
\]
we can deduce that
\begin{align*}
    & \frac{1}{2} \frac{\mathrm{d}}{\mathrm{d t}}    \int_{\Omega}  \rho ( \phi_1 ) | \widetilde{\bm{u}} |^2  \, \mathrm{d}x
    + \int_{\Omega}  \nu ( \phi_1 ) |  D \widetilde{\bm{u}}  |^2  \, \mathrm{d}x
    \\
    & \quad \le | \gamma | \big\| \widetilde{\phi} \big\|_{L^{6}(\Omega)} \| \partial_{t} \bm{u}_2 \|_{L^{3}(\Omega)} \| \widetilde{\bm{u}} \| + C \| \widetilde{\phi} \| \| D \bm{u}_2 \|_{L^{\infty}(\Omega)} \|  \nabla \widetilde{\bm{u}}  \|
    \\
    & \qquad + | \gamma | \sigma_1 \| \widetilde{\bm{u}} \|^2 + | \gamma | \sigma_1 \big| \overline{\widetilde{\phi}} \big| \| \bm{u}_2 \| \| \widetilde{\bm{u}} \| + | \gamma | \big\| \widetilde{\phi} \big\| \| \bm{v}_1 \|_{L^{\infty}(\Omega)} \|  \nabla \bm{u}_2  \|_{L^{\infty}(\Omega)} \|  \widetilde{\bm{u}}  \|
    \\
    & \qquad + \rho^* \| \widetilde{\bm{v}} \| \|  \nabla \bm{u}_2  \|_{L^{\infty}(\Omega)} \| \widetilde{\bm{u}} \| + | \gamma | \| \widetilde{\phi} \|_{L^{6}(\Omega)} \| \bm{v}_2 \|_{L^{\infty}(\Omega)} \| \nabla \bm{u}_2 \|_{L^{3}(\Omega)} \| \widetilde{\bm{u}} \|
    \\
    & \qquad + | \gamma | \|  \nabla \bm{u}_2  \|_{L^{\infty}(\Omega)} \big(  C \big\| \widetilde{\phi} \big\| \| \nabla \mu_{\phi_1} \| \| \widetilde{\bm{u}} \|_{L^{\infty}(\Omega)} + m_\phi^* \|  \nabla \widetilde{\mu_\phi}  \| \| \widetilde{\bm{u}} \|  \big)
    \\
    & \qquad + \| \widetilde{\bm{u}} \| (  \|  \nabla \widetilde{\mu_\phi}  \| + \|  \nabla \widetilde{\mu_\psi}  \|  ) + \| \widetilde{\bm{u}} \|_{L^{\infty}(\Omega)} (  \| \widetilde{\phi} \| \| \nabla \mu_{\phi_2} \| + \| \widetilde{\psi} \| \| \nabla \mu_{\psi_2} \|  )
    \\
    & \quad \le C ( \big| \overline{\widetilde{\phi}} \big| + \big\| \nabla \widetilde{\phi} \big\| ) \big(  C_m  \| \partial_{t} \bm{u}_2 \| + C_m^2 \| \bm{u}_2 \| + C C_m^2 | \gamma | \| \bm{v}_1 \| \|  \bm{u}_2  \|  \big) \| \widetilde{\bm{u}} \|
    \\
    & \qquad + \big(  | \gamma | \sigma_1 \| \bm{u}_2 \| \big| \overline{\widetilde{\phi}} \big| + C_m \rho^* \| \widetilde{\bm{v}} \| \|  \bm{u}_2  \| + C | \gamma | C_m^2 ( \big| \overline{\widetilde{\phi}} \big| + \big\| \nabla \widetilde{\phi} \big\| ) \| \bm{v}_2 \| \| \bm{u}_2 \|  \big) \| \widetilde{\bm{u}} \|
    \\
    & \qquad + | \gamma | \sigma_1 \| \widetilde{\bm{u}} \|^2 + \big(  C | \gamma | C_m^2 \big\| \widetilde{\phi} \big\| \| \nabla \mu_{\phi_1}  \| \| \bm{u}_2 \| + C_m m_\phi^* \| \bm{u}_2 \| \| \nabla \widetilde{\mu_\phi} \|  \big) \| \widetilde{\bm{u}} \|
    \\
    & \qquad + \big(  \|  \nabla \widetilde{\mu_\phi}  \| + \|  \nabla \widetilde{\mu_\psi}  \| + C_m \| \widetilde{\phi} \| \| \nabla \mu_{\phi_2} \| + C_m \| \widetilde{\psi} \| \| \nabla \mu_{\psi_2} \|  \big) \| \widetilde{\bm{u}} \|
    \\
    & \quad \le C (  1 + \| \bm{u}_2 \|^2  ) \int_{\Omega}  \rho ( \phi_1 ) | \widetilde{\bm{u}} |^2  \, \mathrm{d}x
    \\
    & \qquad + C \big(  \| \partial_{t} \bm{u}_2 \|^2 + ( 1 + \| \bm{u}_2 \|^2 ) (  1 + \| \nabla \mu_{\phi_1} \|^2 + \| \nabla \mu_{\phi_2} \|^2 + \| \bm{v}_1 \|^2 + \| \bm{v}_2 \|^2  )  \big) \Big(  \big| \overline{\widetilde{\phi}} \big|^2 + \big\|  \nabla \widetilde{\phi}  \big\|^2  \Big)
    \\
    & \qquad + C \| \nabla \mu_{\psi_2} \|^2 \big\| \nabla \widetilde{\psi} \big\|^2 + C \| \widetilde{\bm{v}} \|^2  + \frac{1}{4} \int_{\Omega} \big( m_\phi ( \phi_1 ) | \nabla \widetilde{\mu_\phi} |^2 + m_\psi ( \psi_1 ) | \nabla \widetilde{\mu_\psi} |^2 \big) \,\mathrm{d}x.
\end{align*}
On the other hand, testing the following equation
\[
\partial_{t}  \widetilde{\phi}  + \widetilde{\bm{v}} \cdot \nabla \phi_1 + \bm{v}_2 \cdot \nabla \widetilde{\phi} + \sigma_1 \overline{\widetilde{\phi}}
= \mathrm{div}\, \big(  m_\phi ( \phi_1 ) \nabla \widetilde{\mu_\phi}  \big) + \mathrm{div}\, \big(  (  m_\phi ( \phi_1 ) - m_\phi ( \phi_2 )  ) \nabla \mu_{\phi_2} \big)
\]
with $\widetilde{\mu_\phi}$, we get
\begin{align*}
    & \frac{1}{2} \frac{\mathrm{d}}{\mathrm{d t}}  \big(  \| \nabla \widetilde{\phi} \|^2 + \sigma_2 \|  \widetilde{\phi} - \overline{\widetilde{\phi}}  \|^2_* \big)
    + \alpha \| \partial_{t} \widetilde{\phi} \|^2 + \int_{\Omega}  m_\phi ( \phi_1 ) | \nabla \widetilde{\mu_\phi} |^2  \, \mathrm{d}x
    \\
    & \quad \le \|  F_\phi''  \|_{C(  [-1+\delta',1-\delta']  )} \| \widetilde{\phi} \| \| \partial_{t} \widetilde{\phi} \| + \|  G  \|_{C^2(  \mathbb{R}^{2}  )} (  \| \widetilde{\phi} \| + \| \widetilde{\psi} \|  ) \| \partial_{t} \widetilde{\phi} \| + \| \widetilde{\bm{v}} \| \| \nabla \widetilde{\mu_\phi} \|
    \\
    & \qquad + C \| \bm{v}_2 \|_{L^{\infty}(\Omega)} \| \nabla \widetilde{\phi} \| \| \nabla \widetilde{\mu_\phi} \| + \sigma_1 \big| \overline{\widetilde{\phi}} \big| \big| \overline{\widetilde{\mu_\phi}} \big| + C \| \widetilde{\phi} \|_{L^{3}(\Omega)} \| \nabla \mu_{\phi_2} \|_{L^{6}(\Omega)} \| \nabla \widetilde{\mu_\phi} \|
    \\
    & \quad \le C \big(  1 + \| \nabla \mu_{\phi_2} \|_{L^{6}(\Omega)}^2 + C_m^2 \| \bm{v}_2 \|^2  \big) \big(  \| \nabla \widetilde{\phi} \|^2 + \big| \overline{\widetilde{\phi}} \big|^2  \big)
    + C \| \nabla \widetilde{\psi} \|^2 \\
    & \qquad + C \| \widetilde{\bm{v}} \|^2 + \frac{1}{4} \int_{\Omega}  m_\phi ( \phi_1 ) | \nabla \widetilde{\mu_\phi} |^2  \, \mathrm{d}x
    + \frac{\alpha}{2} \| \partial_{t} \widetilde{\phi} \|^2.
\end{align*}
Similarly, it holds
\begin{align*}
    & \frac{\beta}{2} \frac{\mathrm{d}}{\mathrm{d t}}   \| \nabla \widetilde{\psi} \|^2 + \alpha \| \partial_{t} \widetilde{\psi} \|^2
    + \int_{\Omega}  m_\psi ( \psi_1 ) | \nabla \widetilde{\mu_\psi} |^2  \, \mathrm{d}x
    \\
    & \quad \le C \big(  1 + \| \nabla \mu_{\psi_2} \|_{L^{6}(\Omega)}^2 + C_m^2 \| \bm{v}_2 \|^2  \big) \| \nabla \widetilde{\psi} \|^2 + C \| \nabla \widetilde{\phi} \|^2
    \\
    & \qquad + C \| \widetilde{\bm{v}} \|^2 + \frac{1}{4} \int_{\Omega}  m_\psi ( \psi_1 ) | \nabla \widetilde{\mu_\psi} |^2  \, \mathrm{d}x
    + \frac{\alpha}{2} \| \partial_{t} \widetilde{\psi} \|^2.
\end{align*}
Combining the above three inequalities, we obtain
\begin{align}
    & \frac{1}{2} \frac{\mathrm{d}}{\mathrm{d t}}  \Big(  \int_{\Omega} \rho ( \phi_1 ) | \widetilde{\bm{u}} |^2 \ \mathrm{d}x + \| \nabla \widetilde{\phi} \|^2 + \sigma_2 \| \widetilde{\phi}-\overline{\widetilde{\phi}} \|^2_* + \beta \| \nabla \widetilde{\psi} \|^2  \Big)
    + \frac{\alpha}{2} \big(  \| \partial_{t} \widetilde{\phi} \|^2 + \| \partial_{t} \widetilde{\psi} \|^2 \big)
    \nonumber \\
    & \qquad + \int_{\Omega}  \Big(  \nu ( \phi_1 ) | D \widetilde{\bm{u}} |^2 + \frac{1}{2} m_\phi ( \phi_1 ) | \nabla \widetilde{\mu_\phi} |^2 + \frac{1}{2} m_\psi ( \psi_1 ) | \nabla \widetilde{\mu_\psi} |^2  \Big)  \, \mathrm{d}x
    \nonumber \\
    & \quad \le C H_1 \Big(  \int_{\Omega} \rho ( \phi_1 ) | \widetilde{\bm{u}} |^2 \ \mathrm{d}x + \| \nabla \widetilde{\phi} \|^2 + \sigma_2 \| \widetilde{\phi}-\overline{\widetilde{\phi}} \|^2_* + \beta \| \nabla \widetilde{\psi} \|^2  \Big) + C H_2 \big| \overline{\widetilde{\phi}} \big|^2 + C \| \widetilde{\bm{v}} \|^2,
    \label{continuous_dependence}
\end{align}
where
\begin{align*}
    H_1 & = \big( 1 + \| \bm{u}_2 \|^2 \big) \big(  1 + \| \nabla \mu_{\phi_1} \|^2 + \| \nabla \mu_{\phi_2} \|^2 + \| \bm{v}_1 \|^2 + \| \bm{v}_2 \|^2  \big)
    \\
    & \quad + \| \partial_{t} \bm{u}_2 \|^2 + \| \nabla \mu_{\phi_2} \|_{L^{6}(\Omega)}^2 + \| \nabla \mu_{\psi_2} \|_{L^{6}(\Omega)}^2,
    \\
    H_2 & = \big( 1 + \| \bm{u}_2 \|^2 \big) \big(  1 + \| \nabla \mu_{\phi_1} \|^2 + \| \nabla \mu_{\phi_2} \|^2 + \| \bm{v}_1 \|^2 + \| \bm{v}_2 \|^2  \big)
    \\
    & \quad + \| \nabla \mu_{\psi_2} \|_{L^{6}(\Omega)}^2 + \| \partial_{t} \bm{u}_2 \|^2.
\end{align*}

Let us first consider the case $( \bm{u}_{0,1},\phi_{0,1},\psi_{0,1} )=( \bm{u}_{0,2},\phi_{0,2},\psi_{0,2} )=( \bm{u}_{0},\phi_{0},\psi_{0} )$, $\bm{v}_1, \bm{v}_2\in S$ and $T=T_0$. This choice of initial data implies that $\overline{\widetilde{\phi}}(t)=0$ in $[0,T_0]$. Then, from the inequality \eqref{continuous_dependence} and Gronwall's lemma, we infer that
\begin{equation*}
    \rho_* \| \widetilde{\bm{u}}(t) \|^2
    \le \int_{\Omega} \rho ( \phi_1 ) | \widetilde{\bm{u}}(t) |^2 \, \mathrm{d}x
    \le C \| \widetilde{\bm{v}} \|^2_{ L^{2} ( 0,T_0;\mathbf{V}_m ) } e^{C\int_{0}^{T_0} H_1(\tau)\, \mathrm{d}\tau},
    \quad \forall\, t\in [0,T_0],
\end{equation*}
which easily implies the continuity of the mapping $\Lambda$ in the topology of $L^{2} ( 0,T_0;\mathbf{V}_m )$. Hence, we can apply Schauder's fixed point theorem to conclude that the mapping $\Lambda$ has a fixed point in $S$. This gives the existence of a solution $(\bm{u}^m,\phi^m,\psi^m,\mu_\phi^m,\mu_\psi^m)$ on $[0,T_\alpha]$ with $T_\alpha=T_0$ to the approximate problem \eqref{eq:semi_galerkin_nsch1}--\eqref{eq:semi_galerkin_nschb} subject to the initial condition \eqref{icappr}.

Next, we investigate the uniqueness by considering the choice
$$
( \bm{u}_{0,1},\phi_{0,1},\psi_{0,1} )=( \bm{u}_{0,2},\phi_{0,2},\psi_{0,2} )=( \bm{u}_{0},\phi_{0},\psi_{0} ),\quad \bm{v}_1=\bm{u}_1, \quad  \bm{v}_2=\bm{u}_2,\quad T=T_0,\quad \delta'=\widetilde{\delta}.
$$
Then, from \eqref{continuous_dependence} and Gronwall's lemma, we can deduce that
\begin{equation*}
    \rho_* \| \widetilde{\bm{u}}(t) \|^2
    + \| \nabla \widetilde{\phi}(t) \|^2
    + \beta \| \nabla \widetilde{\psi}(t) \|^2
    \le C \Big(\widetilde{E}(0) +  \int_{0}^{T} \widetilde{H}_2(\tau) \big| \overline{\widetilde{\phi}}(\tau) \big|^2 \,\mathrm{d}\tau  \Big) e^{C\int_{0}^{T} \widetilde{H}_1 (\tau) \,\mathrm{d}\tau}=0,
\end{equation*}
for all $t\in [0,T_0]$, where
\begin{align*}
    \widetilde{E}(0) & =  \rho^*\| \bm{u}(0) \|^2
    + \| \nabla \widetilde{\phi}(0) \|^2
    + \beta \| \nabla \widetilde{\psi}(0) \|^2
    + \sigma_2 \| \widetilde{\phi}(0) - \overline{\widetilde{\phi}}(0) \|^2_*=0,
    \\
    \widetilde{H}_1 & = \big( 1 + \| \bm{u}_2 \|^2 \big) \big(  1 + \| \nabla \mu_{\phi_1} \|^2 + \| \nabla \mu_{\phi_2} \|^2 + \| \bm{u}_1 \|^2 + \| \bm{u}_2 \|^2  \big)
    \\
    & \quad + \| \partial_{t} \bm{u}_2 \|^2 + \| \nabla \mu_{\phi_2} \|_{L^{6}(\Omega)}^2 + \| \nabla \mu_{\psi_2} \|_{L^{6}(\Omega)}^2, \\
    \widetilde{H}_2 & = \big( 1 + \| \bm{u}_2 \|^2 \big) \big(  1 + \| \nabla \mu_{\phi_1} \|^2 + \| \nabla \mu_{\phi_2} \|^2 + \| \bm{u}_1 \|^2 + \| \bm{u}_2 \|^2  \big)
    \\
    & \quad + \| \nabla \mu_{\psi_2} \|_{L^{6}(\Omega)}^2 + \| \partial_{t} \bm{u}_2 \|^2.
\end{align*}
This yields the uniqueness of the approximate solution $(\bm{u}^m,\phi^m,\psi^m,\mu_\phi^m,\mu_\psi^m)$ on  $[0,T_0]$.
\qed

\subsection{Proof of Theorem \ref{nschsse3d}}

We now consider the $k$-dependent approximate initial data as in \eqref{eq:semi_galerkin_nschi}. The following proposition demonstrates that the corresponding approximate solutions $( \bm{u}^m,\phi^m,\psi^m,\mu_\phi^m,\mu_\psi^m )$  admit a common existing interval independent of $\alpha$, $m$ and $k$,
on which they remain uniformly bounded with respect to these parameters.
\begin{proposition} \label{uniform_approximate_solution}
Let $\Omega \subset \mathbb{R}^{3}$ be a bounded domain of class $C^3$. Suppose that $\alpha\in (0,1)$, $m\in \mathbb{N}_+$ is arbitrary, $k\in \mathbb{N}_+$ is sufficiently large, and Assumptions $(\mathbf{H0})$--$(\mathbf{H2})$ are satisfied. Then, there exists $\widetilde{T}_*>0$, which is independent of $\alpha$, $m$ and $k$, such that problem \eqref{eq:semi_galerkin_nsch1}--\eqref{eq:semi_galerkin_nschi} admits a unique strong solution $( \bm{u}^m,\phi^m,\psi^m,\mu_\phi^m,\mu_\psi^m )$ on $[0,\widetilde{T}_*]$ that satisfies
\begin{align*}
    & \bm{u}^m \in H^1 ( 0,\widetilde{T}_*;\mathbf{V}_m ),\quad  \phi^m,\psi^m \in L^{\infty} ( 0,\widetilde{T}_*;H^{3}_N(\Omega) ),
    \\
    &\partial_{t} \phi^m, \partial_{t} \psi^m \in L^{\infty} ( 0,\widetilde{T}_*;V ) \cap L^{2} ( 0,\widetilde{T}_*;W ) \cap H^1 ( 0,\widetilde{T}_*;H ),
    \\
    & \mu_\phi^m,\mu_\psi^m \in L^{2} ( 0,\widetilde{T}_*;H^{3}(\Omega) ) \cap L^{\infty} ( 0,\widetilde{T}_*;W ) \cap H^1 ( 0,\widetilde{T}_*;V ).
\end{align*}
Moreover, we have the following estimates on $[0,\widetilde{T}_*]$:
\begin{align}
    & \sup\limits_{t \in [0,\widetilde{T}_*]} \| \bm{u}^m (t) \|_{\bm{H}^1_{0,\sigma}(\Omega)} + \sup\limits_{t \in [0,\widetilde{T}_*]} \| \mu_\phi^m (t) \|_{V} + \sup\limits_{t \in [0,\widetilde{T}_*]} \| \mu_\psi^m (t) \|_{V}
    \nonumber \\
    & \quad + \sup\limits_{t \in [0,\widetilde{T}_*]} \sqrt{\alpha} \| \partial_{t} \phi^m(t) \| + \sup\limits_{t \in [0,\widetilde{T}_*]} \sqrt{\alpha} \| \partial_{t} \psi^m(t) \|
    \nonumber \\
    & \quad + \sup\limits_{t \in [0,\widetilde{T}_*]} \| \phi^m(t) \|_{W} + \sup\limits_{t \in [0,\widetilde{T}_*]} \| \psi^m(t) \|_{W}
    \nonumber \\
    & \quad + \sup\limits_{t \in [0,\widetilde{T}_*]} \| F_\phi'( \phi^m(t) ) \| + \sup\limits_{t \in [0,\widetilde{T}_*]} \| F_\psi' ( \psi^m(t) ) \| \le K_2,
    \label{Linfty_bound} \\
    & \int_{0}^{\widetilde{T}_*}  \left(\| \partial_{t} \bm{u}^m(\tau) \|^2 + \| \mathbf{A} \bm{u}^m(\tau) \|^2 + \| \mu^m_\phi(\tau) \|^2_{H^{3}(\Omega)} + \| \mu^m_{\psi}(\tau) \|^2_{H^{3}(\Omega)}\right)  \,\mathrm{d}\tau
    \nonumber \\
    & \quad + \int_{0}^{\widetilde{T}_*} \left( \| \nabla \partial_{t} \phi^m(\tau) \|^2 + \| \nabla \partial_{t} \psi^m(\tau) \|^2 \right) \,\mathrm{d}\tau \le K_3,
    \label{L2_bound}
\end{align}
where $K_2$ and $K_3$ are positive constants that depend on  $\overline{\phi_0}$, $\overline{\psi_0}$, $E_{\mathrm{tot}} ( \bm{u}_0,\phi_0,\psi_0 )$, $\| \bm{u}_0 \|_{\bm{H}^1_{0,\sigma}(\Omega)}$, $\|-\Delta \phi_0 + F_\phi'(\phi_0)\|_V$, $\|-\beta \Delta \psi_0 + F_\psi'(\psi_0)\|_V$, $\widetilde{T}_*$ and the coefficients of the system, but are independent of $m$, $\alpha$ and $k$. Moreover,
\begin{align}
 & ( \phi^m,\psi^m ) \in [-1+\delta( \alpha,k ),1-\delta( \alpha,k )] \times [\delta( \alpha,k ),1-\delta( \alpha,k )] \quad \mathrm{a.e.} \ \mathrm{in} \ \Omega \times [0,\widetilde{T}_*],\label{delta-al}
\end{align}
where $\delta( \alpha,k )\in (0,1/2)$ depends on $\widetilde{T}_*$, $\alpha$, $k$, but is independent of $m$.
\end{proposition}
\noindent \textbf{Proof.}
The initial datum $(\mathbb{P}_m\bm{u}_0,\phi_{0,k},\psi_{0,k})$ is admissible in Proposition \ref{local_approximate_fixpoint}. We denote the corresponding solution constructed in Proposition \ref{local_approximate_fixpoint} by $( \bm{u}^m,\phi^m,\psi^m,\mu_\phi^m,\mu_\psi^m )$ with  $[0,T_m)$ being its maximal existence interval.
By Lemma \ref{approximate_initial_phase_field} and the mass relations for $\overline{\phi^m},\,\overline{\psi^m}$, the argument in \cite{GHW2026} applies and yields a sufficiently large integer  $\widehat{k}$ such that for all $k>\widehat{k}$, $|\overline{\phi^m}(t)|\leq 1-\delta_*$, $|\overline{\psi^m}(t)-1/2|\leq (1/2)-\delta_*$ for all $t\in [0,T_m)$, where $\delta_*\in(0,1/2)$ depends only on $\overline{\phi_0}$, $\overline{\psi_0}$, $c$ and $\sigma_1$, but not on the approximating parameters $m$, $\alpha$ and $k$.
Furthermore, from the following energy identity
\begin{align*}
    & \frac{\mathrm{d}}{\mathrm{d}t} E_{\mathrm{tot}} (\bm{u}^m,\phi^m,\psi^m)
    + \int_{\Omega}  \big(  \nu(\phi^m) | D \bm{u}^m |^2
    + m_\phi(\phi^m) | \nabla \mu_\phi^m |^2
    + m_\psi(\psi^m) | \nabla \mu_\psi^m |^2 \big) \, \mathrm{d}x
    \\
    & \quad + \alpha \big(\|\partial_t \phi^m\|^2+\|\partial_t \psi^m\|^2\big)
    +  \sigma_1 \big( \overline{\phi^m} -c \big) \Big( \int_\Omega    \mu_\phi^m\,\mathrm{d}x - \frac{\gamma}{2}  \| \bm{u}^m \|^2 \Big) = 0,
\end{align*}
we can argue as in \cite{G2022,OGW2026} to conclude
\begin{align}
    & \| \bm{u}^m \|_{L^{\infty} ( 0,T_m;\bm{L}^{2}_\sigma(\Omega) )}
    + \| \bm{u}^m \|_{L^{2} ( 0,T_m;\bm{H}^{1}_{0,\sigma}(\Omega) )} \le C,
    \label{sp_lower_bound} \\
    & \| \phi^m \|_{L^{\infty} ( 0,T_m;V )}
    + \| \psi^m \|_{L^{\infty} ( 0,T_m;V )}
    \nonumber\\
    & \quad + \| \phi^m \|_{L^{2} ( 0,T_m;W )}
    + \| \psi^m \|_{L^{2} ( 0,T_m;W )} \le C, \label{phi_psi_lower_bound} \\
    & \| \mu_\phi^m \|_{L^{2} ( 0,T_m;V )}
    + \| \mu_\psi^m \|_{L^{2} ( 0,T_m;V )} \le C, \label{mu_lower_bound} \\
    & \sqrt{\alpha} \| \partial_{t} \phi^m \|_{L^{2} ( 0,T_m;H )}
    + \sqrt{\alpha} \| \partial_{t} \psi^m \|_{L^{2} ( 0,T_m;H )}
    \nonumber\\
    & \quad + \| \partial_{t} \phi^m \|_{L^{2} ( 0,T_m;V^* )} + \| \partial_{t} \psi^m \|_{L^{2} ( 0,T_m;V^* )} \le C,
    \label{phit_psit_lower_bound}
\end{align}
where the constant $C>0$ depends on $\| \bm{u}_0 \|$, $\| \phi_0 \|_{V}$, $\| \psi_0 \|_{V}$, $\overline{\phi_0}$, $\overline{\psi_0}$, $\Omega$ and the coefficients of the system, but is independent of $m$, $\alpha$ and $k$ (again with $k$ being sufficiently large).

Next, taking $\bm{\eta} = \partial_{t} \bm{u}^m$ in \eqref{eq:semi_galerkin_nsch1}, we obtain
\begin{align*}
     & \frac{1}{2} \frac{\mathrm{d}}{\mathrm{d t}}  \Big(  \int_{\Omega} \nu ( \phi^m ) | D \bm{u}^m |^2  \, \mathrm{d}x
     + \gamma \sigma_1 ( \overline{\phi^m}-c ) \| \bm{u}^m \|^2  \Big)  + \int_{\Omega}  \rho ( \phi^m ) | \partial_{t} \bm{u}^m |^2  \, \mathrm{d}x
     \\
     & \quad = \frac{1}{2} \int_{\Omega}  \nu'( \phi^m ) \partial_{t} \phi^m | D \bm{u}^m |^2  \, \mathrm{d}x - \big( \rho ( \phi^m ) ( \bm{u}^m \cdot \nabla ) \bm{u}^m, \partial_{t} \bm{u}^m \big)
     + \gamma \big( m_\phi ( \phi^m ) ( \nabla \mu_\phi^m \cdot \nabla ) \bm{u}^m, \partial_{t} \bm{u}^m \big)
     \\
     & \qquad + \frac{1}{2} \gamma \sigma_1^2 \big( \overline{\phi^m}-c \big) \| \bm{u}^m \|^2
     + \big( \mu_\phi^m \nabla \phi^m + \mu_\psi^m \nabla \psi^m, \partial_{t} \bm{u}^m \big).
\end{align*}
Similarly, testing \eqref{eq:semi_galerkin_nsch2} and \eqref{eq:semi_galerkin_nsch4} with $\partial_{t} \mu_\phi^m$ and $\partial_{t} \mu_\psi^m$, respectively,
gives
\begin{align*}
    & \frac{\mathrm{d}}{\mathrm{d t}}  \Big(  \int_{\Omega}  \frac{1}{2} m_\phi ( \phi^m ) | \nabla \mu_\phi^m |^2 + \sigma_1 ( \overline{\phi^m}-c ) \mu_\phi^m + \mu_\phi^m \bm{u}^m \cdot \nabla \phi^m  \, \mathrm{d}x  \Big)
    + \big( \partial_{t} \mu_\phi^m, \partial_{t} \phi^m \big)
    \nonumber \\
    & \quad =  \frac{1}{2} \int_{\Omega}  m_\phi' ( \phi^m ) \partial_{t} \phi^m | \nabla \mu_\phi^m |^2  \, \mathrm{d}x
    + \big(  \bm{u}^m \cdot \nabla \partial_{t} \phi^m + \partial_{t} \bm{u}^m \cdot \nabla \phi^m + \sigma_1^2 \big( \overline{\phi^m}-c \big) , \mu_\phi^m \big),
\end{align*}
and
\begin{align*}
    & \frac{\mathrm{d}}{\mathrm{d t}}  \Big(  \int_{\Omega}  \frac{1}{2} m_\psi ( \psi^m ) | \nabla \mu_\psi^m |^2 + \mu_\psi^m \bm{u}^m \cdot \nabla \psi^m  \, \mathrm{d}x  \Big)
    + \big( \partial_{t} \mu_\psi^m, \partial_{t} \psi^m \big)
    \nonumber \\
    & \quad =  \frac{1}{2} \int_{\Omega}  m_\psi' ( \psi^m ) \partial_{t} \psi^m | \nabla \mu_\psi^m |^2  \, \mathrm{d}x
    + \big(  \bm{u}^m \cdot \nabla \partial_{t} \psi^m + \partial_{t} \bm{u}^m \cdot \nabla \psi^m , \mu_\psi^m \big).
\end{align*}
On the other hand, from \eqref{eq:semi_galerkin_nsch3} and \eqref{eq:semi_galerkin_nsch5} we obtain
\begin{align*}
    \big( \partial_{t} \mu_\phi^m, \partial_{t} \phi^m \big)
    & = \frac{\alpha}{2} \frac{\mathrm{d}}{\mathrm{d t}} \| \partial_{t} \phi^m \|^2
    + \| \nabla \partial_{t} \phi^m \|^2
    + \sigma_2 \|  \partial_{t} \phi^m - \overline{\partial_{t} \phi^m}  \|^2_*
    + \int_{\Omega}  F_\phi'' ( \phi^m ) | \partial_{t} \phi^m |^2 \, \mathrm{d}x
    \\
    & \quad + \int_{\Omega} \big( \partial_\phi ^2 G ( \phi^m,\psi^m ) | \partial_{t} \phi^m |^2 + \partial_{\phi,\psi} ^2 G ( \phi^m,\psi^m ) \partial_{t} \phi^m \partial_{t} \psi^m \big) \, \mathrm{d}x,
    \\
    \big( \partial_{t} \mu_\psi^m, \partial_{t} \psi^m \big)
    & = \frac{\alpha}{2} \frac{\mathrm{d}}{\mathrm{d t}} \| \partial_{t} \psi^m \|^2
    + \beta \| \nabla \partial_{t} \psi^m \|^2
    + \int_{\Omega}  F_\psi'' ( \psi^m ) | \partial_{t} \psi^m |^2 \, \mathrm{d}x
    \\
    & \quad + \int_{\Omega} \big( \partial_\psi ^2 G ( \phi^m,\psi^m ) | \partial_{t} \psi^m |^2 + \partial_{\phi,\psi} ^2 G ( \phi^m,\psi^m ) \partial_{t} \phi^m \partial_{t} \psi^m \big) \, \mathrm{d}x.
\end{align*}
Combining the above identities, we can deduce that (cf. \cite{G2022})
\begin{align}
    & \frac{\mathrm{d}}{\mathrm{d t}} H_m
    + \| \nabla \partial_{t} \phi^m \|^2
    + \sigma_2 \|  \partial_{t} \phi^m - \overline{\partial_{t} \phi^m}  \|^2_*
    + \beta \| \nabla \partial_{t} \psi^m \|^2
    + \int_{\Omega}  \rho ( \phi^m ) | \partial_{t} \bm{u}^m |^2  \, \mathrm{d}x
    \nonumber \\
    & \quad \le \int_{\Omega} \big( \partial_\phi ^2 G ( \phi^m,\psi^m ) | \partial_{t} \phi^m |^2 + 2 \partial_{\phi,\psi} ^2 G ( \phi^m,\psi^m ) \partial_{t} \phi^m \partial_{t} \psi^m + \partial_\psi ^2 G ( \phi^m,\psi^m ) | \partial_{t} \psi^m |^2 \big) \, \mathrm{d}x
    \nonumber \\
    & \qquad + \frac{1}{2} \int_{\Omega}  m_\phi' ( \phi^m ) \partial_{t} \phi^m | \nabla \mu_\phi^m |^2  \, \mathrm{d}x
    + \big(  \bm{u}^m \cdot \nabla \partial_{t} \phi^m + \partial_{t} \bm{u}^m \cdot \nabla \phi^m , \mu_\phi^m - \overline{\mu_\phi^m} \big)
    \nonumber \\
    & \qquad + \frac{1}{2} \int_{\Omega}  m_\psi' ( \psi^m ) \partial_{t} \psi^m | \nabla \mu_\psi^m |^2  \, \mathrm{d}x
    + \big(  \bm{u}^m \cdot \nabla \partial_{t} \psi^m + \partial_{t} \bm{u}^m \cdot \nabla \psi^m , \mu_\psi^m - \overline{\mu_\psi^m} \big)
    \nonumber \\
    & \qquad + \sigma_1^2 | \Omega | \big( \overline{\phi^m}-c \big) \overline{\mu_\phi^m}
    + \frac{1}{2} \int_{\Omega}  \nu'( \phi^m ) \partial_{t} \phi^m | D \bm{u}^m |^2  \, \mathrm{d}x
    \notag\\
    &\qquad - \big( \rho ( \phi^m ) ( \bm{u}^m \cdot \nabla ) \bm{u}^m, \partial_{t} \bm{u}^m \big)
    + \gamma \big( m_\phi ( \phi^m ) ( \nabla \mu_\phi^m \cdot \nabla ) \bm{u}^m, \partial_{t} \bm{u}^m \big) \nonumber \\
    & \qquad + \frac{1}{2} \gamma \sigma_1^2 \big( \overline{\phi^m}-c \big) \| \bm{u}^m \|^2
    - \big( \phi^m \nabla \mu_\phi^m + \psi^m \nabla \mu_\psi^m, \partial_{t} \bm{u}^m \big)
    \nonumber \\
    & \quad \le C \big(  \| \partial_{t} \phi^m \|^2 + \| \partial_{t} \psi^m \|^2  \big)
    \nonumber \\
    & \qquad + C \big(  \| \partial_{t} \phi^m \|_{L^{6}(\Omega)} \| \nabla \mu_\phi^m \|_{L^{3}(\Omega)} \| \nabla \mu_\phi^m \| + \| \partial_{t} \psi^m \|_{L^{6}(\Omega)} \| \nabla \mu_\psi^m \|_{L^{3}(\Omega)} \| \nabla \mu_\psi^m \|  \big)
    \nonumber \\
    & \qquad + \| \bm{u}^m \|_{L^{3}(\Omega)} \big(  \| \nabla \partial_{t} \phi^m \| \| \mu_\phi^m - \overline{\mu_\phi^m} \|_{L^{6}(\Omega)} + \| \nabla \partial_{t} \psi^m \| \| \mu_\psi^m - \overline{\mu_\psi^m} \|_{L^{6}(\Omega)}  \big)
    \nonumber \\
    & \qquad + \| \partial_{t} \bm{u}^m \| \big(  \| \nabla \phi^m \|_{L^{6}(\Omega)} \| \mu_\phi^m - \overline{\mu_\phi^m} \|_{L^{3}(\Omega)} + \| \nabla \psi^m \|_{L^{6}(\Omega)} \| \mu_\psi^m - \overline{\mu_\psi^m} \|_{L^{3}(\Omega)}  \big)
    \nonumber \\
    & \qquad + C \big( 1 + \| \mu_\phi^m \| \big) + C \| \partial_{t} \phi^m \|_{L^{6}(\Omega)} \| D \bm{u}^m \|_{L^{3}(\Omega)} \| D \bm{u}^m \|
    \nonumber \\
    & \qquad + \rho^* \| \bm{u}^m \|_{L^{6}(\Omega)} \| \nabla \bm{u}^m \|_{L^{3}(\Omega)} \| \partial_{t} \bm{u}^m \|
    + |\gamma| m_\phi^* \| \nabla \mu_\phi^m \|_{L^{\infty}(\Omega)} \| \nabla \bm{u}^m \| \| \partial_{t} \bm{u}^m \|
    \nonumber \\
    & \qquad + | \gamma | \sigma_1^2 \| \bm{u}^m \|^2
    + \big(  \| \nabla \mu_\phi^m \| + \| \nabla \mu_\psi^m \|  \big) \| \partial_{t} \bm{u}^m \| \nonumber \\
    & \quad \le C \big(  | \overline{\phi^m} |^2 + \| \partial_{t} \phi^m - \partial_{t} \overline{\phi^m} \|_* \| \nabla \partial_{t} \phi^m \| + \| \partial_{t} \psi^m \|_* \| \nabla \partial_{t} \psi^m \|  \big)
    \nonumber \\
    & \qquad + C  \big(  \| \nabla \partial_{t} \phi^m \| + | \partial_{t} \overline{\phi^m} |  \big) \| \nabla \mu_\phi^m \|_{H^2(\Omega)}^{\frac{1}{4}} \| \nabla \mu_\phi^m \|^{\frac{7}{4}}
    + C \| \nabla \partial_{t} \psi^m \| \| \nabla \mu_\psi^m \|_{H^2(\Omega)}^{\frac{1}{4}} \| \nabla \mu_\psi^m \|^{\frac{7}{4}}
    \nonumber \\
    & \qquad + C \| D \bm{u}^m \| \big(  \| \nabla \partial_{t} \phi^m \| \| \nabla \mu_\phi^m \| + \| \nabla \partial_{t} \psi^m \| \| \nabla \mu_\psi^m \|  \big)
    \nonumber \\
    & \qquad + \| \partial_{t} \bm{u}^m \| \big(  \| \phi^m \|_W \| \nabla \mu_\phi^m \| + \|\psi^m \|_W \| \nabla \mu_\psi^m \|  \big)
    + C \big( 1 + \| \mu_\phi^m \| \big)
    \nonumber \\
    & \qquad + C \big(  \| \nabla \partial_{t} \phi^m \| + | \partial_{t} \overline{\phi^m} |  \big) \| \mathbf{A} \bm{u}^m \|^\frac{1}{2} \| D \bm{u}^m \|^{\frac{3}{2}}
    \nonumber \\
    & \qquad + C \| D \bm{u}^m \|^{\frac{3}{2}} \| \mathbf{A} \bm{u}^m \|^\frac{1}{2} \| \partial_{t} \bm{u}^m \|
    + C \| \nabla \mu_\phi^m \|^{\frac{1}{4}} \| \nabla \mu_\phi^m \|^{\frac{3}{4}}_{H^2(\Omega)} \| D \bm{u}^m \| \| \partial_{t} \bm{u}^m \|
    \nonumber \\
    & \qquad + | \gamma | \sigma_1^2 \| \bm{u}^m \|^2
    + \big(  \| \nabla \mu_\phi^m \| + \| \nabla \mu_\psi^m \|  \big) \| \partial_{t} \bm{u}^m \|
    \nonumber \\
    & \quad \le   \xi \big(  \| \nabla \partial_{t} \phi^m \|^2 + \beta \| \nabla \partial_{t} \psi^m \|^2  \big)
    + C_{\xi} \big(  | \overline{\phi^m} |^2 + \| \partial_{t} \phi^m - \partial_{t} \overline{\phi^m} \|_*^2 + \| \partial_{t} \psi^m \|_*^2  \big)
    \nonumber \\
    & \qquad + \xi \big(  \| \nabla \partial_{t} \phi^m \|^2 + \| \nabla \mu_\phi^m \|_{H^{2}(\Omega)}^2 + \beta \| \nabla \partial_{t} \psi^m \|^2 + \| \nabla \mu_\psi^m \|_{H^{2}(\Omega)}^2  \big)
    \nonumber \\
    & \qquad + C_{\xi} \big(  \| \nabla \mu_\phi^m \|^{\frac{14}{3}} + \| \nabla \mu_\psi^m \|^{\frac{14}{3}} +  | \partial_{t} \overline{\phi^m} |^2 \big)
    \nonumber \\
    & \qquad + \xi \big(  \| \nabla \partial_{t} \phi^m \|^2 + \beta \| \nabla \partial_{t} \psi^m \|^2  \big)
    + C_{\xi} \| D \bm{u}^m \|^2 \big(  \| \nabla \mu_\phi^m \|^2 + \| \nabla \mu_\psi^m \|^2  \big)
    \nonumber \\
    & \qquad + \xi \| \partial_{t} \bm{u}^m \|^2
    + C_{\xi} \big(  \| \phi^m \|_W^2 \| \nabla \mu_\phi^m \|^2 + \| \psi^m \|_W^2 \| \nabla \mu_\psi^m \|^2 \big)
    \nonumber \\
    & \qquad + \xi \big(  \| \nabla \partial_{t} \phi^m \|^2 + \| \mathbf{A} \bm{u}^m \|^2 +  \| \partial_{t} \bm{u}^m \|^2  \big)
    + C_{\xi} \big(1+\| D \bm{u}^m \|^6\big)
    \nonumber \\
    & \qquad + \xi \big(  \| \nabla \mu_\phi^m \|^2_{H^{2}(\Omega)} + \| \partial_{t} \bm{u}^m \|^2  \big)
    + C_{\xi} \| \nabla \mu_\phi^m \|^2 \| D \bm{u}^m \|^8
    + | \gamma | \sigma_1^2 \| \bm{u}^m \|^2
    \nonumber \\
    & \qquad + \xi \| \partial_{t} \bm{u}^m \|^2
    + C_{\xi} \big(  \| \nabla \mu_\phi^m \|^2 + \| \nabla \mu_\psi^m \|^2  \big), \label{energy_inequality_original}
\end{align}
where
\begin{align*}
    H_m & \define
    \frac{\alpha}{2} \| \partial_{t} \phi^m \|^2 + \int_{\Omega} \Big( \frac{1}{2} m_\phi ( \phi^m ) | \nabla \mu_\phi^m |^2 + \sigma_1 ( \overline{\phi^m}-c ) \mu_\phi^m - \phi^m \bm{u}^m \cdot \nabla \mu_\phi^m  \Big)\, \mathrm{d}x
    \\
    & \quad + \frac{\alpha}{2} \| \partial_{t} \psi^m \|^2 + \int_{\Omega} \Big( \frac{1}{2} m_\psi ( \psi^m ) | \nabla \mu_\psi^m |^2 - \psi^m \bm{u}^m \cdot \nabla \mu_\psi^m \Big) \, \mathrm{d}x
    \\
    & \quad + \frac{1}{2} \int_{\Omega} \nu ( \phi^m ) | D \bm{u}^m |^2  \, \mathrm{d}x
    + \frac{1}{2} \gamma \sigma_1 ( \overline{\phi^m}-c ) \| \bm{u}^m \|^2.
\end{align*}
From the estimates \eqref{sp_lower_bound}, \eqref{mu_lower_bound} and the $L^\infty$-bounds for $\phi^m$, $\psi^m$, we find that
\begin{align}
    H_m & \ge \frac{\alpha}{2} \| \partial_{t} \phi^m \|^2 + \frac{\alpha}{2} \| \partial_{t} \psi^m \|^2
    + \frac{1}{4} \int_{\Omega}  m_\phi ( \phi^m ) | \nabla \mu_\phi^m |^2  \, \mathrm{d}x
    \nonumber \\
    & \quad + \frac{1}{4} \int_{\Omega}  m_\psi ( \psi^m ) | \nabla \mu_\psi^m |^2  \ \mathrm{d}x
    + \frac{1}{2} \int_{\Omega} \nu ( \phi^m ) | D \bm{u}^m |^2  \, \mathrm{d}x
    \nonumber \\
    & \quad  - C ( 1 + \sigma_1^2 ) - (C+| \gamma | \sigma_1) \| \bm{u}^m \|^2
    \nonumber \\
    & \ge \frac{\alpha}{2} \| \partial_{t} \phi^m \|^2
    + \frac{\alpha}{2} \| \partial_{t} \psi^m \|^2
    + \frac{1}{4} \int_{\Omega}  m_\phi ( \phi^m ) | \nabla \mu_\phi^m |^2  \, \mathrm{d}x
    \nonumber \\
    & \quad + \frac{1}{4} \int_{\Omega}  m_\psi ( \psi^m ) | \nabla \mu_\psi^m |^2  \, \mathrm{d}x
    + \frac{1}{2} \int_{\Omega} \nu ( \phi^m ) | D \bm{u}^m |^2  \, \mathrm{d}x
    - C_4,
    \label{hm_lower_bound} \\
    H_m & \le \frac{\alpha}{2} \| \partial_{t} \phi^m \|^2 + \frac{\alpha}{2} \| \partial_{t} \psi^m \|^2
    +  \int_{\Omega}  m_\phi ( \phi^m ) | \nabla \mu_\phi^m |^2  \, \mathrm{d}x
    \nonumber \\
    & \quad + \int_{\Omega}  m_\psi ( \psi^m ) | \nabla \mu_\psi^m |^2  \, \mathrm{d}x
    + \frac{1}{2} \int_{\Omega} \nu ( \phi^m ) | D \bm{u}^m |^2  \, \mathrm{d}x
    \nonumber \\
    & \quad + C ( 1 + \sigma_1^2 )
    + (C+| \gamma | \sigma_1) \| \bm{u}^m \|^2
    \nonumber \\
    & \le \frac{\alpha}{2} \| \partial_{t} \phi^m \|^2
    + \frac{\alpha}{2} \| \partial_{t} \psi^m \|^2
    + \int_{\Omega}  m_\phi ( \phi^m ) | \nabla \mu_\phi^m |^2  \, \mathrm{d}x
    \nonumber \\
    & \quad + \int_{\Omega}  m_\psi ( \psi^m ) | \nabla \mu_\psi^m |^2  \, \mathrm{d}x
    + \frac{1}{2} \int_{\Omega} \nu ( \phi^m ) | D \bm{u}^m |^2  \, \mathrm{d}x + C_5,
    \label{hm_upper_bound}
\end{align}
where the positive constants $C_4$ and $C_5$ are independent of $m$, $\alpha$ and $k$.
From \eqref{eq:semi_galerkin_nsch2}--\eqref{eq:semi_galerkin_nsch5} and the $L^\infty$-bounds for $\phi^m$, $\psi^m$, we also have
\begin{align*}
    & \|\phi^m \|_W \le C \big(  \| \nabla \mu_\phi^m \| + \alpha \| \partial_{t} \phi^m \| + 1  \big),
    \quad  \| \psi^m \|_W \le C \big(  \| \nabla \mu_\psi^m \| + \alpha \| \partial_{t} \psi^m \| + 1  \big),
    \\
    & \| \partial_{t} \phi^m - \partial_{t} \overline{\phi^m} \|_* \le C \big(  \| \bm{u}^m \| + \| \nabla \mu_\phi^m \|  \big),
    \quad \| \partial_{t} \psi^m \|_* \le C \big(  \| \bm{u}^m \| + \| \nabla \mu_\psi^m \|  \big).
\end{align*}
Hence, it follows from \eqref{energy_inequality_original} that
\begin{align}
    & \frac{\mathrm{d}}{\mathrm{d t}} H_m
    + \| \nabla \partial_{t} \phi^m \|^2
    + \beta \| \nabla \partial_{t} \psi^m \|^2
    + \int_{\Omega}  \rho ( \phi^m ) | \partial_{t} \bm{u}^m |^2  \, \mathrm{d}x
    \nonumber \\
    & \quad \le 8 \xi \big(  \| \nabla \partial_{t} \phi^m \|^2 + \beta \| \nabla \partial_{t} \psi^m \|^2 + \| \partial_{t} \bm{u}^m \|^2 + \| \mathbf{A} \bm{u}^m \|^2 \big)
    \nonumber \\
    & \qquad + 8 \xi \big( \| \nabla \mu_\phi^m \|^2_{H^{2}(\Omega)} + \| \nabla \mu_\psi^m \|^2_{H^{2}(\Omega)}  \big)
    + C_{\xi} (  1 + C_4 + H_m  )^5.
    \label{energy_inequality}
\end{align}

To control the right-hand side of \eqref{energy_inequality}, it remains to estimate $\| \nabla \mu_\phi^m \|_{H^{2}(\Omega)}$, $\| \nabla \mu_\psi^m \|_{H^{2}(\Omega)}$ and $\| \mathbf{A} \bm{u}^m \|$. For $\| \nabla \mu_\phi^m \|_{H^{2}(\Omega)}$, we take gradient on both sides of equation \eqref{eq:semi_galerkin_nsch2} and get
\begin{align*}
    m_\phi ( \phi^m ) \nabla \Delta \mu_\phi^m & = - m_\phi' ( \phi^m ) \Delta \mu_\phi^m \nabla \phi^m - m_\phi' ( \phi^m ) \nabla^2 \mu_\phi^m \nabla \phi^m   - m_\phi' ( \phi^m ) \nabla^2 \phi^m  \nabla \mu_\phi^m \nonumber \\
    & \quad - m_\phi'' ( \phi^m ) (  \nabla \phi^m \cdot \nabla \mu_\phi^m  ) \nabla \phi^m + \nabla \bm{u}^m   \nabla \phi^m + (\nabla^2 \phi^m )  \bm{u}^m + \nabla \partial_{t} \phi^m.
\end{align*}
Applying the inequalities \eqref{gns_inequality} and \eqref{Agmon_inequality} in three dimensions, we can establish the following estimate:
\begin{align*}
    \|  m_\phi ( \phi^m ) \nabla \Delta \mu_\phi^m  \|
    & \le C \| \Delta \mu_\phi^m \|_{L^{3}(\Omega)} \| \nabla \phi^m \|_{L^{6}(\Omega)} + C \| \nabla^2 \mu_\phi^m \|_{L^{3}(\Omega)} \| \nabla \phi^m \|_{L^{6}(\Omega)}
    \\
    & \quad + C \| \nabla^2 \phi^m \| \| \nabla \mu_\phi^m \|_{L^{\infty}(\Omega)} + C \| \nabla \phi^m \|_{L^{6}(\Omega)}^2 \| \nabla \mu_\phi^m \|_{L^{6}(\Omega)}
    \\
    & \quad + \| \nabla \bm{u}^m \|_{L^{3}(\Omega)} \| \nabla \phi^m \|_{L^{6}(\Omega)} + \| \nabla^2 \phi^m \| \| \bm{u}^m \|_{L^{\infty}(\Omega)} + \| \nabla \partial_{t} \phi^m \|
    \\
    & \le C \| \phi^m \|_W \| \nabla \mu_\phi^m \|^{\frac{1}{4}} \| \nabla \mu_\phi^m \|^{\frac{3}{4}}_{H^{2}(\Omega)} + C \|\phi^m \|_W^2 \| \nabla \mu_\phi^m \|^\frac{1}{2} \| \nabla \mu_\phi^m \|_{H^{2}(\Omega)}^\frac{1}{2}
    \\
    & \quad + C \| \phi^m \|_W \| \bm{u}^m \|^\frac{1}{2}_{H^{1}(\Omega)} \| \bm{u}^m \|^\frac{1}{2}_{H^{2}(\Omega)} + \| \nabla \partial_{t} \phi^m \| \\
    & \le \epsilon_1 \| \nabla \mu_\phi^m \|_{H^{2}(\Omega)} + C_{\epsilon_1} \| \phi^m \|_W^4 \| \nabla \mu_\phi^m \|
    \\
    & \quad + \epsilon_2 \| \mathbf{A}  \bm{u}^m \| + C_{\epsilon_2} \| \phi^m \|_W^2 \| D \bm{u}^m \|
    + \| \nabla \partial_{t} \phi^m \|,
\end{align*}
where $\epsilon_1, \epsilon_2>0$ are arbitrary small constants. By the elliptic estimate for the Neumann problem, we have
\begin{align*}
\|\nabla \mu_\phi^m \|_{H^{2}(\Omega)}
&\le
\|\mu_\phi^m-\overline{\mu_\phi^m} \|_{H^{3}(\Omega)}
\leq C\|\Delta (\mu_\phi^m-\overline{\mu_\phi^m})\|_{H^1(\Omega)}
\\
&\le C\big(\|\nabla \Delta  \mu_\phi^m \|+ \|\Delta  \mu_\phi^m \|\big)
\le
C\big(\|\nabla \Delta  \mu_\phi^m \|+ \|\nabla \Delta  \mu_\phi^m \| ^\frac12 \|\nabla \mu_\phi^m \| ^\frac12\big)
\\
&\le C\big(\|\nabla \Delta  \mu_\phi^m \|+ \|\nabla   \mu_\phi^m \|\big)
\\
&\le C_6 \big( \|  m_\phi ( \phi^m ) \nabla \Delta \mu_\phi^m  \| + \| \nabla \mu_\phi^m \| \big),
\end{align*}
where $C_6$ is a positive constant depending on $m_{\phi,*}$ and $\Omega$.
Taking $\epsilon_1=1/(2C_6)$ and denoting $\epsilon=16\epsilon_2^2$, from the above two inequalities we infer that
\begin{equation}\label{mhi_h3}
    \|  \nabla \mu_\phi^m  \|_{H^{2}(\Omega)}^2 \le C_{\epsilon} ( 1+C_4 + H_m )^5 + \epsilon \| \mathbf{A} \bm{u}^m \|^2 + C_7 \| \nabla \partial_{t} \phi^m \|^2.
\end{equation}
Similarly, for $\mu_\psi^m$, we get
\begin{equation} \label{msi_h3}
    \|  \nabla \mu_\psi^m  \|_{H^{2}(\Omega)}^2 \le C_{\epsilon} ( 1+C_4 + H_m )^5 + \epsilon \| \mathbf{A} \bm{u}^m \|^2 + C_8 \| \nabla \partial_{t} \psi^m \|^2.
\end{equation}
Next, applying the argument in \cite{G2022}, we take $\bm{\eta} = \mathbf{A} \bm{u}^m$ in \eqref{eq:semi_galerkin_nsch1} and obtain
\begin{align*}
    - \frac{1}{2} \big( \nu ( \phi^m ) \Delta \bm{u}^m, \mathbf{A} \bm{u}^m \big)
    & = - \big( \rho( \phi^m ) \partial_{t} \bm{u}^m, \mathbf{A} \bm{u}^m \big)
    - \big( \rho( \phi^m ) ( \bm{u}^m \cdot \nabla ) \bm{u}^m , \mathbf{A} \bm{u}^m \big)
    \\
    & \quad + \gamma \big(  m_\phi(\phi^m)( \nabla \mu_\phi^m \cdot \nabla ) \bm{u}^m + \sigma_1 \big(  \overline{\phi^m} - c  \big) \bm{u}^m , \mathbf{A} \bm{u}^m \big)
    \\
    & \quad + \big( \mu_\phi^m \nabla \phi^m + \mu_\psi^m \nabla \psi^m, \mathbf{A} \bm{u}^m \big)
    + \int_{\Omega}  \nu' ( \phi^m ) ( D \bm{u}^m \nabla \phi^m) \cdot \mathbf{A} \bm{u}^m  \, \mathrm{d}x.
\end{align*}
It follows from \cite{G2021} (see also \cite{GMT2019}) that there exists $\pi^m \in C \big( [0,T_m);V \big)$ satisfying
\[
- \Delta \bm{u}^m + \nabla \pi^m = \mathbf{A} \bm{u}^m, \quad  \mathrm{a.e.} \ \mathrm{in} \ Q_{T_m},\quad \text{with}\ \int_\Omega \pi^m\,\mathrm{d}x=0,
\]
and
\[
\| \pi^m \| \le C \| \nabla \bm{u}^m \|^\frac{1}{2} \| \mathbf{A} \bm{u}^m \|^\frac{1}{2}, \quad  \| \pi^m \|_{V} \le C \| \mathbf{A} \bm{u}^m \|,
\]
where the constant $C>0$ depends only on $\Omega$.
Therefore, taking advantage of some estimates in \cite{G2022}, we can deduce that
\begin{align*}
    \frac{1}{2} \big( \nu ( \phi^m ) \mathbf{A} \bm{u}^m, \mathbf{A} \bm{u}^m \big)
    & = - \big( \rho( \phi^m ) \partial_{t} \bm{u}^m, \mathbf{A} \bm{u}^m \big)
    - \big( \rho( \phi^m ) ( \bm{u}^m \cdot \nabla ) \bm{u}^m , \mathbf{A} \bm{u}^m \big)
    \\
    & \quad + \gamma \big(  m_\phi(\phi^m)( \nabla \mu_\phi^m \cdot \nabla ) \bm{u}^m + \sigma_1 \big(  \overline{\phi^m} - c  \big) \bm{u}^m , \mathbf{A} \bm{u}^m \big)
    \\
    & \quad - \big( \phi^m \nabla \mu_\phi^m + \psi^m \nabla \mu_\psi^m, \mathbf{A} \bm{u}^m \big)
    + \frac{1}{2} \big(  \nu' ( \phi^m ) \pi^m \nabla \phi^m , \mathbf{A} \bm{u}^m \big)
    \\
    & \quad + \int_{\Omega}  \nu' ( \phi^m ) ( D  \bm{u}^m\nabla \phi^m)\cdot \mathbf{A} \bm{u}^m  \, \mathrm{d}x
    \\
    & \le \rho^* \| \partial_{t} \bm{u}^m \| \| \mathbf{A} \bm{u}^m \| + \rho^* \| \bm{u}^m \|_{L^{6}(\Omega)} \| \nabla \bm{u}^m \|_{L^{3}(\Omega)} \| \mathbf{A} \bm{u}^m \|
    \\
    & \quad + | \gamma | m_\phi^* \| \nabla \mu_\phi^m \|_{L^{\infty}(\Omega)} \| \nabla \bm{u}^m \| \| \mathbf{A} \bm{u}^m \| + 2 | \gamma | \sigma_1 \| \bm{u}^m \| \| \mathbf{A} \bm{u}^m \|
    \\
    & \quad + \big(  \| \nabla \mu_\phi^m \| + \| \nabla \mu_\psi^m \|  \big) \| \mathbf{A} \bm{u}^m \| + C \| \pi^m \|_{L^{3}(\Omega)} \| \nabla \phi^m \|_{L^{6}(\Omega)} \| \mathbf{A} \bm{u}^m \|
    \\
    & \quad + C \| \nabla \phi^m \|_{L^{6}(\Omega)} \| D \bm{u}^m \|_{L^{3}(\Omega)} \| \mathbf{A} \bm{u}^m \| \\
    & \le \rho^* \| \partial_{t} \bm{u}^m \| \| \mathbf{A} \bm{u}^m \| + C \| D \bm{u}^m \|^{\frac{3}{2}} \| \mathbf{A} \bm{u}^m \|^{\frac{3}{2}}
    \\
    & \quad + C \| \nabla \mu_\phi^m \|^{\frac{1}{4}} \| \nabla \mu_\phi^m \|^{\frac{3}{4}}_{H^{2}(\Omega)} \| D \bm{u}^m \| \| \mathbf{A} \bm{u}^m \| + 2 | \gamma | \sigma_1 \| \bm{u}^m \| \| \mathbf{A} \bm{u}^m \|
    \\
    & \quad + \big(  \| \nabla \mu_\phi^m \| + \| \nabla \mu_\psi^m \|  \big) \| \mathbf{A} \bm{u}^m \| + C \| D \bm{u}^m \|^{\frac{1}{4}} \|\phi^m \|_W \| \mathbf{A} \bm{u}^m \|^{\frac{7}{4}}
    \\
    & \quad + C \| \phi^m \|_W \| D \bm{u}^m \|^\frac{1}{2} \| \mathbf{A} \bm{u}^m \|^{\frac{3}{2}}
    \\
    & \le \frac{\nu_*}{4} \| \mathbf{A} \bm{u}^m \|^2
    + \epsilon \| \nabla \mu_\phi^m \|_{H^{2}(\Omega)}^2
    + C \| \partial_{t} \bm{u}^m \|^2
    + C \big(\| D \bm{u}^m \|^6 +\| D \bm{u}^m \|^2\big)
    \\
    & \quad   + C_{\epsilon} \| \nabla \mu_\phi^m \|^2 \| D \bm{u}^m \|^8
    + C \big( \| \nabla \mu_\phi^m \|^2 + \| \nabla \mu_\psi^m \|^2 \big)
    \\
    & \quad  + C \| \phi^m \|_W^8 \| D \bm{u}^m \|^2
    + C \| \phi^m \|_W^4 \| D \bm{u}^m \|^2,
\end{align*}
which further implies
\begin{align}
    \big( \nu ( \phi^m ) \mathbf{A} \bm{u}^m, \mathbf{A} \bm{u}^m \big)
    & \le 4 \epsilon \| \nabla \mu_\phi^m \|_{H^{2}(\Omega)}^2 + C_{\epsilon} \| \nabla \mu_\phi^m \|^2 \| D \bm{u}^m \|^8
    + C \| \partial_{t} \bm{u}^m \|^2
    \nonumber \\
    & \quad + C \big(\| D \bm{u}^m \|^6 + \| D\bm{u}^m \|^2\big)
    + C \big( \| \nabla \mu_\phi^m \|^2 + \| \nabla \mu_\psi^m \|^2 \big)
    \nonumber \\
    & \quad + C \| \phi^m \|_W^8 \| D \bm{u}^m \|^2
    + C \| \phi^m \|_W^4 \| D \bm{u}^m \|^2
    \nonumber \\
    & \le 4 \epsilon \| \nabla \mu_\phi^m \|_{H^{2}(\Omega)}^2
    + C_9 \| \partial_{t} \bm{u}^m \|^2 + C_{\epsilon} (1+C_4 + H_m )^5.
    \label{sp_h2}
\end{align}
Then, combining \eqref{energy_inequality}, \eqref{mhi_h3}, \eqref{msi_h3} and \eqref{sp_h2} together, we obtain
\begin{align*}
    & \frac{\mathrm{d}}{\mathrm{d t}} H_m
    + \| \nabla \partial_{t} \phi^m \|^2
    + \beta \| \nabla \partial_{t} \psi^m \|^2
    + \rho_* \| \partial_{t} \bm{u}^m \|^2 \nonumber\\
    &\qquad + \omega \big(  \| \nabla \mu_\phi^m \|_{H^{2}(\Omega)}^2 + \| \nabla \mu_\psi^m \|_{H^{2}(\Omega)}^2 + \nu_* \| \mathbf{A} \bm{u}^m \|^2 \big)
    \nonumber \\
    & \quad \le 8 \xi \big(  \| \nabla \partial_{t} \phi^m \|^2 + \beta \| \nabla \partial_{t} \psi^m \|^2 + \| \partial_{t} \bm{u}^m \|^2 \big)
    \nonumber \\
    & \qquad + 8 \xi \big(  \| \mathbf{A} \bm{u}^m \|^2 + \| \nabla \mu_\phi^m \|^2_{H^{2}(\Omega)} + \| \nabla \mu_\psi^m \|^2_{H^{2}(\Omega)}  \big)
    \nonumber \\
    & \qquad + 2 \omega \epsilon \| \mathbf{A} \bm{u}^m \|^2 + \omega C_{10} \big(  \| \nabla \partial_{t} \phi^m \|^2 + \| \nabla \partial_{t} \psi^m \|^2  \big)
    + 4 \omega \epsilon \| \nabla \mu_\phi^m \|^2_{H^{2}(\Omega)} + \omega C_9 \| \partial_{t} \bm{u}^m \|^2 \\
    & \qquad + (  C_{\xi} + \omega C_{\epsilon}  ) (  1 + C_4 + H_m  )^5.
\end{align*}
Setting
$$
\epsilon = \frac{\nu_*}{8 + 16 \nu_*},
\quad
\omega = \frac{1}{ 4 C_{10} + 4 C_{10}/\beta + 4 C_9/\rho_* },
\quad
\xi = \frac{1}{32 + 32/\rho_* + 32/\omega + 32/ ( \omega \nu_*) },
$$
we infer from the above inequality that
\begin{align*}
    & \frac{\mathrm{d}}{\mathrm{d t}} H_m
    + \frac{1}{2} \big(  \| \nabla \partial_{t} \phi^m \|^2 + \beta \| \nabla \partial_{t} \psi^m \|^2 + \rho_* \| \partial_{t} \bm{u}^m \|^2  \big) \nonumber \\
    & \qquad + \frac{\omega}{2} \big(  \| \nabla \mu_\phi^m \|_{H^{2}(\Omega)}^2 + \| \nabla \mu_\psi^m \|_{H^{2}(\Omega)}^2 + \nu_* \| \mathbf{A} \bm{u}^m \|^2  \big)
    \nonumber \\
    & \quad \le C_{11} ( 1 + C_4 + H_m )^5,
\end{align*}
where $C_{11}>0$ is independent of $m$, $\alpha$ and $k$. Therefore, for
\[
0<\widetilde{T} < \frac{1}{ 4C_{11}^{-1} (  1 + C_4 + H_m(0) )^4 },
\]
it holds
\[
1+C_4 + H_m(t) \le \frac{1 + C_4 + H_m (0)}{\big( 1-4C_{11}t( 1+C_4 + H_m(0) )^4 \big)^{\frac{1}{4}}},
\quad \forall\, t \in [0,\widetilde{T}] \cap [0,T_m).
\]
To derive an estimate of $H_m$ that is independent of $m$, $\alpha$ and $k$, we are left to control $\alpha \| \partial_{t} \phi^m(0) \|^2$, $\alpha \| \partial_{t} \psi^m(0) \|^2$, $\| \nabla \mu_\phi^m(0) \|^2$ and $\| \nabla \mu_\psi^m (0)\|^2$.
Testing \eqref{eq:semi_galerkin_nsch3} with $\partial_{t} \phi^m - \partial_t \overline{\phi^m}$ and using  \eqref{eq:semi_galerkin_nsch2}, we get
\begin{align*}
    & \alpha \| \partial_{t} \phi^m - \partial_t \overline{\phi^m} \|^2 + \int_{\Omega}  m_\phi ( \phi^m ) | \nabla \mu_\phi^m |^2  \, \mathrm{d}x
    \\
    & \quad = \big(  \nabla (  -\Delta \phi^m + \sigma_2 \mathcal{N} (\phi^m - \overline{\phi^m}) + F_\phi' ( \phi^m ) + \partial_\phi G ( \phi^m,\psi^m )  ) , m_\phi ( \phi^m ) \nabla \mu_\phi^m - \phi^m \bm{u}^m \big)
    \\
    & \qquad + ( \nabla \mu_\phi^m, \phi^m \bm{u}^m ).
\end{align*}
By continuity of $\bm{u}^m, \phi^m, \psi^m, \partial_t\phi^m, \mu_\phi^m$, it follows from the above identity that
\begin{align*}
    & \alpha \| \partial_{t} \phi^m(0) - \partial_t \overline{\phi^m}(0) \|^2 + \int_{\Omega}  m_\phi ( \phi_{0,k} ) | \nabla \mu_\phi^m(0) |^2  \, \mathrm{d}x
    \\
    & \quad = \big(  \nabla (  -\Delta \phi_{0,k} + F_\phi' ( \phi_{0,k} ) + \partial_\phi G ( \phi_{0,k},\psi_{0,k} ) + \sigma_2 \mathcal{N} (\phi_{0,k} - \overline{\phi_{0,k}} )   ) , m_\phi ( \phi_{0,k} ) \nabla \mu_\phi^m(0) - \phi_{0,k} \bm{u}^m(0) \big)
    \\
    & \qquad + \big( \nabla \mu_\phi^m(0), \phi_{0,k} \bm{u}^m(0) \big),
\end{align*}
which implies
\begin{align*}
    \alpha \| \partial_{t} \phi^m(0) \|^2 + \| \nabla \mu_\phi^m(0) \|^2 & \le 2 \alpha \sigma_1^2 | \Omega | \big( \overline{\phi_{0,k}} - c \big)^2 + C \|   \nabla (  -\Delta \phi_{0,k} + F_\phi' ( \phi_{0,k} )  )  \|^2 \\
    & \quad + C \big(  \| \nabla \phi_{0,k} \|^2 + \| \nabla \psi_{0,k} \|^2  \big)
    + C \| \bm{u}^m(0) \|^2.
\end{align*}
Similarly, it holds
\begin{align*}
    \alpha \| \partial_{t} \psi^m(0) \|^2 + \| \nabla \mu_\psi^m(0) \|^2
    & \le C \|   \nabla (  -\beta \Delta \psi_{0,k} + F_\psi' ( \psi_{0,k} )  )  \|^2
    \\
    & \quad +  C \big(  \| \nabla \phi_{0,k} \|^2 + \| \nabla \psi_{0,k} \|^2  \big)
    + C \| \bm{u}^m(0) \|^2.
\end{align*}
Thus, we deduce from \eqref{hm_upper_bound} and Lemma \ref{approximate_initial_phase_field} that
\begin{align*}
    H_m(0) & \le C \big(  1 + \| \bm{u}_0 \|^2_{\bm{H}_{0,\sigma}^{1}(\Omega)} + \| \phi_0 \|^2_{V} + \| \psi_0 \|^2_{V}  \big)
    \\
    & \quad + C \big(  \|  -\Delta \phi_{0} + F_\phi' ( \phi_{0} )  \|_V^2 + \|   -\beta \Delta \psi_{0} + F_\psi' ( \psi_{0} ) \|_V^2  \big) + C_5 =: K_4,
\end{align*}
for $k$ sufficiently large.
Setting now $\widetilde{T}_* = \big(  8C_{11} ( 1+ C_4 + K_4 )^4  \big)^{-1}$, we have
\begin{equation} \label{higher_order_boundedness_Linfinity}
    0 \le 1 + C_4 + H_m(t) \le  \left(\frac12\right)^{-\frac14}(1+C_4+K_4), \quad \forall\, t \in [0,\widetilde{T}_*] \cap [0,T_m).
\end{equation}
This allows us to define all approximate solutions $( \bm{u}^m,\phi^m,\psi^m,\mu_\phi^m, \mu_\psi^m )$ on the finite interval $[0,\widetilde{T}_*]$, which is independent of $\alpha$, $m$ and $k$ (with $k$ being sufficiently large). Moreover, we have the following uniform estimate
\begin{align}
    & \int_{0}^{t}  \big( \| \nabla \partial_{t} \phi^m \|^2 + \| \nabla \partial_{t} \psi^m \|^2 + \| \partial_{t} \bm{u}^m \|^2 + \| \nabla \mu_\phi^m \|_{H^{2}(\Omega)}^2 + \| \nabla \mu_\psi^m \|_{H^{2}(\Omega)}^2 +\| \mathbf{A} \bm{u}^m \|^2 \big)
    \, \mathrm{d}\tau
    \nonumber \\
    & \quad \le K_5, \quad \forall\, t \in [0,\widetilde{T}_*].
    \label{higher_order_boundedness_L2}
\end{align}
Hence, from \eqref{eq:semi_galerkin_nsch3}, \eqref{eq:semi_galerkin_nsch5} and \cite{GGW2018} it follows that
\begin{equation*}
\phi^m, \psi^m\in L^{\infty} \big( 0,\widetilde{T}_*;W \big),\quad F'_\phi(\phi^m), F'_\psi(\phi^m)\in L^{\infty} \big( 0,\widetilde{T}_*;H \big),
\end{equation*}
with uniform bounds in $\alpha$, $m$, $k$. Then by  \eqref{eq:semi_galerkin_nsch2}, \eqref{eq:semi_galerkin_nsch4} and Lemma \ref{weightedlaplace}, we further get
\begin{equation*}
\mu_\phi^m, \mu_\psi^m \in L^{\infty} \big( 0,\widetilde{T}_*;W\big),
\end{equation*}
with uniform bounds in $m$, $k$ (but dependent on $\alpha$).
Finally, due to Lemma \ref{approximate_initial_phase_field} and Proposition \ref{cvchossr}, we find some $\delta(\alpha,k)\in (0,1/2)$ depending on $\alpha$, $k$, but independent of $m$, such that the strict separation property \eqref{delta-al} holds.
As a consequence, from \eqref{eq:semi_galerkin_nsch3}, \eqref{eq:semi_galerkin_nsch5} and the elliptic regularity theory, we can conclude that
$\phi^m, \psi^m \in L^{\infty} ( 0,\widetilde{T}_*;H^{3}_N(\Omega) )$.
\qed
\medskip

\noindent \textbf{Completion of the Proof of Theorem \ref{nschsse3d}.}
Based on the uniform estimates established in Proposition \ref{uniform_approximate_solution}, we are able to first pass to the limit as $m\to \infty$ in the semi-Galerkin scheme, then take the vanishing viscosity limit as $\alpha \to 0$, and finally pass to the limit as $k \to \infty$ regarding the approximation of the initial data. The limiting procedure follows standard compactness arguments as in \cite{G2022} and is therefore omitted for brevity.
In the end, we obtain a strong solution $( \bm{u},\phi,\psi,\mu_\phi,\mu_\psi )$ to problem \eqref{eq:nsch1}--\eqref{eq:nschi} on $[0,T_0]$ with $T_0=\widetilde{T}_*$ in the sense of Definition \ref{nschssd}. The pressure term $\pi$ can be recovered via the classical approach for Navier--Stokes equations in \cite{G2011}. This completes the proof of Theorem \ref{nschsse3d}-(1).

The uniqueness of strong solutions can be established under additional assumptions described in Theorem \ref{nschsse3d}-(2). Indeed, the established regularity
    \[
    \phi, \psi \in L^{\infty}(0, T_0; W^{2,6}(\Omega)) \cap H^1(0, T_0; V) \subset C^{\frac{5}{16}}([0, T_0]; C(\overline{\Omega}))
    \]
    ensures, by (H\"{o}lder) continuity, the existence of a short time interval $[0, T_1]\subset [0,T_0]$ on which $\phi$ and $\psi$ inherit the strict separation properties of the initial data (cf. \cite{CGGS2026,G2022}).
    This crucial fact guarantees that the singularities of the potentials do not interfere with the estimate of the difference between strong solutions.
    Then the proof of uniqueness is essentially the same as in \cite[Section 4]{G2022}  by showing the continuous dependence of strong solutions on the initial data. Let $(\bm{u}_i,\phi_i,\psi_i,\mu_{\phi,i},\mu_{\psi,i})$, $i \in \{ 1,2 \}$, be two strong solutions to problem \eqref{eq:nsch1}--\eqref{eq:nschi}. Denote their differences by
     \[
     \big(  \widetilde{\bm{u}}, \widetilde{\phi}, \widetilde{\psi}, \widetilde{\mu_\phi}, \widetilde{\mu_\psi}  \big) \define (\bm{u}_1,\phi_1,\psi_1,\mu_{\phi,1},\mu_{\psi,1}) - (\bm{u}_2,\phi_2,\psi_2,\mu_{\phi,2},\mu_{\psi,2}).
     \]
     Taking the difference in equations satisfied by the two solutions and deriving estimates as in \cite{G2022}, due to the concentration-dependent mobilities, we have the following additional terms
     \begin{align*}
         & \int_\Omega \nabla(m_\phi(\phi_1)) \cdot \nabla \widetilde{\mu_\phi} \Delta^2 \widetilde{\phi} \, \mathrm{d} x,
         \quad  \int_\Omega \nabla(m_\psi(\psi_1)) \cdot \nabla \widetilde{\mu_\psi} \Delta^2 \widetilde{\psi} \, \mathrm{d} x,
         \\
         & \int_\Omega \mathrm{div}\, \big( \big( m_\phi(\phi_1)-m_\phi(\phi_2) \big) \nabla \mu_{\phi,2} \big) \Delta^2 \widetilde{\phi} \, \mathrm{d} x,
         \\
         & \int_\Omega \mathrm{div}\, \big(  \big(  m_\psi(\psi_1)-m_\psi(\psi_2)  \big) \nabla \mu_{\psi,2}  \big) \Delta^2 \widetilde{\psi} \, \mathrm{d} x.
     \end{align*}
     In view of equations \eqref{eq:nsch4} and \eqref{eq:nsch6}, we can rewrite the first term as
     \begin{align*}
         & \int_\Omega \nabla(m_\phi(\phi_1)) \cdot \big(  - \nabla \Delta \widetilde{\phi} + F_\phi''(\phi_1) \nabla \phi_1 - F_\phi''(\phi_2) \nabla \phi_2 + \sigma_2 \nabla \mathcal{N} \big( \widetilde{\phi} - \overline{\widetilde{\phi}} \big)  \big) \Delta^2 \widetilde{\phi} \, \mathrm{d} x
         \\
         & \qquad + \int_\Omega \nabla(m_\phi(\phi_1)) \cdot \big(  \nabla (\partial_\phi G(\phi_1,\psi_1) - \partial_\phi G(\phi_2,\psi_2))  \big) \Delta^2 \widetilde{\phi} \, \mathrm{d} x.
     \end{align*}
     Due to the strict separation property of $\phi_1$ and $\phi_2$, the above expression can be bounded by
     \[
     \epsilon \| \Delta^2 \widetilde{\phi} \|^2
     + C_\epsilon \big(\| \Delta \widetilde{\phi} \|^2 + \| \Delta \widetilde{\psi} \|^2\big)
     + C \big|\overline{\widetilde{\phi}}\big|,
     \]
     where $\epsilon>0$ can be arbitrarily small.
    Analogously, the other three terms can be bounded by
    \[
    \epsilon \big(\| \Delta^2 \widetilde{\phi} \|^2 + \beta \| \Delta^2 \widetilde{\psi} \|^2\big)
    + C_\epsilon \big(\| \Delta \widetilde{\phi} \|^2 + \| \Delta \widetilde{\psi} \|^2\big)
    + C \big|\overline{\widetilde{\phi}}\big|.
    \]
    With the above observations, we can derive the desired continuous dependence on the initial data following the argument in \cite{G2022}. This establishes the uniqueness of strong solutions on $[0,T_1]$, and we omit the remaining details.

The proof of Theorem \ref{nschsse3d} is complete.
\qed

\section{Global Well-posedness in 2D} \label{2d_strong_proof}
\setcounter{equation}{0}

This section is devoted to the proof of Theorem \ref{nschsse2d}, namely, the global well-posedness in two dimensions.
We first study the convective Cahn--Hilliard system \eqref{eq:nsch3}--\eqref{eq:nsch6} with a given solenoidal velocity field.
The improved regularity of its solution allows us to obtain further regularity of the velocity through the momentum equation \eqref{eq:nsch1}.
This argument, combined with the existence of a global weak solution to the original problem (cf. \cite{OGW2026}), gives the existence of a global strong solution.

\subsection{A convective Cahn--Hilliard system}

The following result extends \cite[Theorem 2.4]{AGG2024} on a convective Cahn--Hilliard equation with constant mobility and \cite[Theorem 1.2 (A)--(C)]{CGGG2025} on a Cahn--Hilliard equation with nondegenerate mobility but without convection.
\begin{proposition}\label{cvcho_2d_global_sse}
Let $\Omega \subset \mathbb{R}^{2}$ be a bounded domain of class $C^3$.
Suppose that Assumptions $(\mathbf{H1})$ and $(\mathbf{H2})$ hold.
\begin{itemize}
    \item[(1)] Let $\bm{u} \in L^{2} ( 0,\infty;\bm{L}^{2}_\sigma(\Omega) )$ be a given velocity field. Then for any initial data $\phi_0,\psi_0 \in V \cap L^{\infty}(\Omega)$ with $\phi_0 \in [-1,1]$, $\psi_0 \in [0,1]$ almost everywhere in $\Omega$ and $\overline{\phi_0} \in (-1,1)$, $\overline{\psi_0} \in (0,1)$,
    problem \eqref{eq:cvcho1}--\eqref{eq:cvchoi} with $\alpha=0$ admits a unique global weak solution $( \phi,\psi,\mu_\phi,\mu_\psi )$ such that
    \begin{align*}
        & \phi,\psi \in BC \big( [0,\infty);V \big) \cap L^{4}_{\mathrm{uloc}} \big( [0,\infty);W \big) \cap L^{2}_{\mathrm{uloc}} \big( [0,\infty);W^{2,p}(\Omega) \big),
        \\
        & \phi,\psi \in L^{\infty} ( Q ) \ \mathrm{with} \ \phi(x,t) \in (-1,1) \ \mathrm{and} \ \psi(x,t) \in (0,1) \ \mathrm{a.e.} \ \mathrm{in} \ Q,
        \\
        & \partial_{t} \phi,\partial_{t} \psi \in L^{2} \big( 0,\infty;V^* \big), \quad  F_\phi' ( \phi ), F_\psi' ( \psi ) \in L^{2}_{\mathrm{uloc}} \big( [0,\infty);L^{p}(\Omega) \big),
        \\
        & \mu_\phi,\mu_\psi \in L^{2}_{\mathrm{uloc}} \big( [0,\infty);V \big) \ \mathrm{with} \ \nabla \mu_\phi,\nabla \mu_\psi \in L^{2} \big( 0,\infty;H \big),
    \end{align*}
    for all $p \in [2,\infty)$, and
    \begin{align*}
        & \langle  \partial_{t} \phi + \sigma_1 ( \overline{\phi}-c ) ,\zeta \rangle + (  m_\phi( \phi ) \nabla \mu_\phi - \phi \bm{u} , \nabla \zeta ) = 0, && \forall \,\zeta \in V, \ \mathrm{a.e.} \ \mathrm{on} \ (0,\infty),
        \\
        & \mu_\phi = - \Delta \phi + \sigma_2 \mathcal{N} ( \phi - \overline{\phi} ) + F_\phi' ( \phi ) + \partial_\phi G ( \phi,\psi ), && \mathrm{a.e.} \ \mathrm{in} \ Q,
        \\
        & \langle  \partial_{t} \psi ,\zeta \rangle + (  m_\psi( \psi ) \nabla \mu_\psi - \psi \bm{u} , \nabla \zeta ) = 0, && \forall\, \zeta \in V, \ \mathrm{a.e.} \ \mathrm{on} \ (0,\infty),
        \\
        & \mu_\psi = - \beta \Delta \psi + F_\psi' ( \psi ) + \partial_\psi G ( \phi,\psi ), && \mathrm{a.e.} \ \mathrm{in} \ Q,
    \end{align*}
    and $( \phi,\psi )|_{t=0} = ( \phi_0,\psi_0 )$ almost everywhere in $\Omega$.
    Moreover, the following energy identity holds
    \begin{align*}
        & E_{\mathrm{free}} ( \phi(t),\psi(t) ) + \int_{Q_t}  m_\phi( \phi ) | \nabla \mu_\phi |^2 + m_\psi( \psi ) | \nabla \mu_\psi |^2  \, \mathrm{d}(x,\tau)
        \\
        & \qquad + \int_{0}^{t} \sigma_1 | \Omega | ( \overline{\phi}-c ) \overline{\mu_\phi} \, \mathrm{d}\tau + \int_{\Omega_t}  \bm{u} \cdot (  \mu_\phi \nabla \phi + \mu_\psi \nabla \psi  )  \, \mathrm{d}(x,\tau)
        \\
        & \quad = E_{\mathrm{free}} ( \phi_0,\psi_0 ),
        \quad \forall\,  t \in [0,\infty).
    \end{align*}
\item[(2)]  Assume that $\bm{u}\in L^{2} ( 0,\infty;\bm{H}^{1}_{0,\sigma}(\Omega) )$, $\phi_0,\psi_0 \in W$ with $\phi_0 \in [-1,1]$, $\psi_0 \in [0,1]$ almost everywhere in $\Omega$, $\overline{\phi_0} \in (-1,1)$, $\overline{\psi_0} \in (0,1)$ and $-\Delta \phi_0 + F_\phi'(\phi_0), -\beta \Delta \psi_0 + F_\psi'(\psi_0) \in V$, then problem \eqref{eq:cvcho1}--\eqref{eq:cvchoi} with $\alpha=0$ admits a unique global strong solution $( \phi,\psi,\mu_\phi,\mu_\psi )$ such that
    \begin{align*}
        & \phi,\psi \in L^{\infty} ( 0,\infty;W^{2,p}(\Omega) ), \quad  \partial_{t} \phi, \partial_{t} \psi \in L^{2} ( 0,\infty;V ),
        \\
        & \phi,\psi \in L^{\infty} ( Q ) \ \ \mathrm{with} \ \ \phi(x,t) \in (-1,1) \ \mathrm{and} \ \psi(x,t) \in (0,1) \ \mathrm{a.e.} \ \mathrm{in} \ Q,
        \\
        & \mu_\phi,\mu_\psi \in L^{\infty} ( 0,\infty; V) \ \mathrm{with} \ \nabla \mu_\phi,\nabla \mu_\psi \in L^{2} ( 0,\infty;H^{2}(\Omega) ),
        \\
        & F_\phi'( \phi ), F_\psi'( \psi ) \in L^{\infty} ( 0,\infty;L^{p}(\Omega) ),
    \end{align*}
    for all $p \in [2,\infty)$. The strong solution satisfies the system \eqref{eq:cvcho1}--\eqref{eq:cvcho4} with $\alpha=0$ almost everywhere in $Q$ and the boundary conditions in \eqref{eq:cvchob} almost everywhere on $S$.
In addition, if $\bm{u}\in L^{\infty} ( 0,\infty;\bm{L}^{2}_\sigma(\Omega) )$, then $\partial_{t} \phi, \partial_{t} \psi \in L^{\infty} ( 0,\infty;V^* )$.
Moreover, if the assumption $(\mathbf{H2*})$ holds, $(\phi,\psi)$ satisfies the uniform strict separation property on $[0,\infty)$ and
    \begin{align*}
        & \phi, \psi \in L^{\infty} ( 0,\infty;H^{3}_N(\Omega) ),
        \quad \mu_\phi,\mu_\psi \in BC( [0,\infty);V ),
        \quad
        \partial_{t} \mu_\phi,\partial_{t} \mu_\psi \in L^{2} ( 0,\infty;V^* ).
    \end{align*}
\end{itemize}
\end{proposition}
\noindent \textbf{Proof.} We divide the proof of Proposition \ref{cvcho_2d_global_sse} into two parts.
\smallskip 

\noindent \textbf{Part 1. Global weak solutions.} The existence of a global weak solution can be established by standard arguments as in \cite{AGG2024,BB1999,CGGG2025}. We note that in \cite{AGG2024}, it was required that the velocity field satisfies $\bm{u}\in L^{2} ( 0,\infty;\bm{H}^{1}_{0,\sigma}(\Omega) )$. This condition was relaxed to $\bm{u} \in L^{2} ( 0,\infty;\bm{L}^{2}_\sigma(\Omega) )$ in \cite{GKS2026} for a bulk-surface Cahn--Hilliard model with convection, which also applies here.

Next, we investigate the uniqueness. This can be proved by showing the continuous dependence of weak solutions on the initial data, as in \cite{CGGG2025}. Suppose that $( \phi_i,\psi_i,\mu_{\phi_i},\mu_{\psi_i} )$, $i \in \{ 1,2 \}$, are two weak solutions corresponding to the initial data $( \phi_{0,i},\psi_{0,i} )$, respectively.
Define the difference
$( \widetilde{\phi},\widetilde{\psi},\widetilde{\mu_\phi},\widetilde{\mu_\psi} )=( \phi_1,\psi_1,\mu_{\phi_1},\mu_{\psi_1} ) - ( \phi_2,\psi_2,\mu_{\phi_2},\mu_{\psi_2} )$.
Then we have
\begin{align}
    & \big\langle  \partial_{t} \widetilde{\phi} + \sigma_1 \overline{\widetilde{\phi}} ,\zeta \big\rangle + \big(  m_\phi( \phi_1 ) \nabla \widetilde{\mu_\phi} + ( m_\phi( \phi_1 )-m_\phi( \phi_2 ) ) \nabla \mu_{\phi_2} - \widetilde{\phi} \bm{u} , \nabla \zeta \big) = 0,
    \label{ccho_difference_1}
    \\
    & \big\langle  \partial_{t} \widetilde{\psi} ,\zeta \big\rangle + \big(  m_\psi( \psi_1 ) \nabla \widetilde{\mu_\psi} + ( m_\psi( \psi_1 )-m_\psi( \psi_2 ) ) \nabla \mu_{\psi_2} - \widetilde{\psi} \bm{u} , \nabla \zeta \big) = 0, \label{ccho_difference_3}
\end{align}
for all $\zeta \in V$ and almost everywhere in $(0,\infty)$, with
\begin{align}
    & \widetilde{\mu_\phi} = - \Delta \widetilde{\phi} + \sigma_2 \mathcal{N} ( \widetilde{\phi} - \overline{\widetilde{\phi}} ) + F_\phi' ( \phi_1 ) - F_\phi' ( \phi_2 ) + \partial_\phi G ( \phi_1,\psi_1 ) - \partial_\phi G ( \phi_2,\psi_2 ), && \mathrm{a.e.} \ \mathrm{in} \ Q, \label{ccho_difference_2}
    \\
    & \widetilde{\mu_\psi} = - \beta \Delta \widetilde{\psi} + F_\psi' ( \psi_1 ) - F_\psi' ( \psi_2 ) + \partial_\psi G ( \phi_1,\psi_1 ) - \partial_\psi G ( \phi_2,\psi_2 ), && \mathrm{a.e.} \ \mathrm{in} \ Q, \label{ccho_difference_4}
\end{align}
and $ ( \widetilde{\phi},\widetilde{\psi} )|_{t=0} = ( \widetilde{\phi}_0, \widetilde{\psi}_0 ) \define  ( \phi_{0,1},\psi_{0,1} ) - ( \phi_{0,2},\psi_{0,2} ) $.

Testing \eqref{ccho_difference_1} and \eqref{ccho_difference_3} with $g_\phi=\mathcal{N}_{\phi, \phi_1} \big( \widetilde{\phi}-\overline{\widetilde{\phi}} \big)$ and $g_\psi=\mathcal{N}_{\psi, \psi_1} \big( \widetilde{\psi} - \overline{\widetilde{\psi}} \big)$, respectively, applying Lemma \ref{mutiply_diff_formula_general}--(2), we obtain
\begin{align*}
    & \frac{1}{2} \frac{\mathrm{d}}{\mathrm{d t}}  \Big(  \Big\|  \sqrt{m_\phi( \phi_1 )} \nabla g_\phi  \Big\|^2 +  \Big\|  \sqrt{m_\psi( \psi_1 )} \nabla g_\psi  \Big\|^2  \Big)  + \big\| \nabla \widetilde{\phi} \big\|^2 + \beta \big\|  \nabla \widetilde{\psi}  \big\|^2 + \sigma_2 \big\|  \widetilde{\phi} - \overline{\widetilde{\phi}}  \big\|^2_*
    = \sum\limits_{i=1}^6 I_i,
\end{align*}
where
\begin{align*}
    & I_1 = \int_{\Omega}  \big(  \big(  \widetilde{\phi}-\overline{\widetilde{\phi}}  \big) \nabla g_\phi + \big(  \widetilde{\psi} - \overline{\widetilde{\psi}}  \big) \nabla g_\psi  \big) \cdot \bm{u}  \, \mathrm{d}x,
    \\
    & I_2 = - \frac{1}{2} \int_{\Omega} \big( m_\phi''( \phi_1 ) \nabla \mathcal{N} (\partial_{t} \phi_1 - \partial_{t} \overline{\phi_1} ) \cdot \nabla \phi_1 | \nabla g_\phi |^2 + m_\psi''( \psi_1 ) \nabla \mathcal{N} \partial_{t} \psi_1 \cdot \nabla \psi_1 | \nabla g_\psi |^2 \big)\, \mathrm{d}x,
    \\
    & I_3 = - \int_{\Omega} \big[ m_\phi'( \phi_1 ) \nabla \mathcal{N} (\partial_{t} \phi_1 -\partial_{t} \overline{\phi_1}) \cdot ( \nabla^2 g_\phi \nabla g_\phi ) + m_\psi'( \psi_1 ) \nabla \mathcal{N} \partial_{t} \psi_1 \cdot ( \nabla^2 g_\psi \nabla g_\psi ) \big] \, \mathrm{d}x,
    \\
    & I_4 = - \int_{\Omega}  \big[ ( m_\phi ( \phi_1 )-m_\phi ( \phi_2 ) ) \nabla \mu_{\phi_2} \cdot \nabla g_\phi + ( m_\psi ( \psi_1 )-m_\psi ( \psi_2 ) ) \nabla \mu_{\psi_2} \cdot \nabla g_\psi \big]  \, \mathrm{d}x,
    \\
    & I_5 = - (  F_\phi'(\phi_1) - F_\phi'(\phi_2) + \partial_\phi G(\phi_1,\psi_1) - \partial_\phi G(\phi_2,\psi_2) ,  \widetilde{\phi} - \overline{\widetilde{\phi}}  )
    \\
    & \qquad - (  F_\psi'(\psi_1) - F_\psi'(\psi_2) + \partial_\psi G(\phi_1,\psi_1) - \partial_\psi G(\phi_2,\psi_2) ,  \widetilde{\psi} - \overline{\widetilde{\psi}}  ), \\
    & I_6 = \frac{1}{2} \partial_t \overline{\phi_1} \int_\Omega m_\phi'(\phi_1) | \nabla g_\phi |^2 \, \mathrm{d} x.
\end{align*}
It easily follows that
\[
| I_6 | \le \sigma_1\| m_\phi \|_{C^1([-1,1])} \| \nabla g_\phi \|^2.
\]
Next, the terms $I_2$, $I_3$ and $I_4$ can be  estimated as in \cite{CGGG2025} with minor modifications:
\begin{align*}
    | I_2| + | I_3 |
    & \le \epsilon \big(  \| \nabla \widetilde{\phi} \|^2 + \beta \| \nabla \widetilde{\psi} \|^2  \big)
    + C_\epsilon \big(  \| \partial_{t} \phi_1 - \partial_{t} \overline{\phi_1}\|_{V^*}^2
    + \| \phi_1 \|_W^4  \big) \| \nabla g_\phi \|^2
    \\
    & \quad + C_\epsilon \big(  \| \partial_{t} \psi_1 \|_{V^*}^2 + \| \psi_1 \|_W^4  \big) \| \nabla g_\psi \|^2,
\end{align*}
and
\begin{align*}
    | I_4 | & \le \epsilon \big(  \| \nabla \widetilde{\phi} \|^2 + \beta \| \nabla \widetilde{\psi} \|^2  \big)
    + C_\epsilon \big(  \| \nabla \mu_{\phi_2} \|^2 + \| \phi_1 \|_W^4  \big) \| \nabla g_\phi \|^2 + C_\epsilon \big( \| \nabla \mu_{\psi_2} \|^2 + \| \psi_1 \|_W^4 \big) \| \nabla g_\psi \|^2
    \\
    & \quad + C_\epsilon \Big(  \big( 1 + \| \nabla \mu_{\phi_2} \|^2 + \| \phi_1 \|_W^4  \big) \big| \overline{\widetilde{\phi}} \big|^2 + \big( 1 + \| \nabla \mu_{\psi_2} \|^2 + \| \psi_1 \|_W^4 \big) \big| \overline{\widetilde{\psi}} \big|^2  \Big),
\end{align*}
where $\epsilon \in (0,1]$ is an arbitrary constant to be determined later. Below we focus on $I_1$ and $I_5$.
For $I_1$, using H\"{o}lder's inequality, we get
\begin{equation*}
    I_1 \le \| \bm{u} \| \big(  \| \nabla g_\phi \|_{L^{4}(\Omega)} \big\|  \widetilde{\phi}-\overline{\widetilde{\phi}}  \big\|_{L^{4}(\Omega)} + \| \nabla g_\psi \|_{L^{4}(\Omega)} \big\|  \widetilde{\psi}-\overline{\widetilde{\psi}}  \big\|_{L^{4}(\Omega)}  \big).
\end{equation*}
Besides, the interpolation inequality and the Gagliardo--Nirenberg inequality imply
\begin{align*}
    & \big\|  \widetilde{\phi}-\overline{\widetilde{\phi}}  \big\|_{L^{4}(\Omega)}  \le C \big\|  \nabla \widetilde{\phi}  \big\|^{\frac{3}{4}} \|  \nabla g_\phi  \|^{\frac{1}{4}},
    \quad
    \| \nabla g_\phi \|_{L^{4}(\Omega)}  \le C \|  \nabla g_\phi  \|^{\frac{1}{2}} \|  \nabla g_\phi  \|^{\frac{1}{2}}_{V}.
\end{align*}
Applying Lemma \ref{elliptic_estimate_2d}, we can deduce that
\begin{align*}
    \| \nabla g_\phi \|_{L^{4}(\Omega)} \big\|  \widetilde{\phi}-\overline{\widetilde{\phi}}  \big\|_{L^{4}(\Omega)} & \le C \| \nabla g_\phi \|^{\frac{3}{4}} \| \nabla g_\phi \|_V^{\frac{1}{2}} \| \nabla \widetilde{\phi} \|^{\frac{3}{4}} \\
    & \le \| \nabla g_\phi \|^{\frac{3}{4}} \| \nabla \widetilde{\phi} \|^{\frac{3}{4}} \big(  \| \nabla \phi_1 \| \| \phi_1 \|_W \| \nabla g_\phi \| + \| \nabla g_\phi \|^\frac{1}{2} \big\| \nabla \widetilde{\phi} \big\|^\frac{1}{2}  \big)^\frac{1}{2} \\
    & \le C \| \phi_1 \|_W^\frac{1}{2} \| \nabla g_\phi \|^{\frac{5}{4}} \big\| \nabla \widetilde{\phi} \big\|^{\frac{3}{4}} + C \| \nabla g_\phi \| \| \nabla \widetilde{\phi} \|.
\end{align*}
Similar estimates hold for $\psi_i$. Therefore, combining the above estimates, we obtain
\begin{align*}
    I_1 & \le C \| \bm{u} \| \| \phi_1 \|_W^\frac{1}{2} \| \nabla g_\phi \|^{\frac{5}{4}} \big\| \nabla \widetilde{\phi} \big\|^{\frac{3}{4}} + C \| \bm{u} \| \| \nabla g_\phi \| \| \nabla \widetilde{\phi} \|
    \\
    & \quad + C \| \bm{u} \| \| \psi_1 \|_W^\frac{1}{2} \| \nabla g_\psi \|^{\frac{5}{4}} \big\| \nabla \widetilde{\psi} \big\|^{\frac{3}{4}} + C \| \bm{u} \| \| \nabla g_\psi \| \| \nabla \widetilde{\psi} \|
    \\
    & \le \epsilon \big(\big\| \nabla \widetilde{\phi} \big\|^2 + \beta \big\| \nabla \widetilde{\psi} \big\|^2\big)
    + C_{\epsilon} \big(  \| \bm{u} \|^2 + \| \phi_1 \|^4_W  \big) \| \nabla g_\phi \|^2
    \\
    & \quad + C_{\epsilon} \big(  \| \bm{u} \|^2 + \| \psi_1 \|^4_W  \big) \| \nabla g_\psi \|^2.
\end{align*}
The estimate of $I_5$ is the same as in the proof of Proposition \ref{cvchosse}, that is,
\begin{align*}
    I_5 & \le C \big(  \big| \overline{\widetilde{\phi}} \big| + \big| \overline{\widetilde{\psi}} \big|  \big) (  1 + \| \nabla \mu_{\phi_1} \| + \| \nabla \mu_{\phi_2} \| + \| \nabla \mu_{\psi_1} \| + \| \nabla \mu_{\psi_2} \|  ) \\
    & \quad + C (  \| \widetilde{\phi} - \overline{\widetilde{\phi}} \|^2 + \| \widetilde{\psi} - \overline{\widetilde{\psi}} \|^2  )  + C \big( \overline{\widetilde{\phi}} \big)^2 + \big( \overline{\widetilde{\psi}} \big)^2
    \\
    & \le C \big(  \big| \overline{\widetilde{\phi}} \big| + \big| \overline{\widetilde{\psi}} \big|  \big) (  1 + \| \nabla \mu_{\phi_1} \| + \| \nabla \mu_{\phi_2} \| + \| \nabla \mu_{\psi_1} \| + \| \nabla \mu_{\psi_2} \|  ) \\
    & \quad + \epsilon (  \| \nabla \widetilde{\phi} \|^2 + \beta \| \nabla \widetilde{\psi} \|^2  ) + C_\epsilon (  \| \nabla g_\phi \|^2 + \| \nabla g_\psi \|^2  )  + C \big( \overline{\widetilde{\phi}} \big)^2 + \big( \overline{\widetilde{\psi}} \big)^2,
\end{align*}
where we have also used the monotonicity of $F_\phi', F_\psi'$.
Taking $\epsilon=1/8$ and using the estimate
\[
\| \partial_{t} f - \partial_{t} \overline{f} \|_{V^*} \le C \big(  \| \bm{u} \| + \| \nabla \mu_f \|  \big), \quad  \forall f \in \{ \phi_1,\phi_2,\psi_1,\psi_2 \},
\]
we can deduce that
\begin{align*}
    & \frac{\mathrm{d}}{\mathrm{d t}}  \Big(  \Big\|  \sqrt{m_\phi( \phi_1 )} \nabla g_\phi  \Big\|^2 +  \Big\|  \sqrt{m_\psi( \psi_1 )} \nabla g_\psi  \Big\|^2  \Big)  + \big\| \nabla \widetilde{\phi} \big\|^2 + \beta \big\|  \nabla \widetilde{\psi}  \big\|^2 + \sigma_2 \big\|  \widetilde{\phi} - \overline{\widetilde{\phi}}  \big\|^2_*
    \\
    & \quad \le h_1(t) \Big(  \Big\|  \sqrt{m_\phi( \phi_1 )} \nabla g_\phi  \Big\|^2 +  \Big\|  \sqrt{m_\psi( \psi_1 )} \nabla g_\psi  \Big\|^2  \Big) + h_2(t) \big(  \big| \overline{\widetilde{\phi}} \big| + \big| \overline{\widetilde{\psi}} \big| + \big( \overline{\widetilde{\phi}} \big)^2 + \big( \overline{\widetilde{\psi}} \big)^2 \big),
\end{align*}
where
\begin{align*}
    & h_1(\cdot) = C \Big(  1 + \| \bm{u} \|^2  + \sum\limits_{f \in \{ \phi_1,\phi_2,\psi_1,\psi_2 \}} \big(\| \nabla \mu_f \|^2 +\|f\|_{W}^4\big)  \Big) \in L^1(0,T), \\
    & h_2(\cdot) = C \Big(  1 + \sum\limits_{f \in \{ \phi_1,\phi_2,\psi_1,\psi_2 \}} \| \nabla \mu_f \|^2  \Big) \in L^1(0,T).
\end{align*}
An application of Gronwall's lemma easily yields
the continuous dependence on the initial data and the uniqueness of weak solutions.
\smallskip

\noindent \textbf{Part 2. Global strong solutions.}
We prove the existence of global strong solutions by constructing approximate solutions with the help of the system \eqref{eq:cvcho1}--\eqref{eq:cvcho4} with $\alpha\in (0,1)$.
For convenience, we first assume that $\bm{u} \in C^\infty_0 (0,\infty;\bm{H}^{2}_{0,\sigma}(\Omega))$. Consider the system \eqref{eq:cvcho1}--\eqref{eq:cvcho4}  subject to the boundary conditions \eqref{eq:cvchob} and the initial conditions:
\[
\phi|_{t=0} = \phi_{0,k}, \quad \psi|_{t=0} = \psi_{0,k},
\]
where $\phi_{0,k}$ and $\psi_{0,k}$ are defined as in Section  \ref{loc-3D-1}.
It follows from Proposition \ref{cvchossr} that there exists a unique global-in-time strong solution $( \phi_{\alpha,k},\psi_{\alpha,k},\mu_{\phi,\alpha,k},\mu_{\psi,\alpha,k})$ on $[0,\infty)$.

Next, we derive uniform estimates independent of $\alpha$ and $k$. Testing \eqref{eq:cvcho1} and \eqref{eq:cvcho3} with $\partial_{t}  \mu_{\phi,\alpha,k} $ and $\partial_{t}  \mu_{\psi,\alpha,k} $, respectively, we obtain
\begin{align*}
    & \frac{1}{2} \frac{\mathrm{d}}{\mathrm{d t}}  \Big(  \big\|  \sqrt{m_\phi( \phi_{\alpha,k} )} \nabla \mu_{\phi,\alpha,k}  \big\|^2 + \big\|  \sqrt{m_\psi( \psi_{\alpha,k} )} \nabla \mu_{\psi,\alpha,k}  \big\|^2 + \alpha \| \partial_{t} \phi_{\alpha,k} \|^2 + \alpha \| \partial_{t} \psi_{\alpha,k} \|^2  \Big)
    \\
    & \qquad + \| \nabla \partial_{t} \phi_{\alpha,k} \|^2 + \beta \| \nabla \partial_{t} \psi_{\alpha,k} \|^2 + \sigma_2 \| \phi_{\alpha,k}-\overline{\phi_{\alpha,k}} \|_{V^*}^2
    \\
    & \qquad + \int_{\Omega} \big( F_\phi''( \phi_{\alpha,k} ) | \partial_{t} \phi_{\alpha,k} |^2 + F_\psi''( \psi_{\alpha,k} ) | \partial_{t} \psi_{\alpha,k} |^2 \big) \, \mathrm{d}x
    \\
    & \quad = \frac{1}{2} \int_{\Omega} \big( m_\phi'( \phi_{\alpha,k} ) \partial_{t} \phi_{\alpha,k} | \nabla \mu_{\phi,\alpha,k} |^2 + m_\psi'( \psi_{\alpha,k} ) \partial_{t} \psi_{\alpha,k} | \nabla \mu_{\psi,\alpha,k} |^2 \big)
    \, \mathrm{d}x
    \\
    & \qquad - \int_{\Omega} \big[ \partial_\phi ^2 G ( \phi_{\alpha,k},\psi_{\alpha,k} ) (\partial_{t} \phi_{\alpha,k})^2 + \partial_\psi ^2 G ( \phi_{\alpha,k},\psi_{\alpha,k} ) (\partial_{t} \psi_{\alpha,k})^2  \big]
    \, \mathrm{d}x
    \\
    & \qquad
    - \int_{\Omega}  2 \partial_{\phi,\psi} ^2 G ( \phi_{\alpha,k},\psi_{\alpha,k} ) \partial_{t} \phi_{\alpha,k} \partial_{t} \psi_{\alpha,k}  \ \mathrm{d}x
    - (  \bm{u} \cdot \nabla \phi_{\alpha,k} , \partial_{t} \mu_{\phi,\alpha,k} ) - (  \bm{u} \cdot \nabla \psi_{\alpha,k} , \partial_{t} \mu_{\psi,\alpha,k} )
    \\
    &\qquad
    - \sigma_1 \AbsNormal{\Omega} ( \overline{\phi_{\alpha,k}} -c ) \partial_{t}  \overline{\mu_{\phi,\alpha,k}}.
\end{align*}
By calculations similar to those in the proof of \cite[Theorem 2.4]{AGG2024}, we get
\begin{align*}
    (  \bm{u} \cdot \nabla \phi_{\alpha,k} , \partial_{t} \mu_{\phi,\alpha,k} )
    & = \frac{\mathrm{d}}{\mathrm{d t}}  \Big(  \alpha \int_{\Omega} \bm{u} \cdot \nabla \phi_{\alpha,k} \partial_{t} \phi_{\alpha,k} \ \mathrm{d}x  \Big)
    + \alpha \int_{\Omega}  \partial_{t} \bm{u} \cdot \nabla \partial_{t} \phi_{\alpha,k} \phi_{\alpha,k}  \, \mathrm{d}x
    \\
    & \qquad + \int_{\Omega}    \nabla \partial_{t} \phi_{\alpha,k} \cdot \nabla  (\bm{u} \cdot \nabla \phi_{\alpha,k} ) \, \mathrm{d}x
    - \sigma_2 \int_\Omega \phi_{\alpha,k} \bm{u} \cdot \nabla \mathcal{N} (\partial_t \phi_{\alpha,k} - \partial_t \overline{\phi_{\alpha,k}}) \, \mathrm{d}x
    \\
    & \qquad
    - \int_{\Omega}  \bm{u} \cdot \nabla \partial_{t} \phi_{\alpha,k} F_\phi'( \phi_{\alpha,k} )  \, \mathrm{d}x
    \\
    & \qquad + \int_{\Omega} \big[ \bm{u} \cdot \nabla \phi_{\alpha,k} (  \partial_\phi ^2 G ( \phi_{\alpha,k},\psi_{\alpha,k} ) \partial_{t} \phi_{\alpha,k} + \partial_{\phi,\psi} ^2 G ( \phi_{\alpha,k},\psi_{\alpha,k} ) \partial_{t} \psi_{\alpha,k}  ) \big] \, \mathrm{d}x.
\end{align*}
Similar results can be obtained for $(  \bm{u} \cdot \nabla \psi_{\alpha,k} , \partial_{t} \mu_{\psi,\alpha,k} )$. Then it holds
\begin{align*}
    & \frac{\mathrm{d}}{\mathrm{d t}} \mathcal{E}_{\alpha,k} + \| \nabla \partial_{t} \phi_{\alpha,k} \|^2 + \beta \| \nabla \partial_{t} \psi_{\alpha,k} \|^2 + \sigma_2 \| \phi_{\alpha,k}-\overline{\phi_{\alpha,k}} \|_{V^*}^2
    \\
    & \qquad + \int_{\Omega} \big( F_\phi''( \phi_{\alpha,k} ) | \partial_{t} \phi_{\alpha,k} |^2 + F_\psi''( \psi_{\alpha,k} ) | \partial_{t} \psi_{\alpha,k} |^2 \big) \, \mathrm{d}x
    \\
    & \quad = I_1 + I_2 - I_3,
\end{align*}
where
\begin{align*}
    & \mathcal{E}_{\alpha,k} = \frac{1}{2} \Big(  \big\|  \sqrt{m_\phi( \phi_{\alpha,k} )} \nabla \mu_{\phi,\alpha,k}  \big\|^2 + \big\|  \sqrt{m_\psi( \psi_{\alpha,k} )} \nabla \mu_{\psi,\alpha,k}  \big\|^2 + \alpha \| \partial_{t} \phi_{\alpha,k} \|^2 + \alpha \| \partial_{t} \psi_{\alpha,k} \|^2  \Big)
    \\
    & \quad \qquad + \sigma_1 \AbsNormal{\Omega} ( \overline{\phi_{\alpha,k}} -c ) \overline{\mu_{\phi,\alpha,k}}
    + \alpha \Big(  \int_{\Omega} \bm{u} \cdot \nabla \phi_{\alpha,k} \partial_{t} \phi_{\alpha,k} \, \mathrm{d}x
    + \int_{\Omega} \bm{u} \cdot \nabla \psi_{\alpha,k} \partial_{t} \psi_{\alpha,k} \, \mathrm{d}x  \Big),
    \\
    & I_1 = \frac{1}{2} \int_{\Omega} \big( m_\phi'( \phi_{\alpha,k} ) \partial_{t} \phi_{\alpha,k} | \nabla \mu_{\phi,\alpha,k} |^2 + m_\psi'( \psi_{\alpha,k} ) \partial_{t} \psi_{\alpha,k} | \nabla \mu_{\psi,\alpha,k} |^2 \big)  \, \mathrm{d}x,
    \\
    & I_2 = \sigma_1 \AbsNormal{\Omega} \partial_{t} \overline{\phi_{\alpha,k}} \overline{\mu_{\phi,\alpha,k}} - \int_{\Omega}  \big[ \partial_\phi ^2 G ( \phi_{\alpha,k},\psi_{\alpha,k} ) (\partial_{t} \phi_{\alpha,k})^2 + \partial_\psi ^2 G ( \phi_{\alpha,k},\psi_{\alpha,k} ) (\partial_{t} \psi_{\alpha,k})^2 \big] \, \mathrm{d}x
    \\
    & \qquad - 2 \int_{\Omega}  \partial_{\phi,\psi} ^2 G ( \phi_{\alpha,k},\psi_{\alpha,k} ) \partial_{t} \phi_{\alpha,k} \partial_{t} \psi_{\alpha,k}  \, \mathrm{d}x,
    \\
    & I_3 = \alpha \int_{\Omega}  \partial_{t} \bm{u} \cdot \nabla \partial_{t} \phi_{\alpha,k} \phi_{\alpha,k}  \ \mathrm{d}x
     + \int_{\Omega}  \nabla \partial_{t} \phi_{\alpha,k} \cdot \nabla (\bm{u} \cdot \nabla \phi_{\alpha,k} ) \, \mathrm{d}x
    \\
    &\qquad - \sigma_2 \int_\Omega \phi_{\alpha,k} \bm{u} \cdot \nabla \mathcal{N} (\partial_t \phi_{\alpha,k} - \partial_t \overline{\phi_{\alpha,k}}) \, \mathrm{d}x
    \\
    & \qquad + \alpha \int_{\Omega}  \partial_{t} \bm{u} \cdot \nabla \partial_{t} \psi_{\alpha,k} \psi_{\alpha,k}  \, \mathrm{d}x
     + \int_{\Omega}   \nabla \partial_{t} \psi_{\alpha,k} \cdot \nabla (\bm{u} \cdot \nabla \psi_{\alpha,k} ) \, \mathrm{d}x
    \\
    & \qquad
    - \int_{\Omega}  \bm{u} \cdot \nabla \partial_{t} \phi_{\alpha,k} F_\phi'( \phi_{\alpha,k} )  \, \mathrm{d}x
    - \int_{\Omega}  \bm{u} \cdot \nabla \partial_{t} \psi_{\alpha,k} F_\psi'( \psi_{\alpha,k} )  \, \mathrm{d}x
    \\
    & \qquad - \int_{\Omega}  \bm{u} \cdot \big(  \partial_\phi G ( \phi_{\alpha,k},\psi_{\alpha,k} ) \nabla \partial_{t} \phi_{\alpha,k} + \partial_\psi G ( \phi_{\alpha,k},\psi_{\alpha,k} ) \nabla \partial_{t} \psi_{\alpha,k} \big )
    \, \mathrm{d}x.
\end{align*}
Arguing as in the proof of \cite[Theorem 2.4]{AGG2024}, we can estimate $I_3$ as follows:
\begin{align*}
     | I_3 | & \le C \big(  \| \nabla \bm{u} \|^2 + \alpha^3 \| \nabla \bm{u} \|^4 + \alpha^7 \| \nabla \bm{u} \|^8  \big)
     \\
    & \quad \times \big(  \|  \nabla \mu_{\phi,\alpha,k}  \|^2 + \|  \nabla \mu_{\psi,\alpha,k}  \|^2 + \alpha \|  \partial_{t} \phi_{\alpha,k}  \|^2 + \alpha \|  \partial_{t} \psi_{\alpha,k}  \|^2  \big)
    \\
    & \quad + C \alpha \| \partial_{t} \bm{u} \|^2 + C \| \nabla \bm{u} \|^2 + C (  \|  \nabla \mu_{\phi,\alpha,k}  \|^2 + \|  \nabla \mu_{\psi,\alpha,k}  \|^2  )
    \\
    & \quad + \epsilon \big(  \|  \nabla \partial_{t} \phi_{\alpha,k}  \|^2 + \beta \|  \nabla \partial_{t} \psi_{\alpha,k}  \|^2  \big),
\end{align*}
where $\epsilon \in (0,1)$ is a small constant to be determined later. Next, we treat the first two terms $I_1, I_2$. Using the estimates
\begin{align*}
    \big|  \sigma_1 \partial_{t} \overline{\phi_{\alpha,k}}(t) \overline{\mu_{\phi,\alpha,k}}(t)  \big|
    & \le 2 C \sigma_1 e^{-\sigma_1 t} (  1 + \| \nabla \mu_{\phi,\alpha,k}(t) \| + \alpha \| \nabla \partial_t \phi_{\alpha,k} \|  )
    \\
    & \le C_\epsilon \sigma_1 \big(  \| \nabla \mu_{\phi,\alpha,k}(t) \|^2 + e^{-\sigma_1 t}  \big) + \epsilon \| \nabla \partial_t \phi_{\alpha,k} \|^2,
    \\
    \| \partial_{t} \phi_{\alpha,k} \|^2
    & \le C (  |  \partial_{t} \overline{\phi}_{\alpha,k}  |^2 + \| \partial_{t} \phi_{\alpha,k} - \partial_{t} \overline{\phi}_{\alpha,k} \|_{V^*} \| \nabla \partial_{t} \phi_{\alpha,k} \|  )
    \\
    & \le \epsilon \| \nabla \partial_{t} \phi_{\alpha,k} \|^2 + C_{\epsilon} (  \sigma_1 e^{-\sigma_1 t} + \|  \nabla \mu_{\phi,\alpha,k}  \|^2 + \| \bm{u} \|^2  ),
    \\
    \| \partial_{t} \psi_{\alpha,k} \|^2
    & \le \epsilon \beta \| \nabla \partial_{t} \psi_{\alpha,k} \|^2 + C_{\epsilon} (  \|  \nabla \mu_{\psi,\alpha,k}  \|^2 + \| \bm{u} \|^2  ),
\end{align*}
we can deduce that
\[
I_2 \le \epsilon \big(  \|  \nabla \partial_{t} \phi_{\alpha,k}  \|^2 + \beta \|  \nabla \partial_{t} \psi_{\alpha,k}  \|^2  \big)
+ C_{\epsilon} \big(  \sigma_1 e^{-\sigma_1 t} + \| \bm{u} \|^2 + \|  \nabla \mu_{\phi,\alpha,k}  \|^2 + \|  \nabla \mu_{\psi,\alpha,k}  \|^2  \big).
\]
Concerning $I_1$, we observe that
\begin{align*}
    & \Big| \int_{\Omega} \big( m_\phi'( \phi_{\alpha,k} ) \partial_{t} \phi_{\alpha,k} | \nabla \mu_{\phi,\alpha,k} |^2 + m_\psi'( \psi_{\alpha,k} ) \partial_{t} \psi_{\alpha,k} | \nabla \mu_{\psi,\alpha,k} |^2 \big) \, \mathrm{d}x \Big|
    \\
    & \quad \le C \| \partial_{t} \phi_{\alpha,k} \| \| \nabla \mu_{\phi,\alpha,k} \|^2_{L^{4}(\Omega)} + C \| \partial_{t} \psi_{\alpha,k} \| \| \nabla \mu_{\psi,\alpha,k} \|^2_{L^{4}(\Omega)}
    \\
    & \quad \le C \big(  \|  \nabla \mathcal{N} \big(  \partial_{t} \phi_{\alpha,k} - \partial_{t} \overline{\phi_{\alpha,k}}  \big)  \|^\frac{1}{2} \| \nabla \partial_{t} \phi_{\alpha,k} \|^\frac{1}{2}
    + \big| \partial_{t} \overline{\phi_{\alpha,k}} \big|  \big) \| \nabla \mu_{\phi,\alpha,k} \|
    \| \mu_{\phi,\alpha,k} - \overline{\mu_{\phi,\alpha,k}} \|_W
    \\
    & \qquad + C\|  \nabla \mathcal{N} \partial_{t} \psi_{\alpha,k}  \|^\frac{1}{2} \| \nabla \mu_{\psi,\alpha,k} \| \| \mu_{\psi,\alpha,k} - \overline{\mu_{\psi,\alpha,k}} \|_W
    \\
    & \quad \le  C (\| \bm{u} \|^\frac{1}{2} + \| \nabla \mu_{\phi,\alpha,k} \|^\frac12) \| \nabla \partial_{t} \phi_{\alpha,k} \|^\frac{1}{2} \| \nabla \mu_{\phi,\alpha,k} \| \| \mu_{\phi,\alpha,k} - \overline{\mu_{\phi,\alpha,k}} \|_W
    \\
    & \qquad + C \sigma_1 e^{-\sigma_1 t} \| \nabla \mu_{\phi,\alpha,k} \| \| \mu_{\phi,\alpha,k} - \overline{\mu_{\phi,\alpha,k}} \|_W
    \\
    & \qquad + C (\| \bm{u} \| + \| \nabla \mu_{\psi,\alpha,k} \|  )^\frac{1}{2}
    \| \nabla \partial_{t} \psi_{\alpha,k} \|^\frac{1}{2} \| \nabla \mu_{\psi,\alpha,k} \|
    \| \mu_{\psi,\alpha,k} - \overline{\mu_{\psi,\alpha,k}} \|_W .
\end{align*}
Using equation \eqref{eq:cvcho2}, the monotonicity of $F_\phi'$, the uniform bound for $\|\phi_{\alpha,k}\|_V$, $\|\psi_{\alpha,k}\|_V$ and  the fact
$$
|\overline{\mu_{\phi, \alpha,k}}|\leq C \big(1+ \|\nabla \mu_{\phi, \alpha,k}\| + \alpha \|\partial_t\phi_{\alpha,k} - \partial_t \overline{\phi_{\alpha,k}}\|\big),
$$
we can apply \cite[Lemmas 7.3, 7.4]{GGW2018}) and the Poincar\'{e}--Wirtinger inequality to obtain
\begin{align*}
\| \Delta \phi_{\alpha,k} \|^2
& \le C \big(  1+\| \nabla \mu_{\phi,\alpha,k} \| + \alpha \| \nabla \partial_t \phi_{\alpha,k} \|  \big),
\\
 \|\phi_{\alpha,k}\|_{W^{2,4}(\Omega)}
 & \leq C(1+\| \mu_{\phi,\alpha,k} \|_V + \alpha \| \partial_t \phi_{\alpha,k} \|_V) \\
 & \leq C(1+\| \nabla \mu_{\phi,\alpha,k} \| + \alpha \| \nabla \partial_t \phi_{\alpha,k} \|).
\end{align*}
Then it follows from \eqref{eq:cvcho2} and Lemma \ref{elliptic_estimate_2d} that
\begin{align*}
    \| \mu_{\phi,\alpha,k} - \overline{\mu_{\phi,\alpha,k}} \|_W
    & = \|  \mathcal{N}_{\phi, \phi_{\alpha,k}} \big(  \partial_{t} \phi_{\alpha,k} - \partial_{t} \overline{\phi_{\alpha,k}} + \bm{u} \cdot \nabla \phi_{\alpha,k}  \big)  \|_W
    \\
    & \le C \| \nabla \phi_{\alpha,k} \| \| \phi_{\alpha,k} \|_W \|  \nabla \mathcal{N}_{\phi, \phi_{\alpha,k}} \big(  \partial_{t} \phi_{\alpha,k} - \partial_{t} \overline{\phi_{\alpha,k}} + \bm{u} \cdot \nabla \phi_{\alpha,k}  \big)  \|
    \\
    & \quad + C \|  \partial_{t} \phi_{\alpha,k} - \partial_{t} \overline{\phi_{\alpha,k}} + \bm{u} \cdot \nabla \phi_{\alpha,k}  \|
    \\
    & \le C   \|  \phi_{\alpha,k}  \|_W \| \nabla \mu_{\phi,\alpha,k} \| + C \| \nabla \mu_{\phi,\alpha,k} \|^\frac{1}{2} \big(  \| \nabla \partial_{t} \phi_{\alpha,k} \| + \|  \nabla ( \bm{u} \cdot \nabla \phi_{\alpha,k} )  \|  \big)^\frac{1}{2}
    \\
    &\le C \| \nabla \mu_{\phi,\alpha,k} \| \big( 1 + \| \nabla \mu_{\phi,\alpha,k} \| + \alpha \| \nabla \partial_{t} \phi_{\alpha,k} \|  \big)^\frac{1}{2}
    \\
    & \quad + C \| \nabla \mu_{\phi,\alpha,k} \|^\frac{1}{2} \big(  \| \nabla \partial_{t} \phi_{\alpha,k} \|^\frac{1}{2} + \| \bm{u} \|^\frac{1}{2}_V  \|\phi_{\alpha,k} \|_{W^{2,4}(\Omega)}^\frac{1}{2} \big)
    \\
    & \le C \| \nabla \mu_{\phi,\alpha,k} \|^{\frac{3}{2}} + C \sqrt{\alpha} \big(  \| \nabla \mu_{\phi,\alpha,k} \| + \| \nabla \mu_{\phi,\alpha,k} \|^\frac{1}{2} \| \bm{u} \|^\frac{1}{2}_V  \big) \| \nabla \partial_{t}  \phi_{\alpha,k}  \|^\frac{1}{2}
    \\
    & \quad + C \| \nabla \mu_{\phi,\alpha,k} \|^\frac{1}{2} \| \nabla \partial_{t}  \phi_{\alpha,k}  \|^\frac{1}{2} + C \| \bm{u} \|^\frac{1}{2}_V\| \nabla \mu_{\phi,\alpha,k} \|^\frac{1}{2}
    \\
    &\quad + C(1 + \| \bm{u} \|^\frac{1}{2}_V )\| \nabla \mu_{\phi,\alpha,k} \|.
\end{align*}
In a similar manner, we obtain
\begin{align*}
    \| \mu_{\psi,\alpha,k} - \overline{\mu_{\psi,\alpha,k}} \|_W
    & \le C \| \nabla \mu_{\psi,\alpha,k} \|^{\frac{3}{2}} + C \sqrt{\alpha} \big(  \| \nabla \mu_{\psi,\alpha,k} \| + \| \nabla \mu_{\psi,\alpha,k} \|^\frac{1}{2} \| \bm{u} \|^\frac{1}{2}_V  \big) \| \nabla \partial_{t}  \psi_{\alpha,k}  \|^\frac{1}{2}
    \\
    & \quad + C \| \nabla \mu_{\psi,\alpha,k} \|^\frac{1}{2} \| \nabla \partial_{t}  \psi_{\alpha,k}  \|^\frac{1}{2} + C \| \bm{u} \|^\frac{1}{2}_V \| \nabla \mu_{\psi,\alpha,k} \|^\frac{1}{2} \\
    &\quad +C(1 + \| \bm{u} \|^\frac{1}{2}_V ) \| \nabla \mu_{\psi,\alpha,k} \|.
\end{align*}
These estimates imply that
\begin{align*}
    & \Big| \int_{\Omega} \big( m_\phi'( \phi_{\alpha,k} ) \partial_{t} \phi_{\alpha,k} | \nabla \mu_{\phi,\alpha,k} |^2 + m_\psi'( \psi_{\alpha,k} ) \partial_{t} \psi_{\alpha,k} | \nabla \mu_{\psi,\alpha,k} |^2 \big)
    \, \mathrm{d}x \Big|
    \\
    & \quad \le C (1 + \| \bm{u} \|^\frac{1}{2}_V) \| \nabla \mu_{\phi,\alpha,k} \|^{\frac{5}{2}} \| \nabla \partial_{t} \phi_{\alpha,k} \|^\frac{1}{2} + C \| \bm{u} \|^\frac{1}{2}_V  \| \nabla \mu_{\phi,\alpha,k} \|^{\frac{3}{2}} \| \nabla \partial_{t} \phi_{\alpha,k} \|
    \\
    & \qquad + C \| \bm{u} \|_V \| \nabla \mu_{\phi,\alpha,k} \|^{\frac{3}{2}} \| \nabla \partial_{t} \phi_{\alpha,k} \|^\frac{1}{2}
    + C (1 + \| \bm{u} \|_V ) \| \nabla \mu_{\phi,\alpha,k} \|^2 \| \nabla \partial_{t} \phi_{\alpha,k} \|^\frac{1}{2}
    \\
    & \qquad + C\| \nabla \mu_{\phi,\alpha,k} \|^3 \| \nabla \partial_{t} \phi_{\alpha,k} \|^\frac{1}{2} + C\| \nabla \mu_{\phi,\alpha,k} \|^2 \| \nabla \partial_{t} \phi_{\alpha,k} \|
    \\
    & \qquad + C \| \bm{u} \|_V^\frac{1}{2} \| \nabla \mu_{\phi,\alpha,k} \|^2 \| \nabla \partial_{t} \phi_{\alpha,k} \|^\frac{1}{2} + C \sqrt{\alpha} \big(  \| \bm{u} \|^{\frac{5}{2}}_V + \|  \nabla \mu_{\phi,\alpha,k}  \|^{\frac{5}{2}}  \big) \|  \nabla \partial_{t} \phi_{\alpha,k}  \|
    \\
    & \qquad + C \sigma_1 e^{-\sigma_1 t} + C\| \bm{u} \|_V^2  + C (\| \bm{u} \|_V^2 + \| \nabla \mu_{\phi,\alpha,k} \|^2 )\| \nabla \mu_{\phi,\alpha,k} \|^2
    \\
    & \qquad +C (1 + \| \bm{u} \|^\frac{1}{2}_V ) \| \nabla \mu_{\psi,\alpha,k} \|^{\frac{5}{2}} \| \nabla \partial_{t} \psi_{\alpha,k} \|^\frac{1}{2} + C \| \bm{u} \|^\frac{1}{2}_V \| \nabla \mu_{\psi,\alpha,k} \|^{\frac{3}{2}} \| \nabla \partial_{t} \psi_{\alpha,k} \|
    \\
    & \qquad + C \| \bm{u} \|_V \| \nabla \mu_{\psi,\alpha,k} \|^{\frac{3}{2}} \| \nabla \partial_{t} \psi_{\alpha,k} \|^\frac{1}{2}
    + C(1 + \| \bm{u} \|_V ) \| \nabla \mu_{\psi,\alpha,k} \|^2 \| \nabla \partial_{t} \psi_{\alpha,k} \|^\frac{1}{2}
    \\
    & \qquad + C\| \nabla \mu_{\psi,\alpha,k} \|^3 \| \nabla \partial_{t} \psi_{\alpha,k} \|^\frac{1}{2} + C\| \nabla \mu_{\psi,\alpha,k} \|^2 \| \nabla \partial_{t} \psi_{\alpha,k} \|
    \\
    & \qquad + C \| \bm{u} \|_V^\frac{1}{2} \| \nabla \mu_{\psi,\alpha,k} \|^2 \| \nabla \partial_{t} \psi_{\alpha,k} \|^\frac{1}{2} + C \sqrt{\alpha} \big(  \| \bm{u} \|^{\frac{5}{2}}_V + \|  \nabla \mu_{\psi,\alpha,k}  \|^{\frac{5}{2}}  \big) \|  \nabla \partial_{t} \psi_{\alpha,k}  \|
    \\
    & \quad \le C_{\epsilon} (  1 + \| \bm{u} \|^2_V + \| \nabla \mu_{\phi,\alpha,k} \|^2 + \| \nabla \mu_{\psi,\alpha,k} \|^2  ) (  \| \nabla \mu_{\phi,\alpha,k} \|^2 + \| \nabla \mu_{\psi,\alpha,k} \|^2  )
    \\
    & \qquad + (  \epsilon + \alpha (  \| \bm{u} \|^3_V + \| \nabla \mu_{\phi,\alpha,k} \|^3 + \| \nabla \mu_{\psi,\alpha,k} \|^3  )  ) (  \|  \nabla \partial_{t} \phi_{\alpha,k}  \|^2 + \beta \|  \nabla \partial_{t} \psi_{\alpha,k}  \|^2  )
    \\
    & \qquad + C (  \| \bm{u} \|^2_V + \sigma_1 e^{-\sigma_1 t}  ).
\end{align*}
Since
\begin{align*}
    & \alpha \Big| \int_{\Omega} \bm{u} \cdot \nabla \phi_{\alpha,k} \partial_{t} \phi_{\alpha,k} \ \mathrm{d}x + \int_{\Omega} \bm{u} \cdot \nabla \psi_{\alpha,k} \partial_{t} \psi_{\alpha,k} \ \mathrm{d}x \Big|
    \\
    & \quad \le C\alpha  \| \bm{u} \|_{L^\infty(\Omega)} (  \| \nabla \phi_{\alpha,k} \| \| \partial_{t} \phi_{\alpha,k} \| + \| \nabla \psi_{\alpha,k} \| \| \partial_{t} \psi_{\alpha,k} \|  )
    \\
    & \quad \le \frac{\alpha}{8} ( \| \partial_{t} \phi_{\alpha,k} \|^2 + \| \partial_{t} \psi_{\alpha,k} \|^2 ) + C \alpha \| \bm{u} \|_{L^\infty(\Omega)}^2,
    \\
    & \big| \sigma_1 ( \overline{\phi}_{\alpha,k}-c ) \overline{\mu_{\phi,\alpha,k} }\big|  \le C \sigma_1 e^{-\sigma_1 t} (1 + \| \nabla \mu_{\phi,\alpha,k} \| + \alpha \| \partial_t \phi_{\alpha,k}- \partial_t \overline{\phi_{\alpha,k}} \|)
    \\
    & \quad \le \frac{1}{4} \int_{\Omega}  m_\phi ( \phi_{\alpha,k} ) | \nabla \mu_{\phi,\alpha,k} |^2  \, \mathrm{d}x + \frac{\alpha}{8} \| \partial_t \phi_{\alpha,k} \|^2
    + C \sigma_1 e^{-\sigma_1 t},
\end{align*}
we have
\begin{align*}
    & \| \nabla \mu_{\phi,\alpha,k} \|^2 + \| \nabla \mu_{\psi,\alpha,k} \|^2 + \alpha \big(  \| \partial_{t} \phi_{\alpha,k} \|^2 + \| \partial_{t} \psi_{\alpha,k} \|^2  \big)
    \\
    & \quad \le C_0 \sigma_1 e^{-\sigma_1 t} + \alpha C_1 \| \bm{u} \|_{L^\infty(\Omega)}^2 + C_2 \mathcal{E}_{\alpha,k},
    \\
    & \mathcal{E}_{\alpha,k} \le C_3\big(  \| \nabla \mu_{\phi,\alpha,k} \|^2 + \| \nabla \mu_{\psi,\alpha,k} \|^2 + \alpha  \| \partial_{t} \phi_{\alpha,k} \|^2 + \alpha\| \partial_{t} \psi_{\alpha,k} \|^2  \big)
    \\
    &\qquad\ \  + C_0 \sigma_1 e^{-\sigma_1 t} + \alpha C_1 \| \bm{u} \|^2_{L^\infty(\Omega)},
\end{align*}
where $C_0,C_1,C_2,C_3 >0$ may depend on $\Omega$, $\widetilde{E}_{\mathrm{free}}( \phi_0,\psi_0 )$, $\overline{\phi_0}$, $\overline{\psi_0}$, $m_\phi$, $m_\psi$ and the parameters of system \eqref{eq:cvcho1}--\eqref{eq:cvcho4} but are independent of $\alpha,k$.
Combining the above estimates and taking $\epsilon = 1/6$, we obtain
\begin{align}
    & \frac{\mathrm{d}}{\mathrm{d t}}  \mathcal{E}_{\alpha,k}  + \Big[\frac{1}{2} - \alpha \big(  \| \bm{u} \|^3_V + \| \nabla \mu_{\phi,\alpha,k} \|^3 + \| \nabla \mu_{\psi,\alpha,k} \|^3  \big)  \Big]
    \big(  \| \nabla \partial_{t} \phi_{\alpha,k} \|^2 + \beta \| \nabla \partial_{t} \psi_{\alpha,k} \|^2  \big)
    \nonumber \\
    & \quad \le \mathcal{F}_1 \mathcal{E}_{\alpha,k} + \mathcal{F}_2, \label{approximate_high_order_estimate}
\end{align}
where
\begin{align*}
    & \mathcal{F}_1 = C \big(  \| \bm{u} \|^2_V + \alpha^3 \| \bm{u} \|^4_V + \alpha^7 \| \bm{u} \|^8_V + \| \nabla \mu_{\phi,\alpha,k} \|^2 + \| \nabla \mu_{\psi,\alpha,k} \|^2  \big), \\
    & \mathcal{F}_2 = C \big(  \alpha \| \partial_{t} \bm{u} \|^2 + \| \bm{u} \|^2_V + \| \nabla \mu_{\phi,\alpha,k} \|^2 + \| \nabla \mu_{\psi,\alpha,k} \|^2 + \sigma_1 e^{-\sigma_1 t}  \big)
    \\
    & \qquad + \big(  \alpha C_1 \| \bm{u} \|^2_{L^\infty(\Omega)} + C_0 \sigma_1 e^{-\sigma_1 t}  \big) \mathcal{F}_1.
\end{align*}
Since $\bm{u} \in C^{\infty}_0 (0,\infty; \bm{H}^{2}_{0,\sigma}(\Omega))$, it holds $\mathcal{F}_1, \ \mathcal{F}_2 \in L^1(0,\infty)$.
Arguing as in the proof of \cite[Theorem 2.4]{AGG2024}, we find that
\begin{align*}
    & \alpha \| \partial_{t} \phi_{\alpha,k} (0) \|^2 + \|  \nabla \mu_{\phi,\alpha,k}(0)  \|^2 + \alpha \| \partial_{t} \psi_{\alpha,k} (0) \|^2 + \|  \nabla \mu_{\psi,\alpha,k}(0)  \|^2 \\
    & \quad \le C_4 ( 1+ \| \nabla (-\Delta \phi_0+F_\phi'(\phi_0)) \|^2 + \| \nabla  (-\beta \Delta \psi_0+F_\psi'(\psi_0))\|^2  ),
\end{align*}
where $C_4$ is independent of $\alpha, k$.
By the basic energy estimate, we have
\[
\|  \nabla \mu_{\phi,\alpha,k}  \|_{L^{2} ( 0,\infty;H )} + \|  \nabla \mu_{\psi,\alpha,k}  \|_{L^{2} ( 0,\infty;H )} \le C \Big(  1 + \sqrt{\widetilde{E}_{\mathrm{free}} ( \phi_0,\psi_0 )} + \| \bm{u} \|_{L^{2} ( 0,\infty;\bm{L}^{2}_\sigma(\Omega) )}  \Big),
\]
where $C>0$ is independent of $\alpha, k$.
Consequently, it holds
\begin{align*}
    F_1(\alpha)& =\int_{0}^{\infty} \mathcal{F}_1 (\tau)\,\mathrm{d}\tau
    \\
    & \le C \big(  1 + \| \bm{u} \|_{L^{2} ( 0,\infty;V )}^2 + \widetilde{E}_{\mathrm{free}} ( \phi_0,\psi_0 ) + \alpha^3 \| \bm{u} \|_{L^{4} ( 0,\infty;V )}^4 + \alpha^7 \| \bm{u} \|_{L^{8} ( 0,\infty;V )}^8  \big),
    \\
    F_2 (\alpha) &= \int_{0}^{\infty} \mathcal{F}_2(\tau)\, \mathrm{d}\tau
    \\
    & \le C \big(  1 + \| \bm{u} \|_{L^{2} ( 0,\infty;V )}^2 + \widetilde{E}_{\mathrm{free}} ( \phi_0,\psi_0 ) + \alpha \| \partial_{t} \bm{u} \|^2_{L^{2} ( 0,\infty;H )}  \big) \nonumber
    \\
    & \quad + C \big( 1 + \alpha \| \bm{u} \|^2_{L^\infty( Q )}   \big) \big(  1 + \| \bm{u} \|_{L^{2} ( 0,\infty;V )}^2 + \widetilde{E}_{\mathrm{free}} ( \phi_0,\psi_0 ) + \alpha^3 \| \bm{u} \|_{L^{4} ( 0,\infty;V )}^4 + \alpha^7 \| \bm{u} \|_{L^{8} ( 0,\infty;V )}^8  \big),
\end{align*}
where $C>0$ is independent of $\alpha$.
Therefore, there exists $\alpha_0 \in (0,1)$ such that for all $\alpha \in (0,\alpha_0]$, it holds
\begin{align*}
    & \alpha \big[  \| \bm{u} \|^3_{L^{\infty} ( 0,\infty;V )} + (  C_0 \sigma_1 + \alpha C_1 \| \bm{u} \|^2_{L^\infty(Q)} + C_2 ( \mathcal{E}_{\alpha,k}(0) + F_2(\alpha) ) e^{F_1(\alpha)}  )^{\frac{3}{2}}  \big]
    \\
    & \quad \le C \alpha \big( 1 + \| \bm{u} \|^3_{L^{\infty} ( 0,\infty;V )} + \| \bm{u} \|^3_{L^{\infty}(Q)}\big)
    \\
    & \qquad + C \alpha \big[ ( 1 + \| \nabla (-\Delta \phi_0+F_\phi'(\phi_0)) \|^2 + \| \nabla (-\beta\Delta \psi_0+F_\psi'(\psi_0)) \|^2 + \| \bm{u} \|^2_{L^{\infty}(Q)} + F_2(\alpha) ) e^{F_1(\alpha)}  \big]^{\frac{3}{2}}
    \\
    & \quad \le \frac{1}{4}.
\end{align*}

From now on, we consider $\alpha \in (0,\alpha_0]$.
Observing that
\[
\| \nabla \mu_{\phi,\alpha,k}(0) \|^3 + \| \nabla \mu_{\psi,\alpha,k}(0) \|^3 \le ( C_0 \sigma_1 + C_2 \mathcal{E}_{\alpha,k}(0) )^{\frac{3}{2}},
\]
we find
\begin{align*}
    \alpha \big(  \| \bm{u}(0) \|^3_V + \| \nabla \mu_{\phi,\alpha,k}(0) \|^3 + \| \nabla \mu_{\psi,\alpha,k}(0) \|^3  \big)  &
    = \alpha (  \| \nabla \mu_{\phi,\alpha,k}(0) \|^3 + \| \nabla \mu_{\psi,\alpha,k}(0) \|^3  )  \le \frac{1}{4}
\end{align*}
for $\alpha \in (0,\alpha_0]$.
Define
\[T_{\alpha}=
\sup \Big\{  T \ge 0: \alpha \big(  \| \bm{u}(t) \|^3_V + \| \nabla \mu_{\phi,\alpha,k}(t) \|^3 + \| \nabla \mu_{\psi,\alpha,k}(t) \|^3  \big) \le \frac{1}{3}, \ \ \forall\, t \in [0,T]  \Big\}.
\]
Then the continuity of $\bm{u}$, $\mu_{\phi,\alpha,k}$ and $\mu_{\psi,\alpha,k}$ gives $T_{\alpha} > 0$.
We claim that $T_{\alpha} = \infty$. If this is the case, then \eqref{approximate_high_order_estimate} implies that
\begin{align}
\frac{\mathrm{d}}{\mathrm{d t}}  \mathcal{E}_{\alpha,k} (t) + \frac{1}{6} \big(  \| \nabla \partial_{t} \phi_{\alpha,k}(t) \|^2 + \beta \| \nabla \partial_{t} \psi_{\alpha,k}(t) \|^2  \big) \le \mathcal{F}_1(t) \mathcal{E}_{\alpha,k}(t) + \mathcal{F}_2(t),
\quad \forall\, t>0.
\label{high-E-2D}
\end{align}

Assume that by contradiction $T_{\alpha} \in (0,\infty)$, then \eqref{high-E-2D} holds
for all $t \in (0,T_{\alpha})$.
An application of Gronwall's lemma gives
\begin{align*}
    \mathcal{E}_{\alpha,k}(t) & \le \Big(  \mathcal{E}_{\alpha,k}(0) + \int_{0}^{t} \mathcal{F}_2(\tau) \mathrm{d}\tau  \Big) e^{\int_{0}^{t} \mathcal{F}_1(\tau)\, \mathrm{d}\tau}  \le ( \mathcal{E}_{\alpha,k}(0) + F_2(\alpha) ) e^{F_1(\alpha)},
    \quad \forall\, t \in [0,T_{\alpha}].
\end{align*}
This implies that
\begin{align*}
    & \alpha \big( \| \bm{u}(t) \|^3_V + \| \nabla \mu_{\phi,\alpha,k}(t) \|^3 + \| \nabla \mu_{\psi,\alpha,k}(t) \|^3 \big)
    \\
    & \quad \le \alpha \big[  \| \bm{u} \|^3_{L^{\infty} ( 0,\infty;V )} + (  C_0 \sigma_1 + \alpha C_1 \| \bm{u} \|^2_{L^{\infty}(Q)} + C_2 ( \mathcal{E}_{\alpha,k}(0) + F_2(\alpha) ) e^{F_1(\alpha)}  )^{\frac{3}{2}}  \big]
    \\
    & \quad \le \frac{1}{4},
    \quad \forall\, t \in [0,T_{\alpha}].
\end{align*}
By continuity again, there exists some $\epsilon >0$ such that
\[
\alpha \big(  \| \bm{u}(t) \|^3_V + \| \nabla \mu_{\phi,\alpha,k}(t) \|^3 + \| \nabla \mu_{\psi,\alpha,k}(t) \|^3  \big)
\le \frac{1}{3},
\quad \forall\, t \in [0,T_{\alpha}+\epsilon],
\]
which contradicts the definition of $T_{\alpha}$.
Hence, $T_{\alpha}=\infty$ and we can conclude
\[
\mathcal{E}_{\alpha,k}(t) \le ( \mathcal{E}_{\alpha,k}(0) + F_2(\alpha) ) e^{F_1(\alpha)},
\quad \forall\, t \in [0,\infty).
\]
Consequently, it follows that
\begin{align}
    & \| \nabla \mu_{\phi,\alpha,k}(t) \|^2 + \| \nabla \mu_{\psi,\alpha,k}(t) \|^2 + \alpha (  \| \partial_{t} \phi_{\alpha,k}(t) \|^2 + \| \partial_{t} \psi_{\alpha,k}(t) \|^2  ) \nonumber
    \\
    & \quad \le C_0 \sigma_1 + \alpha C_1 \| \bm{u} \|^2_{L^{\infty}(Q)} + C_2 ( \mathcal{E}_{\alpha,k}(0) + F_2(\alpha) ) e^{F_1(\alpha)} \nonumber
    \\
    & \quad \le C_0 \sigma_1 + \alpha_0 C_1 \| \bm{u} \|^2_{L^{\infty}(Q)}
    +C \big(  \| \nabla (-\Delta \phi_0 + F'_\phi(\phi_0)) \|^2 + \| \nabla (-\beta \Delta \phi_0 + F'_\psi(\psi_0)) \|^2\big)e^{F_1(\alpha_0)}
    \nonumber
    \\
    & \qquad + C (  C_0 \sigma_1 + F_2(\alpha_0)  ) e^{F_1(\alpha_0)} \nonumber
    \\
    & \quad =:  K_0,
    \label{estimate_K0}
\end{align}
as well as
\begin{align}
    & \| \nabla \partial_{t} \phi_{\alpha,k} \|^2_{L^{2} ( 0,\infty;H )} + \beta \| \nabla \partial_{t} \psi_{\alpha,k} \|^2_{L^{2} ( 0,\infty;H )} \nonumber
    \\
    & \quad \le 6 \big[  F_2(\alpha) + \mathcal{E}_{\alpha,k}(0)  + \frac{\alpha C_1}{C_2} \| \bm{u} \|^2_{L^\infty(\Omega)} + \frac{C_0 \sigma_1}{C_2} + F_1(\alpha) ( \mathcal{E}_{\alpha,k}(0) + F_2(\alpha) ) e^{F_1(\alpha)}  \big] \nonumber
    \\
    & \quad \le C \big(  F_2(\alpha_0) + \| \nabla (-\Delta \phi_0+F_\phi'(\phi_0)) \|^2 + \| \nabla (-\beta\Delta \psi_0+F_\psi'(\psi_0)) \|^2 + C_0 \sigma_1 \big) \nonumber
    \\
    & \qquad + C(1 + F_1(\alpha_0)) \big(  \| \nabla (-\Delta \phi_0+F_\phi'(\phi_0)) \|^2 + \| \nabla (-\Delta \phi_0+F_\phi'(\phi_0)) \|^2 \big) e^{F_1(\alpha_0)} \nonumber
    \\
    & \qquad + C(1 + F_1(\alpha_0)) (  C_0 \sigma_1 + \alpha_0 C_1 \| \bm{u} \|^2_{L^{\infty}(Q)} + F_2(\alpha_0)  ) e^{F_1(\alpha_0)} \nonumber
    \\
    & \quad =:  K_1.
    \label{estimate_K1}
\end{align}
Combined with \eqref{eq:cvcho2} and \eqref{eq:cvcho4}, the above two estimates imply
\[
\| \phi_{\alpha,k}(t) \|^2_W + \| \psi_{\alpha,k}(t) \|^2_W \le K_2,
\]
and
\begin{align*}
    & \| \phi_{\alpha,k} \|^2_{L^{2}_{\mathrm{uloc}} ( [0,\infty);W^{2,p}(\Omega) )} + \| \psi_{\alpha,k} \|^2_{L^{2}_{\mathrm{uloc}} ( [0,\infty);W^{2,p}(\Omega) )} \\
    & \qquad + \| F_\phi' ( \phi_{\alpha,k} ) \|^2_{L^{2}_{\mathrm{uloc}} ( [0,\infty);L^{p}(\Omega) )} + \| F_\psi' ( \psi_{\alpha,k} ) \|^2_{L^{2}_{\mathrm{uloc}} ( [0,\infty);L^{p}(\Omega) )}  \le K_{3,p},
\end{align*}
for all $p \in [2,\infty)$. The positive constants $K_0$, $K_1$, $K_2$ and $K_{3,p}$ are independent of $\alpha \in (0,\alpha_0]$.

Let us choose a sequence of $\{ \alpha_n \}_{n=1}^\infty$ with $\alpha_n \to 0$, such that there exist some functions
\begin{align*}
    & \phi_k,\psi_k \in L^{\infty} \big( 0,\infty;W \big) \cap L^{2}_{\mathrm{uloc}} \big( [0,\infty);W^{2,p}(\Omega) \big) \cap H^{1}_{\mathrm{uloc}} \big( [0,\infty);V \big), \\
    & \mu_{\phi,k}, \mu_{\psi,k} \in L^{\infty} \big( 0,\infty;V \big),
\end{align*}
as the corresponding weak (weak-*) limits in the sense that
\begin{align*}
    & \phi_{\alpha_n,k} \rightharpoonup \phi_k, \ \psi_{\alpha_n,k} \rightharpoonup \psi_k \ \mathrm{in} \ L^{2} \big( 0,T;W^{2,p}(\Omega) \big) \cap H^{1} \big( 0,T;V \big),
    \\
    & \phi_{\alpha_n,k} \mathop{\rightharpoonup}\limits^* \phi_k, \ \psi_{\alpha_n,k} \mathop{\rightharpoonup}\limits^* \psi_k \ \mathrm{in} \ L^{\infty} \big( 0,\infty;W \big),
    \\
    & \mu_{\phi,\alpha_n,k} \mathop{\rightharpoonup}\limits^* \mu_{\phi,k}, \ \mu_{\psi,\alpha_n,k} \mathop{\rightharpoonup}\limits^* \mu_{\psi,k} \ \mathrm{in} \ L^{\infty} \big( 0,\infty;V \big),
\end{align*}
for all $T>0$.
Applying the Aubin--Lions--Simon lemma (cf. \cite{Si1987}), we also obtain
\[
\phi_{\alpha_n,k} \to \phi_k, \ \ \psi_{\alpha_n,k} \to \psi_k \quad  \mathrm{in} \ BC\big( [0,T];W^{1,q}(\Omega) \big) ,
\]
for all $T>0$ and $q \in [2,\infty)$.
By a standard argument, we can show that $( \phi_k,\psi_k,\mu_{\phi,k},\mu_{\psi,k} )$ satisfies the system \eqref{eq:cvcho1}--\eqref{eq:cvcho4} with $\alpha=0$ almost everywhere in $Q$, \eqref{eq:cvchob} almost everywhere on $S$ and the initial condition
\[
\phi_k|_{t=0} = \phi_{0,k}(x), \quad  \psi_k|_{t=0} = \psi_{0,k}(x)
\quad \text{in}\ \Omega.
\]
Passing to the limit in \eqref{estimate_K0} and \eqref{estimate_K1} gives
\begin{align}
    & \| \nabla \mu_{\phi,k}(t) \|^2 + \| \nabla \mu_{\psi,k}(t) \|^2
    \notag \\
    & \quad \le C_0 \sigma_1 + C \big(  \| \nabla (-\Delta \phi_0+ F_\phi'(\phi_0)) \|^2 + \| \nabla (-\beta \Delta \psi_0+ F_\psi'(\psi_0)) \|^2 + C_0 \sigma_1 + F_2(0)  \big) e^{F_1(0)},
    \label{es-kk-1}
\end{align}
for all $t \in [0,\infty)$, and
\begin{align}
    & \| \nabla \partial_{t} \phi_k \|^2_{L^{2} ( 0,\infty;H )} + \beta \| \nabla \partial_{t} \psi_k \|^2_{L^{2} ( 0,\infty;H )}
    \notag \\
    & \quad \le C \big(  F_2(0) + \| \nabla (-\Delta \phi_0+ F_\phi'(\phi_0)) \|^2 + \| \nabla (-\beta \Delta \psi_0+ F_\psi'(\psi_0)) \|^2 + C_0 \sigma_1  \big) \nonumber
    \notag \\
    & \qquad + C (1 + F_1(0)) (  \| \nabla (-\Delta \phi_0+ F_\phi'(\phi_0)) \|^2 + \| \nabla (-\beta \Delta \psi_0+ F_\psi'(\psi_0)) \|^2 + C_0 \sigma_1 + F_2(0)  ) e^{F_1(0)}.
     \label{es-kk-2}
\end{align}

Following the arguments in \cite{AGG2024}, we can pass to the limit as $k \to \infty$ and use the density argument to obtain the desired strong solution $( \phi,\psi,\mu_\phi,\mu_\psi )$ to problem \eqref{eq:cvcho1}--\eqref{eq:cvcho4} with any given $\bm{u} \in L^{2} ( 0,\infty;\bm{H}_{0,\sigma}^{1}(\Omega) )$ as the solenoidal drift velocity field, satisfying the boundary and initial conditions:
\begin{numcases}{}
    \partial_{\mathbf{n}} \phi = \partial_{\mathbf{n}} \mu_\phi = \partial_{\mathbf{n}} \psi = \partial_{\mathbf{n}} \mu_\psi = 0, \nonumber & a.e. on $S$,
    \\
    \phi|_{t=0} = \phi_0(x), \nonumber \ \ \psi|_{t=0} = \psi_0(x), & a.e. in $\Omega$.
\end{numcases}
In addition, $( \phi,\psi,\mu_\phi,\mu_\psi )$ satisfy the same estimates as in \eqref{es-kk-1} and \eqref{es-kk-2}.
With these estimates and Lemma \ref{weightedlaplace}, we can use equations \eqref{eq:cvcho1}--\eqref{eq:cvcho4} with $\alpha=0$ to obtain
\begin{align*}
    & \phi,\psi \in L^{\infty} ( 0,\infty;W^{2,p}(\Omega) ),
    \quad
    F_\phi'( \phi ),F_\psi' ( \psi ) \in L^{\infty} ( 0,\infty;L^{p}(\Omega) ),
    \quad
    \nabla \mu_\phi, \nabla \mu_\psi \in L^{2} ( 0,\infty;H^{2}(\Omega) ).
\end{align*}
Furthermore, when $\bm{u} \in L^{\infty} \big( 0,\infty;\bm{L}^{2}_\sigma(\Omega) \big)$, by comparison in \eqref{eq:cvcho1} and \eqref{eq:cvcho3}, we get
\[
\| \partial_{t} \phi(t) - \partial_{t} \overline{\phi}(t) \|_{V^*} \le \| \bm{u}(t) \| + m_\phi^* \| \nabla \mu_\phi(t) \|, \quad  \| \partial_{t} \psi(t) \|_{V^*} \le \| \bm{u}(t) \| + m_\psi^* \| \nabla \mu_\psi(t) \|,
\]
for almost all $t >0$, which implies $ \partial_{t} \phi,\partial_{t} \psi \in L^{\infty} ( 0,\infty;V^* )$.

Finally, if the assumption $(\mathbf{H2*})$ holds, the same iteration argument in \cite{GP2025,O2025} can be applied to the equations \eqref{eq:cvcho2} and \eqref{eq:cvcho4} (with $\alpha=0$) to obtain two constants $\delta_\phi \in (0,1)$ and $\delta_\psi \in (0,1/2)$ depending on $\| \mu_\phi \|_{L^{\infty} ( 0,\infty;V )}$ and $\| \mu_\psi \|_{L^{\infty} ( 0,\infty;V )}$, respectively, such that
\[
\phi \in [-1+\delta_\phi,1-\delta_\phi], \quad  \psi \in [\delta_\psi,1-\delta_\psi] \quad  \mathrm{in} \ Q.
\]
It follows from \eqref{eq:cvcho2} that
\begin{align*}
    & -\nabla \Delta \phi = \nabla \mu_\phi - F_\phi''( \phi ) \nabla \phi - \partial_\phi ^2G ( \phi,\psi ) \nabla \phi - \partial_{\phi,\psi} ^2G ( \phi,\psi ) \nabla \psi - \sigma_2 \nabla \mathcal{N} \big( \phi - \overline{\phi} \big) \in L^{\infty} \big( 0,\infty;H \big),
    \\
    & \partial_{t} \mu_\phi = - \Delta \partial_{t} \phi + F_\phi'' ( \phi ) \partial_{t} \phi + \partial_\phi ^2G ( \phi,\psi ) \partial_{t} \phi + \partial_{\phi,\psi} ^2G ( \phi,\psi ) \partial_{t} \psi + \sigma_2 \mathcal{N} \big(  \partial_{t} \phi - \partial_{t} \overline{\phi}  \big) \in L^{2} \big( 0,\infty;V^* \big).
\end{align*}
Therefore, for all $T>0$, the interpolation between $L^{2} ( 0,T;H^{3}(\Omega) )$ and $ H^1( 0,T;V^* )$ gives $\mu_\phi \in BC( [0,T];V )$.
Similarly, we have $\partial_{t} \mu_\psi \in L^{2} ( 0,\infty;V^* )$, $\mu_\psi \in BC( [0,\infty);H )$. Finally, the strict separation property of $\phi, \psi$ and the elliptic estimate yields $\phi, \psi \in L^{\infty} ( 0,\infty;H^{3}(\Omega) )$.

The proof of Proposition \ref{cvcho_2d_global_sse} is now complete.
\qed

\subsection{Proof of Theorem \ref{nschsse2d}}

We are now in a position to prove Theorem \ref{nschsse2d}.

To begin with, we recall that for the initial data given in Theorem \ref{nschsse2d},
it has been shown in \cite[Theorem 2.1]{OGW2026} that problem \eqref{eq:nsch1}--\eqref{eq:nschi} admits a global-in-time (finite energy) weak solution $( \bm{u}_w,\phi,\psi,\mu_\phi,\mu_\psi )$ such that
\begin{align*}
    & \bm{u}_w \in BC_w \big( [0,\infty);\bm{L}^{2}_\sigma(\Omega) \big)\cap L^{2} \big( 0,\infty;\bm{H}^{1}_{0,\sigma}(\Omega) \big),\\
    & \phi,\psi \in BC_w \big( [0,\infty);V \big) \cap L^{2}_{\mathrm{uloc}} \big( [0,\infty);W \big), \\
    & \phi, \psi \in L^{\infty} (Q)\ \ \text{with}\ \ \phi(x,t) \in (-1,1) \ \mathrm{and} \ \psi(x,t) \in (0,1) \ \mathrm{a.e.} \ \mathrm{in} \ Q,\\
    & \mu_\phi,\mu_\psi \in L^{2}_{\mathrm{uloc}} \big( [0,\infty);V \big),\quad F_\phi'(\phi),F_\psi'(\psi) \in L^{2}_{\mathrm{uloc}} \big( [0,\infty);H \big)
\end{align*}
satisfying
\begin{align}
        & - \int_Q \big( \rho \bm{u}_w \cdot \partial_{t} \bm{\theta}
        + \rho (\bm{u}_w \otimes \bm{u}_w) : \nabla \bm{\theta}\big)\, \mathrm{d}(x,t)
        + \int_Q  \nu(\phi) D \bm{u}_w : D \bm{\theta}  \, \mathrm{d}(x,t)  \nonumber
        \\
        & \quad   + \int_Q  \gamma   m_\phi ( \phi ) ( \bm{u}_w \otimes \nabla \mu_\phi) : \nabla \bm{\theta} \, \mathrm{d}(x,t)
        = \int_Q ( \mu_\phi \nabla \phi + \mu_\psi \nabla \psi)\cdot \bm{\theta} \, \mathrm{d}(x,t),
        \label{eq:testns_origin}
\end{align}
for all $\bm{\theta} \in C^{\infty}_0(Q)^2$ with $\mathrm{div}\, \bm{\theta} = 0$,
\begin{align*}
   & \int_Q \big[- \phi \partial_{t} \zeta
   + ( \bm{u} \cdot \nabla \phi) \zeta \big]
   \, \mathrm{d}(x,t)
   + \int_{0}^{\infty} | \Omega | \sigma_1 \big( \overline{\phi} -c \big) \overline{\zeta} \,  \mathrm{d}t
   = - \int_Q  m_\phi(\phi)\nabla \mu_\phi\cdot  \nabla \zeta \, \mathrm{d}(x,t),
   \\
    & \int_Q \big[- \psi \partial_{t} \zeta + ( \bm{u} \cdot \nabla \psi) \zeta \big] \, \mathrm{d}(x,t)
    = - \int_Q m_\psi(\psi)\nabla \mu_\psi \cdot \nabla \zeta \, \mathrm{d}(x,t),
\end{align*}
for all $\zeta \in C_0^{\infty}\big(0,\infty;C^1\big(\overline{\Omega}\big)\big)$,
and that \eqref{eq:nsch4}, \eqref{eq:nsch6} hold almost everywhere in $Q$.
In addition, the following energy inequality holds
\begin{align*}
    & E_{\mathrm{tot}} \big( \bm{u}_w(t),\phi(t),\psi(t) \big) + \int_{Q_{(s,t)}} \big( \nu (\phi) |D \bm{u}_w|^2+ m_\phi(\phi)| \nabla \mu_\phi |^2 +m_\psi(\psi) | \nabla \mu_\psi |^2 \big) \, \mathrm{d}(x,\tau)
    \\
    & \quad + \sigma_1 \int_{Q_{( s,t )}}  \big(  \overline{\phi} - c  \big) \Big(  \mu_\phi - \frac12 \gamma | \bm{u}_w |^2  \Big)  \, \mathrm{d}( x,\tau )
    \le E_{\mathrm{tot}} \big( \bm{u}_w(s),\phi(s),\psi(s) \big),
\end{align*}
for almost all $s \in [0,\infty)$ (including $s=0$) and all $t \in [s,\infty)$.

On the other hand, it readily follows from Proposition \ref{cvcho_2d_global_sse} that $( \phi,\psi,\mu_\phi,\mu_\psi )$ is actually the unique global strong solution to problem \eqref{eq:cvcho1}--\eqref{eq:cvchoi} (with $\alpha=0$) subject to the ``given'' solenoidal velocity field $\bm{u}_w$.

Thus, the remaining task is to improve the regularity of the weak solution $\bm{u}_w$. We achieve this goal by studying a linearized version of the momentum equation \eqref{eq:nsch1}, in which we treat both the external force term and the convective velocity field as given functions. More precisely, let us consider
\begin{align}
        & - \int_Q \big( \rho \bm{v} \cdot \partial_{t} \bm{\theta}
        + \rho (\bm{v}\otimes \bm{u}_w) : \nabla \bm{\theta}\big)\, \mathrm{d}(x,t)
        + \int_Q  \nu(\phi) D \bm{v} : D \bm{\theta}  \, \mathrm{d}(x,t)  \nonumber
        \\
        & \quad   + \int_Q  \gamma   m_\phi ( \phi ) ( \bm{v} \otimes \nabla \mu_\phi) : \nabla \bm{\theta} \, \mathrm{d}(x,t)
        = \int_Q ( \mu_\phi \nabla \phi + \mu_\psi \nabla \psi)\cdot \bm{\theta} \, \mathrm{d}(x,t)
        \label{eq:testns_linearized}
\end{align}
for all $\bm{\theta} \in C^{\infty}_0(Q)^d$ such that $\mathrm{div}\, \bm{\theta} = 0$ and $\bm{v}|_{t=0}=\bm{u}_0$. Under the regularity properties of $( \bm{u}_w,\phi,\psi,\mu_\phi,\mu_\psi )$, especially the enhanced regularity guaranteed by Proposition \ref{cvcho_2d_global_sse}, the above linear problem admits a global weak solution $\bm{v}$ such that
\begin{align*}
& \bm{v} \in BC_w ( [0,\infty);\bm{L}^{2}_\sigma(\Omega) )\cap L^{2} ( 0,\infty;\bm{H}^{1}_{0,\sigma}(\Omega)) \hookrightarrow L^4(0,\infty;\bm{L}^4(\Omega)),
\\
& \partial_{t}  \mathbb{P}_{\sigma} ( \rho \bm{v} )  \in L^{2} \big( 0,\infty;\bm{H}^{1}_{0,\sigma}(\Omega)^* \big).
\end{align*}
The proof of existence follows from a Faedo--Galerkin approximation scheme. Next, if $\bm{v}_1$ and $\bm{v}_2$ are two solutions to the above problem with corresponding regularity and initial data $\bm{u}_{0,1}$ and $\bm{u}_{0,2}$ respectively, then $\widetilde{\bm{v}} = \bm{v}_1  -\bm{v}_2$ satisfies
\begin{align*}
    & \int_Q \big[\partial_{t}  (  \rho \widetilde{\bm{v} } )  \bm{\theta}  -  \rho (\widetilde{\bm{v}} \otimes \bm{u}_w): \nabla \bm{\theta}\big]\,\mathrm{d}(x,t) \\
    &\quad
    + \int_Q \Big[ \nu(\phi) D \widetilde{\bm{v}} : D \bm{\theta} +\gamma m_\phi ( \phi ) ( \widetilde{\bm{v}} \otimes \nabla \mu_\phi ):\nabla \bm{\theta} \Big]
    \,\mathrm{d}(x,t) = 0,
\end{align*}
for all $\bm{\theta} \in C^{\infty}_0(Q)^d$ such that $\mathrm{div}\, \bm{\theta} = 0$. By density argument, the above identity holds for all $\bm{\theta} \in L^{2} \big( 0,\infty;\bm{H}^{1}_{0,\sigma}(\Omega) \big)$, in particular, taking $\bm{\theta}=\widetilde{\bm{v}}$, we obtain
\begin{align*}
& \frac{1}{2} \int_{\Omega} \rho( \phi(t) ) |  \widetilde{\bm{v}}(t)  |^2 \ \mathrm{d}x
+ \int_0^t\int_\Omega   \nu( \phi ) | D \widetilde{\bm{v}} |^2  \ \mathrm{d}(x,\tau)
\\
&\quad =
\frac{1}{2} \int_{\Omega} \rho( \phi(0) ) |  \widetilde{\bm{v}}(0)  |^2 \ \mathrm{d}x +  \frac{1}{4} (\rho_1-\rho_2) \sigma_1 \int_0^t  ( \overline{\phi}(\tau)-c ) \| \widetilde{\bm{v}}(\tau) \|^2\,\mathrm{d}\tau,\quad \forall\, t>0,
\end{align*}
where we have used equation \eqref{eq:cvcho1}.
An application of Gronwall's lemma yields the continuous dependence of $\bm{v}$ on the initial data, which implies that the weak solution $\bm{v}$ is unique. In view of \eqref{eq:testns_origin}, we immediately have $\bm{v}=\bm{u}_w$.

Hence, to prove the existence of the global-in-time strong solutions to problem \eqref{eq:nsch1}--\eqref{eq:nschi}, it suffices to show that
\[
\bm{v} \in BC\big( [0,\infty);\bm{H}^{1}_{0,\sigma}(\Omega) \big) \cap L^{2} \big( 0,\infty;\bm{H}^{2}(\Omega) \big) \cap H^{1} \big( 0,\infty;\bm{L}^{2}_\sigma(\Omega) \big).
\]
Below we derive only the \textit{a priori} estimates for simplicity. Using equation \eqref{eq:cvcho1} with $\bm{u}=\bm{u}_w$, we can (formally) rewrite \eqref{eq:testns_linearized} in the following strong form
\begin{align}
    &  \rho( \phi ) \partial_{t}\bm{v} +  \rho ( \phi ) ( \bm{u}_w \cdot \nabla ) \bm{v}
    - \gamma  m_\phi ( \phi ) (  \nabla \mu_\phi \cdot \nabla ) \bm{v}
    - \mathrm{div}( 2\nu(\phi) D \bm{v}) + \nabla P
    \nonumber  \\
    & \quad = \gamma \sigma_1 \big( \overline{\phi}-c \big) \bm{v}  +   \mu_\phi \nabla \phi +  \mu_\psi \nabla \psi.
    \label{eq:testns_linearized_non_conservative}
\end{align}
Testing \eqref{eq:testns_linearized_non_conservative} by $\bm{\theta}=\partial_{t} \bm{v}$ and $\bm{\theta}= \mathbf{A} \bm{v}$, respectively, we obtain
\begin{align*}
    & \frac{1}{2} \frac{\mathrm{d}}{\mathrm{d t}}  \Big(  \int_{\Omega} \nu ( \phi ) | D \bm{v} |^2  \, \mathrm{d}x + \gamma \sigma_1 ( \overline{\phi}-c ) \| \bm{v} \|^2  \Big)
    + \int_{\Omega}  \rho ( \phi ) | \partial_{t} \bm{v} |^2  \, \mathrm{d}x
    \\
    & \quad = \frac{1}{2} \int_{\Omega}  \nu'( \phi ) \partial_{t} \phi | D \bm{v} |^2  \, \mathrm{d}x - \big( \rho ( \phi ) ( \bm{u}_w \cdot \nabla ) \bm{v}, \partial_{t} \bm{v} \big) +  \gamma \big( m_\phi ( \phi ) ( \nabla \mu_\phi \cdot \nabla ) \bm{v}, \partial_{t} \bm{v} \big)
    \\
    & \qquad +  \frac12 \gamma \sigma_1^2 ( \overline{\phi}-c ) \| \bm{v} \|^2
    + \big( \mu_\phi \nabla \phi + \mu_\psi \nabla \psi, \partial_{t} \bm{v} \big)
    \\
    & \quad \le C \| \partial_{t} \phi \| \| D \bm{v} \| \| \mathbf{A} \bm{v} \| + C (  \| \bm{u}_w \|_{L^{4}(\Omega)} + \| \nabla \mu_\phi \|_{L^{4}(\Omega)}  ) \| \nabla \bm{v} \|^\frac{1}{2} \| \mathbf{A} \bm{v} \|^\frac{1}{2} \| \partial_{t} \bm{v} \|
    \\
    & \qquad + C \sigma_1 e^{-\sigma_1 t} \| \bm{v} \|^2 + \| \mu_\phi \nabla \phi + \mu_\psi \nabla \psi \| \| \partial_{t} \bm{v} \|,
\end{align*}
and
\begin{align*}
    \frac{1}{2} ( \nu ( \phi ) \mathbf{A} \bm{v}, \mathbf{A} \bm{v} )
    &  = - ( \rho( \phi ) \partial_{t} \bm{v}, \mathbf{A} \bm{v} ) - ( \rho( \phi ) ( \bm{u}_w \cdot \nabla ) \bm{v} , \mathbf{A} \bm{v} )
    \\
    & \quad + \gamma \big(  m_\phi(\phi)( \nabla \mu_\phi \cdot \nabla ) \bm{v} + \sigma_1 \big(  \overline{\phi} - c  \big) \bm{v} , \mathbf{A} \bm{v} \big)
    \\
    & \quad - ( \phi \nabla \mu_\phi + \psi \nabla \mu_\psi, \mathbf{A} \bm{v} ) + \frac{1}{2} (  \nu' ( \phi ) \pi \nabla \phi , \mathbf{A} \bm{v} )
     + \int_{\Omega}  \nu' ( \phi ) ( \nabla \phi \cdot D ) \bm{v} \cdot \mathbf{A} \bm{v}  \, \mathrm{d}x
    \\
    & \le \rho^* \| \mathbf{A} \bm{v} \| \| \partial_{t} \bm{v} \| + C \big(  \| \bm{u}_w \|_{L^{4}(\Omega)} + \| \nabla \mu_\phi \|_{L^{4}(\Omega)}  \big) \| \nabla \bm{v} \|^\frac{1}{2} \| \mathbf{A} \bm{v} \|^{\frac{3}{2}}
    \\
    & \quad + C \| \bm{v} \| \| \mathbf{A} \bm{v} \| + C (  \| \pi \| + \| D \bm{v} \|  ) \| \nabla \phi \|_{L^{\infty}(\Omega)} \| \mathbf{A} \bm{v} \|
    \\
    & \quad + \| \mu_\phi \nabla \phi + \mu_\psi \nabla \psi \| \| \mathbf{A} \bm{v} \|
    \\
    & \le \frac{1}{4} \nu_* \| \mathbf{A} \bm{v} \|^2 + 4 (\rho^*)^2 \| \partial_{t} \bm{v} \|^2 + C (  \| \bm{u}_w \|_{L^{4}(\Omega)}^4 + \| \nabla \mu_\phi \|_{L^{4}(\Omega)}^4  ) \| D \bm{v} \|^2 \\
    & \quad + C \| \bm{v} \|^2 + C \| \mu_\phi \nabla \phi + \mu_\psi \nabla \psi \|^2,
\end{align*}
where $\pi \in V_{(0)}$ satisfies $\nabla \pi = \mathbf{A} \bm{v} + \Delta \bm{v}$. Combining the above estimates, we can deduce that
\begin{align*}
    & \frac{1}{2} \frac{\mathrm{d}}{\mathrm{d t}}  \Big(  \int_{\Omega} \nu ( \phi ) | D \bm{v} |^2  \, \mathrm{d}x + \gamma \sigma_1 ( \overline{\phi}-c ) \| \bm{v} \|^2  \Big)  + \frac{1}{2} \rho_* \| \partial_{t} \bm{v} \|^2 + \frac{\nu_* \rho_*}{32 (\rho^*)^2} \| \mathbf{A} \bm{v} \|^2
    \\
    & \quad \le H_1 \Big(  \int_{\Omega} \nu ( \phi ) | D \bm{v} |^2  \, \mathrm{d}x + \gamma \sigma_1 ( \overline{\phi}-c ) \| \bm{v} \|^2  \Big) + H_2,
\end{align*}
with
\begin{align*}
    & H_1 = C \big(  \| \partial_{t} \phi \|^2 + \| \bm{u}_w \|_{L^{4}(\Omega)}^4 + \| \nabla \mu_\phi \|_{L^{4}(\Omega)}^4  \big),
    \\
    & H_2 = C \big(  \sigma_1 e^{-\sigma_1 t} \| \bm{v} \|^2 + \| \mu_\phi \nabla \phi + \mu_\psi \nabla \psi \|^2  \big) + C \| \bm{v} \|^2 H_1.
\end{align*}
Applying Gronwall's lemma, we obtain the desired higher-order estimates that yield the existence of a global strong solution $\bm{v}$ to \eqref{eq:testns_linearized}. By uniqueness, we find $\bm{v}=\bm{u}_w$.

In summary, the global weak solution $( \bm{u}_w,\phi,\psi,\mu_\phi,\mu_\phi )$ is a global strong solution to problem \eqref{eq:nsch1}--\eqref{eq:nschi} in the sense of Definition \ref{nschssd}. The additional regularity and the weak-strong uniqueness can be proved arguing essentially as in \cite{AGG2024,G2021}; therefore, the details are omitted here.

This completes the proof of Theorem \ref{nschsse2d}.
\qed
\appendix
\section{Useful Tools} \label{appendix}
\setcounter{equation}{0}

To handle nondegenerate concentration-dependent mobilities, we report some results for the two operators $\mathcal{N}_{\phi, h}$ and $\mathcal{N}_{\psi, h}$.

\begin{lemma} \label{weightedlaplace}
Assume that $\Omega\subset \mathbb{R}^3$ is a bounded domain of class $C^3$, $m_\phi$ satisfies the conditions in $(\mathbf{H1})$.
We consider the operator $\mathcal{N}_{\phi, h}$ defined in \eqref{ahi_definition}. If $h\in W^{1,6}(\Omega)$, the restriction of $\mathcal{N}_{\phi, h}$ to $L^p_{(0)}(\Omega)$ with $p \in [6/5,6]$ is continuous from $L^p_{(0)}(\Omega)$ to $W^{2,p}(\Omega)\cap V_{(0)}$. If $h \in H^{2}(\Omega)$, the restriction of $\mathcal{N}_{\phi, h}$ to $V_{(0)}$ is continuous from $V_{(0)}$ to $H^{3}_N(\Omega)\cap V_{(0)}$. Moreover, we have the following estimates:
    \begin{align*}
        & \|  \nabla \mathcal{N}_{\phi, h}( f )  \| \le \frac{C}{m_{\phi,*}} \| f \|_{V^*_{(0)}}, \quad \forall f \in V^*_{(0)}, \\
        & \text{if}\ p \in [6/5,3/2], \ \|  \Delta \mathcal{N}_{\phi, h}( f )  \|_{L^{p}(\Omega)} \le \frac{C}{m_{\phi,*}} \left(  1 + \frac{ \|  \nabla m_\phi ( h )  \|_{L^{6}(\Omega)} }{ m_{\phi,*} }  \right) \| f \|_{L^{p}(\Omega)}, \ \forall f \in L^p_{(0)}(\Omega),
        \\
        & \text{if}\ p \in (3/2,2 ],  \ \|  \Delta \mathcal{N}_{\phi, h}( f )  \|_{L^{p}(\Omega)} \le \frac{C}{m_{\phi,*}} \left(  1 + \frac{ \|  \nabla m_\phi ( h )  \|_{L^{6}(\Omega)}^2 }{ m_{\phi,*}^2 }  \right) \| f \|_{L^{p}(\Omega)}, \ \forall f \in L^p_{(0)}(\Omega),
        \\
        & \text{if}\ p \in (2,3], \ \|  \Delta \mathcal{N}_{\phi, h}( f )  \|_{L^{p}(\Omega)} \le \frac{C}{m_{\phi,*}} \left(  1 + \frac{ \|  \nabla m_\phi ( h )  \|_{L^{6}(\Omega)}^3 }{ m_{\phi,*}^3 }  \right) \| f \|_{L^{p}(\Omega)}, \ \forall f \in L^p_{(0)}(\Omega),
        \\
        & \text{if}\ p \in (3,6), \ \|  \Delta \mathcal{N}_{\phi, h}( f )  \|_{L^{p}(\Omega)} \le \frac{C}{m_{\phi,*}} \left(  1 + \frac{ \|  \nabla m_\phi ( h )  \|_{L^{6}(\Omega)}^4 }{ m_{\phi,*}^4 }  \right) \| f \|_{L^{p}(\Omega)}, \ \forall f \in L^p_{(0)}(\Omega),
        \\
        & \|  \Delta \mathcal{N}_{\phi, h}( f )  \|_{L^{6}(\Omega)} \le \frac{C}{m_{\phi,*}} \left(  1 + \frac{ \|  \nabla m_\phi ( h )  \|_{L^{6}(\Omega)}^5 }{ m_{\phi,*}^5 }  \right) \| f \|_{L^{6}(\Omega)},
        \ \forall f \in L^6_{(0)}(\Omega),
        \\
        & \|  \Delta \mathcal{N}_{\phi, h}( f )  \|_{V } \le \frac{C}{m_{\phi,*}} \left(  1 + \frac{ \|  \nabla m_\phi ( h )  \|_{L^{6}(\Omega)}^4 \Big(  \frac{ \|  \nabla^2 m_\phi ( h )  \| }{ m_{\phi,*} } + 1  \Big) }{ m_{\phi,*}^4 }  \right) \| \nabla f \|,\ \forall f \in V_{(0)},
    \end{align*}
    where the constant $C>0$ depends only on $\Omega$. Similar regularity properties and estimates hold for the operator $\mathcal{N}_{\psi, h}$ defined in \eqref{asi_definition}.
\end{lemma}
\noindent \textbf{Proof. } By definition, for any $f \in V^*_{(0)}$, $g = \mathcal{N}_{\phi, h}f\in V_{(0)}$ is the unique weak solution to the elliptic boundary value problem:
\begin{numcases}{}
    -\mathrm{div}\, \big(  m_\phi ( h ) \nabla g  \big) = f, \quad \text{in} \ \Omega,
    \label{elliptic1} \\
    \partial_{\mathbf{n}} g = 0, \qquad\qquad \qquad \ \ \,  \text{on}\  \partial \Omega.\nonumber
\end{numcases}
Testing \eqref{elliptic1} by $g$ and using the Poincar\'{e}--Wirtinger inequality, we easily obtain
\[
m_{\phi,*} \|  \nabla g  \|^2 \le  \int_\Omega m_{\phi}(h)|\nabla g|^2\,\mathrm{d} x\leq C \| f \|_* \| \nabla g \|,
\]
which yields the first estimate. As in \cite{ADG2013ndm}, we rewrite \eqref{elliptic1} in the following form:
\begin{align}
    -\Delta g =  \frac{ \nabla \big( m_\phi( h ) \big) \cdot \nabla g + f}{ m_\phi( h ) }, \quad \mathrm{in} \ \Omega. \label{elliptic2}
\end{align}
Using the Sobolev embedding $V \hookrightarrow L^{6}(\Omega)$ in three dimensions, for any $f \in L^p_{(0)}(\Omega)$, $ p \in [6/5,6] $, we have $f \in V_{(0)}^*$ and $\| f \|_* \le C \| f \|_{L^{p}(\Omega)}$. This implies
\[
\|\nabla g\|=\|  \nabla \mathcal{N}_{\phi, h}f  \| \le \frac{C}{m_{\phi,*}} \| f \|_{L^{p}(\Omega)}.
\]
Substituting the above estimate into \eqref{elliptic2}, we get
\begin{align*}
    \| \Delta g \|_{L^{p}(\Omega)} & \le \frac{1}{m_{\phi,*}} \big(  \| \nabla m_\phi(h) \|_{L^{6}(\Omega)} \|  \nabla g  \| + \| f \|_{L^{p}(\Omega)}  \big) \\
    & \le \frac{C}{m_{\phi,*}} \Big(  \frac{\| \nabla m_\phi(h) \|_{L^{6}(\Omega)}}{m_{\phi,*}} + 1 \Big) \| f \|_{L^{p}(\Omega)}, \quad
    \forall\, p \in \Big[\frac{6}{5},\frac{3}{2}\Big].
\end{align*}
If $f \in L^{p}_{(0)}(\Omega)$ for $p \in(3/2,2]$, from the above estimate, the elliptic estimate, the Wirtinger--Poincar\'{e} inequality, and the Sobolev embedding theorem, we find
\begin{align*}
    \|  \nabla g  \|_{L^{3}(\Omega)}
    & \le C \| g \|_{W^{2,\frac{3}{2}}(\Omega)}
    \le C \|  \Delta g  \|_{L^{\frac{3}{2}}(\Omega)} \\
    & \le \frac{C}{m_{\phi,*}} \Big(  \frac{\| \nabla m_\phi(h) \|_{L^{6}(\Omega)}}{m_{\phi,*}} + 1 \Big) \| f \|_{L^{p}(\Omega)}, \quad  \forall\, p \in \Big(\frac{3}{2},2\Big].
\end{align*}
Consequently, from \eqref{elliptic2} we infer that
\begin{align}
    \| \Delta g \|_{L^{p}(\Omega)}
    & \le \frac{1}{m_{\phi,*}} \big(  \| \nabla m_\phi(h) \|_{L^{6}(\Omega)} \|  \nabla g  \|_{L^{3}(\Omega)} + \| f \|_{L^{p}(\Omega)}  \big)
    \nonumber \\
    & \le \frac{C}{m_{\phi,*}} \Big[  \Big( \frac{\| \nabla m_\phi(h) \|_{L^{6}(\Omega)}}{m_{\phi,*}} \Big)^2 + 1 \Big] \| f \|_{L^{p}(\Omega)}, \quad  \forall\, p \in \Big(\frac{3}{2},2\Big].\notag
\end{align}
Repeating the above procedure, we can obtain the required estimates for $p\in (2,3]$ and $(3,6)$ and $p=6$.

Next, if $h \in H^{2}(\Omega)$ and $f \in V_{(0)}$, we take gradient on both sides of \eqref{elliptic2} to get
\begin{equation*}
    -\nabla \Delta g = \frac{\nabla^2 \big(  m_\phi ( h )  \big)\nabla g }{m_\phi( h )} - \frac{ \nabla \big( m_\phi( h ) \big) \cdot \nabla g}{ m_\phi( h )^2 } \nabla \big( m_\phi( h ) \big) + \frac{ \nabla^2 g  \nabla \big( m_\phi( h ) \big)}{ m_\phi( h ) } + \frac{ \nabla f}{ m_\phi( h ) } - \frac{ f \nabla \big(  m_\phi ( h )  \big) }{ m_\phi( h )^2 }.
\end{equation*}
Then, from the previous derived estimates, the elliptic estimate and the Sobolev embedding theorem, we can deduce that
\begin{align*}
    \|  \nabla \Delta g  \| & \le \frac{ \|  \nabla^2 m_\phi ( h )  \| \| \nabla g \|_{L^{\infty}(\Omega)}}{m_{\phi,*}} + \Big(  \frac{ \| \nabla \big( m_\phi( h ) \big)  \|_{L^{6}(\Omega)}}{m_{\phi,*}}  \Big)^2 \| \nabla g \|_{L^{6}(\Omega)}
    \\
    & \quad + \frac{\|  \nabla^2 g  \|_{L^{3}(\Omega)} \| \nabla \big( m_\phi( h ) \big)  \|_{L^{6}(\Omega)}}{m_{\phi,*}}  + \frac{\|  \nabla f  \|}{m_{\phi,*}}
    + \frac{ \| f \|_{L^{3}(\Omega)} \| \nabla \big( m_\phi( h ) \big)  \|_{L^{6}(\Omega)} }{m_{\phi,*}^2}
    \\
    & \le C\frac{\|  \nabla^2 m_\phi ( h )  \| \|  g \|_{W^{2,4}(\Omega)} }{m_{\phi,*}}
    + C\Big(  \frac{ \| \nabla \big( m_\phi( h ) \big)  \|_{L^{6}(\Omega)}}{m_{\phi,*}}  \Big)^2 \|  g \|_{W}
    \\
    & \quad + C\frac{\|  g  \|_{W^{2,3}(\Omega)} \| \nabla \big( m_\phi( h ) \big)  \|_{L^{6}(\Omega)}}{m_{\phi,*}}
    + \frac{\|  \nabla f  \|}{m_{\phi,*}}
    + C\frac{ \|\nabla f \| \| \nabla \big( m_\phi( h ) \big)  \|_{L^{6}(\Omega)} }{m_{\phi,*}^2}
    \\
    & \le \frac{C}{m_{\phi,*}} \left(  1 + \frac{ \|  \nabla m_\phi ( h )  \|_{L^{6}(\Omega)}^4 \Big(  \frac{ \|  \nabla^2 m_\phi ( h )  \| }{ m_{\phi,*} } + 1  \Big) }{ m_{\phi,*}^4 }  \right) \| \nabla f \|.
\end{align*}
Combining this with the known estimate on $\|\nabla g\|$, we arrive at the conclusion.
\qed
\smallskip

When $d=2$, the following estimates for $\mathcal{N}_{\phi, h}$ and $\mathcal{N}_{\psi, h}$ have been established in  \cite{CGGG2025}:
\begin{lemma} \label{elliptic_estimate_2d}
Assume that $\Omega\subset \mathbb{R}^2$ is a bounded domain of class $C^3$, $m_\phi$ satisfies the conditions in $(\mathbf{H1})$.
If $h \in H^{2}(\Omega)$, then we have
\begin{itemize}
    \item there exists a constant $C>0$ such that
    \[
\|  \mathcal{N}_{\phi, h} f  \|_{H^{2}(\Omega)}
\le C \big(  \| \nabla h \| \| h \|_{H^{2}(\Omega)} \| \nabla \mathcal{N}_{\phi, h} f \| + \| f \|  \big),\ \ \forall\, f \in L^2_{(0)}(\Omega);
\]
    \item there exists a constant $C>0$ such that
    \[
\|  \mathcal{N}_{\phi, h} f  \|_{W^{2,4}(\Omega)}
\le C \big(  \| \nabla h \|^{\frac{1}{4}} \| h \|_{H^{2}(\Omega)}^{\frac{3}{4}} \| \nabla \mathcal{N}_{\phi, h} f \|^{\frac{1}{4}} \| \mathcal{N}_{\phi, h} f \|_{H^{2}(\Omega)}^{\frac{3}{4}} + \| f \|_{L^{4}(\Omega)}  \big),\ \ \forall\, f \in L^4_{(0)}(\Omega);
\]
    \item  there exists a constant $C>0$ such that
   \begin{align*}
        \|  \mathcal{N}_{\phi,h} f  \|_{H^3(\Omega)}
        & \le C \Big(  \Big\|  \frac{m_\phi'(h)}{m_\phi(h)} \nabla h \cdot \nabla \mathcal{N}_{\phi,h} f  \Big\|_V + \Big\| \frac{f}{m_\phi(h)} \Big\|_V  \Big) ,\ \ \forall\, f \in V_{(0)}.
    \end{align*}
\end{itemize}
Similar results hold for the operator $\mathcal{N}_{\psi, h}$ under the corresponding assumptions on $m_\psi$ and $h$.
\end{lemma}

Finally, we recall some useful chain rules (see \cite{BB1999,CGGG2025}):
\begin{lemma}
\label{mutiply_diff_formula_general}
Let $T>0$. Assume that $m_\phi$ satisfies the conditions in $(\mathbf{H1})$.

(1) Let $d=2,3$. If $h \in L^{\infty} ( 0,T;H^{2}(\Omega) ) \cap H^1 (  0,T;H  )$, then we have
    \begin{equation*}
        \big(  \mathcal{N}_{\phi, h} \partial_{t} (f-\overline{f}), f\big)
        = \frac{\mathrm{d}}{\mathrm{d}t} \frac12 \Big\|  \sqrt{m_\phi( h )} \nabla \mathcal{N}_{\phi, h} (f-\overline{f})  \Big\|^2
        + \frac12 \int_{\Omega}  m_\phi'( h ) \partial_{t} h \big| \nabla \mathcal{N}_{\phi, h} (f-\overline{f}) \big|^2  \, \mathrm{d}x,
    \end{equation*}
    for any $f \in L^{\infty} ( 0,T;V ) \cap H^1 ( 0,T;V^* )$ and almost everywhere in $(0,T)$.

(2) Let $d=2$. If $h \in L^{\infty} ( 0,T;V ) \cap L^{4} ( 0,T;H^2(\Omega))\cap H^{1} ( 0,T;V^* )$, then we have
\begin{align*}
    & \big( \mathcal{N}_{\phi, h} \partial_{t} (  f - \overline{f}), f\big)
    \\
    &\quad  = \frac{\mathrm{d}}{\mathrm{d t}}  \frac{1}{2} \big(  \mathcal{N}_{\phi, h} ( f-\overline{f} ) ,  f \big)
    + \frac{1}{2} \int_{\Omega} ( \nabla \mathcal{N} \partial_{t} h ) \cdot m_\phi'' ( h ) \nabla h |  \nabla \mathcal{N}_{\phi, h} ( f-\overline{f} )  |^2  \, \mathrm{d}x
    \\
    & \qquad + \int_{\Omega} ( \nabla \mathcal{N} \partial_{t} h ) \cdot m_\phi' ( h ) \big(  \nabla^2 \mathcal{N}_{\phi, h} ( f-\overline{f} ) \nabla \mathcal{N}_{\phi, h} ( f-\overline{f} )  \big)  \, \mathrm{d}x,
\end{align*}
for any $f \in L^{\infty} ( 0,T;V ) \cap L^{4} ( 0,T;W ) \cap H^{1} (0,T;V^* )$
and almost everywhere in $(0,T)$.

Similar results hold for the operator $\mathcal{N}_{\psi, h}$ under the corresponding assumptions on $m_\psi$ and $h$.
\end{lemma}

\section{Analysis of a Convective Viscous Cahn--Hilliard System with Nondegenerate Mobilities} \label{appendix-2}
\setcounter{equation}{0}

We study the well-posedness of the convective viscous Cahn--Hilliard system \eqref{eq:cvcho1}--\eqref{eq:cvchoi}. The subsequent results extend the previous work \cite{G2022} on a single convective viscous Cahn--Hilliard equation with constant mobility to a coupled convective viscous Cahn--Hilliard system with nondegenerate mobilities.

\begin{proposition} \label{cvchosse}
Let $\Omega \subset \mathbb{R}^{d}$ $(d \in \{ 2,3 \})$ be a bounded domain of class $C^3$. Moreover, suppose that $\alpha\in (0,1)$, Assumptions $(\mathbf{H1})$--$(\mathbf{H2})$ are satisfied, and $\bm{u}\in L^{2} \big( I;\bm{L}^{2}_\sigma(\Omega) \big)$.

(1) For any initial data $\phi_0,\psi_0 \in V \cap L^{\infty}(\Omega)$ with $(\phi_0,\psi_0) \in [-1,1]\times[0,1]$ almost everywhere in $\Omega$ and $( \overline{\phi_0},\overline{\psi_0} ) \in (-1,1) \times (0,1)$, problem \eqref{eq:cvcho1}--\eqref{eq:cvchoi} admits a weak solution $(\phi, \psi)$ on $I$ in the sense of Definition \ref{cvchossd}.

(2) When $d=2$, the weak solution is unique.
While for $d=3$, if there exists a strong solution $( \widetilde{\phi},\widetilde{\psi} )$ to problem \eqref{eq:cvcho1}--\eqref{eq:cvchoi} on $[0,T_0] \subset I$ originating from the same initial data $( \phi_0,\psi_0 )$ with the additional regularity
$\widetilde{\phi},\widetilde{\psi} \in L^{\infty} \big( 0,T_0;W \big)$, then $(\phi, \psi)=( \widetilde{\phi},\widetilde{\psi} )$ almost everywhere in $[0,T_0]$.

In addition, we have the following regularity results:
\begin{itemize}
    \item [$(\mathbf{R1})$] If $\bm{u} \in L^{\infty} ( I;\bm{L}^{3}_\sigma(\Omega) )$, $\partial_t\bm{u}\in L^{\frac{4}{3}}(  I;\bm{L}^{1}_\sigma(\Omega) )$ and
    $\phi_0,\psi_0 \in W$ with
    $-\Delta \phi_0 + F_\phi'( \phi_0 )$, $-\beta\Delta \psi_0 + F_\psi'( \psi_0 ) \in H$,
    then problem \eqref{eq:cvcho1}--\eqref{eq:cvchoi} admits a unique strong solution on $[0,T_r]\subset I$ such that
    \begin{align}
        & \partial_{t} \phi,\partial_{t} \psi \in L^{\infty} ( 0,T_r;H ) \cap L^{2} ( 0,T_r;V ), \label{concentration_dt_r1}\\
        & \phi,\psi \in L^{\infty} ( 0,T_r;W ) \cap L^{2} ( 0,T_r;W^{2,6}(\Omega) ),
        \label{concentration_r1} \\
        & \mu_\phi,\mu_\psi \in L^{\infty} ( 0,T_r;W ). \label{potential_r1}
    \end{align}
    When $d=2$, the strong solution exists on the whole interval $I$ such that
    \begin{align*}
        & \partial_{t} \phi,\partial_{t} \psi \in L^{\infty} ( I;H ) \cap L^{2}_{\mathrm{uloc}} ( I;V ),
        \\
        & \phi,\psi \in L^{\infty} ( I;W ) \cap L^{2}_{\mathrm{uloc}} ( I;W^{2,p}(\Omega) ), \quad  \forall\, p \in [2,\infty),
        \\
        & \mu_\phi,\mu_\psi \in L^{\infty} ( I;W ).
    \end{align*}
    When $d=3$, if $\partial_t \bm{u} \in L^p\big(  I;\bm{L}^{3}_\sigma(\Omega)  \big)$ for some
    $p > 4/3$, then there exist positive times $T_1, T_2 \in I$
    such that
    \begin{align*}
        T_r \ge \min \{ 1, T_1,T_2  \},
    \end{align*}
   where $T_1$ and $T_2$ are determined by \eqref{T_2}--\eqref{T_1} below.
    \item [$(\mathbf{R2})$] Let the assumption of  $(\mathbf{R1})$ hold. Suppose that for some $\delta_{0,\phi} \in (0,1)$ and $\delta_{0,\psi} \in (0,1/2)$, it holds
    \[
    ( \phi_0(x),\psi_0(x) ) \in [-1+\delta_{0,\phi},1-\delta_{0,\phi}] \times [\delta_{0,\psi},1-\delta_{0,\psi}], \quad \forall x \in \Omega.
    \]
    Then there exist $\delta_\phi \in (0,1)$, $\delta_\psi\in (0,1/2)$ depending on $\alpha$ such that
    \[
(  \phi(x,t),\psi(x,t)  ) \in [-1+\delta_\phi,1-\delta_\phi] \times [\delta_\psi,1-\delta_\psi],\quad \forall\, (x,t) \in \Omega \times [0,T_r],
\]
and $\phi, \psi \in L^{2}_{\mathrm{uloc}} \big( 0,T_r;H^{3}(\Omega) \big)$.
\end{itemize}
\end{proposition}

\begin{remark} \label{reason_modification_regularity_cvcho} \rm
In contrast to \cite{G2022}, we prove $(\mathbf{R1})$ using an additional fixed-point argument, time-difference estimates for an approximate problem with regularized potential functions, and choosing time derivatives instead of time-difference quotients as test functions. These modifications are mainly due to concentration-dependent mobilities and the lack of mass conservation resulting from the source term $\sigma_1 (\overline{\phi} - c)$. Both features render the arguments in \cite{G2022} no longer directly applicable.

The application of
Lemma \ref{weightedlaplace} requires that the weight functions $m_\phi(\phi)$ and $m_\psi(\psi)$ belong to $L^\infty(0,T;H^2(\Omega))$. Since this regularity is not available \textit{a priori} for the solution in Definition \ref{cvchossd}, we freeze the coefficients by introducing smooth input functions $(\widetilde\phi, \widetilde\psi)$, which naturally leads to a fixed-point argument. To close the estimates, it is necessary to control the $L^\infty(0,T;W)$-norms of the output functions $(\phi,\psi)$. By equations \eqref{eq:cvcho2} and \eqref{eq:cvcho4}, this further requires estimates for $\partial_t\phi$ and $\partial_t\psi$ in $L^\infty(0,T;H)$. However, these estimates cannot be derived directly from the time-difference argument as in \cite{G2022}. The obstruction comes from terms such as $\partial^+_{t,h} (\overline{F'_\phi(\phi)}) \partial^+_{t,h} \overline{\phi}$, which simply vanishes if $\overline{\phi}$ is conserved. On the other hand, we observe that its formal limit as $h \to 0$, i.e., $\partial_t (\overline{F'_\phi(\phi)}) \partial_t \overline{\phi}$, can be rewritten as
\[
    \partial_t (\overline{F'_\phi(\phi)}) \partial_t \overline{\phi} = \partial_t ( \overline{F'_\phi(\phi)} \partial_t \overline{\phi} ) - \overline{F'_\phi(\phi)} \partial_t^2 \overline{\phi}.
\]
This motivates us to test the time-differentiated equations using time derivatives rather than testing the time-difference equations with difference quotients. In addition, testing the time-differentiated equations by $\mathcal N_{\phi,\phi}(\partial_t\phi-\partial_t\overline{\phi})$ requires additional regularity of both $\phi$ and $\mu_\phi$ (and similarly for $\psi, \mu_\psi$), which is not available \textit{a priori}. We therefore construct more regular approximate solutions and then transfer the regularity to the solution through a combination of uniqueness and compactness arguments.

Due to the highly nonlinear structure induced by concentration-dependent mobilities, we expect only the existence of local-in-time strong solutions in three dimensions, cf. \cite{CGGG2025,CGGS2026} on the singular Cahn--Hilliard equation. Moreover, we anticipate uniqueness for weak solutions only in two dimensions, whereas in three dimensions we obtain a partial result of weak-strong uniqueness.
\end{remark}

\noindent \textbf{Proof of Proposition \ref{cvchosse}.}
The existence of a global weak solution can be proved by standard arguments and is therefore omitted here (cf. \cite{DG2022,O2025} for a similar coupled Cahn--Hilliard system with constant mobilities and without convection). Compared to the \textit{a priori} estimates in \cite{G2022}, below we only mention the estimates for $\mu_\phi,\mu_\psi$ in $ L^{\frac{4}{3}}_{\mathrm{uloc}} \big( I;W^{2,\frac{4}{3}}(\Omega) \big)$ under the assumption $\bm{u}\in L^{2} \big( I;\bm{L}^{2}_\sigma(\Omega) \big)$.
Define $f_\phi = \partial_{t} \phi + \bm{u} \cdot \nabla \phi + \sigma_1 \big( \overline{\phi} - c \big)$. Noticing that
\[
\| \nabla \phi \|_{L^{4}_{\mathrm{uloc}} ( I;L^{4}(\Omega) )} \le C \| \phi \|_{L^{\infty}(Q_T)}^{\frac{1}{2}} \|  \phi  \|_{L^{2}_{\mathrm{uloc}} ( I;W )}^{\frac{1}{2}},
\]
by H\"{o}lder's inequality, we can obtain $f_\phi \in L^{\frac{4}{3}}_{\mathrm{uloc}} \big(  I;L^{\frac{4}{3}}(\Omega)  \big)$.
Besides, from $m_\phi \in BUC^1(\mathbb{R})$, we infer that $\nabla \big( m_\phi ( \phi ) \big) \in L^{4}_{\mathrm{uloc}} \big( I;L^{4}(\Omega) \big)$.
Since $\mu_\phi\in L^2_{\mathrm{uloc}}(I;V)$ can be viewed as a weak solution to the following elliptic boundary value problem:
\begin{numcases}{}
    \mathrm{div}\, \big(  m_\phi ( \phi ) \nabla g  \big) = f_\phi, & in $\Omega$, \nonumber \\
    \partial_{\mathbf{n}} g = 0, & on $\partial \Omega$, \nonumber
\end{numcases}
as in the proof of Lemma \ref{weightedlaplace}, we can conclude
$$
\Delta \mu_\phi \in L^{\frac{4}{3}}_{\mathrm{uloc}} \big( I;L^{\frac{4}{3}}(\Omega) \big).
$$
Similarly, we get $\Delta \mu_\psi \in L^{\frac{4}{3}}_{\mathrm{uloc}} \big( I;L^{\frac{4}{3}}(\Omega) \big)$. Then, by the standard elliptic regularity theory, we obtain $\mu_\phi, \mu_\psi \in L^{\frac{4}{3}}_{\mathrm{uloc}} \big( I;W^{2,\frac{4}{3}}(\Omega) \big)$.

\medskip
\noindent \textbf{Uniqueness.}
Let $( \phi_i,\psi_i )$, $i\in \{ 1,2 \}$,  be two weak solutions to problem \eqref{eq:cvcho1}--\eqref{eq:cvchoi} on certain time interval $[0,T_0] \subset I$ with the initial data $( \phi_{0,i},\psi_{0,i} )$ and the velocity fields $\bm{u}_i$. Define the differences $\widetilde{\phi} = \phi_1 - \phi_2$, $\widetilde{\psi} = \psi_1 - \psi_2$ and $\widetilde{\bm{u}} = \bm{u}_1 - \bm{u}_2$.
Testing the equation satisfied by  $\widetilde{\phi}$ (resp. $\widetilde{\psi}$) with $g_\phi=\mathcal{N}_{\phi, \phi_1} \big(  \widetilde{\phi} - \overline{\widetilde{\phi}}  \big)$ (resp. $g_\psi=\mathcal{N}_{\psi, \psi_1} \big(  \widetilde{\psi} - \overline{\widetilde{\psi}}  \big)$) and using integration by parts, we get
\begin{align*}
    & \big( \partial_{t} \widetilde{\phi} - \overline{\partial_{t} \widetilde{\phi}}, g_\phi \big)
    - \int_{\Omega}  \big(  \widetilde{\phi} - \overline{\widetilde{\phi}}  \big) \nabla g_\phi \cdot \bm{u}_2  \, \mathrm{d}x
    - \int_{\Omega}  \phi_1 \nabla g_\phi \cdot \widetilde{\bm{u}}  \, \mathrm{d}x
    + \big(  \mu_{\phi_1} - \mu_{\phi_2} ,  \widetilde{\phi} - \overline{\widetilde{\phi}}  \big)
    \\
    & \quad = - \int_{\Omega}  \big( m_\phi ( \phi_1 )-m_\phi ( \phi_2 ) \big) \nabla \mu_{\phi_2} \cdot \nabla g_\phi  \, \mathrm{d}x,
    \\
    & \big( \partial_{t} \widetilde{\psi} - \overline{\partial_{t} \widetilde{\psi}} , g_\psi \big) - \int_{\Omega}  \big(  \widetilde{\psi} - \overline{\widetilde{\psi}}  \big) \nabla g_\psi \cdot \bm{u}_2  \, \mathrm{d}x
    - \int_{\Omega}  \psi_1 \nabla g_\psi \cdot \widetilde{\bm{u}}  \ \mathrm{d}x + \big(  \mu_{\psi_1} - \mu_{\psi_2} ,  \widetilde{\psi} - \overline{\widetilde{\psi}}  \big)
    \\
    & \quad = - \int_{\Omega}  \big( m_\psi ( \psi_1 )-m_\psi ( \psi_2 ) \big) \nabla \mu_{\psi_2} \cdot \nabla g_\psi  \, \mathrm{d}x.
\end{align*}
Adding the two identities together, under the additional assumption
$\phi_1,\psi_1 \in L^{\infty} \big( 0,T_0;W \big)$, we can  apply the chain rule in Lemma \ref{mutiply_diff_formula_general}-(1) and obtain
\begin{align*}
    & \frac{1}{2} \frac{\mathrm{d}}{\mathrm{d t}}  \Big(  \Big\|  \sqrt{m_\phi( \phi_1 )} \nabla g_\phi  \Big\|^2 +  \Big\|  \sqrt{m_\psi( \psi_1 )} \nabla g_\psi  \Big\|^2 + \alpha \big\|  \widetilde{\phi} - \overline{\widetilde{\phi}}  \big\|^2 + \alpha \big\|  \widetilde{\psi} - \overline{\widetilde{\psi}}  \big\|^2  \Big)
    \\
    & \qquad + \big\| \nabla \widetilde{\phi} \big\|^2 + \beta \big\|  \nabla \widetilde{\psi}  \big\|^2 + \sigma_2 \big\|  \widetilde{\phi} - \overline{\widetilde{\phi}}  \big\|^2_*
    \\
    & \quad = \int_{\Omega}   \big[  \big(  \widetilde{\phi} - \overline{\widetilde{\phi}}  \big)\nabla g_\phi + \big(  \widetilde{\psi} - \overline{\widetilde{\psi}}  \big)\nabla g_\psi  \big] \cdot \bm{u}_2 \, \mathrm{d}x
    + \int_{\Omega}  \big(  \phi_1 \nabla g_\phi + \psi_1 \nabla g_\psi  \big) \cdot \widetilde{\bm{u}}  \, \mathrm{d}x
    \\
    & \qquad - \frac{1}{2} \int_{\Omega} \Big( m_\phi'( \phi_1 ) \partial_{t} \phi_1 | \nabla g_\phi |^2 + m_\psi'( \psi_1 ) \partial_{t} \psi_1 | \nabla g_\psi |^2  \Big) \, \mathrm{d}x
    \\
    & \qquad - \int_{\Omega} \Big[ \big( m_\phi ( \phi_1 )-m_\phi ( \phi_2 ) \big) \nabla \mu_{\phi_2} \cdot \nabla g_\phi + \big( m_\psi ( \psi_1 )-m_\psi ( \psi_2 ) \big) \nabla \mu_{\psi_2} \cdot \nabla g_\psi \Big] \, \mathrm{d}x \\
    & \qquad - (  F_\phi'(\phi_1) - F_\phi'(\phi_2) + \partial_\phi G(\phi_1,\psi_1) - \partial_\phi G(\phi_2,\psi_2) ,  \widetilde{\phi} - \overline{\widetilde{\phi}}  )
    \\
    & \qquad - (  F_\psi'(\psi_1) - F_\psi'(\psi_2) + \partial_\psi G(\phi_1,\psi_1) - \partial_\psi G(\phi_2,\psi_2) ,  \widetilde{\psi} - \overline{\widetilde{\psi}}  ).
\end{align*}
The right-hand side of the above inequality can be estimated using Lemma \ref{weightedlaplace} and the convexity of $F_\phi$, $F_\psi$:
\begin{align*}
    & \frac{1}{2} \frac{\mathrm{d}}{\mathrm{d t}}  \Big(  \Big\|  \sqrt{m_\phi( \phi_1 )} \nabla g_\phi  \Big\|^2 +  \Big\|  \sqrt{m_\psi( \psi_1 )} \nabla g_\psi  \Big\|^2 + \alpha \big\|  \widetilde{\phi} - \overline{\widetilde{\phi}}  \big\|^2 + \alpha \big\|  \widetilde{\psi} - \overline{\widetilde{\psi}}  \big\|^2  \Big)
    \\
    & \qquad + \big\| \nabla \widetilde{\phi} \big\|^2 + \beta \big\|  \nabla \widetilde{\psi}  \big\|^2 + \sigma_2 \Big\|  \widetilde{\phi} - \overline{\widetilde{\phi}}  \Big\|^2_*
    \\
    & \quad \le \| \bm{u}_2 \| \Big(  \big\|  \widetilde{\phi} - \overline{\widetilde{\phi}}  \big\|_{L^{6}(\Omega)} \|  \nabla g_\phi  \|_{L^{3}(\Omega)} + \big\|  \widetilde{\psi} - \overline{\widetilde{\psi}}  \big\|_{L^{6}(\Omega)} \|  \nabla g_\psi  \|_{L^{3}(\Omega)}  \Big)
    \\
    & \qquad + \|\widetilde{\bm{u}} \| ( \|\phi_1\|_{L^\infty(\Omega)} \|  \nabla g_\phi  \| +\|\psi_1\|_{L^\infty(\Omega)} \|  \nabla g_\psi  \|  )  + C \| \partial_{t} \phi_1 \| \| \nabla g_\phi \|^2_{L^{4}(\Omega)}
    \\
    & \qquad  + C \| \partial_{t} \psi_1 \| \| \nabla g_\psi \|^2_{L^{4}(\Omega)}
    + \| m_\phi ( \phi_1 )-m_\phi ( \phi_2 ) \|_{L^{6}(\Omega)} \| \nabla \mu_{\phi_2} \| \| \nabla g_\phi \|_{L^{3}(\Omega)}
    \\
     &\qquad + \| m_\psi ( \psi_1 )-m_\psi ( \psi_2 ) \|_{L^{6}(\Omega)} \| \nabla \mu_{\psi_2} \| \| \nabla g_\psi \|_{L^{3}(\Omega)}
     + \big| \overline{\widetilde{\phi}} \big| \AbsNormal{\Omega} \big(  \big|  \overline{F_\phi'(\phi_1)}  \big| + \big|  \overline{F_\phi'(\phi_2)}  \big|  \big)
    \\
    & \qquad  + \big| \overline{\widetilde{\psi}} \big| \AbsNormal{\Omega} \big(  \big|  \overline{F_\psi'(\psi_1)}  \big| + \big|  \overline{F_\psi'(\psi_2)}  \big|  \big)
     + C \big( \big\| \widetilde{\phi} \big\| + \big\| \widetilde{\psi} \big\| \big) \big( \big\|  \widetilde{\phi} - \overline{\widetilde{\phi}}  \big\| + \big\|  \widetilde{\psi} - \overline{\widetilde{\psi}}  \big\| \big)
    \\
    & \quad \le C \| \bm{u}_2 \|^2 \Big[  \big( 1+\| \nabla m_\phi ( \phi_1 ) \|_{L^{6}(\Omega)}  \big)^2 \big\| \widetilde{\phi} - \overline{\widetilde{\phi}} \big\|^2 + \big( 1+\| \nabla m_\psi ( \psi_1 ) \|_{L^{6}(\Omega)}  \big)^2 \big\| \widetilde{\psi} - \overline{\widetilde{\psi}} \big\|^2  \Big]
    \\
    & \qquad + \frac{1}{2} \big\|  \nabla \widetilde{\phi}  \big\|^2  + \frac{\beta}{2} \big\| \nabla \widetilde{\psi} \big\|^2  + \| \widetilde{\bm{u}} \|^2 + \frac{1}{2} \big(  \|  \nabla g_\phi  \|^2 + \|  \nabla g_\psi  \|^2  \big)
    \\
    & \qquad + C \Big[  \|  \partial_{t} \phi_1  \| \big( 1+\| \nabla m_\phi ( \phi_1 ) \|_{L^{6}(\Omega)}^2  \big)^2 \big\| \widetilde{\phi} - \overline{\widetilde{\phi}} \big\|^2 + \|  \partial_{t} \psi_1  \| \big( 1+\| \nabla m_\psi ( \psi_1 ) \|_{L^{6}(\Omega)}^2  \big)^2 \big\| \widetilde{\psi} - \overline{\widetilde{\psi}} \big\|^2  \Big]
    \\
    & \qquad + C \Big[  \|  \nabla \mu_{\phi_2}  \|^2 \big( 1+\| \nabla m_\phi ( \phi_1 ) \|_{L^{6}(\Omega)}  \big)^2 \big\| \widetilde{\phi} - \overline{\widetilde{\phi}} \big\|^2 + \|  \nabla \mu_{\psi_2}  \|^2 \big( 1+\| \nabla m_\psi ( \psi_1 ) \|_{L^{6}(\Omega)}  \big)^2 \big\| \widetilde{\psi} - \overline{\widetilde{\psi}} \big\|^2  \Big]
    \\
    & \qquad + C \big( \big| \overline{\widetilde{\phi}} \big| + \big| \overline{\widetilde{\psi}} \big| \big) \big(  1 + \|  \nabla \mu_{\phi_1}  \| + \|  \nabla \mu_{\phi_2}  \|
    + \|  \nabla \mu_{\psi_1}  \| + \|  \nabla \mu_{\psi_2}  \| \big)
    \\
    &\qquad + C \Big[ \big\|  \widetilde{\phi} - \overline{\widetilde{\phi}}  \big\|^2 + \big\|  \widetilde{\psi} - \overline{\widetilde{\psi}}  \big\|^2 +\big( \overline{\widetilde{\phi}} \big)^2 + \big( \overline{\widetilde{\psi}} \big)^2 \Big].
\end{align*}
Consequently, we obtain
\begin{align}
    & \frac{1}{2} \frac{\mathrm{d}}{\mathrm{d t}}  \Big(  \Big\|  \sqrt{m_\phi( \phi_1 )} \nabla g_\phi  \Big\|^2 +  \Big\|  \sqrt{m_\psi( \psi_1 )} \nabla g_\psi  \Big\|^2 + \alpha \big\|  \widetilde{\phi} - \overline{\widetilde{\phi}}  \big\|^2 + \alpha \big\|  \widetilde{\psi} - \overline{\widetilde{\psi}}  \big\|^2  \Big)  \nonumber \\
    & \qquad + \frac{1}{2} \big\| \nabla \widetilde{\phi} \big\|^2 + \frac{\beta}{2} \big\|  \nabla \widetilde{\psi}  \big\|^2 + \sigma_2 \Big\|  \widetilde{\phi} - \overline{\widetilde{\phi}}  \Big\|^2_* \nonumber \\
    & \quad \le C H_1 \Big(  \Big\|  \sqrt{m_\phi( \phi_1 )} \nabla g_\phi  \Big\|^2 +  \Big\|  \sqrt{m_\psi( \psi_1 )} \nabla g_\psi  \Big\|^2 + \alpha \big\|  \widetilde{\phi} - \overline{\widetilde{\phi}}  \big\|^2 + \alpha \big\|  \widetilde{\psi} - \overline{\widetilde{\psi}}  \big\|^2  \Big) + C H_2,
    \label{vCH-dfii-1}
\end{align}
where
\begin{align*}
    H_1 & = \big( 1+\| \nabla m_\phi ( \phi_1 ) \|_{L^{6}(\Omega)}  \big)^2 \big(  \| \bm{u}_2 \|^2 + \|  \nabla \mu_{\phi_2}  \|^2  \big) + \big( 1+\| \nabla m_\psi ( \psi_1 ) \|_{L^{6}(\Omega)}  \big)^2 \big(  \| \bm{u}_2 \|^2 + \|  \nabla \mu_{\psi_2}  \|^2  \big) \\
    & \quad + \|  \partial_{t} \phi_1  \| \big( 1+\| \nabla m_\phi ( \phi_1 ) \|_{L^{6}(\Omega)}^2  \big)^2 + \|  \partial_{t} \psi_1  \| \big( 1+\| \nabla m_\psi ( \psi_1 ) \|_{L^{6}(\Omega)}^2  \big)^2 +1,
    \\
    H_2 &= \big( \big| \overline{\widetilde{\phi}} \big| + \big| \overline{\widetilde{\psi}} \big| \big) \big(  1 + \|  \nabla \mu_{\phi_1}  \| + \|  \nabla \mu_{\phi_2}  \| + \|  \nabla \mu_{\psi_1}  \| + \|  \nabla \mu_{\psi_2}  \|   \big) + \| \widetilde{\bm{u}} \|^2.
\end{align*}
Therefore, using Gronwall's lemma and the mass relations
$$
\overline{\widetilde{\phi}}(t) = e^{-\sigma_1 t}\big(  \overline{\phi_{0,1}} - \overline{\phi_{0,2}}  \big), \quad  \overline{\widetilde{\psi}}(t) =  \overline{\psi_{0,1}} - \overline{\psi_{0,2}},
$$
we infer that
\begin{align}
    &\Big\|  \sqrt{m_\phi( \phi_1 )} \nabla g_\phi  \Big\|^2 +  \Big\|  \sqrt{m_\psi( \psi_1 )} \nabla g_\psi  \Big\|^2 + \alpha \big\|  \widetilde{\phi} - \overline{\widetilde{\phi}}  \big\|^2 + \alpha \big\|  \widetilde{\psi} - \overline{\widetilde{\psi}}  \big\|^2 \le C R_0 e^{\int_{0}^{T_0} H_1(\tau)\, \mathrm{d}\tau},
    \label{vCH-dfii-3}
\end{align}
where
\begin{align*}
    & R_0 = \Big\| \sqrt{m_\phi( \phi_{0,1} )} \nabla \mathcal{N}_{\phi, \phi_{0,1}} \big( \widetilde{\phi}(0) - \overline{\widetilde{\phi}}(0) \big)  \Big\|^2 +  \Big\| \sqrt{m_\phi( \phi_{0,2} )}  \nabla \mathcal{N}_{\psi, \psi_{0,1}} \big( \widetilde{\psi}(0) - \overline{\widetilde{\psi}}(0) \big)  \Big\|^2
    \\
    & \qquad + \alpha \big\|  \widetilde{\phi}(0) - \overline{\widetilde{\phi}}(0)  \big\|^2
    + \alpha \big\|  \widetilde{\psi}(0) - \overline{\widetilde{\psi}}(0)  \big\|^2
    + \| \widetilde{\bm{u}} \|_{L^{2} ( 0,T_0;H )}^2
    \\
    & \qquad + \big( \big|\overline{\widetilde{\phi}}(0)\big| +  \big|\overline{\widetilde{\psi}}(0)\big| \big)  T_0  \big( 1+\widetilde{E}_{\mathrm{free}} ( \phi_{0,1},\psi_{0,1} )  + \| \bm{u}_1 \|_{L^{2} ( 0,T_0;H)}^2  \big)
    \\
    & \qquad + \big( \big|\overline{\widetilde{\phi}}(0)\big| +  \big|\overline{\widetilde{\psi}}(0)\big| \big)  T_0 \big(1+ \widetilde{E}_{\mathrm{free}} ( \phi_{0,2},\psi_{0,2} ) + \| \bm{u}_2 \|_{L^{2} ( 0,T_0;H )}^2  \big).
\end{align*}
When $d=2$, the inequality \eqref{vCH-dfii-1} can be refined as (for details, see the proof of Lemma \ref{cvcho_2d_global_sse}, see also \cite{CGGG2025} for the case $\alpha=0$)
\begin{align}
    & \frac{1}{2} \frac{\mathrm{d}}{\mathrm{d t}}  \Big(  \Big\|  \sqrt{m_\phi( \phi_1 )} \nabla g_\phi  \Big\|^2 +  \Big\|  \sqrt{m_\psi( \psi_1 )} \nabla g_\psi  \Big\|^2 + \alpha \big\|  \widetilde{\phi} - \overline{\widetilde{\phi}}  \big\|^2 + \alpha \big\|  \widetilde{\psi} - \overline{\widetilde{\psi}}  \big\|^2  \Big)
    \nonumber \\
    & \qquad + \frac12 \big\| \nabla \widetilde{\phi} \big\|^2 + \frac{\beta}{2}  \big\|  \nabla \widetilde{\psi}  \big\|^2
    + \sigma_2 \Big\|  \widetilde{\phi} - \overline{\widetilde{\phi}}  \Big\|^2_*
    \nonumber \\
    & \quad \le C \widetilde{H}_1 \Big(  \Big\|  \sqrt{m_\phi( \phi_1 )} \nabla g_\phi  \Big\|^2 +  \Big\|  \sqrt{m_\psi( \psi_1 )} \nabla g_\psi  \Big\|^2 + \alpha \big\|  \widetilde{\phi} - \overline{\widetilde{\phi}}  \big\|^2 + \alpha \big\|  \widetilde{\psi} - \overline{\widetilde{\psi}}  \big\|^2  \Big) + C \widetilde{H}_2,
      \label{vCH-dfii-2}
\end{align}
where
\begin{align*}
    \widetilde{H}_1 & = 1 + \| \bm{u}_2 \|^2 + \| \nabla \phi_1 \|_V^2 + \| \nabla \psi_1 \|_V^2 + \| \partial_{t} \phi_1 \|^2 + \| \partial_{t} \psi_1 \|^2 + \| \nabla \mu_{\phi_2} \|^2 + \| \nabla \mu_{\psi_2} \|^2 ,\\
    \widetilde{H}_2 & = \big( \big| \overline{\widetilde{\phi}} \big| + \big| \overline{\widetilde{\psi}} \big| \big) \big(  1 + \|  \nabla \mu_{\phi_1}  \| + \|  \nabla \mu_{\phi_2}  \|
    +\|  \nabla \mu_{\psi_1}  \| + \|  \nabla \mu_{\psi_2}  \| \big) + \|\widetilde{\bm{u}} \|^2.
\end{align*}
Compared with the three-dimensional case, regularity of weak solutions in two dimensions is sufficient to derive the differential inequality \eqref{vCH-dfii-2}.

Hence, we can deduce from \eqref{vCH-dfii-3} the result of weak-strong uniqueness in three dimensions, and analogously from \eqref{vCH-dfii-2} the uniqueness of weak solutions when $d=2$.
\medskip

\noindent \textbf{Regularity 1.} The above argument immediately yields the uniqueness of strong solutions that possess the additional regularity specified in $(\mathbf{R1})$.
Thus, it remains to prove the existence. The proof consists of several steps.
\smallskip
\\
\textbf{Step 1. A regularized auxiliary problem.}
Consider
\begin{numcases}{}
    \partial_{t} \phi  + \bm{u} \cdot \nabla \phi  + \sigma_1 \big( \overline{\phi } -c \big) = \mathrm{div}\, \big(m_\phi \big( \widetilde{\phi} \big) \nabla \mu_\phi \big),
    \label{eq:apcvcho1schauder} \\
    \mu_\phi  = \alpha \partial_{t} \phi  -\Delta \phi  + \sigma_2 \mathcal{N} \big( \phi  - \overline{\phi } \big) + \big( F_\phi^{\epsilon} \big)'(\phi ) + \partial_\phi G(\phi ,\psi ), \label{eq:apcvcho2schauder} \\
    \partial_{t} \psi  + \bm{u} \cdot \nabla \psi  = \mathrm{div}\, \big(m_\psi \big( \widetilde{\psi} \big) \nabla \mu_\psi \big),
    \label{eq:apcvcho3schauder} \\
    \mu_\psi  = \alpha \partial_{t} \psi  -\beta \Delta \psi  + \big( F_\psi^{\epsilon} \big)'(\psi ) + \partial_\psi G(\phi ,\psi ), \label{eq:apcvcho4schauder}
\end{numcases}
in $Q_T$, subject to the boundary and initial conditions:
\begin{numcases}{}
    \partial_{\mathbf{n}} \phi  = \partial_{\mathbf{n}} \mu_\phi  = \partial_{\mathbf{n}} \psi  = \partial_{\mathbf{n}} \mu_\psi  = 0,  \label{eq:apcvchobschauder} & on $S_T$,
     \\
    \phi |_{t=0} = \phi_{0,k}(x), \ \ \psi |_{t=0} = \psi_{0,k}(x), \label{eq:apcvchoischauder} & in $\Omega$.
\end{numcases}
Here, $\widetilde{\phi}$ and $\widetilde{\psi}$ are given functions that satisfy
\[
\widetilde{\phi},\widetilde{\psi} \in L^{\infty} ( I;W ) \cap W^{1,\infty}(I;H ), \quad  ( \widetilde{\phi},\widetilde{\psi} )|_{t=0} = ( \phi_{0},\psi_{0} ),
\]
and the regularized initial data $\phi_{0,k}$, $\psi_{0,k}$ are defined as in Section \ref{loc-3D-1}.
With the same notation $\mu_{\phi,0} = -\Delta \phi_0 + F_\phi' ( \phi_0 )$, $\mu_{\psi,0} = -\beta \Delta \psi_0 + F_\psi' ( \psi_0 )$ and $\mu_{\phi,0}^k =  h_k \circ \mu_{\phi,0}$, $\mu_{\psi,0}^k =  h_k \circ \mu_{\psi,0}$, we find that
$\mu_{\phi,0}^k,\, \mu_{\psi,0}^k\in  L^\infty(\Omega)$ satisfy
$$
\|\mu_{\phi,0}^k\|\leq \|\mu_{\phi,0}\|,\quad \|\mu_{\psi,0}^k\|\leq \|\mu_{\psi,0}\|,\quad \forall\, k\in \mathbb{N}_+.
$$
Furthermore, the function $\phi_{0,k}\in W$ (resp. $ \psi_{0,k}\in W$) is uniquely determined with $F_\phi' ( \phi_{0,k} )\in H$ (resp. $F_\psi' ( \psi_{0,k})\in H$), both uniformly bounded by $\|\mu_{\phi,0}\|$ (resp. $\|\mu_{\psi,0}\|$), cf. \cite{GMT2019}. The conclusions in Lemma \ref{approximate_initial_phase_field}-(2), (3) (i.e., the strict separation property) also apply. The approximate singular potentials are defined as in \cite{GMT2019,M2019}:
\begin{equation}\notag
    ( F_\phi^{\epsilon} )''(s)
    = \left\{
        \begin{aligned}
            & F_\phi''\big( -1+\epsilon \big), &&\ s \le -1+\epsilon, \\
            & F_\phi''(s), &&\ s \in \big(  -1+\epsilon,1-\epsilon \big),\\
            & F_\phi''\big( 1-\epsilon \big), &&\ s \ge 1-\epsilon,
        \end{aligned}
        \right.
\quad
    ( F_\psi^{\epsilon} )''(s)
    = \left\{
        \begin{aligned}
            & F_\psi''\big( \epsilon \big), &&\ s \le \epsilon, \\
            & F_\psi''(s), &&\ s \in \big( \epsilon, 1-\epsilon \big),\\
            & F_\psi''\big( 1-\epsilon \big), &&\ s \ge 1-\epsilon,
        \end{aligned}
        \right.
\end{equation}
where $\epsilon \in (0,1/2]$, with
\begin{equation*}
    F_\phi^{\epsilon}(0) = \big( F_\phi^{\epsilon} \big)'(0)=F_\psi^{\epsilon}\Big( \frac{1}{2} \Big) = \big( F_\psi^{\epsilon} \big)'\Big( \frac{1}{2} \Big)=0.
\end{equation*}
For the given velocity field $\bm{u}$, we first consider $\bm{u} \in C^{\infty}_0 (I; \bm{C}^{\infty}_{0,\sigma} ( \Omega ))$ and then recover the general case $\bm{u} \in L^{\infty} ( I;\bm{L}^{3}_\sigma(\Omega) ) \cap W^{1,\frac{4}{3}} (  I;\bm{L}^{1}_\sigma(\Omega) )$ by the density argument (cf. \cite{AGG2024}).

A standard argument like in \cite{GMT2019} yields that problem \eqref{eq:apcvcho1schauder}--\eqref{eq:apcvchoischauder} admits a unique weak solution $( \phi^{k,\epsilon},\psi^{k,\epsilon} )$ on $I$ such that
\begin{align*}
    & \phi^{k,\epsilon},\psi^{k,\epsilon} \in BC_{w} ( I;V ) \cap L^{2}_{\mathrm{uloc}} (I;W )\cap
    H^1( I;H ),
    \\
    & \mu_\phi^{k,\epsilon},\mu_\psi^{k,\epsilon} \in L^{2}_{\mathrm{uloc}} ( I;W ),
    \quad  \big( F_\phi^{\epsilon} \big)'\big( \phi^{k,\epsilon} \big), \big( F_\psi^{\epsilon} \big)'\big( \psi^{k,\epsilon} \big) \in L^{2}_{\mathrm{uloc}} ( I ;H )
\end{align*}
with uniform estimates with respect to $\epsilon, k$ in the corresponding spaces.
In particular, the $L^2_tH^2_x$-estimate of $\mu_\phi^{k,\epsilon}$ follows from Lemma \ref{weightedlaplace} and equations \eqref{eq:apcvcho1schauder}--\eqref{eq:apcvcho2schauder}:
\begin{align}
    \| \nabla \mu_\phi^{k,\epsilon} \|
    &\le C \big(  \big\|  \partial_{t} \phi^{k,\epsilon} - \partial_{t}  \overline{\phi^{k,\epsilon}}   \big\| + \| \bm{u} \|_{L^{3}(\Omega)} \big\| \nabla \phi^{k,\epsilon} \big\|  \big),
    \label{muh1estimate} \\
    \| \nabla \phi^{k,\epsilon} \|_{L^{6}(\Omega)}
    & \le C \big\|\nabla \phi^{k,\epsilon}  \big\|_{H^1(\Omega)}
     \le C \big( 1 + \big\|  \partial_{t} \phi^{k,\epsilon} - \partial_{t}  \overline{\phi^{k,\epsilon}}   \big\| + \big\| \nabla \mu_\phi^{k,\epsilon} \big\|  \big) \nonumber \\
    & \le C \big( 1 + \big\|  \partial_{t} \phi^{k,\epsilon} - \partial_{t}  \overline{\phi^{k,\epsilon}}   \big\| + \| \bm{u} \|_{L^{3}(\Omega)} \big),
    \label{phih2estimate} \\
    \|  \Delta \mu_\phi^{k,\epsilon} \|
    & \le C \big( 1+\big\| \nabla m_\phi ( \widetilde{\phi} ) \big\|^2_{L^{6}(\Omega)}  \big) \big(  \big\|  \partial_{t} \phi^{k,\epsilon} - \partial_{t}  \overline{\phi^{k,\epsilon}}   \big\| + \| \bm{u} \|_{L^{3}(\Omega)} \big\| \nabla \phi^{k,\epsilon} \big\|_{L^{6}(\Omega)}  \big) \nonumber \\
    & \le C \big( 1+\big\| \nabla m_\phi ( \widetilde{\phi} ) \big\|^2_{L^{6}(\Omega)}  \big) \big(  1 + \| \bm{u} \|_{L^{3}(\Omega)}  \big) \big\|  \partial_{t} \phi^{k,\epsilon} - \partial_{t}  \overline{\phi^{k,\epsilon}}   \big\| \nonumber \\
    & \quad + C \big( 1+\big\| \nabla m_\phi ( \widetilde{\phi} ) \big\|^2_{L^{6}(\Omega)}  \big) \| \bm{u} \|_{L^{3}(\Omega)} ( 1 + \| \bm{u} \|_{L^{3}(\Omega)} ) . \label{muh2estimate}
\end{align}
Similar estimates hold for $\mu_\psi^{k,\epsilon}$.

Exploiting a time-difference argument similar to that in \cite{G2022}, we can obtain higher-order Sobolev estimates of $( \phi^{k,\epsilon},\psi^{k,\epsilon},\mu_\phi^{k,\epsilon},\mu_\psi^{k,\epsilon} )$.
Since
$\partial^+_{t,h} \phi^{k,\epsilon} \to \partial_{t} \phi^{k,\epsilon}$ and $\partial^+_{t,h} \psi^{k,\epsilon} \to \partial_{t} \psi^{k,\epsilon}$ in $L^{2} (0,T;H )$,
we infer that
\begin{align*}
    & \|  \partial_{t} \phi^{k,\epsilon}  \|_{L^{\infty} ( 0,T;H )} + \|  \partial_{t} \phi^{k,\epsilon}  \|_{L^{2} ( 0,T;V )}
    + \|  \partial_{t} \psi^{k,\epsilon}  \|_{L^{\infty} ( 0,T;H )} + \|  \partial_{t} \psi^{k,\epsilon}  \|_{L^{2} (0,T;V )} \\
    & \quad \le C \big(  \alpha, T, k, \epsilon, \| \phi_{0,k} \|_{W}, \| \psi_{0,k} \|_{W}, \| \bm{u} \|_{L^\infty(I;L^3(\Omega) )}, \| \partial_t \bm{u} \|_{L^2( I;H )}  \big).
\end{align*}
From equation \eqref{eq:apcvcho1schauder} and Lemma \ref{weightedlaplace}, we have
\[
\| \nabla \mu_\phi^{k,\epsilon} \| \le C \big(  \big\|  \partial_{t} \phi^{k,\epsilon} - \overline{\partial_{t} \phi^{k,\epsilon}}  \big\| + \| \bm{u} \|_{L^{3}(\Omega)} \big\|  \phi^{k,\epsilon} - \overline{\phi^{k,\epsilon}}  \big\|_{L^{6}(\Omega)}\big),
\]
which implies
\[
\|  \mu_\phi^{k,\epsilon}  \|_{L^{\infty} (0,T;V )} \le C \big(  \alpha, T, k, \epsilon, \| \phi_{0,k} \|_{W}, \| \psi_{0,k} \|_{W}, \| \bm{u} \|_{L^\infty(I;L^3(\Omega) )}, \| \partial_t \bm{u} \|_{L^2( I;H )}  \big).
\]
Therefore, it follows from \eqref{eq:apcvcho2schauder} and \cite[Theorem A.1]{GMT2019} that
\begin{align*}
& \|  \phi^{k,\epsilon}  \|_{L^{\infty} (0,T;W )}
 + \big\|\big( F_\phi^{\epsilon} \big)'\big( \phi^{k,\epsilon} \big)\big\|_{L^\infty(0,T;L^2(\Omega))}
\\
&\quad \le C \big(  \alpha, T, k, \epsilon, \| \phi_{0,k} \|_{W}, \| \psi_{0,k} \|_{W}, \| \bm{u} \|_{L^\infty(I;L^3(\Omega) )}, \| \partial_t \bm{u} \|_{L^2( I;H )}    \big).
\end{align*}
Using equation \eqref{eq:apcvcho1schauder} again, we obtain
\begin{align*}
    \| \Delta \mu_\phi^{k,\epsilon} \| \le C \big(  1 + \|  \nabla m_\phi (  \widetilde{\phi}  )  \|_{L^6(\Omega)}^2  \big) \big(  \big\|  \partial_{t} \phi^{k,\epsilon} - \overline{\partial_{t} \phi^{k,\epsilon}}  \big\| + \| \bm{u} \|_{L^{3}(\Omega)} \| \nabla \phi^{k,\epsilon}  \|_{L^{6}(\Omega)}  \big),
\end{align*}
which implies
\[
\|  \mu_\phi^{k,\epsilon}  \|_{L^{\infty} (0,T;W )} \le C \big(  \alpha, T, k, \epsilon, \| \phi_{0,k} \|_{W}, \| \psi_{0,k} \|_{W}, \| \bm{u} \|_{L^\infty(I;L^3(\Omega) )}, \| \partial_t \bm{u} \|_{L^2( I;H )}   \big).
\]
Similar arguments apply to $\psi^{k,\epsilon}$ and $\mu_\psi^{k,\epsilon}$. In summary, we have shown
\begin{align*}
    & \phi^{k,\epsilon}, \psi^{k,\epsilon} \in L^{\infty} ( 0,T;W ) \cap H^1 ( 0,T;V ) \cap W^{1,\infty} ( 0,T;H ), \\
    & \mu_\phi^{k,\epsilon},\mu_\psi^{k,\epsilon} \in L^{\infty} ( 0,T;W ),
\end{align*}
for all $T>0$, although the associated estimates may depend on $\epsilon,k$.

Next, we show that $ \mu_\phi^{k,\epsilon}, \mu_\psi^{k,\epsilon} \in H^1 ( 0,T;H )$. In view of equations \eqref{eq:apcvcho1schauder}--\eqref{eq:apcvcho2schauder}, we write
\begin{align*}
    \mu_\phi^{k,\epsilon} = \big(  I + \alpha \mathcal{A}_{\phi, \widetilde{\phi}}  \big)^{-1} \Big(  &- \alpha \sigma_1 \big(  \overline{\phi^{k,\epsilon}} -c  \big) - \alpha \bm{u} \cdot \nabla \phi^{k,\epsilon} - \Delta \phi^{k,\epsilon} \\
    &+ \sigma_2 \mathcal{N} \big(  \phi^{k,\epsilon} - \overline{\phi^{k,\epsilon} }  \big)+ \big(  F_\phi^{\epsilon}  \big)' ( \phi^{k,\epsilon} ) + \partial_\phi G (  \phi^{k,\epsilon},\psi^{k,\epsilon}  )  \Big).
\end{align*}
From $\phi^{k,\epsilon}, \psi^{k,\epsilon} \in L^\infty(0,T;W)\cap H^1(0,T;V)$ and $\bm{u} \in C^\infty_0 (I;C^\infty_{0,\sigma}(\Omega))$, we find that
\begin{align*}
    f_\phi & \define - \alpha \sigma_1 \big(  \overline{\phi^{k,\epsilon}} -c  \big) - \alpha \bm{u} \cdot \nabla \phi^{k,\epsilon} - \Delta \phi^{k,\epsilon} + \sigma_2 \mathcal{N} \big(  \phi^{k,\epsilon} - \overline{\phi^{k,\epsilon} }  \big)
    \\
    & \quad\ + \big(  F_\phi^{\epsilon}  \big)' ( \phi^{k,\epsilon} ) + \partial_\phi G (  \phi^{k,\epsilon},\psi^{k,\epsilon}  ) \in H^1 (0,T;V^*) \cap L^\infty(0,T;H).
\end{align*}
For any given $h \in C^\infty( \mathbb{R} ; C^\infty(\overline{\Omega}))$, it holds
\[
\frac{\mathrm{d}}{\mathrm{d} t} \mathcal{A}_{\phi,h} (f) = - \mathrm{div}\, (m_\phi'(h) \partial_t h \nabla f), \quad \forall f \in V.
\]
Thus, the map $\mathcal{A}_{\phi,h(\cdot)}:\mathbb{R} \to \mathcal{L}(V;V^*) \cap \mathcal{L}(W;H)$, $t \to \mathcal{A}_{\phi,h(t)}$, belongs to $C^1(\mathbb{R}; \mathcal{L}(V;V^*) \cap \mathcal{L}(W;H))$.
Therefore, by the implicit function theorem (see, e.g., \cite[Theorem 7.13-2]{C2013}), we have
\begin{align*}
    \frac{\mathrm{d}}{\mathrm{d} t} \big( (I + \alpha \mathcal{A}_{\phi,h})^{-1} \widetilde{f}_1 \big) =
    -\alpha (I + \alpha \mathcal{A}_{\phi,h})^{-1} \frac{\mathrm{d}}{\mathrm{d} t} \mathcal{A}_{\phi,h} (I + \alpha \mathcal{A}_{\phi,h})^{-1} \widetilde{f}_1
    + (I + \alpha \mathcal{A}_{\phi,h} )^{-1} \partial_t \widetilde{f}_1,
\end{align*}
for any $\widetilde{f}_1 \in C^\infty( \mathbb{R} ; C^\infty(\overline{\Omega}))$.
Taking a convergent sequence $\{ h_i \}_{i=1}^\infty$ such that
\begin{align*}
    & h_i \stackrel{*}{\rightharpoonup} \widetilde{\phi} \ \text{in $L^\infty(0,T;W)$}, \quad \partial_t h_i \stackrel{*}{\rightharpoonup} \partial_t \widetilde{\phi} \ \text{in $L^\infty(0,T;H)$}, \quad  h_i \to \widetilde{\phi} \ \text{in $H^1(0,T;H)$}.
\end{align*}
The Aubin--Lions--Simon lemma further implies
\[
h_i \to \widetilde{\phi} \ \text{in $C([0,T];W^{1,4}(\Omega))$}.
\]
Then we can verify that
\[
(I + \mathcal{A}_{\phi,h_i})^{-1} \to (I + \alpha\mathcal{A}_{\phi,\widetilde{\phi}})^{-1} \ \text{in $C([0,T];\mathcal{L}(V^*;V) \cap \mathcal{L}(H;W) )$}.
\]
Hence, for any $\widetilde{f}_2\in L^2(0,T; H)$, it follows that
\begin{align*}
    & \Big((I + \alpha\mathcal{A}_{\phi,h_i})^{-1} \frac{\mathrm{d}}{\mathrm{d} t} \mathcal{A}_{\phi,h} (I + \alpha\mathcal{A}_{\phi,h_i})^{-1} \widetilde{f}_1,\widetilde{f}_2\Big) \\
    & \quad = \int_\Omega m_\phi'(h_i) \partial_t h_i \nabla (I + \alpha\mathcal{A}_{\phi,h_i})^{-1} \widetilde{f}_1 \cdot \nabla (I + \alpha\mathcal{A}_{\phi,h_i})^{-1} \widetilde{f}_2 \, \mathrm{d} x \\
    & \quad \to \int_\Omega m_\phi'(\widetilde{\phi}) \partial_t \widetilde{\phi} \nabla (I + \alpha\mathcal{A}_{\phi,\widetilde{\phi}})^{-1} \widetilde{f}_1 \cdot \nabla (I + \alpha\mathcal{A}_{\phi,\widetilde{\phi}})^{-1} \widetilde{f}_2 \, \mathrm{d} x,
\end{align*}
which means
\[
(I + \alpha\mathcal{A}_{\phi,h_i})^{-1} \frac{\mathrm{d}}{\mathrm{d} t} \mathcal{A}_{\phi,h_i} (I + \alpha\mathcal{A}_{\phi,h_i})^{-1} \widetilde{f}_1 \rightharpoonup (I + \alpha\mathcal{A}_{\phi,\widetilde{\phi}})^{-1} \frac{\mathrm{d}}{\mathrm{d} t} \mathcal{A}_{\phi,\widetilde{\phi}} (I + \alpha\mathcal{A}_{\phi,\widetilde{\phi}})^{-1} \widetilde{f}_1
\]
in $L^2(0,T;H)$. This combined with $(I + \alpha \mathcal{A}_{\phi,h_i})^{-1} \partial_t \widetilde{f}_1 \to (I + \alpha \mathcal{A}_{\phi,\widetilde{\phi}})^{-1} \partial_t \widetilde{f}_1$ in $L^2(0,T;V)$ yields that
\[
(I +\alpha  \mathcal{A}_{\phi,h_i})^{-1} \widetilde{f}_1 \rightharpoonup (I + \alpha \mathcal{A}_{\phi,\widetilde{\phi}})^{-1} \widetilde{f}_1\quad \text{in $H^1(0,T;H)$}.
\]
Moreover, from
\begin{align*}
    & \Big|  \int_\Omega m_\phi'(\widetilde{\phi}) \partial_t \widetilde{\phi} \nabla (I + \alpha \mathcal{A}_{\phi,\widetilde{\phi}})^{-1} \widetilde{f}_1 \cdot \nabla (I + \alpha \mathcal{A}_{\phi,\widetilde{\phi}})^{-1} \widetilde{f}_2 \, \mathrm{d} x  \Big| \\
    & \quad \le \| m_\phi \|_{C^1(\mathbb{R})} \|  \partial_t \widetilde{\phi}  \|_{L^\infty(0,T;H)} \| \nabla (I + \alpha \mathcal{A}_{\phi,\widetilde{\phi}})^{-1} \widetilde{f}_1 \|_{L^2(0,T;L^6(\Omega))} \| \nabla (I + \alpha  \mathcal{A}_{\phi,\widetilde{\phi}})^{-1} \widetilde{f}_2 \|_{L^2(0,T;L^3(\Omega))} \\
    & \quad \le C_{m_\phi} \|  \partial_t \widetilde{\phi}  \|_{L^\infty(0,T;H)} (1 + \NormNormal{\nabla \widetilde{\phi}}_{L^\infty(0,T;L^6(\Omega))})^3 \| \widetilde{f}_1 \|_{L^2(0,T;H)} \| \widetilde{f}_2 \|_{L^2(0,T;H)},
\end{align*}
we can deduce that
\[
\Big\|  \frac{\mathrm{d}}{\mathrm{d} t} \big( (I + \alpha \mathcal{A}_{\phi,\widetilde{\phi}})^{-1} \widetilde{f}_1 \big)  \Big\|_{L^2(0,T;H)} \le C_{m_\phi} \| \partial_t \widetilde{\phi} \|_{L^\infty(0,T;H)} (1 + \NormNormal{\widetilde{\phi}}_{L^\infty(0,T;W)})^3 \| \widetilde{f}_1 \|_{L^2(0,T;H)\cap H^1(0,T;V^*)}.
\]
Taking a sequence $\{ \widetilde{f}_i \}_{i=1}^\infty$ such that
\begin{align*}
    & \widetilde{f}_i \stackrel{*}{\rightharpoonup} f_\phi \ \text{in $L^\infty(0,T;H)$}, \quad  \widetilde{f}_i \to f_\phi \ \text{in}\  H^1(0,T;V^*)\cap L^2(0,T;L^2(\Omega)),
\end{align*}
we can repeat the above approximate procedure to obtain
\[
(I + \alpha \mathcal{A}_{\phi,\widetilde{\phi}})^{-1} \widetilde{f}_i \rightharpoonup (I + \alpha \mathcal{A}_{\phi,\widetilde{\phi}})^{-1} f_\phi \quad  \text{in $H^1(0,T;H)$}.
\]
This yields $ \mu_\phi^{k,\epsilon} \in H^1( 0,T;H)$. Similarly, we have $ \mu_\psi^{k,\epsilon} \in H^1( 0,T;H)$.

Differentiating \eqref{eq:apcvcho2schauder} with respect to time gives
\begin{align*}
    \alpha \partial_{t} ^2\phi^{k,\epsilon} & = \partial_{t} \mu_\phi^{k,\epsilon}  + \Delta \partial_{t} \phi^{k,\epsilon} - \big(  F_\phi^{\epsilon}  \big)'' ( \phi^{k,\epsilon} )\partial_{t} \phi^{k,\epsilon} - \partial^2_{\phi,\psi} G ( \phi^{k,\epsilon},\psi^{k,\epsilon} ) \partial_{t} \psi^{k,\epsilon} \\
    & \quad - \partial_\phi ^2 G ( \phi^{k,\epsilon},\psi^{k,\epsilon} ) \partial_{t} \phi^{k,\epsilon} - \sigma_2 \mathcal{N} \big(  \partial_{t} \phi^{k,\epsilon} - \partial_{t} \overline{\phi^{k,\epsilon}} \big).
\end{align*}
Since all terms on the right-hand side are well defined, we find $\partial_{t} ^2\phi^{k,\epsilon} \in L^{2} ( 0,T;V^* )$. Similarly, it holds $\partial_{t} ^2\psi^{k,\epsilon} \in L^{2} ( 0,T;V^* )$.

Next, differentiating \eqref{eq:apcvcho1schauder} and \eqref{eq:apcvcho3schauder} with respect to time and testing the resultants with $g_\phi = \mathcal{N}_{\phi, \widetilde{\phi}} \big(  \partial_{t} \phi^{k,\epsilon} - \partial_{t}  \overline{\phi^{k,\epsilon}}   \big)$ and $g_\psi = \mathcal{N}_{\psi, \widetilde{\psi}} \big(  \partial_{t} \psi^{k,\epsilon}   \big)$, respectively,  we obtain
\begin{align}
    & \frac{1}{2} \frac{\mathrm{d}}{\mathrm{d t}}  \Big(  \Big\|  \sqrt{m_\phi( \widetilde{\phi} )} \nabla g_\phi  \Big\|^2 +  \Big\|  \sqrt{m_\psi( \widetilde{\psi} )} \nabla g_\psi  \Big\|^2 + \alpha \big\|  \partial_{t} \phi^{k,\epsilon}  - \partial_{t} \overline{\phi^{k,\epsilon}}   \big\|^2 + \alpha \big\|  \partial_{t} \psi^{k,\epsilon} \big\|^2  \Big)  \notag
    \\
    & \qquad + \big\|  \nabla \partial_{t}  \phi^{k,\epsilon}  \big\|^2 + \beta \big\|  \nabla \partial_{t} \psi^{k,\epsilon}   \big\|^2 + \sigma_2 \big\|  \partial_{t} \phi^{k,\epsilon}  - \partial_{t} \overline{\phi^{k,\epsilon}}   \big\|^2_* \notag
    \\
    & \quad = \int_{\Omega}  \bm{u} \cdot \big(  \big(  \partial_{t} \phi^{k,\epsilon}  - \partial_{t} \overline{\phi^{k,\epsilon}}   \big) \nabla g_\phi +  \partial_{t} \psi^{k,\epsilon} \nabla g_\psi  \big)  \, \mathrm{d}x
    - \int_{\Omega}   \partial_{t} \bm{u}   \cdot \big(  g_\phi \nabla \phi^{k,\epsilon} + g_\psi \nabla \psi^{k,\epsilon}  \big)  \, \mathrm{d}x \notag
    \\
    & \qquad - \frac{1}{2} \int_{\Omega} \big( m_\phi'( \widetilde{\phi} ) \partial_{t} \widetilde{\phi} | \nabla g_\phi |^2 + m_\psi'( \widetilde{\psi} ) \partial_{t} \widetilde{\psi} | \nabla g_\psi |^2 \big) \, \mathrm{d}x \notag
    \\
    & \qquad - \int_{\Omega} \big( m_\phi'( \widetilde{\phi} ) \partial_{t} \widetilde{\phi} \nabla \mu_\phi^{k,\epsilon} \cdot \nabla g_\phi + m_\psi'( \widetilde{\psi} ) \partial_{t} \widetilde{\psi}  \nabla \mu_\psi^{k,\epsilon} \cdot \nabla g_\psi \big) \, \mathrm{d}x \notag
    \\
    & \qquad - \big(  \partial_{t}  \big(  F_\phi^{\epsilon} \big)'\big( \phi^{k,\epsilon} \big) + \partial_{t}  \partial_\phi G(\phi^{k,\epsilon},\psi^{k,\epsilon}) ,  \partial_{t}  \phi^{k,\epsilon} - \partial_{t}  \overline{\phi^{k,\epsilon}}  \big) \notag
    \\
    & \qquad - \big(  \partial_{t}  \big(  F_\psi^{\epsilon} \big)'\big( \psi^{k,\epsilon} \big) + \partial_{t}  \partial_\psi G(\phi^{k,\epsilon},\psi^{k,\epsilon}) ,  \partial_{t}  \psi^{k,\epsilon}  \big) \notag
    \\
    & \quad \le \| \bm{u} \|_{L^{3}(\Omega)} \big(  \big\|  \partial_{t}  \phi^{k,\epsilon} - \partial_{t}  \overline{\phi^{k,\epsilon}}  \big\| \|  \nabla g_\phi  \|_{L^{6}(\Omega)} + \big\|  \partial_{t}  \psi^{k,\epsilon}  \big\| \|  \nabla g_\psi  \|_{L^{6}(\Omega)}  \big) \notag
    \\
    & \qquad + \|  \partial_{t}  \bm{u}  \| \big(  \| g_\phi \|_{L^{\infty}(\Omega)} \| \nabla \phi^{k,\epsilon} \| + \| g_\psi \|_{L^{\infty}(\Omega)} \| \nabla \psi^{k,\epsilon} \|  \big) \notag
    \\
    & \qquad + C \|  \partial_{t} \widetilde{\phi}  \| \|  \nabla g_\phi  \|^2_{L^{4}(\Omega)} + C \|  \partial_{t} \widetilde{\psi}  \| \|  \nabla g_\psi  \|^2_{L^{4}(\Omega)} \notag
    \\
    & \qquad + C \| \partial_{t}  \widetilde{\phi} \| \| \nabla \mu_\phi^{k,\epsilon} \|_{L^{6}(\Omega)} \|  \nabla g_\phi  \|_{L^{3}(\Omega)} + C \| \partial_{t}  \widetilde{\psi} \| \| \nabla \mu_\psi^{k,\epsilon} \|_{L^{6}(\Omega)} \|  \nabla g_\psi  \|_{L^{3}(\Omega)} \notag
    \\
    & \qquad + \AbsNormal{\Omega} \partial_{t}  \overline{\big(  F_\phi^{\epsilon} \big)'\big( \phi^{k,\epsilon} \big)}  \partial_{t}  \overline{ \phi^{k,\epsilon} }
    + C \big(  \| \partial_{t}  \phi^{k,\epsilon} \| + \| \partial_{t}  \psi^{k,\epsilon} \|  \big) \Big(  \big\| \partial_{t}  \phi^{k,\epsilon} - \overline{\partial_{t}  \phi^{k,\epsilon}}  \big\| + \big\| \partial_{t}  \psi^{k,\epsilon} \big\|  \Big) \notag
    \\
    & \quad \le C H_1^{k,\epsilon} \Big(  \big\|  \partial_{t}  \phi^{k,\epsilon} - \partial_{t}  \overline{\phi^{k,\epsilon}}  \big\|^2 + \big\|  \partial_{t}  \psi^{k,\epsilon}  \big\|^2  \Big)
    + C H_2^{k,\epsilon}
    \notag \\
    & \qquad + \AbsNormal{\Omega} \partial_{t}  \Big( \overline{\big(  F_\phi^{\epsilon} \big)'\big( \phi^{k,\epsilon} \big)} \partial_{t}  \overline{ \phi^{k,\epsilon}}  \Big)
    - \AbsNormal{\Omega} \overline{\big(  F_\phi^{\epsilon} \big)'\big( \phi^{k,\epsilon} \big)} \partial_{t} ^2 \overline{ \phi^{k,\epsilon} },
    \label{D-k-e}
    \end{align}
where
\begin{align*}
     H_1^{k,\epsilon} &= 1+ (1+\|  \bm{u}  \|_{L^{3}(\Omega)}) \Big[  \big(  1 + \big\| \nabla m_\phi ( \widetilde{\phi} ) \big\|_{L^{6}(\Omega)}  \big)^2 + \big(  1 + \big\| \nabla m_\psi ( \widetilde{\psi} ) \big\|_{L^{6}(\Omega)}  \big)^2  \Big]
    \nonumber \\
    & \quad + \| \partial_t \widetilde{\phi} \| \big(  1 + \big\| \nabla m_\phi ( \widetilde{\phi} ) \big\|_{L^{6}(\Omega)}  \big)^3 \big(  1 + \|  \bm{u}  \|_{L^{3}(\Omega)}  \big)
    \nonumber \\
    & \quad + \| \partial_t \widetilde{\psi} \| \big(  1 + \big\| \nabla m_\psi ( \widetilde{\psi} ) \big\|_{L^{6}(\Omega)}  \big)^3  \big(  1 + \|  \bm{u}  \|_{L^{3}(\Omega)}  \big),
    \nonumber \\
    H_2^{k,\epsilon}
    & = \|  \partial_{t} \bm{u}  \|^2 +  \big(  1 + \big\| \nabla m_\phi ( \widetilde{\phi} ) \big\|_{L^{6}(\Omega)}  \big)^3 \| \partial_t \widetilde{\phi}  \| ( 1 + \| \bm{u}  \|_{L^{3}(\Omega)})^3
    \nonumber \\
    & \quad + \big(  1 + \big\| \nabla m_\psi ( \widetilde{\psi} ) \big\|_{L^{6}(\Omega)}  \big)^3 \|  \partial_t \widetilde{\psi}  \| (1 + \|  \bm{u}  \|_{L^{3}(\Omega)})^3 + \Big( \partial_t \overline{\phi^{k,\epsilon}} \Big)^2.
    \nonumber
\end{align*}
In particular, the constant $C>0$ in \eqref{D-k-e} is independent of the approximating parameter $\epsilon$.
Recall that there exists a constant $C_m>0$ depending on $m = \max \{ | \overline{\phi_{0,k}} |, | c | \}$ and $F_\phi$ but independent of $\epsilon$ (cf. \cite[Proposition 4.5]{M2019}) such that
\begin{align*}
    \Big| \overline{\big(  F_\phi^{\epsilon} \big)'\big( \phi^{k,\epsilon} \big)} \Big| & \le C_m \int_\Omega \big(  F_\phi^{\epsilon} \big)'\big( \phi^{k,\epsilon} \big) \big(  \phi^{k,\epsilon} - \overline{\phi^{k,\epsilon}}  \big) \, \mathrm{d} x + C_m \\
    & \le C_m \big(  \mu_\phi^{k,\epsilon} - \overline{\mu_\phi^{k,\epsilon}} + \alpha \partial_{t} \phi^{k,\epsilon} - \alpha \overline{\partial_{t} \phi^{k,\epsilon}} - \partial_\phi G (\phi^{k,\epsilon},\psi^{k,\epsilon}) ,  \phi^{k,\epsilon} - \overline{\phi^{k,\epsilon}}  \big) + C_m \\
    & \le C_m'\big(  1 + \|  \nabla \mu_\phi^{k,\epsilon}  \| + \alpha \big\|  \partial_{t} \phi^{k,\epsilon} - \overline{\partial_{t} \phi^{k,\epsilon}}  \big\|  \big).
\end{align*}
Besides, equation \eqref{eq:cvcho1} implies
\[
\| \nabla \mu_\phi^{k,\epsilon} \| \le C \big(  \big\|  \partial_{t} \phi^{k,\epsilon} - \overline{\partial_{t} \phi^{k,\epsilon}}  \big\| + \| \bm{u} \|_{L^{3}(\Omega)} \big\|  \phi^{k,\epsilon} - \overline{\phi^{k,\epsilon}}  \big\|_{L^{6}(\Omega)}  \big).
\]
Then by Young's inequality, we get
$$
\Big| \overline{\big(  F_\phi^{\epsilon} \big)'\big( \phi^{k,\epsilon} \big)} \Big| \le \frac{\alpha}{ 8 \AbsNormal{\Omega} ( 1+\sigma_1 ) } \big\|  \partial_{t} \phi^{k,\epsilon} - \overline{\partial_{t} \phi^{k,\epsilon}}  \big\|^2 + C_{\alpha} \big(  1 + \| \bm{u} \|_{L^{3}(\Omega)}  \big).
$$
Set
\begin{align*}
    E^{k,\epsilon} & = \Big\|  \sqrt{m_\phi( \widetilde{\phi} )} \nabla g_\phi  \Big\|^2 +  \Big\|  \sqrt{m_\psi( \widetilde{\psi} )} \nabla g_\psi  \Big\|^2 + \alpha \big\|  \partial_{t} \phi^{k,\epsilon}  - \partial_{t} \overline{\phi^{k,\epsilon}}   \big\|^2 + \alpha \big\|  \partial_{t} \psi^{k,\epsilon}   \big\|^2 \\
    & \quad - 2 \AbsNormal{\Omega} \overline{\big(  F_\phi^{\epsilon} \big)'\big( \phi^{k,\epsilon} \big)} \partial_{t}  \overline{ \phi^{k,\epsilon} }  + 4 \sigma_1 C_{\alpha} \big(  1 + \| \bm{u} \|_{L^{\infty} ( I;L^{3}(\Omega) )}  \big).
\end{align*}
Using the fact $\partial_{t}  \overline{\phi^{k,\epsilon}} (t) = - \sigma_1 \big(  \overline{\phi_{0,k}} - c  \big) e^{-\sigma_1 t}$, we observe that
\[
    E^{k,\epsilon} \ge \frac{\alpha}{2} \big\|  \partial_{t} \phi^{k,\epsilon}  - \partial_{t} \overline{\phi^{k,\epsilon}}   \big\|^2 + \alpha \big\|  \partial_{t} \psi^{k,\epsilon}   \big\|^2
\]
and
\begin{align*}
    \Big| \overline{\big(  F_\phi^{\epsilon} \big)'\big( \phi^{k,\epsilon}(t) \big)} \partial_{t} ^2 \overline{ \phi^{k,\epsilon} } (t) \Big| & \le \frac{\sigma_1^2 \alpha}{4\AbsNormal{\Omega}(1+\sigma_1)} \big\|  \partial_{t} \phi^{k,\epsilon}(t) - \overline{\partial_{t} \phi^{k,\epsilon}} (t)  \big\|^2 + 2 C_\alpha \big(  1 + \| \bm{u} \|_{L^{3}(\Omega)}  \big) \sigma_1 e^{- \sigma_1 t}
    \\
    & \le \frac{\sigma_1}{2\AbsNormal{\Omega}} E^{k,\epsilon}(t) + 2 C_\alpha \big(  1 + \| \bm{u} \|_{L^{3}(\Omega)}  \big) \sigma_1 e^{- \sigma_1 t}.
\end{align*}
Therefore, from \eqref{D-k-e} we infer that
\begin{align}
     & \frac{1}{2} \frac{\mathrm{d}}{\mathrm{d t}} E^{k,\epsilon} + \frac{1}{2} \big\|  \nabla \partial_{t}  \phi^{k,\epsilon}  \big\|^2 + \frac{\beta}{2} \big\|  \nabla \partial_{t} \psi^{k,\epsilon}   \big\|^2
     \notag \\
     & \quad \le C H_1^{k,\epsilon} E^{k,\epsilon} + C H_2^{k,\epsilon} + 2 C_{\alpha} \big(  1 + \| \bm{u} \|_{L^{\infty} ( I;L^{3}(\Omega) )} \big) \sigma_1 e^{-\sigma_1 t},
     \label{E-k-e}
\end{align}
where the positive constants $C,C_\alpha$ are independent of $\epsilon$. From equations \eqref{eq:apcvcho1schauder}--\eqref{eq:apcvcho2schauder}, we infer that
\begin{align*}
    \partial_{t} \phi^{k,\epsilon} & = -\bm{u} \cdot \nabla \phi^{k,\epsilon} - \sigma_1 \big( \overline{\phi^{k,\epsilon}}-c \big) \\
    & \quad - \frac{ I - \big( I + \alpha \mathcal{A}_{\phi, \widetilde{\phi}} \big)^{-1} }{\alpha}  \Big( - \alpha \sigma_1 \big(  \overline{\phi^{k,\epsilon}} -c  \big) - \alpha \bm{u} \cdot \nabla \phi^{k,\epsilon} - \Delta \phi^{k,\epsilon} \\
    & \qquad \qquad \qquad \qquad \qquad \qquad + \sigma_2 \mathcal{N} \big(  \phi^{k,\epsilon} - \overline{\phi^{k,\epsilon} }  \big)+ \big(  F_\phi^{\epsilon}  \big)' ( \phi^{k,\epsilon} ) + \partial_\phi G (  \phi^{k,\epsilon},\psi^{k,\epsilon}  )  \Big),
\end{align*}
which enables us to define $\partial_{t} \phi^{k,\epsilon}(0)$ in $H$ by continuity. The same argument works for $\partial_{t} \psi^{k,\epsilon}(0)$. From \eqref{E-k-e} and Gronwall's lemma, we deduce that
\begin{align}
\big\|  \partial_{t} \phi^{k,\epsilon}  (t) \big\|^2 + \big\|  \partial_{t} \psi^{k,\epsilon}  (t) \big\|^2 + \int_0^t \big( \big\|  \nabla \partial_{t}  \phi^{k,\epsilon} (\tau) \big\|^2 +  \big\|  \nabla \partial_{t} \psi^{k,\epsilon} (\tau)  \big\|^2 \big)\,\mathrm{d} \tau\leq C,
\quad \forall\, t\in [0,T],
\notag
\end{align}
for all $T>0$, where $C>0$ is independent of $\epsilon$.
\smallskip
\\
\textbf{Step 2. The limit auxiliary system as $\epsilon\to 0$.}
With the aforementioned uniform estimates, we can pass to the limit as $\epsilon \to 0$ to obtain a strong solution $( \phi^k,\psi^k  )$ to the following system:
\begin{numcases}{}
    \partial_{t} \phi  + \bm{u} \cdot \nabla \phi  + \sigma_1 \big( \overline{\phi } -c \big) = \mathrm{div}\, \big(m_\phi \big( \widetilde{\phi} \big) \nabla \mu_\phi \big)
    \label{eq:cvcho1schauder}, \\
    \mu_\phi  = \alpha \partial_{t} \phi  -\Delta \phi  + \sigma_2 \mathcal{N} \big( \phi  - \overline{\phi } \big) + F_\phi'(\phi ) + \partial_\phi G(\phi ,\psi ) \label{eq:cvcho2schauder},\\
    \partial_{t} \psi  + \bm{u} \cdot \nabla \psi  = \mathrm{div}\, \big(m_\psi \big( \widetilde{\psi} \big) \nabla \mu_\psi \big)
    \label{eq:cvcho3schauder},\\
    \mu_\psi  = \alpha \partial_{t} \psi  -\beta \Delta \psi  + F_\psi'(\psi ) + \partial_\psi G(\phi ,\psi ) \label{eq:cvcho4schauder},
\end{numcases}
in $Q_T$, subject to the boundary and initial conditions:
\begin{numcases}{}
    \partial_{\mathbf{n}} \phi  = \partial_{\mathbf{n}} \mu_\phi  = \partial_{\mathbf{n}} \psi  = \partial_{\mathbf{n}} \mu_\psi  = 0, \label{eq:cvchobschauder} & on $S_T$,\\
    \phi |_{t=0} = \phi_{0,k}(x), \ \psi |_{t=0} = \psi_{0,k}(x), \label{eq:cvchoischauder} & in $\Omega$.
\end{numcases}
The uniqueness easily follows since the solution $( \phi^k,\psi^k  )$ is sufficiently regular.

To justify subsequent calculations, we prove the strict separation property of $(\phi^k,\psi^k)$. Arguing as in \cite{MZ2004} (see also \cite{G2022,O2025}), let us  rewrite equation \eqref{eq:cvcho2schauder} in the following form:
\[
\alpha \partial_{t} \phi^k - \Delta \phi^k + F_\phi' ( \phi^k ) = \widetilde{h},
\]
where $\widetilde{h} \define  \mu_\phi^k - \partial_\phi G (  \phi^k,\psi^k  ) - \sigma_2 \mathcal{N} (\phi^k - \overline{\phi^k})$.
Since $\mu_\phi^k \in L^{\infty} \big( 0,T;W \big)$ for any $T>0$, by the Sobolev embedding theorem, we have $\widetilde{h} \in L^{\infty} \big( 0,T;L^{\infty}(\Omega) \big)$.
Define $h_{\pm}(t) \define  \pm \| \widetilde{h}(t) \|_{L^{\infty}(\Omega)}$ and consider the following ordinary differential equations:
\[
\alpha y_{\pm}' + F_\phi' ( y_{\pm} ) =  h_{\pm}, \quad  \ \ y_{\pm}|_{t=0} \define  \pm \|  \phi_{0,k} \|_{L^{\infty}(\Omega)}.
\]
Then, a comparison principle for second-order parabolic equations gives
\[
y_-(t) \le \phi^k(x,t) \le y_+(t), \quad \forall\, t \in [0,T], \  x \in \Omega.
\]
Since $y_{\pm} \in (-1,1)$, applying \cite[Proposition A.3]{MZ2004}, we obtain a constant
\[
\delta = \delta (  \|  \phi_{0,k} \|_{L^{\infty}(\Omega)}, \| \widetilde{h}(t) \|_{L^{\infty} ( 0,T;L^{\infty}(\Omega) )}  ) \in (0,1)
\]
such that
\[
| y_-(t) |,\ | y_+(t) | \le 1 - \delta , \quad  \forall\, t \in [0,T].
\]
Hence, $\phi^k$ is strictly separated from $\pm1$ on  $[0,T]$. An analogous result holds for $\psi^k$.

By an argument similar to that for $( \phi^{k,\epsilon},\psi^{k,\epsilon},\mu_\phi^{k,\epsilon},\mu_\psi^{k,\epsilon} )$, we can show that $\partial_{t} \mu_\phi^k, \partial_{t} \mu_\psi^k \in L^{2} \big( 0,T; H \big)$ and $\partial_{t}^2 \phi^k,\partial_{t}^2 \phi^k \in L^{2} \big( 0,T;V^* \big)$.
Setting $g_\phi = \mathcal{N}_{\phi, \widetilde{\phi}} \big(  \partial_{t} \phi^k - \partial_{t}  \overline{\phi^k}   \big)$ and $g_\psi = \mathcal{N}_{\psi, \widetilde{\psi}}  \partial_{t} \psi^k$, it follows from Agmon's inequality, Lemma \ref{weightedlaplace} and Young's inequality that
\begin{align*}
    & \left|\int_{\Omega}  g_\phi \big( \partial_{t} \bm{u}  \big) \cdot \nabla \phi^k  \, \mathrm{d}x\right|
    = \left|\int_{\Omega}  \phi^k \big( \partial_{t} \bm{u}  \big) \cdot \nabla g_\phi  \, \mathrm{d}x\right| \\
    & \quad \le \|\phi^k\|_{L^\infty(\Omega)} \| \partial_{t} \bm{u} \|_{L^{1}(\Omega)} \| \nabla g_\phi \|_{L^\infty(\Omega)} \\
    & \quad \le C \| \partial_{t} \bm{u} \|_{L^{1}(\Omega)} \| \nabla g_\phi \|_{V}^{\frac{1}{2}} \| \nabla g_\phi \|_{H^{2}(\Omega)}^{\frac{1}{2}} \\
    & \quad \le C \| \partial_{t} \bm{u} \|_{L^{1}(\Omega)} \big(  1 + \|  \nabla m_\phi ( \widetilde{\phi} )  \|_{L^{6}(\Omega)}  \big)^3 \big(  1 + \|  \nabla^2 m_\phi ( \widetilde{\phi} )  \|  \big)^{\frac{1}{2}} \big\|  \partial_{t} \phi^k - \partial_{t} \overline{\phi^k}  \big\|^{\frac{1}{2}} \big\|  \nabla \partial_{t} \phi^k \big\|^{\frac{1}{2}} \\
    & \quad \le \frac{1}{2} \big\|  \nabla \partial_{t} \phi^k \big\|^2 + C \| \partial_{t} \bm{u} \|_{L^{1}(\Omega)}^{\frac{4}{3}} \big(  1 + \|  \nabla m_\phi ( \widetilde{\phi} )  \|_{L^{6}(\Omega)}  \big)^4 \big(  1 + \|  \nabla^2 m_\phi ( \widetilde{\phi} )  \|  \big)^{\frac{2}{3}} \big(  1+ \big\|  \partial_{t} \phi^k - \partial_{t} \overline{\phi^k}  \big\|^2  \big).
\end{align*}
The term $\int_{\Omega}  g_\psi  (\partial_{t} \bm{u})  \cdot \nabla \psi^k  \, \mathrm{d}x$ can be treated similarly. Keeping the above estimates in mind, in analogy to \eqref{E-k-e} for $( \phi^{k,\epsilon}, \psi^{k,\epsilon}, \mu_\phi^{k,\epsilon},\mu_\psi^{k,\epsilon} )$, we can deduce that
\begin{equation}\label{approximate_solution_higher_order_sobolev_estimate_differential}
    \frac{\mathrm{d}}{\mathrm{d t}} E^k + \frac{1}{2} \big\|  \nabla \partial_{t}  \phi^k  \big\|^2 + \frac{\beta}{2} \big\|  \nabla \partial_{t} \psi^k   \big\|^2 \le C_{*} H_1^k E^k + C_{*} H_2^k,
\end{equation}
where
\begin{align*}
    E^k & \define  \Big\|  \sqrt{m_\phi( \widetilde{\phi} )} \nabla g_\phi  \Big\|^2 +  \Big\|  \sqrt{m_\psi( \widetilde{\psi} )} \nabla g_\psi  \Big\|^2 + \alpha \big\|  \partial_{t} \phi^k  - \partial_{t} \overline{\phi^k}   \big\|^2 + \alpha \big\|  \partial_{t} \psi^k  - \partial_{t} \overline{\psi^k}   \big\|^2
    \\
    & \quad - 2 \AbsNormal{\Omega} \overline{ F_\phi'\big( \phi^k \big)} \partial_{t}  \overline{ \phi^k }  + 4 \sigma_1 C_{\alpha} \big(  1 + \| \bm{u} \|_{L^{\infty} ( I;L^{3}(\Omega) )  } \big),
    \\
    H_1^k & \define  1 + \| \bm{u} \|_{L^{3}(\Omega)} \big(  \big(  1 + \|  \nabla m_\phi ( \widetilde{\phi} )  \|_{L^{6}(\Omega)}  \big)^2 + \big(  1 + \|  \nabla m_\psi ( \widetilde{\psi} )  \|_{L^{6}(\Omega)}  \big)^2  \big)
    \\
    & \quad + \| \partial_{t} \bm{u} \|_{L^{1}(\Omega)}^{\frac{4}{3}} \big(  1 + \|  \nabla m_\phi ( \widetilde{\phi} )  \|_{L^{6}(\Omega)}  \big)^4 \big(  1 + \|  \nabla^2 m_\phi ( \widetilde{\phi} )  \|  \big)^{\frac{2}{3}}
    \\
    & \quad + \| \partial_{t} \bm{u} \|_{L^{1}(\Omega)}^{\frac{4}{3}} \big(  1 + \|  \nabla m_\psi ( \widetilde{\psi} )  \|_{L^{6}(\Omega)}  \big)^4 \big(  1 + \|  \nabla^2 m_\psi ( \widetilde{\psi} )  \|  \big)^{\frac{2}{3}}
    \\
    & \quad + \| \partial_{t} \widetilde{\phi} \| \big(  1 + \|  \nabla m_\phi ( \widetilde{\phi} )  \|_{L^{6}(\Omega)}  \big)^4 + \| \partial_{t} \widetilde{\psi} \| \big(  1 + \|  \nabla m_\psi ( \widetilde{\psi} )  \|_{L^{6}(\Omega)}  \big)^4
    \\
    & \quad + \| \partial_{t} \widetilde{\phi} \| \big(  1+ \| \bm{u} \|_{L^{3}(\Omega)}  \big) \big(  1 + \|  \nabla m_\phi ( \widetilde{\phi} )  \|_{L^{6}(\Omega)}  \big)^3
    \\
    & \quad + \| \partial_{t} \widetilde{\psi} \| \big(  1+ \| \bm{u} \|_{L^{3}(\Omega)}  \big) \big(  1 + \|  \nabla m_\psi ( \widetilde{\psi} )  \|_{L^{6}(\Omega)}  \big)^3,
    \\
    H_2^k & \define  \| \partial_{t} \bm{u} \|_{L^{1}(\Omega)}^{\frac{4}{3}} \big(  1 + \|  \nabla m_\phi ( \widetilde{\phi} )  \|_{L^{6}(\Omega)}  \big)^4 \big(  1 + \|  \nabla^2 m_\phi ( \widetilde{\phi} )  \|  \big)^{\frac{2}{3}}
    \\
    & \quad + \| \partial_{t} \bm{u} \|_{L^{1}(\Omega)}^{\frac{4}{3}} \big(  1 + \|  \nabla m_\psi ( \widetilde{\psi} )  \|_{L^{6}(\Omega)}  \big)^4 \big(  1 + \|  \nabla^2 m_\psi ( \widetilde{\psi} )  \|  \big)^{\frac{2}{3}}
    \\
    & \quad + \| \partial_{t} \widetilde{\phi} \| (1 + \| \bm{u} \|_{L^{3}(\Omega)})^3 \big(  1 + \|  \nabla m_\phi ( \widetilde{\phi} )  \|_{L^{6}(\Omega)}  \big)^3
    \\
    & \quad + \| \partial_{t} \widetilde{\psi} \| (1 + \| \bm{u} \|_{L^{3}(\Omega)})^3 \big(  1 + \|  \nabla m_\psi ( \widetilde{\psi} )  \|_{L^{6}(\Omega)}  \big)^3
    \\
    & \quad + \big(  1 + \| \bm{u} \|_{L^{3}(\Omega)} \big) \sigma_1 e^{-\sigma_1 t} + \big| \overline{\phi^k} \big|^2.
\end{align*}
Here, the positive constant $C_{*}$ depends on $\widetilde{E}_{\text{free}}(\phi_0,\psi_0)$, $\alpha$, $\overline{\phi_0}$, $\overline{\psi_0}$, $c$, and other parameters of the system \eqref{eq:cvcho1}--\eqref{eq:cvcho4} but is independent of $k$.
Moreover, it holds
\begin{align*}
    E^k & \le \frac{1}{m_{\phi,*}} \| \partial_{t} \phi^k  - \partial_{t} \overline{\phi^k} \|_*^2 + \frac{1}{m_{\psi,*}} \| \partial_{t} \psi^k \|_*^2 + \frac{3}{2} \alpha \| \partial_{t} \phi^k  - \partial_{t} \overline{\phi^k} \|^2 + \frac{3}{2} \alpha \| \partial_{t} \psi^k \|^2 \\
    & \quad + 8 \sigma_1 C_\alpha (1 + \| \bm{u} \|_{L^\infty(I;L^3(\Omega))}), \\
    E^k & \ge \frac{m_{\phi,*}}{(m_\phi^*)^2} \| \partial_{t} \phi^k  - \partial_{t} \overline{\phi^k} \|_*^2 + \frac{m_{\psi,*}}{(m_\psi^*)^2} \| \partial_{t} \psi^k \|_*^2 + \frac{1}{2} \alpha \| \partial_{t} \phi^k  - \partial_{t} \overline{\phi^k} \|^2 + \frac{1}{2} \alpha \| \partial_{t} \psi^k \|^2.
\end{align*}
From
\begin{align*}
    \mu_\phi^{k}(0) = \big(  I + \alpha \mathcal{A}_{\phi, \phi_0}  \big)^{-1} \Big(  &- \alpha \sigma_1 \big(  \overline{\phi^{k}} -c  \big) - \alpha \bm{u}(0) \cdot \nabla \phi_{0,k} - \Delta \phi_{0,k} \\
    &+ \sigma_2 \mathcal{N} \big(  \phi_{0,k} - \overline{\phi_{0,k} }  \big)+   F_\phi' ( \phi_{0,k} ) + \partial_\phi G (  \phi_{0,k},\psi_{0,k}  )  \Big)
\end{align*}
we infer that
\begin{align*}
    \| \mu_\phi^k(0)  \| & \le C \big(  \| \bm{u}(0) \cdot \nabla \phi_{0,k} \|_* + \|  - \Delta \phi_{0,k} + F_\phi'(\phi_{0,k})  \| + 1  \big) \\
    & \le C \big(  1 + \|  - \Delta \phi_0 + F_\phi'(\phi_0)  \| + \| \bm{u}(0) \|  \big).
\end{align*}
A similar estimate holds for $\| \mu_\psi^k(0)  \|$. Consequently,
\begin{align*}
    E^k(0) & \le C \frac{1 + \alpha}{ \alpha } \big(  1 + \| -\Delta \phi_{0,k}+F'_\phi(\phi_{0,k}) \| + \| -\beta \Delta \psi_{0,k} + F'_\psi(\psi_{0,k}) \| + \| \mu_\phi^k(0) \| + \|  \mu_\psi^k(0) \|  \big)^2 \\
    & \quad + 8 \sigma_1 C_\alpha (1 + \| \bm{u} \|_{L^\infty(I;L^3(\Omega))}) \\
    & \le C_\alpha' \big( 1 + \| \bm{u}(0) \|^2 + \|  - \Delta \phi_{0} + F_\phi'(\phi_{0})  \|^2 + \|  - \beta \Delta \psi_{0} + F_\phi'(\phi_{0})  \|^2  \big) \\
    & \quad + 8 \sigma_1 C_\alpha (1 + \| \bm{u} \|_{L^\infty(I;L^3(\Omega))}) \define E_0.
\end{align*}
Here, the positive constants $C_\alpha$ and $C_\alpha'$ depend on $\overline{\phi_0}$, $\overline{\psi_0}$, $\alpha$, and coefficients of the system \eqref{eq:cvcho1}--\eqref{eq:cvcho4}, but they are independent of $k$ (for $k$ sufficiently large).

Set
\[
h_{\bm{u}}(t) \define  \int_{0}^{t}  \| \partial_{t} \bm{u} (\tau)\|_{L^{1}(\Omega)}^{\frac{4}{3}} \, \mathrm{d}\tau, \quad t \in I, \quad u_{L_3} = \| \bm{u} \|_{L^{\infty} (I;L^{3}(\Omega) )}.
\]
In addition, let $a, b$ be two positive constants such that
\begin{align*}
    & \| \widetilde{\phi} \|_{L^\infty(I;W)} \le a, \quad \| \widetilde{\psi} \|_{L^\infty(I;W)} \le a, \quad  \| \partial_t \widetilde{\phi} \|_{L^\infty(I;H)} \le b, \quad \| \partial_t \widetilde{\psi} \|_{L^\infty(I;H)} \le b.
\end{align*}
Then it follows from \eqref{approximate_solution_higher_order_sobolev_estimate_differential} and Gronwall's lemma that
\begin{align}
    & \frac{\alpha}{2}\big\|  \partial_{t} \phi^k(t)  - \partial_{t} \overline{\phi^k}(t)   \big\|^2 + \frac{\alpha}{2} \big\| \partial_{t} \psi^k(t) \big\|^2  \le E^k(t)
    \nonumber \\
    & \quad \le E_0 e^{C_{*} \int_{0}^{t} H_1^k(\tau)\, \mathrm{d}\tau}
    \notag \\
    &\qquad + C_{H_2}C_{*} \big[ ( 1+a )^{\frac{14}{3}} h_{\bm{u}}(t) + b (1+u_{L_3})^3 ( 1+a )^3 t + ( 1+u_{L_3} ) ( 1 - e^{-\sigma_1 t} ) \big] e^{C_{*} \int_{0}^{t} H_1^k(\tau)\, \mathrm{d}\tau}
    , \label{estimatephipsi}
\end{align}
for all $t \in [0,T]$, where
\[
\int_{0}^{t} H_1^k (\tau)\,\mathrm{d}\tau \le C_{H_1} ( 1+a )^{\frac{14}{3}} h_{\bm{u}}(t) + \big[ 1 + C_{H_1} u_{L_3} ( 1+a )^2 + C_{H_1} b ( 1+a )^4 + C_{H_1} b ( 1+u_{L_3} ) ( 1+a )^3 \big]t,
\]
for all $t \in [0,T]$. We note that $C_{H_1}, C_{H_2}$ are two positive constants that only depend on $m_\phi$, $m_\psi$.
Thanks to \eqref{phih2estimate}--\eqref{muh2estimate}, we can conclude that the $L^{\infty} \big( I;W \big)$ norm of $\phi^k,\ \psi^k,\ \mu_\phi^k,\ \mu_\psi^k$ and the $L^{\infty} \big( 0,T;H \big)$ norm as well as the $L^{2} \big( I;V \big)$ norm of $\partial_{t} \phi^k,\ \partial_{t} \psi^k$ are uniformly bounded by $a$, $b$, $T$, $\widetilde{E}_{\mathrm{free}}( \phi_{0},\psi_{0} )$, $\| -\Delta \phi_{0}+F'_\phi(\phi_{0}) \|$, $\| -\beta \Delta \psi_{0} + F'_\psi(\psi_{0}) \|$, $\| \partial_{t} \bm{u} \|_{L^{\frac{4}{3}} (I;L^{1}(\Omega) )}$ and $\| \bm{u} \|_{L^{\infty} (I;L^{3}(\Omega) )}$.
\smallskip
\\
\textbf{Step 3. The limit auxiliary system as $k\to \infty$.} Based on the above estimates that are uniform with respect to the parameter $k$,  we can extract a convergence subsequence of $( \phi^k, \psi^k, \mu_\phi^k, \mu_\psi^k )$ such that, as $k\to \infty$, it holds
\begin{align*}
    & \phi^k \mathop{\rightharpoonup}\limits^* \phi, \ \psi^k \mathop{\rightharpoonup}\limits^* \psi \ \mathrm{in} \ L^{\infty} \big( 0,T;W \big) , \\
    & \partial_{t}  \phi^k \mathop{\rightharpoonup}\limits^* \partial_{t}  \phi, \ \partial_{t}  \psi^k \mathop{\rightharpoonup}\limits^* \partial_{t}  \psi \ \mathrm{in} \ L^{\infty} \big( 0,T;H \big) , \\
    & \partial_{t}  \phi^k \rightharpoonup \partial_{t}  \phi, \ \partial_{t}  \psi^k \rightharpoonup \partial_{t}  \psi \ \mathrm{in} \ L^{2} \big( 0,T;V \big) , \\
    & \mu_\phi^k \mathop{\rightharpoonup}\limits^* \mu_\phi, \ \mu_\psi^k \mathop{\rightharpoonup}\limits^* \mu_\psi \ \mathrm{in} \ L^{\infty} \big( 0,T;W \big) .
\end{align*}
The limit functions $(\phi,\psi,\mu_\phi,\mu_\psi)$ solve the following system:
\begin{numcases}{}
    \partial_{t} \phi + \bm{u} \cdot \nabla \phi + \sigma_1 \big( \overline{\phi} -c \big) = \mathrm{div}\, (m_\phi (\widetilde{\phi}) \nabla \mu_\phi) \label{eq:spapcvcho1schauder},\\
    \mu_\phi = \alpha \partial_{t} \phi -\Delta \phi + \sigma_2 \mathcal{N} \big( \phi - \overline{\phi} \big) + F_\phi'(\phi) + \partial_\phi G(\phi,\psi) \label{eq:spapcvcho2schauder} ,\\
    \partial_{t} \psi + \bm{u} \cdot \nabla \psi = \mathrm{div}\, (m_\psi (\widetilde{\psi}) \nabla \mu_\psi) \label{eq:spapcvcho3schauder},\\
    \mu_\psi = \alpha \partial_{t} \psi -\beta \Delta \psi + F_\psi'(\psi) + \partial_\psi G(\phi,\psi) \label{eq:spapcvcho4schauder},
\end{numcases}
almost everywhere in $Q_T$, subject to the following boundary and initial conditions:
\begin{numcases}{}
    \partial_{\mathbf{n}} \phi = \partial_{\mathbf{n}} \mu_\phi = \partial_{\mathbf{n}} \psi = \partial_{\mathbf{n}} \mu_\psi = 0 \label{eq:spapcvchobschauder} & a.e. on $S_T$,\\
    \phi|_{t=0} = \phi_0(x), \ \psi|_{t=0} = \psi_0(x) \label{eq:spapcvchoischauder} & in $\Omega$.
\end{numcases}
In addition, the solution $(\phi,\psi,\mu_\phi,\mu_\psi)$ is unique since it is sufficiently regular. By a density argument for the velocity field, we can check that the above solvability result and higher-order estimates also hold in the general case $\bm{u} \in L^{\infty} \big( I;\bm{L}^{3}_\sigma(\Omega) \big) \cap W^{1,\frac{4}{3}} \big(I;\bm{L}^{1}_\sigma(\Omega) \big)$. The details are omitted for brevity.
\smallskip
\\
\textbf{Step 4. The fixed-point argument.}
For any given $T>0$, we introduce
\begin{align*}
    & X_T \define  L^{\infty} \big( 0,T;W \big) \cap W^{1,\infty} \big( 0,T;H \big), \\
    & Y_T \define  C \big(  [0,T] ; C (  \overline{\Omega}  )  \big), \\
    & S_{T,a,b}^{\phi} \define  \Big\{  f \in X_T : \ \| f(t) \|_{W} \le a, \ \| \partial_{t} f \| \le b, \ \mathrm{a.e.} \ t\in [0,T], \ f(0)=\phi_{0}  \Big\}, \\
    & S_{T,a,b}^{\psi} \define  \Big\{  f \in X_T : \ \| f(t) \|_{W} \le a, \ \| \partial_{t} f \| \le b, \ \mathrm{a.e.} \ t\in [0,T], \ f(0)=\psi_{0}  \Big\},
\end{align*}
where $a, b$ are two positive constants that are sufficiently large and compatible with the initial data $(\phi_0, \psi_0)$.
By the Aubin--Lions--Simon lemma, we easily find
$X_T \hookrightarrow\hookrightarrow Y_T$.
For a given pair $( \widetilde{\phi},\widetilde{\psi} ) \in S_{T,a,b}^{\phi} \times S_{T,a,b}^{\psi}$, we denote by $(\phi, \psi)$ the unique solution to problem \eqref{eq:spapcvcho1schauder}--\eqref{eq:spapcvchoischauder} on $[0,T]$.
This allows us to define the map
\begin{align*}
  \Lambda:X_T \times X_T &\to Y_T \times Y_T\\
  \big(  \widetilde{\phi},\widetilde{\psi}  \big) &\mapsto \big( \phi,\psi \big).
\end{align*}
Then, finding the desired solution to problem \eqref{eq:cvcho1}--\eqref{eq:cvchoi} is equivalent to finding a fixed point of the map $\Lambda$, the existence of which is guaranteed by Schauder's fixed‑point theorem (see, e.g., \cite[Theorem 9.12-(a)]{C2013}).

Let $T_{a,b}\in (0,1]$ be a sufficiently small positive constant such that
\begin{align*}
    & C_{H_2}\big[ ( 1+a )^{\frac{14}{3}} h_{\bm{u}}(T_{a,b}) + b (1+u_{L_3})^3 ( 1+a )^3 T_{a,b} + ( 1+u_{L_3} ) ( 1 - e^{-\sigma_1 T_{a,b}} ) \big] \leq 1,
    \\
    & C_{H_1} ( 1+a )^{\frac{14}{3}} h_{\bm{u}}(T_{a,b}) + \big[ 1 + C_{H_1} u_{L_3} ( 1+a )^2 + C_{H_1} b ( 1+a )^4 + C_{H_1} b ( 1+u_{L_3} ) ( 1+a )^3 \big]T_{a,b}\leq 1.
\end{align*}
It follows from \eqref{estimatephipsi} that
\begin{align*}
    & \big\|  \partial_{t} \phi(t)  - \partial_{t} \overline{\phi}(t)   \big\|^2 + \big\|  \partial_{t} \psi(t)  \big\|^2 \le \frac{2}{\alpha} ( E_0+C_{*} ) e^{C_{*}},\quad \forall\, t\in [0,T_{a,b}].
\end{align*}
This estimate combined with an argument similar to that in \cite{G2022,GMT2019} gives
\begin{align*}
    & \| \phi(t) \|_W \le C\big(  1 + \|  \mu_\phi - \overline{\mu_\phi} -\alpha \big( \partial_{t} \phi- \overline{\partial_{t} \phi} \big)- \sigma_2 \mathcal{N} ( \phi - \overline{\phi} ) - \partial_\phi G(\phi,\psi) + \overline{\partial_\phi G(\phi,\psi)}  \|  \big)
    \\
    & \qquad \quad \ \ \le C \big(  1 +  \big\|  \partial_{t} \phi(t)  - \partial_{t} \overline{\phi}(t)   \big\|
    + \| \nabla \mu_\phi \| + \| G \|_{C^2(\mathbb{R}^2)}  \big) \\
    & \qquad \quad \ \ \le C \big( 1 + \big\|  \partial_{t} \phi(t)  - \partial_{t} \overline{\phi}(t)   \big\| + \| \bm{u}(t) \| \big) \\
    & \qquad \quad \ \ \le C_1 \Big(  \sqrt{\frac{2}{\alpha} ( E_0+C_{*} )} e^{\frac{C_{*}}{2}} +1 + u_{L_3} \Big), \\
    & \| \psi(t) \|_W \le C_1 \Big(  \sqrt{\frac{2}{\alpha} ( E_0+C_{*} )} e^{\frac{C_{*}}{2}} + 1+  u_{L_3} \Big),
\end{align*}
for almost all $ t \in [0,T_{a,b}]$, where the constant $C_{*}$ takes the same value as in \eqref{approximate_solution_higher_order_sobolev_estimate_differential}, and $C_1$ depends only on the initial data, $\Omega$ and coefficients of the system.  We choose
\begin{align*}
a_0 &= C_1 \Big(  \sqrt{\frac{2}{\alpha} ( E_0+C_{*} )} e^{\frac{C_{*}}{2}} +  1+u_{L_3} \Big)  + \|\phi_0\|_W + \|\psi_0\|_W,
\\
b_0 &= \sqrt{\frac{2}{\alpha} ( E_0+C_{*} )} e^{\frac{C_{*}}{2}} + \sqrt{2\sigma_1 |\Omega|},
\end{align*}
and then take $T_{a_0,b_0}\in (0,1]$ correspondingly.
It is easily seen that $S^\phi_{T_{a_0,b_0},a_0,b_0}$,  $S^\psi_{T_{a_0,b_0},a_0,b_0}$ are convex and closed sets. Both are compact in $Y_{T_{a_0,b_0}}$. Moreover, $\Lambda$ maps $S^\phi_{T_{a_0,b_0},a_0,b_0} \times S^\psi_{T_{a_0,b_0},a_0,b_0}$ into itself.

It remains to prove that $\Lambda$ is continuous from $S_{T_{a_0,b_0},a_0,b_0}^{\phi} \times S_{T_{a_0,b_0},a_0,b_0}^{\psi}$ to $Y_{T_{a_0,b_0}} \times Y_{T_{a_0,b_0}}$, where $S_{T_{a_0,b_0},a_0,b_0}^{\phi} \times S_{T_{a_0,b_0},a_0,b_0}^{\psi}$ is endowed with the topology inherited from $Y_{T_{a_0,b_0}} \times Y_{T_{a_0,b_0}}$.
Suppose that both pairs, $\Lambda ( \widetilde{\phi}_i,\widetilde{\psi}_i ) = ( \phi_i,\psi_i ) , \ i \in \{ 1,2 \}$,
are defined on $[0,T_{a_0,b_0}]$. As above (see the  uniqueness part), we set
\[
( \Phi,\Psi,M_\phi,M_\psi ) = ( \phi_1,\psi_1,\mu_{\phi_1},\mu_{\psi_1} ) - ( \phi_2,\psi_2,\mu_{\phi_2},\mu_{\psi_2} ),
\]
which satisfies
\begin{numcases}{}
    \partial_{t} \Phi + \bm{u} \cdot \nabla \Phi + \sigma_1 \overline{\Phi} = \mathrm{div}\, \big( m_\phi (\widetilde{\phi}_1) \nabla M_\phi  \big) + \mathrm{div}\, \big(  \big(  m_\phi( \widetilde{\phi}_1 ) - m_\phi( \widetilde{\phi}_2 )  \big) \nabla \mu_{\phi_2}  \big), \label{eq:cvcho1schauderdifference} & \\
    M_\phi = \alpha \partial_{t} \Phi -\Delta \Phi + \sigma_2 \mathcal{N} \big( \Phi - \overline{\Phi} \big) + F_\phi'(\phi_1) - F_\phi'(\phi_2)
    \nonumber \\
    \qquad \quad + \partial_\phi G(\phi_1,\psi_1) - \partial_\phi G(\phi_2,\psi_2) \label{eq:cvcho2schauderdifference},\\
    \partial_{t} \Psi + \bm{u} \cdot \nabla \Psi = \mathrm{div}\, \big( m_\psi (\widetilde{\psi}_1) \nabla  M_\psi \big) + \mathrm{div}\, \big(  \big(  m_\psi( \widetilde{\psi}_1 ) - m_\psi( \widetilde{\psi}_2 )  \big) \nabla \mu_{\psi_2}  \big) \label{eq:cvcho3schauderdifference},\\
    M_\psi = \alpha \partial_{t} \Psi -\beta \Delta \Psi + F_\psi'(\psi_1) - F_\psi'(\psi_2) + \partial_\psi G(\phi_1,\psi_1) - \partial_\psi G(\phi_2,\psi_2) \label{eq:cvcho4schauderdifference},
\end{numcases}
almost everywhere in $\Omega \times (0,T_{a_0,b_0})$, subject to the following boundary and initial conditions:
\begin{numcases}{}
    \partial_{\mathbf{n}} \Phi = \partial_{\mathbf{n}} \Psi = 0, \label{eq:cvchobschauderdifference}& a.e. on $\partial \Omega \times (0,T_{a_0,b_0})$, \\
    \Phi|_{t=0} = 0, \ \Psi|_{t=0} = 0, \label{eq:cvchoischauderdifference} & in $\Omega$.
\end{numcases}
Since $\overline{\Phi}=\overline{\Psi}=0$, we test \eqref{eq:cvcho1schauderdifference} and \eqref{eq:cvcho3schauderdifference} with $g_\phi = \mathcal{N}_{\phi, \widetilde{\phi}_1} \Phi$ and $g_\psi = \mathcal{N}_{\psi, \widetilde{\psi}_1} \Psi$, respectively, using the monotonicity of $F_\phi'$, $F'_\psi$, we have
\begin{align*}
    & \frac{1}{2} \frac{\mathrm{d}}{\mathrm{d t}}  \Big(  \Big\|  \sqrt{m_\phi( \widetilde{\phi}_1 )} \nabla g_\phi  \Big\|^2 +  \Big\|  \sqrt{m_\psi( \widetilde{\psi}_1 )} \nabla g_\psi  \Big\|^2 + \alpha \|  \Phi  \|^2 + \alpha \|  \Psi  \|^2  \Big)  \\
    & \qquad + \big\| \nabla \Phi \big\|^2 + \beta \big\|  \nabla \Psi  \big\|^2 + \sigma_2 \big\|  \Phi - \overline{\Phi}  \big\|^2_* \\
    & \quad \le \| \bm{u} \|_{L^{3}(\Omega)} \Big(  \|  \Phi  \|_{L^{6}(\Omega)} \|  \nabla g_\phi  \| + \|  \Psi  \|_{L^{6}(\Omega)} \|  \nabla g_\psi  \|  \Big)   + C \| \partial_{t} \widetilde{\phi}_1 \| \| \nabla g_\phi \|^2_{L^{4}(\Omega)}
    \\
    & \qquad + C \| \partial_{t} \widetilde{\psi}_1 \| \| \nabla g_\psi \|^2_{L^{4}(\Omega)}  + \| m_\phi ( \widetilde{\phi}_1 )-m_\phi ( \widetilde{\phi}_2 ) \|_{L^{\infty}(\Omega)} \| \nabla \mu_{\phi_2} \| \| \nabla g_\phi \|
    \\
    & \qquad + \| m_\psi ( \widetilde{\psi}_1 )-m_\psi ( \widetilde{\psi}_2 ) \|_{L^{\infty}(\Omega)} \| \nabla \mu_{\psi_2} \| \| \nabla g_\psi \|
    + C \big( \| \Phi \|^2 + \| \Psi \|^2 \big)
    \\
    & \quad \le C \| \bm{u} \|_{L^{3}(\Omega)}^2 \big(  \| \Phi \|^2 + \| \Psi \|^2  \big) + \frac{1}{2} \| \nabla \Phi \|^2 + \frac{\beta}{2} \| \nabla \Psi \|^2 \\
    & \qquad + C \big( 1+b_0 ( 1+a_0 )^4 \big) \big(  \| \Phi \|^2 + \| \Psi \|^2  \big)
     + \| \nabla \mu_{\phi_2} \|^2 \| \Phi \|^2
     + \| \nabla \mu_{\psi_2} \|^2 \| \Psi \|^2
    \\
    & \qquad + C \| \widetilde{\phi}_1 - \widetilde{\phi}_2 \|^2_{C( \overline{\Omega} )} + C \| \widetilde{\psi}_1 - \widetilde{\psi}_2 \|^2_{C( \overline{\Omega} )}.
\end{align*}
Hence, we find that $( \Phi,\Psi ) \to 0$ in $C\big( [0,T_{a_0,b_0}];H \big)$ as $( \widetilde{\phi}_1,\widetilde{\psi}_1 ) \to ( \widetilde{\phi}_2,\widetilde{\psi}_2 )$ in $Y_{T_{a_0,b_0}}$.
Thanks to the complex interpolation
$H^{\frac{15}{8}}(\Omega) = \big[H,H^{2}(\Omega)\big]_{\frac{15}{16}}$
and the embedding
$H^{\frac{15}{8}}(\Omega) \hookrightarrow C \big( \overline{\Omega} \big)$,
it follows from $(\phi_1,\psi_1), (\phi_2,\psi_2)\in S^{\phi}_{a_0,b_0,T_{a_0,b_0}} \times S^{\psi}_{a_0,b_0,T_{a_0,b_0}}$ that
$( \Phi,\Psi ) \to 0$ in $Y_{T_{a_0,b_0}}$.
As a consequence, we can apply Schauder's fixed-point theorem to establish the existence of a local-in-time strong solution to problem \eqref{eq:cvcho1}--\eqref{eq:cvchoi} with the regularity listed in (\textbf{R1}).

In addition, if
\[
u_p = \| \partial_{t} \bm{u} \|_{L^{p} ( 0,1;L^{3}(\Omega) )} <\infty,\quad
\text{for some} \ p >\frac{4}{3},
\]
then it holds
\[
h_{\bm{u}}(t) \le | \Omega |^{\frac{8}{9}} u_p^{\frac{4}{3}} t^{1-\frac{4}{3p}}, \quad \forall\, t\in [0,1].
\]
Consequently, we can take $T_{a_0,b_0}= \min\{ 1, T_1, T_2 \} $, where $T_1=T_1(u_{L_3},u_p)\,,T_2=T_2(u_{L_3},u_p)>0$ are defined by
\begin{align}
    & C_{H_2}\big[ ( 1+a_0 )^{\frac{14}{3}} | \Omega |^{\frac{8}{9}} u_p^{\frac{4}{3}} T_2^{1-\frac{4}{3p}} + b_0 (1+u_{L_3})^3 ( 1+a_0 )^3 T_{2} + ( 1+u_{L_3} )  \sigma_1 T_{2}  \big]= 1, \label{T_2}
    \\
    & C_{H_1} ( 1+a_0 )^{\frac{14}{3}} | \Omega |^{\frac{8}{9}} u_p^{\frac{4}{3}} T_1^{1-\frac{4}{3p}} \notag \\
    &\quad + \big[ 1 + C_{H_1} u_{L_3} ( 1+a_0 )^2 + C_{H_1} b_0 ( 1+a_0 )^4 + C_{H_1} b_0 ( 1+u_{L_3} ) ( 1+a_0 )^3 \big]T_{1}=1.
    \label{T_1}
\end{align}
\begin{remark} \label{length_cvcho} \rm
   The above argument implies that if $\partial_t \bm{u} \in L^p(I;\bm{L}^3_\sigma(\Omega))$ for some $p >4/3$, then there is a positive lower bound for the existence time $T_r=T_{a_0,b_0}$ depending on $\| \bm{u} \|_{W^{1,p}(I;L^3(\Omega))}$, but not on the specific choice of $\bm{u}$.
\end{remark}
In the two-dimensional case, for $g_\phi = \mathcal{N}_{\phi, \phi} ( \partial_{t} \phi - \partial_{t} \overline{\phi} )$ and $g_\psi = \mathcal{N}_{\psi, \psi} \partial_{t} \psi$, we refine the higher-order estimate as follows:
\begin{align}
    & \frac{1}{2} \frac{\mathrm{d}}{\mathrm{d t}}  \Big(  \Big\|  \sqrt{m_\phi( \phi )} \nabla g_\phi  \Big\|^2 +  \Big\|  \sqrt{m_\psi( \psi )} \nabla g_\psi  \Big\|^2 + \alpha \big\|  \partial_{t} \phi  - \partial_{t} \overline{\phi}  \big\|^2 + \alpha \big\|  \partial_{t} \psi  \big\|^2 -2 \AbsNormal{\Omega} F_\phi'( \phi ) \partial_{t} \overline{\phi}   \Big)
    \nonumber \\
    & \qquad + \big\|  \nabla \partial_{t}  \phi  \big\|^2 + \beta \big\|  \nabla \partial_{t} \psi  \big\|^2 + \sigma_2 \big\|  \partial_{t} \phi - \partial_{t} \overline{\phi}   \big\|^2_*
    \nonumber \\
    & \quad = \int_{\Omega}  \bm{u} \cdot \big(  \big(  \partial_{t} \phi - \partial_{t} \overline{\phi}   \big) \nabla g_\phi + \partial_{t} \psi \nabla g_\phi  \big)  \ \mathrm{d}x - \int_{\Omega}  \partial_{t} \bm{u} \cdot \big(  \phi \nabla g_\phi + \psi \nabla g_\psi  \big)  \, \mathrm{d}x
    \nonumber \\
    & \qquad - \frac{1}{2} \int_{\Omega} \big( m_\phi'( \phi ) \partial_{t} \phi | \nabla g_\phi |^2 + m_\psi'( \psi ) \partial_{t} \psi | \nabla g_\psi |^2 \big) \, \mathrm{d}x
    \nonumber \\
    & \qquad - \int_{\Omega} \big( m_\phi'( \phi ) \partial_{t} \phi \nabla \mu_\phi \cdot \nabla g_\phi + m_\psi'( \psi ) \partial_{t} \psi  \nabla \mu_\psi \cdot \nabla g_\psi \big) \, \mathrm{d}x \nonumber \\
    & \qquad - \big(  \partial_{t}  F_\phi'\big( \phi \big) + \partial_{t}  \partial_\phi G(\phi,\psi) ,  \partial_{t}  \phi \big)
     - \big(  \partial_{t}  F_\psi'\big( \psi \big) + \partial_{t}  \partial_\psi G(\phi,\psi) ,  \partial_{t}  \psi  \big) \nonumber \\
    & \qquad + \AbsNormal{\Omega} \overline{\partial_{t}  \partial_\phi G(\phi,\psi)} \partial_{t}  \overline{\phi} - \AbsNormal{\Omega} \overline{F_\psi'\big( \psi \big)} \partial_{t}^2 \overline{\phi} \nonumber \\
    & \quad \le \| \bm{u} \|_{L^{3}(\Omega)} \big(  \| \partial_{t} \phi - \partial_{t} \overline{\phi} \|_{L^{6}(\Omega)} \| \nabla g_\phi \| + \| \partial_{t} \psi \|_{L^{6}(\Omega)} \| \nabla g_\psi \|  \big) \nonumber \\
    & \qquad + \| \partial_{t} \bm{u} \|_{L^{1}(\Omega)} (  \|  \nabla g_\phi  \|_{L^{\infty}(\Omega)} + \| \nabla g_\psi \|_{L^{\infty}(\Omega)}  ) \nonumber \\
    & \qquad + C \big(  \| \partial_{t} \phi \| \| \nabla g_\phi \|^2_{L^{4}(\Omega)} + \| \partial_{t} \psi \| \| \nabla g_\psi \|^2_{L^{4}(\Omega)}  \big) \nonumber \\
    & \qquad + C \big(  \| \partial_{t} \phi \| \| \nabla \mu_\phi \|_{L^{4}(\Omega)} \| \nabla g_\phi \|_{L^{4}(\Omega)} + \| \partial_{t} \psi \| \| \nabla \mu_\psi \|_{L^{4}(\Omega)} \| \nabla g_\psi \|_{L^{4}(\Omega)}  \big) \nonumber \\
    & \qquad + C (  \| \partial_{t} \phi - \partial_{t} \overline{\phi} \|^2 + \| \partial_{t} \psi \|^2 + \sigma_1 e^{-\sigma_1 t}  )  + \sigma_1^2 \AbsNormal{\Omega} e^{-\sigma_1 t} \| F_\phi'( \phi ) \|_{L^{1}(\Omega)}. \label{higer_order_estimate_2d}
\end{align}
Using Lemma \ref{elliptic_estimate_2d}, we can obtain the following estimates:
\begin{align*}
    & \| \bm{u} \|_{L^{3}(\Omega)} \|  \partial_{t} \phi - \partial_{t} \overline{\phi}  \|_{L^{6}(\Omega)} \| \nabla g_\phi \|
    \\
    & \quad \le C \| \bm{u} \|_{L^{3}(\Omega)} \| \nabla g_\phi \| \| \partial_{t} \phi - \partial_{t} \overline{\phi} \|^{\frac{1}{3}} \| \partial_{t} \phi - \partial_{t} \overline{\phi} \|_V^{\frac{2}{3}}
    \\
    & \quad \le \epsilon \| \nabla \partial_{t} \phi \|^2 + C_\epsilon \| \bm{u} \|_{L^{3}(\Omega)}^{\frac{3}{2}} \| \nabla g_\phi \|^{\frac{3}{2}} \| \partial_{t} \phi - \partial_{t} \overline{\phi} \|^{\frac{1}{2}},
    \\
    & \| \partial_{t} \bm{u} \|_{L^{1}(\Omega)} \| \nabla g_\phi \|_{L^{\infty}(\Omega)}
     \le C \| \partial_{t} \bm{u} \|_{\bm{L}^{1}(\Omega)} \| \nabla g_\phi \|^{\frac{1}{2}} \| \nabla g_\phi \|^{\frac{1}{2}}_{H^{2}(\Omega)} \\
    & \quad \le C \| \partial_{t} \bm{u} \|_{L^{1}(\Omega)} \| \nabla g_\phi \|^{\frac{1}{2}} \Big(  \| \nabla \phi \|^{\frac{1}{3}} \| \nabla \phi \|^{\frac{2}{3}}_{W^{1,4}(\Omega)} \| \nabla g_\phi \| + \| \nabla \phi \|^{\frac{2}{3}} \| \nabla \phi \|^{\frac{4}{3}}_{W^{1,4}(\Omega)} \| \nabla g_\phi \|
    \\
    & \qquad + \| \nabla \phi \|_{W^{1,4}(\Omega)} \| \nabla g_\phi \|_{L^{4}(\Omega)} + \| \nabla \phi \|^{\frac{1}{3}} \| \nabla \phi \|^{\frac{2}{3}}_{W^{1,4}(\Omega)} \| \nabla g_\phi \|_V
    \\
    & \qquad + (1 + \| \nabla \phi \|^{\frac{1}{3}} \| \nabla \phi \|^{\frac{2}{3}}_{W^{1,4}(\Omega)}) \| \partial_{t} \phi - \partial_{t} \overline{\phi} \| + \| \nabla \partial_{t} \phi \|  \Big)^{\frac{1}{2}}
    \\
    & \quad \le C_\epsilon (  1 + \|  \partial_{t} \bm{u}  \|_{L^1(\Omega)}^{\frac{4}{3}}  ) (  1 + \| \partial_{t} \phi - \partial_{t} \overline{\phi} \|^2  ) + \epsilon \| \nabla \partial_{t} \phi \|^2,
    \\
    & \| \partial_{t} \phi \| \| \nabla g_\phi \|_{L^{4}(\Omega)}^2
     \le C \| \partial_{t} \phi \| \| \nabla g_\phi \| \| \nabla g_\phi \|_V  \\
     &\quad \le C \| \nabla g_\phi \| (  \| \partial_{t} \phi - \partial_{t} \overline{\phi} \|^2 + \sigma_1 e^{-\sigma_1 t}  ) + C \| \nabla g_\phi \|^2  (  1 + \alpha \|  \partial_{t} \phi  \|^2  ),
    \\
    & \| \partial_{t} \phi \| \| \nabla \mu_\phi \|_{L^{4}(\Omega)} \| \nabla g_\phi \|_{L^{4}(\Omega)}
    \\
    & \quad \le C \| \partial_{t} \phi \| \| \nabla \mu_\phi \|^\frac{1}{2} \| \nabla \mu_\phi \|_V^\frac{1}{2} \| \nabla g_\phi \|^\frac{1}{2} \| \nabla g_\phi \|_V^\frac{1}{2}
    \\
    & \quad \le C ( 1 + \| \partial_{t} \phi - \partial_{t} \overline{\phi} \|  ) ( \| \bm{u} \| + \| \nabla g_\phi \|  )^\frac{1}{2} (  1 + \| \partial_{t} \phi - \partial_{t} \overline{\phi} \| + \| \partial_{t} \phi - \partial_{t} \overline{\phi} \| \| \nabla g_\phi \|  )^\frac{1}{2} \| \nabla g_\phi \|^\frac{1}{2}
    \\
    & \qquad (  ( 1 + \| \partial_{t} \phi - \partial_{t} \overline{\phi} \| ) \| \nabla g_\phi \| + \| \partial_{t} \phi - \partial_{t} \overline{\phi} \|  )^\frac{1}{2} \\
    & \quad \le C ( 1 + \| \partial_{t} \phi - \partial_{t} \overline{\phi} \|^4 ),
\end{align*}
where $\epsilon \in (0,1)$ is an arbitrary constant.
Similar estimates hold for terms involving $\psi$.
Therefore, applying Gronwall's lemma to inequality \eqref{higer_order_estimate_2d} on $I$ if it is finite, or applying Gronwall's lemma on $[0,1]$ and the uniform Gronwall lemma on $[1,\infty)$ if $I=[0,\infty)$,
we can establish uniform estimates for strong solutions in the two-dimensional case. This enables us to obtain a global strong solution with the regularity properties listed in (\textbf{R1}).
\medskip

\noindent \textbf{Regularity 2.} As in the proof of (\textbf{R1}), we rewrite equation \eqref{eq:cvcho2} in the following form:
\[
\alpha \partial_{t} \phi -\Delta \phi + F_\phi (  \phi  ) = \mu_\phi - \partial_\phi G ( \phi,\psi ).
\]
Then, repeating the arguments therein, we can prove the strict separation property for $\phi$ on any interval $ [0,T_r]$ whenever it holds $\mu_\phi \in L^{\infty} ( 0,T_r;W )$. An analogous argument holds for $\psi$.
Then, from \eqref{eq:cvcho2}, \eqref{eq:cvcho4} and the elliptic regularity theory, it follows that
$
\phi,\psi \in L^{2}_{\mathrm{uloc}} ( 0,T_r;H^{3}(\Omega) )$.

The proof of Proposition \ref{cvchosse} is complete.
\qed

\section*{Declarations}
\noindent \textbf{Conflict of interest.} The authors have no competing interests to declare that are relevant to the content of this article.
\smallskip

\noindent \textbf{Funding.}
Ouyang was supported by the China Scholarship Council (CSC) Program (Grant No. 202506100096).
Wu was supported by the Natural Science Foundation of Shanghai 25ZR1401023.
\smallskip

\noindent \textbf{Acknowledgments.}
Grasselli is a member of Gruppo Nazionale per l'Ana\-li\-si Matematica, la Probabilit\`{a} e le loro Applicazioni (GNAMPA), Istituto Nazionale di Alta Matematica (INdAM). His research is part of the activities of ``Dipartimento di Eccellenza 2023--2027'' of Politecnico di Milano.
Wu is a member of the Key Laboratory of Mathematics for Nonlinear Sciences (Fudan University), Ministry of Education of China.
\smallskip

\noindent \textbf{Data availability.}
Data sharing not applicable to this article as no datasets were generated or analyzed during the current study.

\end{document}